%% file: main.tex
\documentclass[11pt]{amsart}
\usepackage{amsmath}
\usepackage{amssymb}
\usepackage{amsthm}
\usepackage{cite}
\usepackage{mathtools}
\usepackage{embrac}
\usepackage{bbm} 
\usepackage[all]{xy}
\usepackage{extarrows}
\usepackage{graphicx}
\usepackage{multirow}
\usepackage{float}
\usepackage{caption}
\usepackage{subcaption}
\usepackage{enumitem}
\usepackage{verbatim}
\usepackage{xcolor}
\usepackage{leftindex}
\usepackage{fancyhdr}
\usepackage{setspace}
\usepackage[hidelinks]{hyperref}
\hypersetup{pdfauthor={Yang Chen and Haoran Wang}}

\theoremstyle{plain}
\newtheorem{theorem}{Theorem}[section]
\newtheorem{theoA}{Theorem}

\newtheorem{corollary}[theorem]{Corollary}
\newtheorem{proposition}[theorem]{Proposition}
\newtheorem{lemma}[theorem]{Lemma}

\theoremstyle{definition}
\newtheorem{definition}[theorem]{Definition}

\theoremstyle{remark}
\newtheorem{remark}[theorem]{Remark}

\def\ZZ{{\mathbb Z}}

\def\QQ{{\mathbb Q}}

\def\FF{{\mathbb F}}

\def\1{{\mathbbm 1}}

\newcommand{\op}{\operatorname}

\newcommand{\wh}{\widehat}
\newcommand{\mf}{\mathfrak}
\newcommand{\mc}{\mathcal}
\newcommand{\mr}{\mathrm}

\newcommand{\vv}{\operatorname{val}_{p}}

\newcommand{\tb}{\textbf}
\newcommand{\mar}[1]{\marginpar{\tiny #1}}
\newcommand{\plim}{\varprojlim}

\newcommand{\brho}{\overline{\rho}}

\newcommand{\quash}[1]{}

\title[Homological branching laws for GL2]{Homological branching laws for $p$-adic and mod $p$ representations of ${\rm GL}_2$ over a $p$-adic field}

\author{Yang Chen}
\address{School of Mathematical Sciences,
Peking University, Beijing 100871}
\email{chenyang2601@163.com}

\author{Haoran Wang}
\address{Academy for Multidisciplinary Studies,
Capital Normal University, Beijing 100048}
\email{haoran@cnu.edu.cn}

\date{} 

\begin{document}

\begin{abstract}
Let $F$ be a finite extension of $\QQ_p$ and $p\geq 5$. In this paper, we establish homological branching laws for mod $p$ representations of $\mathrm{GL}_2(F)$ restricted to split and nonsplit maximal tori, treating principal and special series and, when $F=\mathbb Q_p$, supersingular representations. We determine the derived Hom-complexes in both orders for nonexceptional characters. At exceptional characters, we calculate their Euler--Poincar\'e characteristics under a $p$-torsion-freeness hypothesis. We also establish corresponding results for a class of unitary Banach space representations.

\end{abstract}

\maketitle

\setcounter{tocdepth}{1}
\tableofcontents


\input{chapter_1}
\input{chapter_2}

\input{chapter_3}

\input{chapter_4}

\input{chapter_5}

\bibliographystyle{amsalpha}
\bibliography{ref}

\end{document}

%% file: chapter_1.tex
\section{Introduction}

\subsection{Classical and homological branching laws}

\noindent\medskip

A fundamental question in representation theory concerns branching laws. Given a pair of groups $(G,H)$, where $H$ is a subgroup of $G$, the branching laws describe the behavior of representations of $G$ under
restriction to $H$. For complex representations of $p$-adic reductive groups over a finite extension $F$ of $\QQ_p$,
branching multiplicities give the local representation-theoretic interpretation of automorphic periods and connect them to the arithmetic of $L$-functions.

The classical theorem of Tunnell \cite{Tunnell} and Saito \cite{Saito} addressed the toric periods in their local representation-theoretic form. Let $G=\mr{GL}_2(F)$ and let $H$ be a maximal torus of $G$. Concretely, $H$ is either the nonsplit torus $L^\times$ (where $L$ is a quadratic extension of $F$) or the diagonal (split) torus $T$. They proved a multiplicity-one theorem, namely, 
 \[
 \dim_{\mathbb{C}} \mr{Hom}_H(\pi|_H,\eta) \leq 1,
\]
for every infinite-dimensional irreducible smooth representation $\pi$ of $G$ and every smooth character $\eta$ of $H$, and they gave a criterion for nonvanishing in terms of local $\varepsilon$-factors.

The Tunnell--Saito theorem is a part of a broader family of multiplicity-one results within the Gan–Gross–Prasad program \cite{GGP}. Aizenbud--Gourevitch--Rallis--Schiffmann \cite{AGRS} and Sun--Zhu \cite{SZ} proved multiplicity-one theorems for the general linear and classical pairs $(G,H)$ occurring in this framework. Sakellaridis--Venkatesh \cite{SV} placed periods on spherical varieties in a broader framework and formulated the relative Langlands program.

Seminal work on $p$-adic representations  (e.g. \cite{BL94} \cite{ST-JAMS} \cite{Bre03} \cite{AHHV-JAMS}) and the $p$-adic Langlands program (e.g. \cite{Colmez} \cite{Breuil-ICM} \cite{Paskunas-IHES} \cite{Breuil-Paskunas} \cite{CDP} \cite{CEGGPS1} \cite{BHHMS2}), together with its significant arithmetic achievements (\!\!\cite{Kisin}\cite{Emerton}), motivate the corresponding branching questions with mod $p$ and $p$-adic coefficients. A question of Harris led  to pioneering results by Morra \cite{Morra} and Dospinescu \cite{Dospinescu} in this direction. Assume for the moment that $F = \QQ_p$ and $H = L^{\times}$, a nonsplit torus. In characteristic $p$, Morra \cite{Morra} proved that every infinite-dimensional absolutely irreducible admissible smooth mod $p$ representation $\pi$ of $\op{GL}_2(\QQ_p)$ satisfies 
\[
\operatorname{Hom}_{L^\times}(\pi|_{L^\times}, \eta) = 0,
\]
for any smooth character $\eta$ of $ L^\times$, and computed the higher Ext-groups in both argument orders after restricting to $\mc{O}_L^{\times}$. Dospinescu \cite[Th\'eor\`eme 6.3.1]{Dospinescu} proved an analogous Hom-space vanishing for the locally analytic vectors of nonordinary Banach space representations of $\op{GL}_2(\QQ_p)$. Thus the classical multiplicity may vanish uniformly while the restriction remains homologically nontrivial.

These results motivate the study of branching laws through
derived Hom-complexes rather than degree-zero Hom-spaces alone.
Our formulation treats both split and nonsplit maximal tori
and considers the following objects in the category with
fixed central character:
\[
\operatorname{RHom}_{H,\zeta}(\pi|_H,\eta),
\qquad
\operatorname{RHom}_{H,\zeta}(\eta,\pi|_H).
\]
For nonexceptional characters, we determine both the degree
and the dimension of the unique nonzero cohomology group
of each complex. For exceptional characters, assuming that
$H/Z$ has no $p$-torsion, we establish cohomological
boundedness and compute the Euler--Poincar\'e characteristics. We also establish corresponding results
for the unitary Banach representations considered below.

The study of $\op{RHom}$ complexes and Euler--Poincar\'e characteristics, in the context of complex representations, lies at the heart of the theory of homological branching laws --- a program systematically developed by Prasad \cite{Prasad-ICM,Prasad-IHES}. For complex representations of the pair $(\mr{GL}_{n+1}(F),\mr{GL}_n(F))$, Chan--Savin \cite{Chan-Savin}  proved higher-Ext vanishing for generic pairs in the quotient order; Chan \cite{Chan} determined the degree-zero branching law in the opposite order.

We now describe our results in more detail; they characterize the homological branching laws for mod $p$ and $p$-adic representations of $\op{GL}_2$ upon restriction to maximal tori.

\subsection{The main theorems}

\noindent\medskip

Let $p\geq 5$ be a prime.  Let $F$ be a finite extension of $\QQ_p$ of degree $d$. Let $G = \mathrm{GL}_2(F)$. Let  $H$ be either the diagonal torus $T$ or $L^{\times}$, where $L$ is a quadratic field extension of $F$. Let $Z$ denote the center of $G$.

\subsubsection{Homological branching for mod $p$ representations}

\noindent\medskip

Let $\FF$ be a sufficiently large finite extension of $\FF_p$. Let $\pi$ be an infinite-dimensional absolutely irreducible admissible smooth representation of $G$ defined over $\FF$. If $\pi$ is supersingular, we assume $F= \QQ_p$. Let $\eta: H\to \FF^{\times}$ be a smooth character. 

In characteristic $p$, the idea of working in the derived category of {\em unbounded} complexes of smooth representations was initiated in \cite{Schneider}. We regard $\pi|_H$ and $\eta$ as objects in the unbounded derived category $\op{D}(H)$. The derived multiplicity complexes
\[
\op{RHom}_{H} (\pi|_{H},\eta), \qquad \op{RHom}_{H} (\eta, \pi|_{H})
\]
are objects of $\op{D}(\FF)$, the derived category of $\FF$-vector spaces. If $\eta|_{Z} \neq \omega_{\pi}$, where $\omega_{\pi}$ is the central character of $\pi$, then it is straightforward to see that there are quasi-isomorphisms
\[
\op{RHom}_{H} (\pi|_{H},\eta) \simeq  \op{RHom}_{H} (\eta, \pi|_{H}) \simeq 0.
\]
Hence, we assume $\eta|_{Z} = \omega_{\pi}$ from now on. We write 
\[
\zeta: = \omega_{\pi} = \eta|_{Z},
\]
and study the derived multiplicity complexes
\[
\op{RHom}_{H,\zeta} (\pi|_{H},\eta), \qquad \op{RHom}_{H,\zeta} (\eta, \pi|_{H}),
\]
where $\op{RHom}_{H,\zeta}$ means the derived Hom-complex in the category of smooth $\FF$-representations of $H$ on which $Z$ acts by $\zeta$. Fixing the action of $Z$ removes the extraneous self-cohomology contributed by $Z$. This observation is also essential for passing from smooth mod $p$ representations to unitary Banach representations.

We prove that $\op{RHom}_{H,\zeta} (\pi|_{H}, \eta) $ and $\op{RHom}_{H,\zeta} (\eta, \pi|_{H}) $ are in fact objects in $\op{D}_{\rm fd}(\FF)$, the full subcategory of $\op{D} (\FF)$ consisting of objects with finite-dimensional cohomology in each degree. For a cohomologically bounded object $C  \in \op{D}_{\rm fd}(\FF)$, we define the Euler--Poincar\'e characteristic
\[
\chi(C) : =  \sum_{i \in \ZZ} (-1)^i \dim_{\FF} H^i (C). 
\]
To state our results, we set
\[
\delta_{H} = \left\{
    \begin{array}{cl}
        0, & \text{ if $H = L^{\times}$} ,\\
        1, & \text{ if $H = T$,}
    \end{array}\right. 
\]
and set the multiplicity 
\[
m_{\pi} = \left\{
    \begin{array}{cl}
        1, & \text{ if $\pi$ is not supersingular} ,\\
        2, & \text{ if $\pi$ is supersingular (and $F = \QQ_p$).}
    \end{array}\right. 
\]
\begin{definition} \label{def::intro_1}
Let $\eta: H \to \FF^{\times}$ be a smooth character such that $\eta|_{F^{\times}} =\zeta$.  We say $\eta$ is {\em exceptional for $\pi$} if one of the following holds:
\begin{itemize}
\item[(i)] $H = L^{\times}$ and the pair $(\pi ,\eta)$ is of the form 
\[
(\op{St}\otimes\chi\circ\op{det}, \chi\circ  \op{N}_{L/F}),
\]
where $\chi$ is some smooth character of $F^{\times}$ and $\op{N}_{L/F}: L^{\times}\to F^{\times}$ is the norm map.
\item[(ii)] $H = T$ and the pair $(\pi ,\eta)$ lies in the set 
\[
\left\{ (\op{Ind}_B^G\chi_1\boxtimes \chi_2, \chi_1\boxtimes \chi_2) ,(\op{Ind}_B^G\chi_1\boxtimes \chi_2, \chi_2\boxtimes \chi_1),(\op{St}\otimes \chi\circ\op{det}, \chi\boxtimes \chi) \right\},
\]
where $\chi_1,\chi_2, \chi$ are smooth characters of $F^{\times}$, and $\chi_1\neq \chi_2$.
\end{itemize}
Otherwise, we say $\eta$ is {\em nonexceptional for $\pi$}. In particular, if $\pi$ is a supersingular representation of $\op{GL}_2(\QQ_p)$, every smooth character $\eta: H\to \FF^{\times}$ satisfying $\eta|_{Z} = \zeta$ is nonexceptional for $\pi$.
\end{definition}

\begin{theoA} \label{thm::intro_A}
Let $\eta: H \to \FF^{\times}$ be a smooth character such that $\eta|_{F^{\times}} =\zeta$.
\begin{itemize}  
\item[(i)] Assume $\eta$ is nonexceptional for $\pi$. Then the complexes $\op{RHom}_{H,\zeta} (\pi|_{H}, \eta) $ and $\op{RHom}_{H,\zeta} (\eta, \pi|_{H}) $ are each cohomologically concentrated in a single degree. More precisely, we have quasi-isomorphisms in $\op{D}(\FF)$
\begin{align*}
\op{RHom}_{H,\zeta} (\pi|_{H}, \eta) &  \simeq \FF^{\oplus m_{\pi}} [-d],\\
 \op{RHom}_{H, \zeta} (\eta, \pi|_{H}) & \simeq  \FF^{\oplus m_{\pi}} [-\delta_H].
\end{align*}

\item[(ii)] Assume $\eta$ is exceptional for $\pi$. We further assume $H/Z$ has no $p$-torsion. Then the complexes $\op{RHom}_{H ,\zeta} (\pi|_{H} , \eta) $ and $\op{RHom}_{H,\zeta} (\eta, \pi|_{H}) $ are cohomologically bounded. Moreover, we have
\begin{align*}
\chi ( \op{RHom}_{H,\zeta} (\pi|_{H}, \eta)) & =(-1)^d m_{\pi}, \\
  \chi ( \op{RHom}_{H, \zeta} (\eta, \pi|_{H})) & = (-1)^{\delta_H} m_{\pi}.
\end{align*}
\end{itemize}
\end{theoA}

\begin{remark}\label{intro_rk_1}
\begin{itemize}
\item[(i)] When $F\ne\mathbb Q_p$, supersingular representations of ${\rm GL}_2(F)$ are not finitely presented (\!\!\cite{Schraen-Crelle}\cite{Wu-MAMA}). Because finite presentation is essential to our argument, we treat supersingular representations only when $F = \QQ_p$.

\item[(ii)] In fact, we establish Theorem  \ref{thm::intro_A}(i) for all $p$ when $\pi$ is a principal series or special series representation. When $\pi$ is a supersingular representation of $\op{GL}_2(\QQ_p)$, we prove the results for $p\geq 3$ when $H = T$, and also when $H = L^{\times}$ with $L$ unramified over $\QQ_p$.

\item[(iii)] Theorem \ref{thm::intro_A} is used to establish the homological branching laws for mod $p$ representations of ${\rm SL}_2$ in the forthcoming work \cite{LWW}.
\end{itemize}
\end{remark}

For $F=\mathbb Q_p$, the mod $p$ local Langlands correspondence
gives a uniform formulation of these results. Recall that it
associates a finite-length admissible smooth
$\mathbb F$-representation $\pi(\brho)$ of
$\mathrm{GL}_2(\mathbb Q_p)$ to a continuous two-dimensional
Galois representation
\[
\brho:
\operatorname{Gal}(\overline{\mathbb Q}_p/\mathbb Q_p)
\longrightarrow \mathrm{GL}_2(\mathbb F).
\]
Let $\zeta$ denote the central character of $\pi(\brho)$,
and let $\eta:H\to\mathbb F^\times$ be a smooth character
such that $\eta|_Z=\zeta$.

Combining Theorem \ref{thm::intro_A} with the description
of the irreducible constituents of $\pi(\brho)$,
we obtain the following corollary, using the notion of
nonexceptionality defined below.

\begin{definition}
We say that $\eta$ is  {\em nonexceptional for $\pi(\brho)$} if it is nonexceptional, in the sense of Definition \ref{def::intro_1}, for every infinite-dimensional irreducible constituent of $\pi(\brho)$. Otherwise we say $\eta$ is {\em exceptional for $\pi(\brho)$}.
\end{definition}

\begin{corollary}  \label{cor::intro}
\noindent
\begin{itemize}  
\item[(i)] Assume $\eta$ is nonexceptional for $\pi(\brho)$. Then the complexes $\op{RHom}_{H,\zeta} (\pi(\brho)|_{H}, \eta) $ and $\op{RHom}_{H,\zeta} (\eta, \pi(\brho)|_{H}) $ are each cohomologically concentrated in a single degree. More precisely, we have quasi-isomorphisms  in $\op{D}(\FF)$
\begin{align*}
\op{RHom}_{H,\zeta} (\pi(\brho)|_{H}, \eta) &  \simeq \FF^{\oplus 2} [-1] \\
 \op{RHom}_{H, \zeta} (\eta, \pi(\brho)|_{H}) & \simeq  \FF^{\oplus 2}[-\delta_H].
\end{align*}

\item[(ii)] Assume $\eta$ is exceptional for $\pi(\brho)$. Then the complexes $\op{RHom}_{H,\zeta} (\pi(\brho)|_{H}, \eta) $ and $\op{RHom}_{H,\zeta} (\eta, \pi(\brho)|_{H}) $ are cohomologically bounded. Moreover, we have
\begin{align*}
\chi ( \op{RHom}_{H,\zeta} (\pi(\brho)|_{H}, \eta)) & = -2\\
  \chi ( \op{RHom}_{H, \zeta} (\eta, \pi(\brho)|_{H})) & = 2 (-1)^{\delta_H} .
\end{align*}
\end{itemize}
\end{corollary}
Thus, on the representations supplied by the mod $p$ local Langlands correspondence, the absolute value of the 
Euler--Poincar\'e characteristic is uniformly $2$.

\subsubsection{Homological branching for $p$-adic representations}
\noindent\medskip

Let $E$ be a sufficiently large finite extension of $\QQ_p$ with residue field $\FF$. Let $\mc{O}$ be the ring of integers of $E$, and let $\varpi$ be a uniformizer of $\mc{O}$. Let $\Pi$ be an infinite-dimensional absolutely irreducible admissible unitary $E$-Banach space representation of $G = \op{GL}_2(F) $, with a $G$-invariant $\mc{O}$-lattice $\Pi^0$. When $F$ is not $\QQ_p$, we further assume $\Pi$ is {\em ordinary} in the sense that it is a subquotient of a unitary parabolic induction of a unitary character. In the following, let $\zeta$ denote the central character of $\Pi$, which always exists by \cite{Dospinescu-Schraen}. Let $\eta : H\to E^{\times}$ be a unitary character such that $\eta|_{Z} = \zeta$. Let $\eta^{0}$ be an $H$-invariant $\mc{O}$-lattice of $\eta$.

\begin{definition} 
We say $\eta$ is {\em nonexceptional for $\Pi$} if $\eta^0/\varpi$ is nonexceptional (in the sense of Definition \ref{def::intro_1}) for every infinite-dimensional irreducible constituent of $(\Pi^0/\varpi \Pi^0)^{\rm ss}$.
\end{definition}

 Set
\[
m_{\Pi} = \left\{
    \begin{array}{cl}
        1, & \text{ if $\Pi$ is ordinary} ,\\
        2, & \text{ if $\Pi$ is nonordinary (and $F = \QQ_p$).}
    \end{array}\right. 
\]

\begin{theoA} \label{thm::intro_C}
Assume $\eta$ is nonexceptional for $\Pi$. Then
\[
\dim_E \big( \plim_n {\rm Ext}_{\op{Mod}^{\rm sm}_{H,\zeta}(\mc{O}/\varpi^n)}^i ( \Pi^0/ \varpi^{n} \Pi^0,  \eta^0/\varpi^n  ) \big)[1/p]  = \left\{\begin{array}{cl}
        m_{\Pi} ,& \text{ if } i = d,\\
        0,& \text{ otherwise},
        \end{array}\right. 
\]
and 
\[
\dim_E  \big(\plim_n {\rm Ext}_{\op{Mod}^{\rm sm}_{H,\zeta}(\mc{O}/\varpi^n)}^i( \eta^0/\varpi^n  , \Pi^0/ \varpi^{n} \Pi^0)\big)[1/p]   = \left\{\begin{array}{cl}
        m_{\Pi} ,& \text{ if } i = \delta_H,\\
        0,& \text{ otherwise}.
        \end{array}\right.
\]
\end{theoA}

\begin{remark}
When $H= L^{\times}$, Theorem \ref{thm::intro_A} implies that $\Pi|_{L^{\times}}$ is admissible (see Lemma \ref{lemma::adm_L}). We prove in Theorem \ref{thm::Bananch_nonsplit_tori} that there are natural isomorphisms
\[
{\rm Ext}_{{\rm Ban}_{L^{\times},\zeta}^{\rm adm}(E)}^i (\Pi, \eta) \simeq \big( \plim_n {\rm Ext}_{\op{Mod}^{\rm sm}_{L^{\times},\zeta}(\mc{O}/\varpi^n)}^i ( \Pi^0/ \varpi^{n} \Pi^0,  \eta^0/\varpi^n  ) \big)\left[ 1/p \right],
\]
and
\[
{\rm Ext}_{{\rm Ban}_{L^{\times},\zeta}^{\rm adm}(E)}^i (\eta, \Pi) \simeq \big( \plim_n {\rm Ext}_{\op{Mod}^{\rm sm}_{L^{\times},\zeta}(\mc{O}/\varpi^n)}^i ( \eta^0/\varpi^n , \Pi^0/ \varpi^{n} \Pi^0 ) \big)[1/p],
\]
where ${\rm Ext}_{{\rm Ban}_{L^{\times},\zeta}^{\rm adm}(E)}^i$ denotes the $i$-th Ext-group computed in the abelian category of admissible $E$-Banach space representations on which $Z$ acts by $\zeta$.
\end{remark}

\subsection{Outline of the proofs}
We now sketch the proofs of Theorems \ref{thm::intro_A} and \ref{thm::intro_C}.

\subsubsection{Proof sketch for Theorem \ref{thm::intro_A}} 

In the proof of Theorem \ref{thm::intro_A}(i) for principal series and special series representations, the main ingredients are Frobenius reciprocity and Heyer’s computation of the right adjoint to compact induction in the case needed here. We only sketch the proof for 
\[
\pi = \op{Ind}_{B}^{G} \chi_1\boxtimes \chi_2, ~\chi_1\neq \chi_2.
\]
The proof for special series representations is similar.

In the case $H= L^{\times}$, since $G = B L^{\times}$, we have an isomorphism 
\[
 \pi |_{L^\times}\cong \op{c-Ind}_{F^\times}^{L^\times}\chi_1\chi_2. 
\]
After twisting by $\eta^{-1}$, both representations have trivial central character, and we obtain quasi-isomorphisms
\begin{align*}
\op{RHom}_{L^{\times},\zeta} (\op{c-Ind}_{F^\times}^{L^\times}\chi_1\chi_2 , \eta) & \cong \op{RHom}_{L^{\times}/F^{\times}} (\op{c-Ind}_{\{1\}}^{L^\times/F^{\times}}\FF, \FF) \\
\op{RHom}_{L^{\times},\zeta} (\eta , \op{c-Ind}_{F^\times}^{L^\times}\chi_1\chi_2) & \cong \op{RHom}_{L^{\times}/F^{\times}} (\FF , \op{c-Ind}_{\{1\}}^{L^\times/F^{\times}}\FF) .
\end{align*}
Since $L^{\times}/F^{\times}$ is compact, the compact induction functor $\op{c-Ind}_{\{1\}}^{L^\times / F^{\times}}: \op{D}(\FF) \to \op{D}(L^{\times}/F^{\times})$ is the same as the usual induction functor $\op{Ind}_{\{1\}}^{L^\times / F^{\times}}$, and it admits a left adjoint given by the restriction functor $\op{Res}_{\{1\}}^{L^\times/F^\times}$; this gives the second quasi-isomorphism in Theorem \ref{thm::intro_A}(i). On the other hand, we can show that the functor $\op{c-Ind}_{\{1\}}^{L^\times / F^{\times}}$ also admits a right adjoint, denoted by $\mathcal{R}_{\{1\}}^{L^{\times} / F^{\times}}$. The first quasi-isomorphism in Theorem \ref{thm::intro_A}(i) is an immediate consequence of the following lemma, which follows from \cite{Hey_gl}.

\begin{lemma}
We have a quasi-isomorphism in ${\rm D}(\FF)$,
\[
\mathcal{R}_{\{1\}}^{L^{\times} / F^{\times}} \FF \simeq \FF[-d].
\]
\end{lemma}

In the case $H = T$, the proof is similar. However, since $T/Z\simeq F^{\times}$ is not compact and has a valuation direction, the complex $\op{RHom}_{T ,\zeta}(\eta, \pi|_{T})$ is cohomologically concentrated in degree one if $\eta$ is nonexceptional for $\pi$. When $\eta$ is exceptional for $\pi$, assuming $T/Z$ is $p$-torsion-free,  we show $ \op{RHom}_{T ,\zeta}(\eta, \eta) \in \op{D^b_{fd}(\FF)}$ and 
\[
\chi (\op{RHom}_{T ,\zeta}(\eta, \eta) ) = 0.
\]
This implies Theorem \ref{thm::intro_A}(ii) for $\pi$.

\bigskip

The study of supersingular representations is the technical
heart of the paper. The proofs for ramified and unramified
quadratic extensions of $\mathbb Q_p$ follow the same strategy.
We sketch the proof of
\[
\operatorname{RHom}_{L^\times,\zeta}
(\pi|_{L^\times},\eta)
\simeq \mathbb F^{\oplus 2}[-1]
\]
when $\pi$ is supersingular and $L/\mathbb Q_p$ is ramified.

The starting point is Breuil's presentation of supersingular
representations \cite{Bre03}. Up to twist, $\pi$ fits into
a short exact sequence
\[
0\longrightarrow
\operatorname{ind}_{KZ}^{G}\sigma
\xrightarrow{\;\Phi\;}
\operatorname{ind}_{KZ}^{G}\sigma
\longrightarrow \pi\longrightarrow 0,
\]
where $\sigma$ is a Serre weight, extended to $KZ$, and
$\Phi$ is the standard Hecke operator. Mackey decomposition
gives
\[
\left(\operatorname{ind}_{KZ}^{G}\sigma\right)|_{L^\times}
\simeq
\bigoplus_{k\geq 0}
\operatorname{ind}_{L_k^\times}^{L^\times}\sigma_k,
\]
where the subgroups $L_k^\times$ and representations
$\sigma_k$ are described explicitly in the main text.

Applying $\operatorname{Ext}_{L^\times,\zeta}^*(-,\eta)$
to the short exact sequence reduces the problem to computing
the maps $\Phi_i^*$ induced by $\Phi$ on the Ext-groups
of the restricted compact induction. These Ext-groups are
products of the corresponding Ext-groups of the Mackey
summands and vanish outside degrees zero and one; see Proposition \ref{prop::Ext^i_L_ramified}.

The key calculation is the action induced by the Hecke operator on these Ext-groups, including the maps between the Mackey summands. We carry out these calculations in Lemma \ref{lemma::Phi_L_ram},
Proposition \ref{prop::Phi_k_L_ramified}, and
Lemma \ref{lemma::Phi_Ext^i_L_ram_1}. They give
\[
\ker\Phi_0^*=\operatorname{coker}\Phi_1^*=0,
\qquad
\dim_{\mathbb F}\operatorname{coker}\Phi_0^*
=
\dim_{\mathbb F}\ker\Phi_1^*=1.
\]
The resulting long exact sequence therefore yields
\[
\operatorname{Ext}_{L^\times,\zeta}^i
(\pi|_{L^\times},\eta)
\simeq
\begin{cases}
\mathbb F^{\oplus 2}, & i=1,\\
0, & i\neq 1,
\end{cases}
\]
which proves the desired quasi-isomorphism.

For the split torus, we first work with
$T_0:=T(\mathbb Z_p)Z$, whose quotient by $Z$ is compact.
We then pass from $T_0$ to $T$ using
\[
1\longrightarrow T_0/Z
\longrightarrow T/Z
\longrightarrow T/T_0
\longrightarrow 1,
\qquad T/T_0\simeq\mathbb Z.
\]
This introduces an additional group-cohomological calculation
for the valuation direction.

\subsubsection{Proof sketch for Theorem \ref{thm::intro_C}}

We sketch the proof for $H=L^\times$. The key step is
to lift the mod $p$ calculation to a compatible system
of finite free modules over $\mathcal O/\varpi^n$.
Starting from Theorem \ref{thm::intro_A}(i), we establish
cohomological concentration and length formulas modulo
$\varpi^n$, together with compatibility under reduction.
More precisely, set
\[
M_n :=
\operatorname{Ext}^{d}_{\operatorname{Mod}^{\rm sm}_{L^\times,\zeta}
(\mathcal O/\varpi^n)}
\bigl(\Pi^0/\varpi^n\Pi^0,\eta^0/\varpi^n\eta^0\bigr).
\]
The Ext-groups in the other degrees vanish, and
\[
\operatorname{length}_{\mathcal O}(M_n)=n m_\Pi.
\]
Moreover, the transition maps induce isomorphisms
\[
M_n\otimes_{\mathcal O}\mathcal O/\varpi^j
\xrightarrow{\;\sim\;} M_j
\qquad (1\leq j\leq n).
\]

In particular, $M_n/\varpi M_n\simeq M_1$ has dimension
$m_\Pi$ over $\mathbb F$. Nakayama's lemma therefore gives
a surjection
\[
(\mathcal O/\varpi^n)^{\oplus m_\Pi}\twoheadrightarrow M_n.
\]
Since both modules have length $n m_\Pi$, this surjection
is an isomorphism. Thus each $M_n$ is free of rank $m_\Pi$
over $\mathcal O/\varpi^n$. Choosing bases compatible with
the transition maps gives
\[
\varprojlim_n M_n\simeq \mathcal O^{\oplus m_\Pi}.
\]
Inverting $p$ yields the first dimension formula in
Theorem \ref{thm::intro_C}.

For the opposite order, the Ext-groups are
concentrated in degree zero, and the same argument applies to
\[
N_n :=
\operatorname{Hom}_{L^\times}
\bigl(\eta^0/\varpi^n\eta^0,\Pi^0/\varpi^n\Pi^0\bigr).
\]
The case $H=T$ follows the same strategy, with degree zero
in the opposite order replaced by degree
$\delta_H=1$.

\subsection{Organization of the paper}

In \S \ref{sec::preliminary}, we recall preliminaries concerning the derived category of mod $p$ representations of $p$-adic groups and the classification of mod $p$ representations of $\op{GL}_2(F)$. 

In \S \ref{sec::nonsplit_tori} and \S \ref{sec::split_torus}, we prove Theorem \ref{thm::intro_A}. When $H = L^{\times}$, it is obtained by combining Theorems \ref{nonsplitPS}, \ref{thm::St_nonsplit_nongeneric}, \ref{thm::branching_ramified} and \ref{thm::branching_unr}. When $H = T$, it is obtained by combining Theorems \ref{thm::PS_T}, \ref{thm::sp_T}, \ref{thm::ps_T_EP}, \ref{thm::sp_T_EP} and \ref{thm::branching_ss_T}. 

In \S\ref{sec::mod_p_LLC}, we recall the mod $p$ local Langlands correspondence for $\op{GL}_2(\QQ_p)$. In \S\ref{sec::branching_pi(brho)}, we prove Corollary \ref{cor::intro}.

In \S\ref{sec::banach}, we study $p$-adic Banach space representations. Theorem \ref{thm::intro_C} is proved in Theorem \ref{thm::Bananch_nonsplit_tori_1} when $H = L^{\times}$ and in Theorem \ref{thm::Bananch_split_tori} when $H = T$.

\subsection{AI declaration}

During the preparation of this paper, the authors used ChatGPT and Gemini to assist with language editing and proofreading. They used ChatGPT 5.6 and Gemini 3 to help complete the proof of Proposition \ref{inj-to-inj} and to identify Remark \ref{remark::p-torsion-free}.

\subsection{Acknowledgements}

H.W. thanks Jean-Fran\c cois Dat and Yongquan Hu for their encouragement and constant support. He also thanks Li Cai, Yangyu Fan and Chong Zhang for asking him many years ago about branching laws for $p$-adic representations. We thank Pierre Colmez, Yiwen Ding, Gabriel Dospinescu, Yongquan Hu, Stefano Morra, Dipendra Prasad, Zicheng Qian, Zhixiang Wu and Liang Xiao for helpful discussions, suggestions and their feedback.

The main results of this paper were announced at the 10th International Congress of Chinese Mathematicians (ICCM) in January 2026. We thank the organizing committee for the invitation and for inviting us to write the expository paper  \cite{CW-ICCM}. 

This work is supported by the National Key R\&D Program of China 2023YFA1009702, the National Natural Science Foundation of China (Grant No. 12371011), and the Beijing Natural Science Foundation (Grant No. 1262003).

\subsection{Notation} \label{sec::notation}

Let $E$ be a finite extension of $\QQ_p$ with ring of integers $\mathcal O$, uniformizer $\varpi$, and residue field $\FF$.
Let $(A,\mathfrak{m})$ be a complete local Noetherian $\mc{O}$-algebra with residue field $\FF$. Let $G$ be a locally profinite group. Let ${\rm Mod}^{\rm sm}_G(A)$ denote the category of smooth representations of $G$ over $A$ in the sense of \cite[Def. 2.2.1]{Emerton-Ord1}. In particular, this definition requires that every vector in a smooth representation is annihilated by some power of $\mf{m}$. Let ${\rm Mod}^{\rm pro}_G(A)$ denote the category of profinite augmented $G$-representations over $A$ in the sense of \cite[Def. 2.1.6]{Emerton-Ord1}. Its objects are profinite linearly topological $A[[K]]$-modules, where $K$ is some (equivalently, any) compact open subgroup of $G$, with an action of $A[G]$ such that the two actions are the same when restricted to $A[K]$, and its morphisms are $G$-equivariant continuous homomorphisms of topological $A[[K]]$-modules. 

For $V \in {\rm Mod}^{\rm sm}_G(A)$, let $V^{\vee} : = {\rm Hom}_{\mathcal{O}}(V , E/\mathcal{O})$ denote the Pontryagin dual of $V$, equipped with its natural profinite topology. Then $V^{\vee} \in {\rm Mod}^{\rm pro}_G(A)$. Conversely, for $M \in {\rm Mod}^{\rm pro}_G(A)$, let $M^{\vee} : = \op{Hom}_{\mathcal{O}}^{\mathrm{cont}}(M, E/\mathcal{O})$, where $E/\mathcal{O}$ carries the discrete topology and $M^{\vee}$ is equipped with the compact-open topology. Pontryagin duality induces an anti-equivalence of categories between ${\rm Mod}^{\rm sm}_G(A)$ and ${\rm Mod}^{\rm pro}_G(A)$; see \cite[Lemma 2.2.7]{Emerton-Ord1}.

If $G$ is profinite, then ${\rm Mod}^{\rm pro}_G(A)$ is the same as the category of compact $A[[G]]$-modules with continuous homomorphisms; we let ${\rm Mod}^{\rm pro.fg}_G(A)$ denote the  full subcategory of ${\rm Mod}^{\rm pro}_G(A)$ consisting of finitely generated  $A[[G]]$-modules.

Let $Z$ be a subgroup of the center $Z_G$ of $G$. Let $\zeta:Z\to A^{\times}$ be a smooth character. We denote by ${\rm Mod}^{\rm sm}_{G,\zeta}(A)$ the subcategory of ${\rm Mod}^{\rm sm}_G(A)$ consisting of smooth representations on which $Z$ acts by $\zeta$. Let ${\rm Mod}^{\rm pro}_{G,\zeta}(A)$ denote the subcategory of ${\rm Mod}^{\rm pro}_G(A)$ which is anti-equivalent to ${\rm Mod}^{\rm sm}_{G,\zeta}(A)$ under Pontryagin duality.

 If $G$ is a $p$-adic Lie group, we denote by $\operatorname{dim}(G)$ its dimension as a $\QQ_p$-analytic manifold. 

%% file: chapter_2.tex
\section{Preliminaries} \label{sec::preliminary}
\subsection{The derived category}
Let $G$ be a locally profinite group. We denote by $\mr{C}(G)$ the category of cochain complexes of $\op{Mod}^{\mr{sm}}_{G}(\FF)$, and denote by $\mr{K}(G)$ its homotopy category. Then the category $\mr{K}(G)$ carries the structure of a triangulated category with the translation functor $[1]$. We localize $\mr{K}(G)$ with respect to the class of quasi-isomorphisms and obtain $\mr{D}(G)$, the derived category of {\em unbounded} complexes in $\op{Mod}^{\mr{sm}}_{G}(\FF)$. 
We denote by $\textbf{q}:\mr{K}(G)\to \mr{D}(G)$ the localization functor. 

A complex $I \in \mr{C}(G)$ is called {\em homotopically injective} if $\op{Hom}_{\mr{K}(G)}(X,I)=0$ for all acyclic complexes $X$. Recall that the abelian category $\op{Mod}^{\mr{sm}}_{G}(\FF)$ is Grothendieck abelian by \cite[Lem. 1]{Schneider} and \cite[Lem. 3.2]{Vigneras}.  
\begin{lemma}\label{equiv}
     \begin{enumerate}
        \item[(i)] If $I\in \mr{K}(G)$ is homotopically injective, then $\op{Hom}_{\mr{K}(G)}(X,I) \cong \op{Hom}_{\mr{D}(G)}(X,I)$ for all $X\in \mr{K}(G)$. 
        \item[(ii)] For any complex $X$ in $\mr{K}(G)$, there exists a quasi-isomorphism $X\to I$ such that $I$ is homotopically injective. 
    \end{enumerate}
\end{lemma}
\begin{proof}
    This lemma follows from \cite[Thm. 14.3.1(i)(ii)]{KS06}. 
\end{proof}
Let $\mr{K}_{\mr{hi}}(G)$ denote the full subcategory of $\mr{K}(G)$ consisting of homotopically injective objects and let $\iota: \mr{K}_{\mr{hi}}(G)\to \mr{K}(G)$ denote the inclusion functor. 
It follows from Lemma \ref{equiv} that the restriction of $\textbf{q}$ to the subcategory $\mr{K}_{\mr{hi}}(G)$ is fully faithful and essentially surjective. 
\begin{corollary}
    \begin{enumerate}
        \item [(i)] The localization functor $\textbf{q}:\mr{K}(G)\to \mr{D}(G)$ induces an equivalence of categories between $\mr{K}_{\mr{hi}}(G)$ and $\mr{D}(G)$. 
        \item [(ii)]The localization functor $\textbf{q}:\mr{K}(G)\to \mr{D}(G)$ admits a right adjoint $\textbf{i}:\mr{D}(G)\to \mr{K}(G)$, which is the composition of $\iota$ and a chosen quasi-inverse of $\textbf{q}\circ\iota$. 
    \end{enumerate}
\end{corollary}
	
\begin{proof}
    This corollary follows from \cite[Thm. 14.3.1(iii)(v)]{KS06}. 
\end{proof}

It also follows from Lemma \ref{equiv} that right derived functors exist: 
\begin{proposition} 
    Let $\mc A$ be an abelian category and let $F:\op{Mod}^{\mr{sm}}_{G}(\FF)\to \mc A$ be a left exact functor. 
    A right derived functor $\mr RF:\mr{D}(G)\to \mr{D}(\mc A)$ exists and is given by the composition 
    $$\mr RF:\mr{D}(G)\stackrel{\tb{i}}{\to}\mr{K}(G)\stackrel{K(F)}{\longrightarrow}\mr{K}(\mc A)\stackrel{\tb{q}_{\mc A}}{\to}\mr{D}(\mc A),$$
    where $\tb{q}_{\mc A}$ denotes the localization functor and $K(F)$ is the natural extension of $F$ to the homotopy categories. 
\end{proposition}
\begin{proof} 
    The proposition follows from \cite[Thm. 14.3.1(vi)]{KS06}. 
\end{proof}

   If $F:\op{Mod}^{\mr{sm}}_{G}(\FF)\to \mc A$ is exact, we simply write $F:\mr{D}(G)\to \mr{D}(\mc A)$ for the right derived functor $\mr RF$. 

\quash{
The following Brown representability theorem is a special case of \cite[Cor. 2.3.10]{Hey_la}. 
\begin{proposition}[Brown representability] 
    Let $\mc D$ be a triangulated category and let $F:\mr{D}(G)\to \mc D$ be a triangulated functor. Assume $F$ commutes with small direct sums. Then $F$ admits a triangulated right adjoint. 
\end{proposition}
}

The category $\mr{D}(G)$ carries the structure of a closed symmetric monoidal category such that the functors $\_\otimes_{\FF}X$ are triangulated for any $X\in\mr{D}(G)$\cite{SS23}. The tensor unit is $\mathbf{1}: = \FF[0]$. By \cite[Prop. 2.3.11]{Hey_la}, the functor $\_\otimes_{\FF}X$ admits a right adjoint denoted by $\op{hom}(X,\_):\mr{D}(G)\to \mr{D}(G)$. See \cite[Rem. 2.3.12]{Hey_la} for an explicit description of the functor $\op{hom}(X,\_)$.

The category $\mr{D}(G)$ is naturally enriched over $\mr{D}(\FF) : = \mr{D}(\{1\})$, the unbounded derived category of the category of $\FF$-vector spaces. The derived Hom-complex is a bifunctor
\[
\mr{RHom}_{G}(\_,\_) := \mr{RHom}_{\op{Mod}^{\mr{sm}}_{G}(\FF)}(\_,\_) : \mr{D}(G)^{\mr{op}} \times \mr{D}(G) \to \mr{D}(\FF),
\]
which is a triangulated functor in each variable. For $X,Y\in \op{D}(G)$, if $\mr{RHom}_{G}(X,Y) $ lies in $\op{D_{fd}^b}(\FF)$, the full subcategory of $\op{D} (\FF)$ consisting of objects with finite-dimensional bounded cohomology, we define the Euler--Poincar\'e characteristic
\[
\chi(\mr{RHom}_{G}(X,Y) ) : =  \sum_{i\in \ZZ} (-1)^i \dim_{\FF} H^i (\mr{RHom}_{G}(X,Y) ). 
\]

\begin{definition}
    Let $i\in\ZZ$. Let $X$, $Y$ be objects of $\mr{D}(G)$. The \emph{$i$-th extension group} of $X$ by $Y$ is the group  
    $$\op{Ext}^i(X,Y):= \op{Hom}_{\mr{D}(G)}(X,Y[i])\cong \op{Hom}_{\mr{D}(G)}(X[-i],Y). $$
\end{definition}

For $X,Y\in \mr{D}(G)$, we have an isomorphism
$$\op{Ext}^i(X,Y)\cong  H^i (\mr{RHom}_{G}(X,Y)). $$
If $\pi_1$, $\pi_2 \in\op{Mod}^{\mr{sm}}_{G}(\FF)$, the extension group $\op{Ext}^i(\pi_1[0],\pi_2[0]) $ vanishes if $i<0$ and if $i\geq 0$,
\[
\op{Ext}^i(\pi_1[0],\pi_2[0]) \cong \op{Ext}^i_{G}(\pi_1,\pi_2),
\]
where $\op{Ext}^i_{G}(\pi_1,\pi_2)$ denotes the usual extension group computed in $\op{Mod}^{\mr{sm}}_{G}(\FF)$.

\quash{
\begin{lemma}
    Let $X$, $Y$ be objects of $\op{Mod}_H^{\mathrm{sm}}(\FF)$. Let $Y\to I^\bullet$ be an injective resolution. Then 
    $$\op{Ext}^i_{H}(X,Y)\cong\op{Hom}_{\mr{K}(H)}(X[0],I^\bullet[i]). $$
\end{lemma}
\begin{proof}
    By definition, $\op{Ext}^i_{H}(X,Y):=\op{Hom}_{\mr{D}(H)}(X,Y[i])$. Since $Y\to I^\bullet$ is an injective resolution, shifting gives $Y[i]\to I^\bullet[i]$ which is a quasi-isomorphism. 
    Hence in the derived category, we have $Y[i]\cong I^\bullet[i]$. By Lemma \ref{equiv}, one obtains 
    $$\op{Hom}_{\mr{D}(H)}(X[0],Y[i])\cong\op{Hom}_{\mr{D}(H)}(X[0],I^\bullet[i])=\op{Hom}_{\mr{K}(H)}(X[0],I^\bullet[i]). $$
\end{proof}
}

\begin{lemma} \label{lemma::Borel--Wallach}
Let $\pi_1,\pi_2$ be two representations of $G$. Assume that there exists a central element $z\in G$ such that the action of $z$ on $\pi_1$ and $\pi_2$ is given by multiplication by distinct scalars. Then
\[
\op{Ext}^i_{G} (\pi_1,\pi_2) = 0,~ \forall i.
\] 
\end{lemma}
\begin{proof}
This is \cite[Lem. 4.2.4]{Hey_gl}, which refers to \cite[I, \S4]{Borel-Wallach}.
\end{proof}

Let $H$ be a closed normal subgroup of $G$. Let $\zeta:H\to \FF^{\times}$ be a fixed smooth character. For our purposes, we assume there is a smooth character $\eta : G \to \FF^{\times}$ such that $\eta|_H = \zeta$. Let ${\rm Mod}^{\rm sm}_{G,\zeta}(\FF) $ denote the abelian category of smooth representations of $G$ over $\FF$ with fixed action of $H$ given by $\zeta$. It is straightforward to see that the functor $\pi \mapsto \pi\otimes\eta^{-1}$ induces an equivalence of categories between  ${\rm Mod}^{\rm sm}_{G,\zeta}(\FF) $ and ${\rm Mod}^{\rm sm}_{G/H}(\FF) $. The category ${\rm Mod}^{\rm sm}_{G,\zeta}(\FF) $ is then Grothendieck abelian. We may replace the abelian category ${\rm Mod}^{\rm sm}_{G}(\FF) $ by ${\rm Mod}^{\rm sm}_{G,\zeta}(\FF) $ in the above discussion.  The associated unbounded derived category will be denoted by $\op{D}_{\zeta}(G)$. We have the derived Hom-complex 
\[
\mr{RHom}_{G,\zeta}(\_,\_) := \mr{RHom}_{\op{Mod}^{\mr{sm}}_{G,\zeta}(\FF)}(\_,\_) : \mr{D}_{\zeta}(G)^{\mr{op}} \times \mr{D}_{\zeta}(G) \to \mr{D}(\FF).
\]

If $\pi_1,\pi_2 \in {\rm Mod}^{\rm sm}_{G,\zeta}(\FF) $, let $\op{Ext}^{i}_{G,\zeta}(\pi_1,\pi_2) $ denote the $i$-th extension group computed ${\rm Mod}^{\rm sm}_{G,\zeta}(\FF) $. We have an isomorphism
\begin{equation} \label{eq::isom_ext_twist}
\op{Ext}^{i}_{G,\zeta}(\pi_1,\pi_2)\cong \op{Ext}^i_{G/H}(\pi_1 \otimes \eta^{-1},\pi_2 \otimes \eta^{-1}),~ \forall i.
\end{equation}

\subsection{Functors}
Let $G$ be a locally profinite group and let $H$ be a closed subgroup of $G$. For $\pi\in\op{Mod}^{\mr{sm}}_{H}(\FF)$, define the smooth induction $\op{Ind}_{H}^{G} \pi$ by
\begin{equation*}
    \op{Ind}_{H}^{G}\pi:=\left\{f:G\to \pi \left|\begin{aligned}
        &f(hg)=\pi(h) f(g)\text{ for all }h \in H,\ g\in G ~\text{and}\\ 
        &\exists K\subseteq G \text{ open subgroup with }f(gx)=f(g) \text{ for all }g\in G, x\in K
        \end{aligned}\right.\right\}.
\end{equation*}
The group $G$ acts smoothly on $\op{Ind}_{H}^{G}\pi$ by right translation.

Frobenius reciprocity shows that the functor $\op{Ind}_{H}^{G}:\op{Mod}^{\mr{sm}}_{H}(\FF)\to \op{Mod}^{\mr{sm}}_{G}(\FF)$ is right adjoint to the restriction functor $\op{Res}_{H}^{G}:  \op{Mod}^{\mr{sm}}_{G}(\FF) \to  \op{Mod}^{\mr{sm}}_{H}(\FF)$. As a right adjoint, the functor $\op{Ind}_{H}^{G}$ is left exact.

Let $\op{c-Ind}_{H}^G \pi$ denote the subspace of $\op{Ind}_{H}^G \pi$ consisting of the functions $f$ such that the image of the support of $f$ in the Hausdorff space $H\backslash G$ is compact. The space $\op{c-Ind}_{H}^G \pi$ is stable under the action of $G$ and it is a smooth representation of $G$. This gives the compact induction functor $\op{c-Ind}_{H}^G:\op{Mod}^{\mr{sm}}_{H}(\FF)\to \op{Mod}^{\mr{sm}}_{G}(\FF)$. 

\begin{lemma}\label{sec}
    The quotient map $G\to H\backslash G$ has a continuous section. 
\end{lemma}
\begin{proof}
    This is \cite[Lemma 2.3]{AHV19}. 
\end{proof}

\begin{lemma} \label{lemma::compact_ind_exact}
    The compact induction functor $\op{c-Ind}^G_{H}: \op{Mod}^{\mr{sm}}_{H}(\FF)\to \op{Mod}^{\mr{sm}}_{G}(\FF)$ is exact. 
\end{lemma}
\begin{proof}
    By Lemma \ref{sec} and the proof of \cite[Prop. 5]{Her12}, we have a functorial isomorphism of $\FF$-vector spaces 
    $$\op{c-Ind}^G_{H}V\stackrel{\sim}{\longrightarrow} C^{\infty}_{c}(H\backslash G,V),$$
    where $C^{\infty}_{c}(H\backslash G,V)$ denotes the space of compactly supported locally constant functions $H\backslash G\to V$. Then the claim follows from the exactness of $C^{\infty}_{c}(H\backslash G,-)$. 
\end{proof}

\quash{
\begin{lemma} \label{lemma::compact_ind_exact}
    The compact induction functor $\op{c-Ind}^G_{H}: \op{Mod}^{\mr{sm}}_{H}(\FF)\to \op{Mod}^{\mr{sm}}_{G}(\FF)$ is exact. 
\end{lemma}
\begin{proof}
It is straightforward to check that the functor $\op{c-Ind}_{H}^G$ is left exact. One then checks that the functor $\op{c-Ind}_{H}^G$ commutes with small direct sums. It then follows from \cite[Prop. 8.3.27]{KS06} that $\op{c-Ind}_{H}^G$ admits a right adjoint. Consequently, it is right exact. 
\end{proof}
}

It is straightforward to check that the functor $\op{c-Ind}_{H}^G$ commutes with small direct sums. Then by \cite[Prop. 8.3.27]{KS06}, $\op{c-Ind}_{H}^G$ admits a right adjoint, denoted by
\[
R_H^G : \op{Mod}^{\mr{sm}}_{G}(\FF) \to \op{Mod}^{\mr{sm}}_{H}(\FF).
\]
As a right adjoint, the functor $R_H^G$ is left exact.

If $H\subseteq G$ is an open subgroup, we use $\op{ind}_{H}^{G}\pi$ to denote $\op{c-Ind}^G_{H} \pi$. For $v\in \pi$ and $g\in G$ we denote by $[g,v]$ the unique element of $\op{ind}_{H}^{G}\pi$ whose support is $H g^{-1}$ and $[g,v](g^{-1}) = v$. We have $g'\cdot[g,v]=[g' g,v],\ [gh,v]=[g,\pi(h)v]$ for all $g,g'\in G$, $h\in H$. In this case $R_H^G$ is the restriction functor $ \op{Res}_{H}^{G}$.

The exact functor $\op{c-Ind}_H^G$ extends naturally to a triangulated functor on the derived categories
\[
\op{c-Ind}_H^G:\mr{D}(H)\to \mr{D}(G).
\] 

\begin{lemma} \label{lemma::proj_formula}
Let $H$ be a closed subgroup of $G$. Then there is an isomorphism in $\mr{D}(G)$ which is natural in $X \in \mr{D}(H)$ and $Y \in \mr{D}(G)$:
\[
 \op{c-Ind}_H^G (X \otimes_{\FF} \op{Res}^G_H Y) \cong \op{c-Ind}^G_H(X) \otimes_{\FF} Y.
\]
\end{lemma}
\begin{proof}
This is the projection formula proved in \cite[Lemma 2.2.3]{Hey_gl}.
\end{proof}

Following \cite{Hey_gl}, let 
\[
\mc R^G_H : \mr{D}(G)\to \mr{D}(H) 
\]
denote the right derived functor $\mr{R} R_H^G$ of $R_H^G$. The following lemma is taken from \cite{Hey_gl}.

\begin{lemma}\label{lemma::rightadjoint}
    Let $G$ be a $p$-adic Lie group and let $H$ be a closed subgroup of $G$. Let $\eta: G \to \FF^\times$ be a smooth character. Assume $G$ is abelian. Then 
    $$\mathcal{R}^{G}_{H}(\eta[k])\cong (\op{Res}_{H}^{G}\eta) [k-\op{dim}(G)+\op{dim}(H)],~ \forall k\in \ZZ.$$ 
\end{lemma}
\begin{proof}
Since $\eta$ is a character, the functor $(\eta^{-1})[-k]\otimes_{\FF}\_$ is right and also left adjoint to the functor $\eta[k]\otimes_{\FF}\_$. Note that 
\[
\op{hom}(\eta[k],\mathbf{1}) \cong (\eta^{-1})[-k].
\]
The functor $\op{hom}(\eta[k],\mathbf{1})\otimes_{\FF}\_$ is then right adjoint to $\eta[k]\otimes_{\FF}\_$. Hence $\eta[k]$ is dualizable. By \cite[Lem. 2.3.6]{Hey_gl}, we have 
\[
\mathcal{R}_{H}^{G}(\eta[k])\cong \mathcal{R}_{H}^{G}(\mathbf{1})\otimes\op{Res}_{H}^{G}(\eta[k]).
\] 
It then follows from \cite[Prop. 2.3.9]{Hey_gl} that 
\[
\mathcal{R}_{H}^{G}(\mathbf{1}) \cong \omega_{H} \otimes_{\FF}  \operatorname{Res}_{H}^{G}\op{hom}( \omega_{G},\mathbf{1}),
\]
where $\omega_G$ and $\omega_H$ are the dualizing characters of $G$ and $H$ respectively, defined in \cite[Def. 2.3.3]{Hey_gl}. By \cite[Prop. 2.3.4]{Hey_gl},
\[
\omega_{G}\cong\mathfrak{d}_{G}[\op{dim}(G)] \text{ and } \omega_H \cong \mathfrak{d}_{H}[\dim(H)],
\]
where $\mathfrak{d}_{G}$ and $\mathfrak{d}_{H}$ are the reduction mod $p$ of the orientation characters given by the adjoint action of $G$ and $H$ on their Lie algebras. Since $G$ and $H$ are assumed abelian, the characters $\mathfrak{d}_{G}$ and $\mathfrak{d}_{H}$ are trivial. The assertion then follows. 
\end{proof}

\begin{remark}
If $G$ is an open subgroup of $\mathbf{G}(F)$ where $\mathbf{G}$ is a connected reductive group over a finite extension $F$ of $\QQ_p$, then \cite[Lem. 2.10]{SS23} implies that $\mathfrak{d}_{G}$ is trivial.
\end{remark}

\quash{
\begin{lemma}\label{sec}
    The quotient map $H\to H'\backslash H$ has a continuous section. 
\end{lemma}
\begin{proof}
    See \cite[Lemma 2.3]{AHV19}. 
\end{proof}
\begin{lemma}
    The compact induction functor $\op{c-Ind}^H_{H'}: \op{Mod}^{\mr{sm}}_{H'}(\FF)\to \op{Mod}^{\mr{sm}}_{H}(\FF)$ is exact. 
\end{lemma}
\begin{proof}
    By Lemma \ref{sec} and the proof of \cite[Proposition 5]{Her12}, one has an isomorphism of $H$-representations 
    $$\op{c-Ind}^H_{H'}V\stackrel{\sim}{\longrightarrow} C^{\infty}_{c}(H'\backslash H,V),$$
    where $C^{\infty}_{c}(H'\backslash H,V)$ denotes the space of compactly supported locally constant functions $H'\backslash H\to V$ with a left action of $H$ coming from the right action on $H'\backslash H\to V$. 
    Then the claim follows from the exactness of $C^{\infty}_{c}(H'\backslash H,-)$. 
\end{proof}
}

\subsection{Filtrations}
Let $G$ be a $p$-adic Lie group. Let $H$ and $H'$ be closed subgroups of $G$. The group ${H'}\times H$ acts on $G$ via $(x,y)\cdot g:= xgy^{-1}$. Let $Z\subseteq G$ be a ${H'}\times H$-invariant subset. For $V\in \op{Mod}^{\rm{sm}}_{H'}(\FF)$, we denote by $\op{c-Ind}_{H'}^{Z} V$ the $\FF$-vector space of locally constant functions $f:Z\to V$ 
which satisfy $f(xz)=xf(z)$ for all $x\in {H'}$, $z\in Z$ and have compact support in ${H'}\backslash Z$. The group $H$ acts smoothly on $\op{c-Ind}_{H'}^Z V$ by right translation. The proof of Lemma \ref{lemma::compact_ind_exact} also shows that the functor $\op{c-Ind}_{H'}^Z: \op{Mod}^{\rm{sm}}_{H'}(\FF)\to \op{Mod}^{\rm{sm}}_{H}(\FF)$ is exact and hence it extends to a triangulated functor $\op{c-Ind}_{H'}^Z: \mathrm{D}({H'})\to \mathrm{D}(H)$ by applying $\op{c-Ind}_{H'}^Z$ termwise to a complex. 

The proofs of the following lemmas are similar to those in \cite[Lem. 3.3.3, 3.2.4]{Hey_gl}, in which $H$ and ${H'}$ are assumed to be parabolic subgroups of a connected reductive group $G$ over a $p$-adic field. 
We repeat the proof verbatim from \cite{Hey_gl} to make the present paper self-contained. We first recall a notation. Given a closed subgroup $H$ of $G$ and $g\in G$, put $g(H):=g{H}g^{-1}$ and denote by $$g_*: \op{Mod}^{\rm{sm}}_{H}(\FF)\to \op{Mod}^{\rm{sm}}_{g(H)}(\FF)$$ the transport of structure along the conjugation map $g(H) \xrightarrow{\cong} H$, 
i.e. for $\rho$ a representation of $H$, $x\in g(H)$, $g_*(\rho)(x):=\rho(g^{-1}xg)$. 
\begin{lemma}\label{lemma:doublecoset_ind}
    For every $n\in G$ such that $H'nH$ is locally closed in $G$, there is a natural isomorphism of functors $\mathrm{D}({H'}) \rightarrow \mathrm{D}(H)$: 
    $$
    \op{c-Ind}_{{H'}}^{{H'}nH}\stackrel{\cong}{\Longrightarrow} \op{c-Ind}_{n^{-1}({H'}) \cap H}^H \operatorname{Res}_{n^{-1}({H'}) \cap H}^{n^{-1}({H'})} (n^{-1})_*.
    $$
\end{lemma}
\begin{proof}
    Since $H$ is closed in $G$, $H$ is an $l$-group. Note that $H$ acts transitively on the $l$-space $n^{-1}({H'}) \backslash n^{-1}({H'}) H$ by right multiplication and $n^{-1}({H'}) \cap H$ is the stabilizer. 
    By \cite[1.6. Corollary]{BZ76}, we have a homeomorphism $$n^{-1}({H'}) \cap H\backslash H \xrightarrow{\cong} n^{-1}({H'}) \backslash n^{-1}({H'}) H$$ induced by the inclusion $H\hookrightarrow n^{-1}({H'}) H$. 
    Also the homeomorphism $n^{-1}({H'}) H \xrightarrow{n \cdot} {H'} n H$ induces a homeomorphism $n^{-1}({H'}) \backslash n^{-1}({H'}) H \xrightarrow{\cong} {H'} \backslash {H'} n H$. Thus for any $V \in \op{Mod}^{\rm{sm}}_{H'}(\FF)$ the natural map 
    $$
    \begin{aligned}
    \op{c-Ind}_{H'}^{{H'} n H} V & \longrightarrow \op{c-Ind}_{n^{-1}({H'}) \cap H}^H \operatorname{Res}_{n^{-1}({H'}) \cap H}^{n^{-1}({H'})} n_*^{-1} V, \\
    f & \longmapsto[q \mapsto f(n q)]
    \end{aligned}
    $$
    is an isomorphism in $\op{Mod}^{\rm{sm}}_{H}(\FF)$. As all functors involved are exact, the assertion follows.
\end{proof}
\begin{lemma}\label{lemma::filtration_ind}
    Let $Z \subseteq Z^{\prime} \subseteq G$ be ${H'} \times H$-invariant locally closed subsets such that $Z$ is open in $Z^{\prime}$. For every smooth ${H'}$-representation $V$ one has an exact sequence 
    $$
    0 \longrightarrow \op{c-Ind}_{H'}^Z V \longrightarrow \op{c-Ind}_{H'}^{Z^{\prime}} V \longrightarrow \op{c-Ind}_{H'}^{Z^{\prime}\backslash Z} V\longrightarrow 0
    $$
    of smooth $H$-representations. Here, the first map is extension by zero and the second map is restriction of functions.   

In particular, if $Z' \backslash Z$ is open in $Z'$, then the above sequence splits.
\end{lemma}
\begin{proof}
    Injectivity of the first map and exactness in the middle are clear. Choose a continuous section of the projection map $G \rightarrow {H'} \backslash G$; for any ${H'}\times H$-invariant subset $Y\subseteq G$, it restricts to a continuous 
    section $\sigma: {H'} \backslash Y \rightarrow Y$. For any $\FF$-vector space $W$, denote by $\mathcal{C}_c^{\infty}({H'} \backslash Y, W)$ the space of locally constant functions ${H'}\backslash Y\to W$ with compact support. We have an $\FF$-linear isomorphism 
    $$\op{c-Ind}_{H'}^Y V \xrightarrow{\cong} \mathcal{C}_c^{\infty}({H'} \backslash Y, V), f \mapsto f \circ \sigma.$$ 
    Put $Z^{\prime \prime}:=Z^{\prime} \backslash Z.$ Under these identifications, it suffices to show the restriction map $\mathcal{C}_c^{\infty}\left({H'} \backslash Z^{\prime}, V\right) \to \mathcal{C}_c^{\infty}\left({H'} \backslash Z^{\prime \prime}, V\right)$ is surjective. 
    
    Let $f\in \mathcal{C}_c^{\infty}\left({H'} \backslash Z^{\prime \prime}, V\right)$. Write $\operatorname{Supp}(f)=\bigsqcup_{i=1}^r \Omega_i^{\prime \prime}$ for some compact open subsets $\Omega_i^{\prime \prime} \subseteq {H'} \backslash Z^{\prime \prime}$ such that $f$ is constant on each $\Omega_i^{\prime \prime}$. 
    Since $Z'$ is locally closed in $G$, ${H'}\backslash Z'$ is an $l$-space and admits
    a basis consisting of compact open subsets. Choose compact open subsets $\Omega_i^{\prime} \subseteq {H'} \backslash Z^{\prime}$ such that $\Omega_i^{\prime} \cap {H'} \backslash Z^{\prime \prime}=\Omega_i^{\prime \prime}$ for all $i$. 
    Replacing each $\Omega_i^{\prime}$ by $\Omega_i^{\prime} \backslash \bigcup_{j \neq i} \Omega_j^{\prime}$ if necessary, we may further assume that the $\Omega_i^{\prime}$ are pairwise disjoint. Extending each $f|_{ \Omega_i^{\prime \prime}}$ constantly to $\Omega_i^{\prime}$ then yields an extension of $f$ in $\mathcal{C}_c^{\infty}\left({H'} \backslash Z^{\prime}, V\right)$.     
\end{proof}

\subsection{Irreducible mod $p$ representations of $\op{GL}_2(F)$} \label{section::sm_rep_GL_2}

Let $F$ be a finite extension of $\QQ_p$. Let $G:=\op{GL}_2(F)$, $K:=\op{GL}_2(\mc O_F)$ and let $Z$ be the center of $G$. Denote the upper triangular Borel subgroup of $G$ by $B$. We first recall the classification theorem of Barthel--Livn\'{e} \cite[Thm. 34]{BL94}.

\begin{theorem}
    Every irreducible admissible smooth representation of $G=\op{GL}_2(F)$ over $\overline{\FF}_p$ which admits a central character falls into one of the following four families: 
    \begin{enumerate}
        \item[(i)]Smooth characters $\chi\circ \op{det}$, where $\chi:F^\times\to \overline{\FF}_p^\times$ is a smooth character. 
        \item[(ii)]Irreducible principal series representations $\op{Ind}_B^G\chi_1\boxtimes\chi_2$ with $\chi_i: F^{\times} \to \overline{\FF}_p^{\times}$ smooth characters and $\chi_1\neq\chi_2$; the character $\chi_1\boxtimes\chi_2 : B\to \overline{\FF}_p^{\times}$ is defined by sending $\bigl(\begin{smallmatrix}
a & b \\ 0& d
\end{smallmatrix} \bigr)$ to $\chi_1(a) \chi_2(d)$.
        \item[(iii)]Special series representations $\op{St}\otimes (\chi\circ \op{det})$, where $\chi:F^\times\to \overline{\FF}_p^\times$ is a smooth character and $\op{St}$ denotes the Steinberg representation which is given by the short exact sequence
\begin{equation} \label{eq::def_St}
0\to \overline{\FF}_p \to \op{Ind}_{B}^{G} \overline{\FF}_p \to \op{St} \to 0. 
\end{equation}
        \item[(iv)]Supersingular representations in the sense of Barthel--Livn\'{e}. 
    \end{enumerate}
\end{theorem}

When $F=\QQ_p$, all the supersingular representations of $\op{GL}_2(\QQ_p)$ are described explicitly by Breuil \cite{Bre03}. For $0\leq r\leq p-1$, we denote by $\sigma=\op{Sym}^{r}(\overline{\FF}_p^2)$ the subspace of $\overline{\FF}_p[X, Y]$ consisting of homogeneous polynomials of degree $r$. The group $\op{GL}_2(\FF_p)$ acts on $\sigma$ by
    $$\left(\begin{array}{cl}
        a & b \\
        c & d
        \end{array}\right)\cdot f(X, Y ):=f(aX+cY, bX+dY),~ \left(\begin{array}{cl}
            a & b \\
            c & d
            \end{array}\right)\in \op{GL}_2(\FF_p).$$ 
Then $K$ acts on $\sigma$ via the reduction mod $p$ map $K \to \op{GL}_2(\FF_p)$. We view $\sigma$ as a smooth $KZ$-representation by letting $p$ act trivially. 
Let $\mathbf{X}$ (resp. $\mathbf{Y}$) be the $\overline{\FF}_p$-linear projection of $\sigma$ onto the one-dimensional subspace $\overline{\FF}_p X^r$ (resp. $\overline{\FF}_p Y^r$) with respect to the basis $X^r,X^{r-1}Y,\ldots, Y^r$. It follows from \cite[Prop. 8]{BL94} that there exists an endomorphism $\Phi\in \op{End}_{G}(\op{ind}_{KZ}^{G}(\sigma))$ satisfying 
\begin{equation}\label{eq::formula_Phi}
    \Phi([1,v])=\left[\left(\begin{array}{cl}
        p & 0 \\
        0 & 1
        \end{array}\right)^{-1}, \mathbf{Y}(v)\right]+\sum_{i=0}^{p-1}\left[\left(\begin{array}{cl}
            1 & i \\
            0 & p
            \end{array}\right)^{-1}, \mathbf{X}\left(\left(\begin{array}{cl}
                1 & i \\
                0 & 1
                \end{array}\right)\cdot v\right)\right],~ \forall v\in \sigma,
\end{equation}
such that the Hecke algebra $\op{End}_{G}(\op{ind}_{KZ}^{G}(\sigma))$ is isomorphic to $\overline{\FF}_p[\Phi]$.

For $0\leq r\leq p-1$ and $\lambda\in \overline{\FF}_p$, we define the smooth $G$-representation $\pi(r,\lambda)$ by the short exact sequence (see \cite[Thm. 19]{BL94})
    \begin{align} \label{ss}
        0\to\op{ind}_{KZ}^{G}(\sigma)\xrightarrow{\Phi-\lambda}\op{ind}_{KZ}^{G}(\sigma)\to \pi(r,\lambda)\to 0. 
    \end{align}
Let $\omega : Z \cong \QQ_p^{\times} \to \overline{\FF}_p^{\times}$ denote the character sending $a\in \ZZ_p^{\times}$ to $a \pmod{p}$ and $p$ to $1$. It is straightforward to check that the center $Z \cong \QQ_p^{\times}$ acts on $\op{ind}_{KZ}^{G}(\sigma)$ (and hence on $\pi(r,\lambda)$) by the character $\omega^r$. The following theorem is \cite[Thm. 1.1]{Bre03}.

\begin{theorem}  \label{thm::Breuil}
    For $0\leq r\leq p-1$, the representations $\pi(r,0)$ are irreducible. Moreover, every supersingular representation of  $\op{GL}_2(\QQ_p)$ is isomorphic to $\pi(r,0)\otimes(\chi\circ \op{det})$, for some $0\leq r\leq p-1$ and $\chi:\QQ_p^{\times} \to \overline{\FF}_p^{\times}$ a smooth character. 
\end{theorem}

%% file: chapter_3.tex
\section{Homological branching for mod $p$ representations: the nonsplit-tori case}\label{sec::nonsplit_tori}

Let $F$ be a finite extension of $\QQ_p$ of degree $d$. Let $G=\op{GL}_2(F)$. Let $L=F[\sigma]$ be a quadratic extension of $F$. Write $X^2-aX-b,~a,b\in F$ for the minimal polynomial of $\sigma$ over $F$. 
Fix the $F$-basis $\{\sigma,1\}$ of $L$. Then the left multiplication by $L^\times$ on $L$ gives an embedding 
\begin{equation} \label{eq::embed_L_in_G}
\iota: L^\times\to G,\quad x+y\sigma\mapsto \left(\begin{array}{cl}
    x+ay & y\\
    yb & x
\end{array}\right).
\end{equation}
We identify $L^\times$ with the image $\iota(L^\times)$ in $G$. 

Let $\FF$ be a finite extension of $\FF_p$. In this section, we consider the homological branching laws for an absolutely irreducible admissible smooth $\FF$-representation $\pi$ of $G$ with respect to $L^{\times}$. In \S \ref{sec::PS_St_nonsplit}, we assume $\pi$ is either a principal series representation or a special series representation of $\op{GL}_2(F)$. In \S \ref{sec::supersingular_nonsplit}, we assume $F = \QQ_p$ and $\pi$ is a supersingular representation of $G$. 

 Let $\zeta: Z \to \FF^{\times}$ denote the central character of $\pi$. Let $\eta:L^{\times}\to \FF^{\times}$ be a smooth character. By Lemma  \ref{lemma::Borel--Wallach}, if $\eta|_{Z}$ is different from $\zeta$, then
\[
\op{RHom}_{L^\times }(\pi|_{L^\times},\eta) \simeq \op{RHom}_{L^\times}(\eta, \pi|_{L^\times}) \simeq 0 . 
\]
Therefore, we always assume $\eta|_{Z} = \zeta$ from now on, and we consider the derived Hom-complexes
\[
\op{RHom}_{L^\times,\zeta }(\pi|_{L^\times},\eta) ,\qquad \op{RHom}_{L^\times,\zeta}(\eta, \pi|_{L^\times}) .
\]
In this section, the complexes $\op{RHom}_{L^\times,\zeta }(\pi|_{L^\times},\eta) $ and $\op{RHom}_{L^\times,\zeta}(\eta, \pi|_{L^\times}) $ are computed explicitly, except in exceptional cases. In exceptional cases, we determine when the complexes are objects in $\op{D^b_{fd}}(\FF)$ and compute their Euler–Poincar\'e characteristic.

\begin{definition} \label{def::exceptional_L}
\begin{itemize}
\item[(i)] If $\pi$ is an irreducible principal series representation or a special series representation, we say $\eta$ is {\em exceptional for $\pi$} if the pair $(\pi ,\eta)$ is of the form $(\op{St}\otimes\chi_1\circ\op{det}, \chi_1\circ  \op{N}_{L/F})$ for some smooth character $\chi_1$ of $F^{\times}$, where $\op{N}_{L/F}: L^{\times}\to F^{\times}$ is the norm map. Otherwise we say $\eta$ is {\em nonexceptional for $\pi$}.

\item[(ii)] If $F = \QQ_p$ and $\pi$ is an irreducible supersingular representation of $\op{GL}_2(\QQ_p)$, if $p\geq 5$ (resp. $p\geq 3$) and $L$ is ramified (resp. unramified) over $\QQ_p$, then every $\eta$ is called {\em nonexceptional for $\pi$}.
\end{itemize}
\end{definition}

\subsection{Principal series and special series}\label{sec::PS_St_nonsplit}

Assume $p\geq 2$. Let $\chi=\chi_1\boxtimes\chi_2 : T\to \FF^{\times}$ be a smooth character of the diagonal torus $T$ of $G$. Then the central character of the parabolic induced representation $\op{Ind}_B^G \chi$ is $\chi_1\chi_2$. If $\chi_1 = \chi_2$, twisting the short exact sequence \eqref{eq::def_St} by $\chi_1\circ \det$ and restricting to $L^{\times}$, we have a short exact sequence
\begin{equation}\label{eq::twist_St_2}
0\to (\chi_1\circ\op{det})|_{L^{\times}}\to (\op{Ind}_{B}^{G}\chi_1\boxtimes\chi_1)|_{L^{\times}} \to (\op{St}\otimes(\chi_1\circ\op{det}))|_{L^{\times}} \to 0. 
\end{equation}
Let $\op{N}_{L/F}: L^{\times}\to F^{\times}$ be the norm map.

\begin{lemma}
We have $(\chi_1\circ \det)|_{L^{\times}} = \chi_1\circ  \op{N}_{L/F}$.
\end{lemma}
\begin{proof}
This is a straightforward computation which follows from \eqref{eq::embed_L_in_G}.
\end{proof}

\begin{lemma}\label{0}
 We have $G=BL^\times$. 
\end{lemma}
\begin{proof}
     Let $V = F^2$ be the space of row vectors, equipped with the natural right action of $G = \mathrm{GL}_2(F)$. The standard Borel subgroup $B$ is the stabilizer of the line spanned by $(0,1)$. Hence, the space of cosets $B \backslash G$ is naturally identified with the projective line $\mathbb{P}(V)$ of one-dimensional $F$-subspaces of $V$.
    
    The embedding $\iota \colon L^\times \hookrightarrow G$ endows $V$ with the structure of a one-dimensional vector space over $L$ (specifically, it identifies $V$ with the $F$-linear dual $L^\vee$). Then $L^\times$ acts transitively on $V\backslash\{ 0\}$. Consequently, the right action of $L^\times$ on the set of $F$-lines $\mathbb{P}(V)$ is also transitive. 
    
    This transitivity means that $B \backslash G$ consists of a single right $L^\times$-orbit. Since the orbit of the identity coset $B$ is exactly $B \backslash B L^\times$, we conclude that $B \backslash G = B \backslash B L^\times$, which yields $G = B L^\times$.
    \end{proof}

\begin{lemma} \label{lemma::Ind_res_L}
There is an isomorphism of $L^{\times}$-representations
    $$
    (\op{Ind}_{B}^{G} \chi)|_{L^\times}\cong\op{c-Ind}_{F^\times}^{L^\times}\chi_1\chi_2. 
    $$
\end{lemma}
\begin{proof}
  By Lemma \ref{0} and Lemma \ref{lemma:doublecoset_ind}, 
\begin{align*}
  (\op{Ind}_{B}^{G} \chi)|_{L^\times} =   (\op{c-Ind}_{B}^{G} \chi)|_{L^\times} \cong \op{c-Ind}_{B\cap L^{\times}}^{L^{\times}} \op{Res}^B_{B\cap L^{\times}} \chi.
\end{align*}
Then the lemma follows because $B\cap L^\times=F^\times$.  
\end{proof}

We recall an obvious lemma in homological algebra.
\begin{lemma}\label{lemma::concentrate_one_deg}
Let $\mathcal{A}$ be an abelian category. Let $C \in \op{C}(\mathcal{A})$ be a cochain complex in $\mathcal{A}$. If there is a $i_0 \in \ZZ$ such that $H^i (C) = 0$ for $i\neq i_0$, then there is a quasi-isomorphism 
\[
C \simeq H^{i_0}(C) [-i_0],
\]
in $\op{D}(\mathcal{A})$, the derived category of complexes in $\mathcal{A}$.
\end{lemma}
\begin{proof}
By assumption, the canonical truncation $\tau_{\leq i_0} C$ is quasi-isomorphic to $C$; and it is also quasi-isomorphic to $\tau_{\geq i_0} (\tau_{\leq i_0}C) = H^{i_0}(C) [-i_0]$.
\end{proof}

\begin{theorem}\label{nonsplitPS}
    Assume $\pi$ is either $\op{Ind}_B^G \chi$ (not necessarily $\chi_1 \neq \chi_2$) or $\op{St}\otimes(\chi_1\circ\op{det})$. Let $\zeta$ denote the central character of $\pi$. Let $\eta: L^{\times} \to \FF^{\times}$ be a smooth character such that $\eta|_{F^{\times}} = \zeta$. If $\pi$ is a special series, we further assume $\eta \neq \chi_1\circ  \op{N}_{L/F}$. Then 
\[
\op{RHom}_{L^{\times},\zeta} (\pi|_{L^{\times}}, \eta)   \simeq \FF [-d] \qquad \op{RHom}_{L^{\times},\zeta} (\eta, \pi|_{L^{\times}})   \simeq \FF [0] .
\]
\end{theorem}
\begin{proof}
We first prove the statements for $\pi = \op{Ind}_B^G \chi$. By Lemma  \ref{lemma::Ind_res_L},
\begin{align*}
 H^i( \op{RHom}_{L^\times,\zeta}((\op{Ind}_B^G \chi)|_{L^\times},\eta))
        & \cong H^i(\op{RHom}_{L^\times,\zeta}(\op{c-Ind}_{F^\times}^{L^\times}\chi_1\chi_2,\eta)) \\
&\cong \op{Hom}_{\mathrm{D}(L^\times/F^\times)}(\op{c-Ind}_{\{1\}}^{L^\times/F^\times}\FF, \FF[i]),
\end{align*}
where the second isomorphism follows from  \eqref{eq::isom_ext_twist} and Lemma \ref{lemma::proj_formula}. Since the functor $ \op{c-Ind}_{\{1\}}^{L^\times/F^\times} : \mr{D}(\{1\}) \to \mr{D}(L^{\times}/F^{\times})$ admits a right adjoint $\mathcal{R}_{\{1\}}^{L^\times/F^\times}$, we have 
\begin{align*}
\op{Hom}_{\mathrm{D}(L^\times/F^\times)}(\op{c-Ind}_{\{1\}}^{L^\times/F^\times}\FF, \FF[i]) & \cong \op{Hom}_{\mathrm{D}(\{1\})}(\FF, \mathcal{R}_{\{1\}}^{L^\times/F^\times}(\FF[i])) \cong \op{Hom}_{\mathrm{D}(\{1\})}(\FF, \FF[i-d]) \\
& \cong \op{Ext}_{\{1\}}^{i-d}(\FF,\FF),
\end{align*}
where the second isomorphism follows from Lemma \ref{lemma::rightadjoint} because $L^{\times}/F^{\times}$ is a $p$-adic Lie group of dimension $d$. Since $\op{Ext}_{\{1\}}^{i-d}(\FF,\FF)$ is of dimension $1$ if $i=d$ and is zero if $i\neq d$, the complex $\op{RHom}_{L^{\times},\zeta} (\pi|_{L^{\times}}, \eta) $ is cohomologically concentrated in degree $d$. Then by Lemma \ref{lemma::concentrate_one_deg}, there is a quasi-isomorphism
\[
\op{RHom}_{L^{\times},\zeta} (\pi|_{L^{\times}}, \eta)   \simeq \FF [-d].
\]

On the other hand, by Lemma  \ref{lemma::Ind_res_L} again we have   
 \begin{align*}
 H^i(\op{RHom}_{L^\times,\zeta}(\eta,(\op{Ind}_B^G \chi)|_{L^\times}))
       & \cong H^i( \op{RHom}_{L^\times,\zeta}(\eta,\op{c-Ind}_{F^\times}^{L^\times}\chi_1\chi_2)) \\
&  \cong \op{Hom}_{\op{D}(L^\times/F^\times)}(\FF,\op{c-Ind}_{\{1\}}^{L^\times/F^\times}\FF[i]),                                                            
 \end{align*}
where the second isomorphism follows from  \eqref{eq::isom_ext_twist} and Lemma \ref{lemma::proj_formula}. Since $L^\times/F^\times$ is compact, the functor $\op{c-Ind}_{\{1\}}^{L^\times/F^\times}$ is isomorphic to $\op{Ind}_{\{1\}}^{L^\times/F^\times}$, which admits an exact left adjoint $\op{Res}_{\{1\}}^{L^\times/F^\times}$ by Frobenius reciprocity. Then
\[
\op{Hom}_{\op{D}(L^\times/F^\times)}(\FF,\op{c-Ind}_{\{1\}}^{L^\times/F^\times}\FF[i]) \cong
\op{Hom}_{\op{D}(\{1\})}(\op{Res}_{1}^{L^\times}\FF,\FF[i]) \cong  \op{Ext}_{\{1\}}^i(\FF,\FF),
\] 
which is of dimension $1$ if $i = 0$ and is zero if $i\geq 1$. Therefore, the complex $\op{RHom}_{L^{\times},\zeta} (\eta, \pi|_{L^{\times}}) $ is cohomologically concentrated in degree $0$ and there is a quasi-isomorphism
\[
\op{RHom}_{L^{\times},\zeta} (\eta, \pi|_{L^{\times}})   \simeq \FF [0].
\]

Now assume $\pi =\op{St}\otimes(\chi_1\circ\op{det}) $ is a special series representation. Since we assume $\eta \neq \chi_1\circ\op{N}_{L/F}$, Lemma \ref{lemma::Borel--Wallach} gives
\[
\op{RHom}_{L^\times,\zeta}(\chi_1\circ\op{N}_{L/F},\eta) \simeq \op{RHom}_{L^\times ,\zeta}(\eta, \chi_1\circ\op{N}_{L/F}) \simeq 0.
\]
Then the short exact sequence \eqref{eq::twist_St_2} induces quasi-isomorphisms
\begin{align*}
\op{RHom}_{L^\times,\zeta}(\pi|_{L^\times},\eta) & \cong \op{RHom}_{L^\times,\zeta}((\op{Ind}_{B}^{G}\chi_1\boxtimes\chi_1)|_{L^\times},\eta) \\
 \op{RHom}_{L^\times,\zeta}(\eta, \pi|_{L^\times}) & \cong \op{RHom}_{L^\times,\zeta}(\eta, (\op{Ind}_{B}^{G}\chi_1\boxtimes\chi_1)|_{L^\times}).
\end{align*}
Then the results for $\pi$ follow from the above discussion.
\end{proof}

\begin{corollary}
The short exact sequence \eqref{eq::twist_St_2} is nonsplit.
\end{corollary}

\begin{proof}
By Theorem \ref{nonsplitPS}, $(\chi_1\circ\op{det})|_{L^{\times}}$ is never a quotient of $(\op{Ind}_{B}^{G}\chi_1\boxtimes\chi_1)|_{L^{\times}}$, so \eqref{eq::twist_St_2} is nonsplit.
\end{proof}

For the rest of this subsection, we assume $\pi = \op{St}\otimes(\chi_1\circ\op{det})$ and $\eta = \chi_1\circ  \op{N}_{L/F}$. We will compute the Euler–Poincar\'e characteristics 
\[
\chi(\op{RHom}_{L^\times,\zeta}(\pi|_{L^\times},\eta)) \qquad \chi( \op{RHom}_{L^\times,\zeta}(\eta, \pi|_{L^\times}) ),
\]
under the assumption that the quotient group $L^\times/F^\times$ has no $p$-torsion. 

The following proposition is proved by Pa\v{s}k\={u}nas \cite[Cor. 5.17]{Paskunas-IHES} in the much harder case $H = \op{GL}_2(\QQ_p)$; And we follow his proof.

\begin{proposition}\label{inj-to-inj}
    Let $H$ be a topologically finitely generated abelian $p$-adic Lie group. 
    The inclusion functor 
    $$\iota: \op{Mod}_H^{\mathrm{lfin}}(\FF) \hookrightarrow \op{Mod}_H^{\mathrm{sm}}(\FF)$$ 
    maps injective objects to injective objects. 

    Consequently, let $\pi_1, \pi_2 \in \op{Mod}_H^{\mathrm{lfin}}(\FF)$, then $\iota$ induces isomorphisms
    $$ \op{Ext}_{\op{Mod}_H^{\mathrm{lfin}}(\FF)}^i(\pi_1, \pi_2) \cong \op{Ext}_{\op{Mod}_H^{\mathrm{sm}}(\FF)}^i(\pi_1, \pi_2)\cong \op{Ext}_{\op{Mod}_H^{\mathrm{pro}}(\FF)}^i(\pi_2^\vee, \pi_1^\vee). $$
     \end{proposition}
    
    \begin{proof}
    Let $J$ be an injective object in $\op{Mod}_H^{\mathrm{lfin}}(\FF)$, and let $J \hookrightarrow E$ be its injective envelope in the full smooth category $\op{Mod}_H^{\mathrm{sm}}(\FF)$. To show that $J$ is injective in $\op{Mod}_H^{\mathrm{sm}}(\FF)$, it suffices to show that $E = J$.
            
            Let $e \in E$ be an arbitrary vector. We define $M := \FF[H]e \subseteq E$ to be the smooth $H$-subrepresentation generated by $e$, and let $N := M \cap J$. 
            
            Since $E$ is a smooth representation, the generating vector $e$ is fixed by some compact open subgroup $U \subseteq H$. Because $H$ is abelian, we have $u \cdot (h \cdot e) = h \cdot (u \cdot e) = h \cdot e$ for all $h \in H$ and $u \in U$. Thus, the subgroup $U$ fixes $M$ pointwise, making both $M$ and $N$ modules over the discrete group algebra $\FF[H/U]$.
            
            Since $H$ is topologically finitely generated, the quotient group $H/U$ is a finitely generated discrete abelian group. Consequently, its group algebra $R := \FF[H/U]$ is a commutative Noetherian ring, which is also a finitely generated $\FF$-algebra. 
            Since $M = R \cdot e$ is a cyclic $R$-module, $M$ is a Noetherian $R$-module. As $N$ is an $R$-submodule of $M$, $N$ is also a finitely generated $R$-module.
            
            We observe that $N \subseteq J$. Since $J \in \op{Mod}_H^{\mathrm{lfin}}(\FF)$, any finitely generated $H$-subrepresentation of $J$ must be of finite length. Because $N$ is finitely generated as an $R$-module, it is also finitely generated as an $H$-representation, which implies that $N$ has finite length as an $H$-representation. 
            By Zariski's Lemma (Hilbert's Nullstellensatz), since $R$ is a finitely generated $\FF$-algebra, for any maximal ideal $\mathfrak{m} \subset R$, the residue field $R/\mathfrak{m}$ is a finite extension of $\FF$, hence a finite-dimensional $\FF$-vector space. Because $N$ is of finite length as an $H$-representation, its simple Jordan-H\"older factors (which correspond to simple $R$-modules) are all finite-dimensional $\FF$-vector spaces. Thus, $N$ is a finite-dimensional $\FF$-vector space, which further implies that $N$ has finite length as an $R$-module.
            Since $N$ is a Noetherian $R$-module of finite length, all of its associated primes are maximal ideals of $R$:
            \[
            \op{Ass}_R(N) \subseteq \op{MaxSpec}(R).
            \]
            
            Next, we claim that $N \subseteq M$ is an essential extension of $R$-modules. Indeed, let $W \subseteq M$ be any non-zero $R$-submodule. Since $U$ acts trivially on $M$, $W$ is also a non-zero smooth $H$-subrepresentation of $E$. Since $J \hookrightarrow E$ is an essential extension in $\op{Mod}_H^{\mathrm{sm}}(\FF)$, we must have $W \cap J \neq 0$. Thus, 
            \[
            W \cap N = W \cap (M \cap J) = W \cap J \neq 0,
            \]
            proving the claim.
            
            Over a commutative Noetherian ring, an essential extension of modules preserves the set of associated primes. Therefore, we obtain:
            \[
            \op{Ass}_R(M) = \op{Ass}_R(N) \subseteq \op{MaxSpec}(R).
            \]
            Since $M$ is a finitely generated $R$-module and all of its associated primes are maximal ideals, $M$ has finite support and has finite length as an $R$-module. Again, since the residue fields $R/\mathfrak{m}$ are finite-dimensional over $\FF$ for all $\mathfrak{m} \in \op{MaxSpec}(R)$, the finite length of $M$ over $R$ implies that $M$ is a finite-dimensional $\FF$-vector space. In particular, $M$ is of finite length as an $H$-representation, i.e., $M \in \op{Mod}_H^{\mathrm{lfin}}(\FF)$.
            
            Now, consider the $H$-subrepresentation $J + M \subseteq E$. Since $J \in \op{Mod}_H^{\mathrm{lfin}}(\FF)$ and $M$ is of finite length, their sum $J + M$ also lies in $\op{Mod}_H^{\mathrm{lfin}}(\FF)$. 
            By the injectivity of $J$ in $\op{Mod}_H^{\mathrm{lfin}}(\FF)$, the inclusion $J \hookrightarrow J + M$ must split in $\op{Mod}_H^{\mathrm{lfin}}(\FF)$, and hence also in $\op{Mod}_H^{\mathrm{sm}}(\FF)$. 
            However, since $J \hookrightarrow E$ is an essential extension in $\op{Mod}_H^{\mathrm{sm}}(\FF)$, the subextension $J \hookrightarrow J + M$ must also be essential. 
            An essential extension can split if and only if it is an isomorphism. Therefore, we must have $J = J + M$, which implies that $M \subseteq J$, and hence $e \in J$. 
            Since $e \in E$ was chosen arbitrarily, we conclude that $E = J$, establishing the preservation of injective objects.
    \end{proof}

\begin{lemma}\label{lemma::fg_abel}
Let $H$ be a topologically finitely generated abelian $p$-adic Lie group. 
    \begin{itemize}
\item[(i)] There exists a topological group isomorphism
    $$ H \cong \ZZ^r \times \ZZ_p^s \times C, $$
     where $r, s \ge 0$ are non-negative integers and $C$ is a finite abelian group.

\item[(ii)] Let $\chi :H\to \FF^{\times}$ be a smooth character. Assume $H$ is $p$-torsion-free. Then 
        \[\op{dim}_{\FF} \op{Ext}_{H}^k(\chi,\chi)= \binom{r+s}{k}.\]
\end{itemize}    
\end{lemma}
    
    \begin{proof}
   (i) Since $H$ is an abelian $p$-adic Lie group, it admits an open torsion-free abelian pro-$p$ subgroup isomorphic to $\ZZ_p^s$ for some integer $s \ge 0$. Thus, $H$ contains a compact open pro-$p$ subgroup $H_0 \cong \ZZ_p^s$.
    
    Consider the canonical quotient $Q := H / H_0$. Since $H_0$ is open in $H$, the quotient group $Q$ is discrete. Furthermore, since $H$ is topologically finitely generated, $Q$ is a finitely generated discrete abelian group and decomposes as
    $$ Q \cong \ZZ^r \times C \times P, $$
    where $C’$ is a finite abelian group of order coprime to $p$, and $P$ is a finite $p$-group. 
    
    Let $\pi: H \twoheadrightarrow Q$ be the natural projection map. Consider the subgroup $H_1 := \pi^{-1}(P) \subseteq H$. Since $H_0 \subseteq H_1$ and $H_1 / H_0 \cong P$ is finite, $H_1$ is a compact open subgroup of $H$. As an extension of the finite $p$-group $P$ by the pro-$p$ group $H_0 \cong \ZZ_p^s$, $H_1$ is a compact pro-$p$ group. By the structure theory of finitely generated abelian pro-$p$ groups, $H_1$ decomposes into its torsion-free part and its $p$-torsion subgroup 
    $$ H_1 \cong \ZZ_p^s \times P' $$
    for some finite $p$-group $P'$.
    
    Now, consider the remaining quotient $H / H_1 \cong Q / P \cong \ZZ^r \times C$. We have a short exact sequence of topological groups:
    $$ 0 \to H_1 \to H \to \ZZ^r \times C \to 0. $$
    Because $\ZZ^r$ is a free abelian group, $C$ is a finite group of order prime to $p$ and $H_1$ is a pro-$p$ group, $\op{Ext}^1_{\mathbb{Z}}(\ZZ^r\times C, H_1) = 0$, so the sequence splits and (i) follows.
    
    (ii) We follow the approach of \cite[\S3.2]{Paskunas-IHES}. By Proposition \ref{inj-to-inj}, we can compute the extensions in the subcategory $\op{Mod}^{\mr{lfin}}_{H}(\FF)$ of $\op{Mod}^{\mr{sm}}_{H}(\FF)$. Up to twist, we can assume $\chi = \FF$ is the trivial character. We then replace $H$ with its pro-$p$ completion $\wh{H}\cong \ZZ_p^{r+s}$ (see the proof of \cite[Lem. 3.32]{Paskunas-IHES}). It then suffices to compute the dimension of $\op{Ext}_{\mf{C}_{\wh{H}}(\FF)}^k(\FF,\FF)$, where $\mf{C}_{\wh{H}}(\FF)$ denotes the subcategory of ${\rm Mod}^{\rm pro}_{\wh{H}}(\FF)$ corresponding to $\op{Mod}^{\mr{lfin}}_{\wh{H}}(\FF)$ under Pontryagin duality.
 
Note that $\FF[[\wh{H}]]\cong\FF[[x_1,\dots,x_{r+s}]]$. As in the proof of \cite[Cor. 3.35]{Paskunas-IHES}, we use the projective resolution $P_{\bullet}\to \FF\to 0$ of $\FF$ with
$$P_{k}:=\FF[[x_1,\dots,x_{r+s}]]^{\binom{r+s}{k}} $$ and differentials ($\{e_{i_1,\dots,i_k}\}_{i_1<\cdots<i_k}$ denotes the standard basis of $P_k$)
\[
d_k:P_k\to P_{k-1},~ e_{i_1,\dots,i_k}\mapsto \Sigma_{j=1}^{k} (-1)^{j+1}x_{i_j}e_{i_1,\dots,\wh{i}_j,\dots,i_k},
\]
to compute $\op{Ext}_{\mf{C}_{\wh{H}}(\FF)}^k(\FF,\FF)$, and the result follows.
\end{proof}

\begin{theorem} \label{thm::St_nonsplit_nongeneric}
    Let $\pi = \op{St}\otimes(\chi_1\circ\op{det})$ and $\eta = \chi_1\circ  \op{N}_{L/F}$. Assume $L^\times/F^\times$ has no $p$-torsion. Then the complexes  $\op{RHom}_{L^{\times},\zeta} (\pi|_{L^{\times}}, \eta) $ and $\op{RHom}_{L^{\times},\zeta} (\eta, \pi|_{L^{\times}}) $ are objects in $\op{D^b_{fd}}(\FF)$, and 
\[
\chi ( \op{RHom}_{L^{\times},\zeta} (\pi|_{L^{\times}}, \eta))  =(-1)^d,\qquad \chi ( \op{RHom}_{L^{\times},\zeta} (\eta, \pi|_{L^{\times}}))  =1.
\]
\end{theorem}

\begin{proof}
The short exact sequence \eqref{eq::twist_St_2} induces a distinguished triangle
\[
\op{RHom}_{L^\times,\zeta}(\pi|_{L^\times},\eta)\to \op{RHom}_{L^\times,\zeta}((\op{Ind}_{B}^{G}\chi_1\boxtimes\chi_1)|_{L^\times},\eta)\to \op{RHom}_{L^\times,\zeta}(\chi_1\circ\op{N}_{L/F},\eta)\xrightarrow{+1}.
\]
By \eqref{eq::isom_ext_twist},
\[
H^i(\op{RHom}_{L^\times,\zeta}(\chi_1\circ\op{N}_{L/F},\eta)) \simeq \op{Ext}_{L^\times,\zeta}^i(\chi_1\circ\op{N}_{L/F},\eta)\simeq  \op{Ext}_{L^\times/F^\times}^i(\FF,\FF).
\]

Since $L^\times/F^\times$ is assumed $p$-torsion-free, its pro-$p$ completion $\widehat{L^\times/F^\times}$ is isomorphic to $ \ZZ_p^{d}$, where $d$ is the degree of $F$ over $\QQ_p$. By Lemma \ref{lemma::fg_abel}, $\op{RHom}_{L^\times,\zeta}(\chi_1\circ\op{N}_{L/F},\eta)$ is in $\op{D^b_{fd}}(\FF)$ and 
\[
\chi(\op{RHom}_{L^\times,\zeta}(\chi_1\circ\op{N}_{L/F},\eta) ) = 0.
\]
Then it follows from Theorem \ref{nonsplitPS} that $\op{RHom}_{L^{\times},\zeta} (\pi|_{L^{\times}}, \eta)  \in \op{D^b_{fd}}(\FF)$ and
\[
\chi(\op{RHom}_{L^{\times},\zeta} (\pi|_{L^{\times}}, \eta) ) = \chi( \op{RHom}_{L^\times,\zeta}((\op{Ind}_{B}^{G}\chi_1\boxtimes\chi_1)|_{L^\times},\eta) ) = (-1)^d.
\]

The statements for $\op{RHom}_{L^{\times},\zeta} (\eta, \pi|_{L^{\times}})$ and $\chi(\op{RHom}_{L^{\times},\zeta} (\eta, \pi|_{L^{\times}}))$ can be proved similarly. 
\end{proof} 

\begin{remark} \label{remark::p-torsion-free}
We can make the $p$-torsion-free assumption on $L^\times/F^\times$ more explicit. Let $L^1:=\ker  (\op{N}_{L/F}:L^\times\to F^\times )$ denote the subgroup of norm-one elements of $L$. Let $\tau$ be the generator of $\op{Gal}(L/F)$. The map sending $ xF^\times$ to $\frac{x}{\tau(x)}$ gives an isomorphism
\[
  \Phi:L^\times/F^\times\xrightarrow{\ \sim\ }L^1.
\]
Indeed, this map is well-defined and is clearly injective; and Hilbert 90 shows that it is also surjective.

If $p = 2$, then  $L^\times/F^\times$ always has $2$-torsion. If $p$ is odd, then $L^\times/F^\times$ has $p$-torsion  if and only if $\mu_p \subset L$ and $\mu_p \not\subset F$, where $\mu_p \subset \overline{L}$ is the subgroup of $p$-th root of unity. Indeed, if $p=2$, the element $-1\in L^\times$ has norm $1$, so it gives an element of order $2$ in $L^1$. Now suppose $p$ is odd. An element of order $p$ in $L^1$ is a nontrivial $\zeta\in\mu_p \cap L$ such that $ \zeta \tau(\zeta)=1$. If $\mu_p\subset F$, then $\tau(\zeta)=\zeta$, so norm one gives $\zeta^2=1$; because $p$ is odd, this forces $\zeta=1$. Conversely, suppose $\mu_p\subset L$ but $\mu_p\not\subset F$. Then $\tau : \mu_p\to \mu_p$ is a nontrivial automorphism, given by $\zeta_p \mapsto \zeta_p^k$ for some integer $k$, where $\zeta_p$ is a primitive root. Hence $k \equiv -1 \pmod{p}$ because $\tau$ is an automorphism of order two. Therefore every $\zeta\in\mu_p$ has norm one.

From the above discussion, for odd prime $p$, $L^\times/F^\times$ has $p$-torsion precisely in the case
\[
  [F(\mu_p):F]=2,\quad L=F(\mu_p).
\]
\end{remark}

\subsection{Supersingular representations of ${\rm GL}_2(\QQ_p)$} \label{sec::supersingular_nonsplit}
We study the homological branching for supersingular representations of $G:={\rm GL}_2(\QQ_p)$. The case where $L$ is ramified over $\QQ_p$ is treated in \S  \ref{subsection::ramified} under the assumption that $p\geq 5$, while the case where $L$ is unramified over $\QQ_p$ is treated in \S  \ref{subsection::unramified} under the assumption that $p\geq 3$. We let $K$ denote ${\rm GL}_2(\ZZ_p)$ and let $Z$ be the center of $G$. Recall that for $0\leq r\leq p-1$, we denote by $\sigma=\op{Sym}^{r}(\FF^2)$ the subspace of $\FF[X, Y]$ consisting of homogeneous polynomials of degree $r$, equipped with an action of $KZ$ introduced in \S \ref{section::sm_rep_GL_2}.

Let $U(\FF_p) : =  \bigl(\begin{smallmatrix}
1 &\FF_p \\ 0& 1
\end{smallmatrix} \bigr)$ be the upper triangular unipotent subgroup of $\op{GL}_2(\FF_p)$.  Then the group algebra $\FF[ U(\FF_p)]$ is isomorphic to $\FF[z'] / z'^p$ by identifying $\bigl(\begin{smallmatrix}
1 & -1 \\ 0& 1
\end{smallmatrix} \bigr) -1$ with $z'$. The $\FF[ U(\FF_p)]$-module $\sigma$ is cyclic, generated by any $f(X,Y) \in \sigma$ with $f(0, Y) \neq 0$, and is then isomorphic to $\FF[z'] / z'^{r+1}$, see \cite[\S3]{Alperin}.  Let $e_0 := Y^r$, $e_i : = z'^i e_0$, $1\leq i\leq r$. Then $e_0,\ldots, e_r$ form an $\FF$-basis of $\sigma$. It is straightforward to check that $e_r =(-1)^r r! X^r$ and $z'e_r = 0$.

Let $\sigma^{\vee}$ be the $\FF$-linear dual of $\sigma$, with dual basis $e_0^{\vee},\ldots , e_r^{\vee}$. Then $\sigma^{\vee}$ is an $\FF[ U(\FF_p)]$-module given by 
\begin{equation} \label{eq::action_on_sigma_dual}
(g \cdot f )(v) = f(g^{-1} v),~ g\in U(\FF_p),~ f\in \sigma^{\vee}, ~ v\in \sigma.
\end{equation}
The map sending $z$ to $\bigl(\begin{smallmatrix}
1 & 1 \\ 0& 1
\end{smallmatrix} \bigr) -1 \in \FF[U(\FF_p)]$ induces an isomorphism $\FF[z]/z^p \simeq \FF[U(\FF_p)]$. It follows that $\sigma^{\vee}$ is an $\FF[z] / z^{p}$-module and
\[
(z f) (v) = f(z'v),~ f\in \sigma^{\vee},~v\in \sigma.
\]
We have $e_i^{\vee} = z^{r - i} e_r^{\vee}$ with $z^{r+1} e_r^{\vee} = 0$, and $\sigma^{\vee}$ is a cyclic $\FF[z] /z^p$-module generated by $e_r^{\vee}$.

\begin{lemma} \label{lemma::structure_sigma_dual}
Let $\gamma \in \op{GL}_2(\ZZ_p)$ such that $\gamma^{p^n} \neq 1$ for all $n\geq 1$ and $\gamma^p \equiv \bigl(\begin{smallmatrix}
1 & 0 \\ 0& 1
\end{smallmatrix} \bigr)  \pmod{p}$. Let $\alpha: \gamma^{\ZZ_p} \to U(\FF_p)$ be a group homomorphism sending $\gamma$ to $\bigl(\begin{smallmatrix}
1 & 1 \\ 0& 1
\end{smallmatrix} \bigr) \in U(\FF_p)$.
\begin{enumerate}
\item[(i)]  Via the map $\alpha$, $\sigma$ is a smooth representation of $\gamma^{\ZZ_p}$ generated by $e_0$.

\item[(ii)] The map sending $\gamma$ to $1+z$ gives an isomorphism of $\FF$-algebras $\FF[[\gamma^{\ZZ_p}]] \simeq \FF[[z]]$. The $\FF$-linear dual $\sigma^{\vee}$ is then a cyclic $\FF[[\gamma^{\ZZ_p}]] \simeq \FF[[z]]$-module generated by $e_r^{\vee}$, and there is an isomorphism of modules $ \sigma^{\vee} \simeq \FF[z]/z^{r+1}$, which is compatible with the natural surjection of rings $\FF[[z]] \simeq \FF[[\gamma^{\ZZ_p}]] \to  \FF[U(\FF_p)] \simeq \FF[z] /z^p$. Moreover,
\[
(z f) (v) = f(z'v),~ f\in \sigma^{\vee},~v\in \sigma,
\]
where $z': = \gamma^{-1} - 1 \in \FF[\gamma^{\ZZ_p}]$.
\end{enumerate} 
\end{lemma}
\begin{proof}
The assumption on $\gamma$ implies that the group $\gamma^{\ZZ_p}$ is $p$-torsion-free and is isomorphic to the additive group $\ZZ_p$. Then both (i) and (ii) are clear from the definition and the above discussion.
\end{proof}

\subsubsection{The ramified case}  \label{subsection::ramified}
Assume $p \geq 5$. Let $L$ denote a ramified quadratic extension of  $\QQ_p$ in $\overline{\QQ}_p$. Let $\varpi_{L}\in\mathcal O_L$ be a uniformizer satisfying $\varpi_L^2 = up$ where $u\in\ZZ_p^\times$. With the choice of a $\ZZ_p$-basis $\{\varpi_L,1\}$ of $\mathcal O_L$, we have an embedding
$$\iota:\mc O_L\to \op{M}_2(\ZZ_p),\ a+\varpi_L b\mapsto\left(\begin{array}{cl}
    a & b \\
    upb & a
    \end{array}\right),$$
    for any $a,b\in \ZZ_p$. 

\begin{lemma} \label{lemma::representative_L_ramified}
    A set of representatives of the double coset $KZ\backslash G/L^\times$ is given by 
    $$  \left\{\left(\begin{array}{cl}
        p^{-k} & 0 \\
        0 & 1
        \end{array}\right),\ k\in\ZZ_{\geq 0}\right\}. $$
\end{lemma}

\begin{proof}
    Since $G=BL^\times$, one has a natural surjective map 
 \[
\pi: (KZ\cap B) \backslash B\twoheadrightarrow KZ\backslash BL^\times/L^\times.
\]
    A set of representatives of $ (KZ \cap B) \backslash B$ is given by  $\bigl(\begin{smallmatrix}
p^{\ZZ} &\QQ_p/ \ZZ_p \\ 0& 1
\end{smallmatrix} \bigr)$. 

First, we show that for any $\beta_1 = \bigl(\begin{smallmatrix}
p^{k_1} & n_1 \\            0 & 1
\end{smallmatrix} \bigr)$ with $\op{val}_p(n_1)<0$, there is  $\beta_2 = \bigl(\begin{smallmatrix}
p^{k_2} & 0 \\            0 & 1
\end{smallmatrix} \bigr) $ such that $\pi(\beta_1)=\pi(\beta_2)$.  If $\vv(n_1)\geq k_1$,  we take $k_2 = k_1$, and we take $l = 1 - p^{-k_1} n_1 \varpi_L \in \mc{O}_L^{\times}$, then  
\begin{equation}\label{eq::representative_1}
\beta_1 \iota(l) \beta_2^{-1} = \left(\begin{array}{cl}
            1 - up^{1  - 2k_1} n_1^2 & 0  \\
            -up^{1 - 2k_1}n_1  & 1
            \end{array}\right)\in K.
\end{equation}
If $\vv(n_1)\leq k_1-1$, we take $k_2 = k_1 - 2\vv(n_1)$, and we take $ l = un_1^{-1} - p^{k_1-1}n_1^{-2} \varpi_L \in L^{\times} $, then 
\begin{equation}\label{eq::representative_2}
\beta_1 \iota(l) \beta_2^{-1} = \left(\begin{array}{cl}
            0 &  u - p^{2k_1 - 1} n_1^{-2}\\
          - up^{2 \vv(n_1)} n_1^{-2}    & un_1^{-1}
            \end{array}\right)\in K.
\end{equation}

Now assume $\beta_i = \bigl(\begin{smallmatrix}
p^{k_i} & 0 \\            0 & 1
\end{smallmatrix} \bigr),~ i =1,2$, such that $\pi(\beta_1 ) = \pi (\beta_2)$. Then there is $l = a + b\varpi_L \in L^{\times}$ such that 
\[
\beta_1 \iota(l)\beta_2^{-1} =\left(\begin{array}{cl}
                p^{k_1-k_2}a & p^{k_1}b \\
                up^{1-k_2}b & a
                \end{array}\right) \in K.
\]
Then either $p^{k_1-k_2}a^2$ or $p^{k_1 - k_2 +1}b^2$ must lie in $\ZZ_p^{\times}$. If $p^{k_1-k_2}a^2 \in \ZZ_p^{\times}$, since both $p^{k_1-k_2}a$ and $a$ are in $\ZZ_p$, we must have $a\in \ZZ_p^{\times}$ and $k_1 = k_2$. If $p^{k_1 - k_2 +1}b^2 \in \ZZ_p^{\times}$, we must have $p^{k_1} b\in \ZZ_p^{\times}$ and $k_1+k_2 = 1$. This finishes the proof of the lemma.
\end{proof}

For $k\in\ZZ_{\geq 0}$, let $s_k:= \bigl(\begin{smallmatrix}
    p^{-k} & 0 \\
    0 & 1
    \end{smallmatrix} \bigr)$
and let $L_k^\times:=L^\times \bigcap s_{k}^{-1}Ks_{k}Z$. Let $\sigma_k : = (s_k^{-1})_{*}\op{Res}^{KZ}_{s_kL^\times s_k^{-1}\cap KZ}\sigma$; this is a representation of $L_k^{\times}$.  By Mackey's formula \cite[Thm. 1.1]{Yamamoto}, we have $L^{\times}$-equivariant decompositions
\begin{equation} \label{eq::Mackey_1}
(\op{ind}_{KZ}^{G}\sigma)|_{L^\times}\cong \bigoplus_{s_k \in KZ\backslash G/L^\times}\op{ind}_{KZ}^{KZs_k L^\times}\sigma\cong \bigoplus_{k\geq 0}\op{ind}_{L_k^\times}^{L^\times}\sigma_k,
\end{equation}
where the second isomorphism is given by $[s_k^{-1},v]\mapsto [1,v]$, $v\in \sigma$. 

The following lemma describes the restriction of the Hecke operator $\Phi$ to $ \op{ind}_{L_k^\times}^{L^\times}\sigma_k$. 

\begin{lemma}\label{lemma::Phi_L_ram}
    For $k\in \ZZ_{\geq 0}$,  denote by  $\Phi_k:=\Phi|_{\op{ind}_{L_k^\times}^{L^\times}\sigma_k}$. Then $\Phi_k= (\Phi^{+}_k,  \Phi^{-}_k)$, where the morphisms $\Phi^{\pm}_k$ are given by 
\begin{align*}
\Phi^{+}_k:  \op{ind}_{L_k^\times}^{L^\times}\sigma_k & \to \op{ind}_{L_{k+1}^\times}^{L^\times}\sigma_{k+1}, ~k \geq 0\\ 
 [1,X^{r-j}Y^j] & \mapsto \sum_{i=0}^{p-1}[1-\varpi_L p^{k}i, i^j X^r] ;\\
\Phi^{-}_0:   \op{ind}_{L_0^\times}^{L^\times}\sigma_0 & \to \op{ind}_{L_0^\times}^{L^\times}\sigma_0 \\
[1,X^{r-j}Y^j] & \mapsto  \varpi_L [1, \delta_{j,r}u^{-r}X^r] ; \\
\Phi^{-}_k :  \op{ind}_{L_k^\times}^{L^\times}\sigma_k & \to \op{ind}_{L_{k-1}^\times}^{L^\times}\sigma_{k-1}, ~ k \geq 1\\ 
[1,X^{r-j}Y^j] & \mapsto [1, \delta_{j,r}Y^r] .
\end{align*}
\end{lemma}

\begin{proof}
The map $\Phi_k$ is determined by $\Phi ([s_k^{-1}, v ]),~v\in \sigma$.  By \eqref{eq::formula_Phi}, we have 
$$\Phi\left(\left[ \left(\begin{array}{cl}
    p^{-k} & 0 \\
    0 & 1
   \end{array}\right)^{-1},X^{r-j}Y^{j}\right] \right) = \left[  \left(\begin{array}{cl}
        p^{-k+1} & 0 \\
        0 & 1
        \end{array}\right)^{-1},\delta_{j,r}Y^r \right]+\sum_{i=0}^{p-1} \left[ \left(\begin{array}{cl}
        p^{-k-1} & p^{-1}i \\
        0 & 1
        \end{array}\right)^{-1},i^jX^r \right]. $$
Let $\Phi_k^{\pm}$ be given by:
\begin{align*}
\Phi^+_k \left(\left[ \left(\begin{array}{cl}
    p^{-k} & 0 \\
    0 & 1
   \end{array}\right)^{-1},X^{r-j}Y^{j}\right] \right) & = \sum_{i=0}^{p-1} \left[ \left(\begin{array}{cl}
        p^{-k-1} & p^{-1}i \\
        0 & 1
        \end{array}\right)^{-1},i^jX^r \right] \\
\Phi^-_k \left(\left[ \left(\begin{array}{cl}
    p^{-k} & 0 \\
    0 & 1
   \end{array}\right)^{-1},X^{r-j}Y^{j}\right] \right) & = \left[  \left(\begin{array}{cl}
        p^{-k+1} & 0 \\
        0 & 1
        \end{array}\right)^{-1},\delta_{j,r}Y^r \right].
\end{align*}

Since $\vv(p^{-1} i )\geq -k-1$, by \eqref{eq::representative_1} we have
   \begin{align*}
\left(\begin{array}{cl}
        p^{-k-1} & p^{-1}i \\
        0 & 1
        \end{array}\right)^{-1} & = \left(\begin{array}{cl}
            1 & -p^{k}i \\
            -up^{k+1}i & 1
            \end{array}\right) \left(\begin{array}{cl}
    p^{-k-1} & 0 \\
    0 & 1
    \end{array}\right)^{-1}  \left(\begin{array}{cl}
                    1-up^{1+2k}i^2 & 0 \\
                    -up^{2+2k}i & 1
                    \end{array}\right)^{-1} \\
& = \iota( 1 - p^k i \varpi_L) s_{k+1}^{-1} \left(\begin{array}{cl}
                    1-up^{1+2k}i^2 & 0 \\
                    -up^{2+2k}i & 1
                    \end{array}\right)^{-1},
\end{align*}
    where $ \left(\begin{smallmatrix}
                    1-up^{1+2k}i^2 & 0 \\
                    -up^{2+2k}i & 1
                    \end{smallmatrix}\right) \in 1+p M_2(\ZZ_p)$ acting trivially on $\sigma$. Thus this gives the formula for $\Phi^+_k$ using the isomorphisms in \eqref{eq::Mackey_1}.

The formula for $\Phi^-_k$, $k\geq 1$ follows directly from the definition. For $\Phi_0^-$, it suffices to use the identity
 $$\left(\begin{array}{cl}
        p & 0 \\
        0 & 1
        \end{array}\right)^{-1}= \left(\begin{array}{cl}
            0 & 1 \\
            up & 0
            \end{array}\right)\left(\begin{array}{cl}
                1 & 0 \\
                0 & 1
                \end{array}\right)^{-1} \left(\begin{array}{cl}
                    0 & p \\
                    up & 0
                    \end{array}\right)^{-1} ,$$
and the fact $\left(\begin{array}{cl}
    0 & p \\
    up & 0
    \end{array}\right)^{-1}Y^r=u^{-r}X^r$. 
\end{proof}

Recall that the center $Z$ of $G$ acts on $\op{ind}_{KZ}^{G}\sigma$ by the character $\omega^r$ introduced in \S\ref{section::sm_rep_GL_2}.

\begin{lemma}\label{lemma::action_of_L_k}
\begin{enumerate}
\item[(i)]  For $k \in \ZZ_{\geq 0}$, $L_k^\times =  (1+ \varpi_{L}p^k \ZZ_p) \cdot Z$.  

\item[(ii)] The action of $L_k^{\times}$ on $\sigma_k$ is given by
\[
(1 + \varpi_L p^k b) p^n a \cdot X^{r-j} Y^j = \omega^r(a) X^{r-j} (bX + Y)^j,~ a\in \ZZ_p^{\times},~b\in \ZZ_p, ~n\in \ZZ. 
\]
\end{enumerate}
\end{lemma}

\begin{proof}
(i) For $l=a+\varpi_L b\in L^\times$, one has $s_kls_k^{-1}\in KZ$ if and only if $a\in\QQ_p^\times$ and $\op{val}_p(b)-\op{val}_p(a)\geq k$.  It follows that $L_k^\times =  (1+ \varpi_{L}p^k \ZZ_p) \cdot Z$. (ii) is a direct computation. 
\end{proof}

Let $\pi: = \pi(r,0) = \op{ind}_{KZ}^{G}\sigma/\Phi$ be the supersingular representation introduced in \eqref{ss}. Let $\zeta: = \omega^r$ denote the central character of $\pi$. Let $\eta: L^{\times} \to \FF^{\times}$ be a smooth character such that $\eta|_{Z} = \zeta$. By Lemma \ref{lemma::proj_formula}, we have an isomorphism 
\begin{equation}\label{eq::twist_iso_eta}
(\op{ind}^{L^{\times}}_{L_k^{\times}}  \sigma_k) \otimes \eta^{-1} \simeq \op{ind}^{L^{\times}}_{L_k^{\times}}  ( \sigma_k\otimes \eta^{-1}|_{L_k^{\times}}),
\end{equation}
given by the map sending $[g,v]\otimes 1$  to $ \eta(g) [g,v]$. The representation $\sigma_k\otimes \eta^{-1}|_{L_k^{\times}}$ has trivial $Z$-action. Thus it is naturally a $L_k^{\times}/Z$-representation. We let $\overline{\sigma}_k$ denote the $L_k^{\times}/Z$-representation $\sigma_k \otimes\eta^{-1}|_{L_k^{\times}}$.

By \eqref{eq::isom_ext_twist}, we have isomorphisms
\begin{equation} \label{eq::Ext_reduce_mod_center}
\op{Ext}^i_{L^{\times},\zeta } (\op{ind}^{L^{\times}}_{L_k^{\times}} \sigma_k ,\eta) \simeq \op{Ext}^i_{L^{\times}/Z} (\op{ind}^{L^{\times}/Z}_{L_k^{\times}/Z} \overline{\sigma}_k, \FF ) \simeq \op{Ext}^i_{ \FF[[L^{\times}/Z]]} (\FF, ( \op{ind}^{L^{\times}/Z}_{L_k^{\times}/Z} \overline{\sigma}_k)^{\vee} ),~ \forall i,
\end{equation}
and
\[
\op{Ext}^i_{L^{\times},\zeta } (\eta, \op{ind}^{L^{\times}}_{L_k^{\times}} \sigma_k) \simeq \op{Ext}^i_{L^{\times}/Z} (\FF, \op{ind}^{L^{\times}/Z}_{L_k^{\times}/Z}\overline{\sigma}_k ) \simeq \op{Ext}^i_{ \FF[[L^{\times}/Z]]} (( \op{ind}^{L^{\times}/Z}_{L_k^{\times}/Z} \overline{\sigma}_k)^{\vee}, \FF ),~ \forall i, 
\]
where $\op{Ext}^i_{ \FF[[L^{\times}/Z]]}$ stands for the $i$-th extension group in ${\rm Mod}^{\rm pro}_{L^{\times}/Z}(\FF)$. Since $L^{\times}/Z$ is compact, the category ${\rm Mod}^{\rm pro}_{L^{\times}/Z}(\FF)$ is equivalent to the category of compact $\FF[[L^{\times}/Z]]$-modules.

Let $\gamma_k : = (1 + \varpi_L)^{p^k} \in L_k^{\times} $, $k\geq 0$. 

\begin{lemma} \label{lemma::p-torsion_free_ramified}
\begin{itemize}
\item[(i)] The quotient group $L_k^{\times}/Z$ is isomorphic to $\ZZ_p$, for all $k\geq 0$.

\item[(ii)] The composite $i: \gamma_k^{\ZZ_p} \to L_k^{\times} \to  L_k^{\times}/Z$ is an isomorphism of $p$-adic Lie groups.
\end{itemize}
\end{lemma}
\begin{proof}
(i) Since we assume $p\geq 5$, by \cite[Prop. 4.3.12]{Coh07}, the $p$-adic logarithm map gives an isomorphism between $1+\varpi_L^k \mc{O}_L$ and $\varpi_L^k \mc{O}_L$, for all $k \geq 1$. The same map also gives an isomorphism between $1+p^k \ZZ_p$ and $p^k \ZZ_p$, for all $k\geq 1$. Since $Z \cap (1 + \varpi_L^{1+2k} \mc{O}_L) = 1+p^{1+k} \ZZ_p$, the $p$-adic logarithm map induces an isomorphism
\[
L_k^{\times}/Z = (1 + \varpi_L^{1+2k} \mc{O}_L) Z / Z \simeq   \varpi_L^{1+2k} \mc{O}_L / p^{1+k} \ZZ_p \simeq \varpi_L^{1+2k} \ZZ_p.
\]
This proves (i).

(ii) The map $i$ is clearly a homomorphism of groups. It is straightforward to check that $ \gamma_k^{\ZZ_p} \cap Z = \{1\}$, so $i$ is injective. For the surjectivity, note that the map 
\[
\gamma_k^{\ZZ_p} \to  (L_k^{\times}/Z) / ( L_k^{\times}/Z)^p \simeq \ZZ/p\ZZ
\]
induced by $i$ is surjective. Then $i$ is surjective by Nakayama's lemma.
\end{proof}

As a consequence of Lemma \ref{lemma::p-torsion_free_ramified}, 
\[
\FF[[L_k^{\times}/Z]] \simeq \FF[[z_k]],
\]
with $z_k : = \gamma_k - 1$.

We define a map $\alpha_k : \gamma_k^{\ZZ_p} \to U(\FF_p)$ by sending $\gamma_k$ to $\bigl(\begin{smallmatrix}
1 & 1 \\ 0& 1
\end{smallmatrix} \bigr) \in U(\FF_p)$. By Lemma \ref{lemma::structure_sigma_dual}, $\sigma$ is a representation of $\gamma_k^{\ZZ_p}$ via the map $\alpha_k$. Comparing with the action explicitly given in Lemma \ref{lemma::action_of_L_k}(ii), $\sigma$ is isomorphic to $\overline{\sigma}_k$ as $\gamma_k^{\ZZ_p}$-representations. The vectors $\{ X^{r-j} Y^j,~ 0\leq j\leq r\}$ form  an $\FF$-basis of $\overline{\sigma}_k$. Let $e_0 : = Y^r \in \overline{\sigma}_k$, $e_i : = z'^i_k e_0$, where $z'_k: = \gamma_k^{-1} - 1$. Then $e_r = (-1)^r r! X^r$. Let $e_0^{\vee},\ldots, e_r^{\vee}$ be the dual basis of $\overline{\sigma}_k^{\vee}$, the $\FF$-linear dual of $\overline{\sigma}_k$. We have
\[
(z_k f) (v) = f(z_k'v),~ f\in \overline{\sigma}_k^{\vee},~v\in \overline{\sigma}_k.
\]

\begin{lemma} \label{lemma::structure_sigma_k}
\begin{enumerate}
\item[(i)] The inclusion $\FF [[L_{k+1}^{\times}/Z ]] \to \FF [[L_k^{\times} /Z]] $ corresponds to the map $\FF[[z_{k+1}]] \to \FF[[z_k]]$ sending $z_{k+1}$ to $z_k^p$.

\item[(ii)] The $\FF[[L_k^{\times}/Z ]]$-module $\overline{\sigma}_k^{\vee}$ is cyclic, generated by $e_r^{\vee} $. In particular, $\overline{\sigma}_k^{\vee}$ is isomorphic to $\FF[[ z_k]]/  z_k^{r+1} $.

\item[(iii)] The $\FF$-algebra $\FF[[L^{\times} /Z]]$ is isomorphic to $(\FF[t]/t(2 + t))[[z_0]] $. And the  $\FF[[L^{\times} /Z]]$-module $( \op{ind}^{L^{\times}/Z }_{L_k^{\times} /Z } \overline{\sigma}_k)^{\vee}$ is cyclic and is isomorphic to $(\FF[t]/t(2 + t))[[z_0]]/z_0^{p^k(r+1)}$.
\end{enumerate}
\end{lemma}
\begin{proof}
(i) This is because in $\FF[[z_k]]$, we have
\[
z_{k+1} = \gamma_{k+1} - 1 = (1 + \varpi_L)^{p^{k+1}}-1 =((1 + \varpi_L)^{p^k} - 1)^p = z_k^p.
\]
(ii) follows from Lemma \ref{lemma::structure_sigma_dual}(ii). (iii) We have $L^{\times} /Z \simeq \ZZ/2\ZZ \times L_0^{\times} /Z $. Then 
\[
\FF[[L^{\times}/Z ]] \simeq \FF[\ZZ/2\ZZ] \otimes_{\FF} \FF[[L_0^{\times}/Z ]] \simeq (\FF[t]/t(2 + t))[[z_0]],
\] 
where $t$ corresponds to $[\overline{\varpi_L}] - 1$. It then follows from (i) and (ii) that
\[
( \op{ind}^{L^{\times}/Z }_{L_k^{\times} /Z } \overline{\sigma}_k)^{\vee} \simeq \FF[[L^{\times} /Z]]\otimes_{\FF[[L^{\times}_k /Z]]} \overline{\sigma}_k^{\vee} \simeq (\FF[t]/t(2 + t))[[z_0]]/z_0^{p^k(r+1)}.
\]
\end{proof}

\begin{proposition} \label{prop::Ext^i_L_ramified}
\begin{enumerate}
\item[(i)] We have 
\[
 \op{Ext}^i_{ \FF[[L^{\times}/Z]]} (( \op{ind}^{L^{\times}/Z}_{L_k^{\times}/Z} \overline{\sigma}_k)^{\vee}, \FF ) \simeq \left\{\begin{array}{cl}
                                                              \FF  ,~ & \text{if $i = 0,1$,}\\
                                                               0,~ &\text{otherwise}.
                                                                \end{array}\right.
\]
\item[(ii)] We have
\[
\op{Ext}^i_{ \FF[[L^{\times}/Z]]} (\FF, ( \op{ind}^{L^{\times}/Z}_{L_k^{\times}/Z} \overline{\sigma}_k)^{\vee} ) =   \left\{\begin{array}{cl}
                                                     (t+2)z_0^{p^{k}(r+1) - 1} \FF[[z_0]]/z_0^{p^k(r+1)} ,~ & \text{if $i = 0$,}\\
 (t+2)\frac{ \FF[[z_0]]/z_0^{p^k(r+1)}}{(z_0)}  ,~ & \text{if $i = 1$,}\\
                                                               0,~ &\text{otherwise}.
                                                                \end{array}\right.
\]
\end{enumerate}
In particular, $ \op{Ext}^i_{ \FF[[L^{\times}/Z]]} (( \op{ind}^{L^{\times}/Z}_{L_k^{\times}/Z} \overline{\sigma}_k)^{\vee}, \FF )$ and $\op{Ext}^i_{ \FF[[L^{\times}/Z]]} (\FF, ( \op{ind}^{L^{\times}/Z}_{L_k^{\times}/Z} \overline{\sigma}_k)^{\vee} )$ are one-dimensional if $i = 0,1$.
\end{proposition}
\begin{proof}
Let $A$ denote the $\FF$-algebra $\FF[t]/(t(2 + t))$. Then $\FF[[L^{\times}/Z]] \simeq A[[z_0]]$. Consequently, by Lemma \ref{lemma::structure_sigma_k}(iii), 
\[
( \op{ind}^{L^{\times}/Z}_{L_k^{\times}/Z} \overline{\sigma}_k)^{\vee} \simeq A[[z_0]]/z_0^{p^k(r+1)}.
\]

(i) To compute $ \op{Ext}^i_{A[[z_0]]} (A[[z_0]]/ z_0^{p^k(r+1)}, \FF )$, we choose the following projective resolution of $A[[z_0]] /z_0^{p^k(r+1)}$ in ${\rm Mod}^{\rm pro}_{L^{\times}/Z}(\FF)$:
\[
0 \to A[[z_0]] \xrightarrow{z_0^{p^k(r+1)}} A[[z_0]] \to A[[z_0]]/z_0^{p^k(r+1)} \to 0.
\]
Then 
\begin{align*}
\op{Ext}^i_{A[[z_0]] } (A[[z_0]]/ z_0^{p^k(r+1)}, \FF )  & = H^i ( \op{Hom}_{A[[z_0]]} (A[[z_0]] ,\FF ) \xrightarrow{z_0^{p^k(r+1)}}  \op{Hom}_{A[[z_0]]} (A[[z_0]] ,\FF ) ) \\
& \simeq \left\{\begin{array}{cl}
                                                             \FF ,~ & \text{if $i = 0,1$}\\
                                                               0,~ &\text{otherwise}.
                                                                \end{array}\right.
\end{align*}

(ii) Let $P_\bullet$ be the resolution 
\[
0\to \FF[[z_0]] \xrightarrow{z_0} \FF[[z_0]] \to 0,
\]
 and let $Q_{\bullet}$ be the resolution
\[
\cdots  \xrightarrow{2+t}  \FF[t]/t(2+t) \xrightarrow{t} \FF[t]/t(2+t)  \xrightarrow{2+t} \FF[t]/t(2+t)  \xrightarrow{t} \FF[t]/t(2+t)  \to \FF.
\]
Then the total complex ${\rm Tot}_{\bullet} (P_{\bullet} \widehat{\otimes}_{\FF} Q_{\bullet})$ gives a projective resolution of $\FF$ in ${\rm Mod}^{\rm pro}_{L^{\times}/Z}(\FF)$. Since the completed tensor product of compact module with a finite-dimensional $\FF$-vector space is the same as the usual tensor product, we have
\[
{\rm Tot}_{\bullet} (P_{\bullet} \widehat{\otimes}_{\FF} Q_{\bullet})  = {\rm Tot}_{\bullet} (P_{\bullet}\otimes_{\FF} Q_{\bullet}).
\]
Explicitly, this resolution is 
\begin{equation} \label{eq::tensor_resolution}
\cdots   \xrightarrow{\bigl(\begin{smallmatrix}
t & -z_0 \\            0 & 2+t
\end{smallmatrix} \bigr)}  A[[z_0]]\oplus A[[z_0]] \xrightarrow{\bigl(\begin{smallmatrix}
2+t & z_0 \\            0 & t
\end{smallmatrix} \bigr)}  A[[z_0]]\oplus A[[z_0]]  \xrightarrow{(t,-z_0) } A[[z_0]] \to \FF \to 0.
\end{equation}
Then 
\begin{equation*}
\begin{multlined}
\op{Ext}^i_{ A[[z_0]]} (\FF, A[[z_0]]/z_0^{p^k(r+1)} ) = H^i ( A[[z_0]]/z_0^{p^k(r+1)}  \xrightarrow{d^0 = \bigl(\begin{smallmatrix}
t \\  -z_0
\end{smallmatrix} \bigr) } A[[z_0]]/z_0^{p^k(r+1)} \oplus A[[z_0]]/z_0^{p^k(r+1)} \\
  \xrightarrow{d^1 =\bigl(\begin{smallmatrix}
2+t & 0 \\  z_0 & t
\end{smallmatrix} \bigr) }   A[[z_0]]/z_0^{p^k(r+1)} \oplus A[[z_0]]/z_0^{p^k(r+1)}  \xrightarrow{d^2 = \bigl(\begin{smallmatrix}
t & 0 \\  -z_0 & 2+t
\end{smallmatrix} \bigr)}  A[[z_0]]/z_0^{p^k(r+1)} \oplus A[[z_0]]/z_0^{p^k(r+1)} \\  \xrightarrow{d^3 =\bigl(\begin{smallmatrix}
2+t & 0 \\  z_0 & t
\end{smallmatrix} \bigr) }  \cdots ).
\end{multlined}
\end{equation*}

Let $\alpha= a(z_0) + b(z_0) t \in \ker d^0$, $a(z_0),b(z_0)\in \FF[[z_0]]/z_0^{p^k(r+1)}$. Then 
\[
 \left\{\begin{array}{cl}
      a(z_0) t + b(z_0) t^2 & = (a(z_0) -2 b(z_0))t = 0,          \\
  -z_0 a(z_0) -z_0 b(z_0) t & = 0.
                                                                \end{array}\right.
\]
Hence $b(z_0) \in z_0^{p^{k}(r+1) - 1} \FF[[z_0]]/z_0^{p^{k}(r+1)}$ and $a(z_0) = 2 b(z_0)$. Then 
\[
H^0 = \{ (2+t) b(z_0),~b(z_0) \in z_0^{p^{k}(r+1) - 1} \FF[[z_0]]/z_0^{p^{k}(r+1)}\}.
\]

For $H^1$, if $ \bigl(\begin{smallmatrix}
a(z_0) + b(z_0)t \\  c(z_0) + d(z_0) t
\end{smallmatrix} \bigr) \in \ker d^{1}$, then we have 
\[
 \left\{\begin{array}{cl}
      a(z_0) & = 0,          \\
 c(z_0) & = - z_0 b(z_0) + 2d(z_0) .
                                                                \end{array}\right.
\]
Hence 
\[
\ker d^1 = \left\{ \left(\begin{array}{cl}
      b(z_0) t \\
     -z_0 b(z_0) + d(z_0)(2+t)
        \end{array}\right) ,~ b(z_0),d(z_0) \in \FF[[z_0]] / z_0^{p^k(r+1)} \right\}.
\]
Let $a_1(z_0) + b_1(z_0) t \in A[[z_0]]/z_0^{p^k(r+1)}$, $a_1(z_0),b_1(z_0)\in \FF[[z_0]]/z_0^{p^k(r+1)}$. We have
\[
d^0 (a_1(z_0) + b_1(z_0) t) = \left(\begin{array}{cl}
      (a_1(z_0) -2 b_1(z_0))t  \\
     -z_0 a_1(z_0) -z_0 b_1(z_0)t
        \end{array}\right).
\]
Assume $a_1(z_0) = b(z_0) + 2b_1(z_0)$. Then 
\[
-z_0 a_1(z_0) -z_0 b_1(z_0)t = -z_0 b(z_0) - z_0 b_1(z_0)(2+t).
\]
Therefore,
\[
H^1 = \ker d^1/{\rm Im}d^0 \simeq   (t+2)\frac{ \FF[[z_0]]/z_0^{p^k(r+1)}}{(z_0)}.
\]

Finally, we show $H^i =0$ for $i\geq 2$. Assume $i$ is even and $ \bigl(\begin{smallmatrix}
a(z_0) + b(z_0)t \\  c(z_0) + d(z_0) t
\end{smallmatrix} \bigr) \in \ker d^{i}$. Then
\[
 \left\{\begin{array}{cl}
      a(z_0) & = 2 b(z_0),          \\
 c(z_0) & =  z_0 b(z_0)  .
                                                                \end{array}\right.
\]
Hence 
\[
\ker d^i = \left\{ \left(\begin{array}{cl}
      b(z_0) (2+t) \\
     z_0 b(z_0) + d(z_0)t
        \end{array}\right) ,~ b(z_0),d(z_0) \in \FF[[z_0]] / z_0^{p^k(r+1)} \right\}.
\]
Then it follows from $d^{i-1}  \left( \bigl(\begin{smallmatrix}
 b(z_0) \\  d(z_0) 
\end{smallmatrix} \bigr)  \right) = \bigl(\begin{smallmatrix}
 b(z_0) (2+t) \\   z_0 b(z_0) + d(z_0)t
\end{smallmatrix} \bigr)$ that $H^i = 0$. The proof for $i$ odd is similar. 
\end{proof}

The morphisms $\Phi_k^{\pm}$ in Lemma \ref{lemma::Phi_L_ram} induce maps on $\op{ind}^{L^{\times}/Z }_{L_k^{\times} /Z } \overline{\sigma}_k$. By abuse of notation, we denote the induced maps by $\Phi_k^{\pm}$. We have the following proposition.

\begin{proposition} \label{prop::Phi_k_L_ramified}
\begin{enumerate}
\item[(i)] For $k\geq 0$, $\Phi^{+}_k:  \op{ind}_{L_k^\times/Z}^{L^\times/Z} \overline{\sigma}_k  \to \op{ind}_{L_{k+1}^\times /Z }^{L^\times /Z } \overline{\sigma}_{k+1}$ is injective. And under the identification in Lemma \ref{lemma::structure_sigma_k}(iii), the dual map 
\[
(\Phi^+_k)^{\vee}: (\FF[t]/t(2 + t))[[z_0]]/z_0^{p^{k+1}(r+1)} \to (\FF[t]/t(2 + t))[[z_0]]/z_0^{p^k(r+1)}
\]
is multiplication by some $u_k(z_0) \in  1 + z_0 \FF[[z_0]]$.

\item[(ii)] For $k \geq 1$, $\Phi^{-}_k :  \op{ind}_{L_k^\times/Z}^{L^\times /Z} \overline{\sigma}_k   \to \op{ind}_{L_{k-1}^\times /Z}^{L^\times /Z} \overline{\sigma}_{k-1}$ is surjective. And under the identification in Lemma \ref{lemma::structure_sigma_k}(iii), the dual map 
\[
(\Phi^-_k)^{\vee}: (\FF[t]/t(2 + t))[[z_0]]/z_0^{p^{k-1}(r+1)} \to (\FF[t]/t(2 + t))[[z_0]]/z_0^{p^k(r+1)}
\]
is multiplication by $z_0^{(p^k - p^{k-1})(r+1)}$.

\item[(iii)] Under the identification in Lemma \ref{lemma::structure_sigma_k}(iii), the dual map 
\[
(\Phi^-_0)^{\vee}: (\FF[t]/t(2 + t))[[z_0]]/z_0^{r+1} \to (\FF[t]/t(2 + t))[[z_0]]/z_0^{r+1}
\]
is multiplication by $\eta(\varpi_L) ((-1)^ru^{-r} /r!)(1+t)  z_0^{r} $.
\end{enumerate}
\end{proposition}
\begin{proof}
Recall the elements $z_k' : = \gamma_k^{-1} - 1$ and $t = [\overline{\varpi_L} ] - 1$ in $ \FF[L^{\times} / Z]$. The vectors 
\[
 (1+t)^i z'^m_0 z_k'^l [1,e_0],~ i \in \{0,1\},~ 0\leq m\leq p^k-1, ~ 0\leq l\leq r,
\]
form a basis of $\op{ind}_{L_k^\times/Z}^{L^\times/Z} \overline{\sigma}_k$, and their dual vectors $ ((1+t)^i z'^m_0 z_k'^l [1,e_0])^{\vee}$ form a basis of $(\op{ind}_{L_k^\times/Z}^{L^\times/Z} \overline{\sigma}_k)^{\vee}$. We have for $m\geq 1$ and $i\in \ZZ/2\ZZ$,
\[
(1+t)z_0 \cdot ((1+t)^i z'^m_0 z_k'^l [1,e_0])^{\vee} = ((1+t)^{i-1} z'^{m-1}_0 z_k'^{l} [1,e_0])^{\vee}.
\]
It follows that $(\op{ind}_{L_k^\times/Z}^{L^\times/Z} \overline{\sigma}_k)^{\vee}$ is generated by the dual vector $(z_0'^{p^k-1} z_k'^{r}[1,e_0])^{\vee}$.

(i) Let $v_i : = [1-\varpi_L p^k i, X^r] \in  \op{ind}_{L_{k+1}^\times /Z }^{L^\times /Z } \overline{\sigma}_{k+1},~ 0\leq i\leq p-1$.  Then by Lemma \ref{lemma::Phi_L_ram}
\[
\Phi_k^+ ([1 , X^{r-j} Y^j]) = \sum_{i=0}^{p-1} i^j v_i.
\]
Since $v_0,\ldots, v_{p-1}$ are linearly independent over $\FF$, a Vandermonde argument shows that $\Phi_k^+ ([1 , X^{r-j} Y^j ])~,0\leq j\leq r$ are also  linearly independent over $\FF$. Since $\Phi_k^+$ is $L^\times$-equivariant, the same argument proves injectivity on every translate $g[1,\overline{\sigma}_k]$. Moreover, the image of this translate is supported on the $L_k^\times$-coset $gL_k^\times$; hence the images corresponding to distinct elements of $L^\times/L_k^\times$ have disjoint supports. Since
\[
\operatorname{ind}_{L_k^\times/Z}^{L^\times/Z}\overline{\sigma}_k
=\bigoplus_{g\in L^\times/L_k^\times}g[1,\overline{\sigma}_k],
\]
it follows that $\Phi_k^+$ is injective. The $\FF$-linear dual map $(\Phi_k^+)^{\vee}$ is then surjective. Thus $(\Phi_k^+)^{\vee}$ is given by multiplication by a unit $u_k(z_0)$ of $(\FF[t]/t(2 + t))[[z_0]]$. We have
\begin{align*}
(\Phi_k^+)^{\vee} ((z_0'^{p^{k+1}-1} z_{k+1}'^{r}[1,e_0])^{\vee}) (z_0'^{p^k-1} z_k'^{r}[1,e_0]) & = (z_0'^{p^{k+1}-1} z_{k+1}'^{r}[1,e_0])^{\vee} (\Phi_k^+ (z_0'^{p^k-1} z_k'^{r}[1,e_0]))\\
& = (z_0'^{p^{k+1}-1} z_{k+1}'^{r}[1,e_0])^{\vee} (z_0'^{p^k-1} (\Phi_k^+(z_k'^{r}[1,e_0]))) \\
& = (z_0'^{p^{k+1}-1} z_{k+1}'^{r}[1,e_0])^{\vee} (z_0'^{p^k-1} (\Phi_k^+((-1)^r r![1,X^r]))). \\
\end{align*}
Since $\gamma_k^{-i} \equiv 1 - \varpi_L p^k i \pmod{p^{k+1}}$, we have $v_i = \gamma_k^{-i} [1, X^r]$, and by Lemma \ref{lemma::Phi_L_ram}
\begin{equation*} \label{eq::formula_Phi_k^+}
\Phi_k^+ ([1 , X^r]) = \sum_{i=0}^{p-1}   \gamma_k^{-i} [1, X^r] .
\end{equation*} 
It follows that
\begin{align*}
(z_0'^{p^{k+1}-1} z_{k+1}'^{r}[1,e_0])^{\vee} (z_0'^{p^k-1} (\Phi_k^+((-1)^r r![1,X^r]))) & = (z_0'^{p^{k+1}-1} z_{k+1}'^{r}[1,e_0])^{\vee} (  \sum_{i=0}^{p-1} z_0'^{p^k-1} \gamma_k^{-i} [1, (-1)^r r!X^r] )\\
& =  (z_0'^{p^{k+1}-1} z_{k+1}'^{r}[1,e_0])^{\vee} (  \sum_{i=0}^{p-1} z_0'^{p^k-1} (z_k'+1)^{i} z_{k+1}'^r [1, e_0] ).
\end{align*}
Since $(z_0'^{p^{k+1}-1} z_{k+1}'^{r}[1,e_0])^{\vee} (z_0'^m [1,e_0]) =0 $ if $0\leq m < p^{k+1}-1 + p^{k+1}r = p^{k+1}(r+1) -1$, we have
\begin{align*}
(z_0'^{p^{k+1}-1} z_{k+1}'^{r}[1,e_0])^{\vee} (  \sum_{i=0}^{p-1} z_0'^{p^k-1} (z_k'+1)^{i} z_{k+1}'^r [1, e_0] ) & = (z_0'^{p^{k+1}-1} z_{k+1}'^{r}[1,e_0])^{\vee} ( z_0'^{p^k-1} z_k'^{p-1} z_{k+1}'^r [1, e_0] )\\
& = (z_0'^{p^{k+1}-1} z_{k+1}'^{r}[1,e_0])^{\vee} ( z_0'^{p^{k+1}-1}  z_{k+1}'^r [1, e_0] ) \\
& = 1.
\end{align*}
This proves that $u_k(z_0) \in  1 + z_0 \FF[[z_0]]$.

(ii) By Lemma \ref{lemma::Phi_L_ram} and the projection-formula identification \eqref{eq::twist_iso_eta}, for $k\geq 1$ the map $\Phi_k^-$ sends the generator $[1, e_0] \in \op{ind}_{L_k^\times/Z}^{L^\times /Z} \overline{\sigma}_k$ to the generator $[1,e_0] \in \op{ind}_{L_{k-1}^\times/Z}^{L^\times /Z} \overline{\sigma}_{k-1}$. Hence it is surjective. For $m\geq 0$, $i\in\{0,1\}$,
\begin{align*}
(\Phi_k^-)^{\vee} ((z_0'^{p^{k-1} - 1}z_{k-1}'^r [1,e_0])^{\vee}) ((1+t)^i z_0'^{m} [1,e_0]) & =(z_0'^{p^{k-1}(r+1)-1} [1,e_0])^{\vee} (\Phi_k^- ((1+t)^i z_0'^{m} [1,e_0]))\\
& = (z_0'^{p^{k-1}(r+1)-1} [1,e_0])^{\vee} ((1+t)^i z_0'^{m} \Phi_k^-([1,e_0])) \\
& = (z_0'^{p^{k-1}(r+1)-1} [1,e_0])^{\vee} ((1+t)^i z_0'^{m} [1 , e_0]) \\
& = \left\{\begin{array}{cl}
                                                               1,~ & \text{if $m = p^{k-1}(r+1)-1$ and $i=0$,}\\
                                                               0,~ &\text{else},
                                                                \end{array}\right.
\end{align*}
where the third identity follows from Lemma \ref{lemma::Phi_L_ram}. Therefore, 
\[
(\Phi_k^-)^{\vee} ((z_0'^{p^{k-1} -1}z_{k-1}'^r [1,e_0])^{\vee}) =  (z_0'^{p^{k-1}(r+1)-1} [1,e_0])^{\vee} = z_0^{(p^k - p^{k-1})(r+1)} (z_0'^{p^{k}(r+1)-1} [1,e_0])^{\vee}.
\]
This gives (ii).

(iii) By Lemma \ref{lemma::Phi_L_ram} and the projection-formula identification \eqref{eq::twist_iso_eta}, for $m\geq 0$, $i \in \{0,1\}$, we have
\begin{align*}
(\Phi_0^-)^{\vee}( (z_0'^r [1,e_0])^{\vee}) ((1+t)^i z_0'^m [1,e_0]) & =  (z_0'^r [1,e_0])^{\vee} (\Phi_0^- ( (1+t)^i z_0'^m [1,e_0])) \\
& = (z_0'^r [1,e_0])^{\vee} ( \eta(\varpi_L^{i}) u^{-r}(1+t)^{i+1}  z_0'^m [1 ,X^r]) \\
& = (z_0'^r [1,e_0])^{\vee} ( \eta(\varpi_L^{i})u^{-r}(1+t)^{i+1} z_0'^m  [1, (-1)^r z_0'^r e_0 /r!]) \\
& =  (z_0'^r [1,e_0])^{\vee} ( ( \eta(\varpi_L^{i})(-1)^ru^{-r}/r!) (1+t)^{i+1} z_0'^{m +r} [1,  e_0 ])\\
& = \left\{\begin{array}{cl}
                                                               \eta(\varpi_L) (-1)^r u^{-r}/ r!,~ & \text{if $m = 0$ and $i =1$,}\\
                                                               0,~ &\text{else},
                                                                \end{array}\right.
\end{align*}
where the second identity follows from \eqref{eq::twist_iso_eta} and Lemma \ref{lemma::Phi_L_ram}. Therefore, 
\begin{align*}
(\Phi_0^-)^{\vee} ( (z_0'^r [1,e_0])^{\vee}) &= (\eta(\varpi_L) (-1) (-1)^r u^{-r} / r!) ( (1+t) [1,e_0])^{\vee} \\
& = (\eta(\varpi_L) (-1)^r u^{-r}/r!) (1+t)  z_0^{r} (z_0'^{r} [1,e_0])^{\vee}.
\end{align*}
This proves (iii).
\end{proof}

The following two lemmas describe the induced maps of $\Phi^{\pm}_k$ between extension groups. Let $(\Phi^{\pm}_{k})^{*}$ (resp. $(\Phi^{\pm}_{k})_{*}$) denote the map between the extension groups $\op{Ext}^i_{\FF[[L^{\times}/Z]]}(\FF, (\op{ind}^{L^{\times}/Z}_{L_{k}^{\times}/Z} \overline{\sigma}_{k})^{\vee})$ (resp. $\op{Ext}^i_{\FF[[L^{\times}/Z]]}((\op{ind}^{L^{\times}/Z}_{L^{\times}_{k}/Z} \overline{\sigma}_{k})^{\vee},\FF)$) induced by $(\Phi^{\pm}_{k})^{\vee}$.

\begin{lemma} \label{lemma::Phi_Ext^i_L_ram_1}
Under the isomorphisms in Proposition \ref{prop::Ext^i_L_ramified}(i), the induced maps have the following description.
\begin{enumerate}
\item[(i)] For $k\geq 0$, the map 
\[
(\Phi^{+}_{k})^{*} : \op{Ext}^i_{\FF[[L^{\times}/Z]]}(\FF, (\op{ind}^{L^{\times}/Z}_{L^{\times}_{k+1}/Z} \overline{\sigma}_{k+1})^{\vee}) \to  \op{Ext}^i_{\FF[[L^{\times}/Z]]}(\FF, (\op{ind}^{L^{\times}/Z}_{L^{\times}_{k}/Z} \overline{\sigma}_{k})^{\vee})
\]
is zero if $i=0$, and is the identity map if $i=1$.

\item[(ii)] For $k \geq 1$, the map 
\[
(\Phi^{-}_{k} )^{*} : \op{Ext}^i_{\FF[[L^{\times}/Z]]}(\FF, (\op{ind}^{L^{\times}/Z}_{L^{\times}_{k-1}/Z} \overline{\sigma}_{k-1})^{\vee}) \to  \op{Ext}^i_{\FF[[L^{\times}/Z]]}(\FF, (\op{ind}^{L^{\times}/Z}_{L^{\times}_{k}/Z} \overline{\sigma}_{k})^{\vee})
\]
is the identity map if $i=0$, and is zero if $i=1$.

\item[(iii)] If $1\leq r\leq p-1$, then the map 
\[
(\Phi^{-}_{0} )^{*}: \op{Ext}^i_{\FF[[L^{\times}/Z]]}(\FF, (\op{ind}^{L^{\times}/Z}_{L^{\times}_{0}/Z} \overline{\sigma}_{0})^{\vee}) \to  \op{Ext}^i_{\FF[[L^{\times}/Z]]}(\FF, (\op{ind}^{L^{\times}/Z}_{L^{\times}_{0}/Z} \overline{\sigma}_{0})^{\vee})
\]
is zero for $i=0,1$.

\item[(iv)] If $r=0$, then the map 
\[
(\Phi^{-}_{0} )^{*}: \op{Ext}^i_{\FF[[L^{\times}/Z]]}(\FF, (\op{ind}^{L^{\times}/Z}_{L^{\times}_{0}/Z} \overline{\sigma}_{0})^{\vee}) \to  \op{Ext}^i_{\FF[[L^{\times}/Z]]}(\FF, (\op{ind}^{L^{\times}/Z}_{L^{\times}_{0}/Z} \overline{\sigma}_{0})^{\vee})
\]
is multiplication by $\eta(\varpi_L)$ for $i=0,1$.
\end{enumerate}
\end{lemma}
\begin{proof}
All statements follow directly from Proposition \ref{prop::Phi_k_L_ramified} and Proposition \ref{prop::Ext^i_L_ramified}(ii). 
\end{proof}

\begin{lemma} \label{lemma::Phi_Ext^i_L_ram_2}
Under the isomorphisms in Proposition \ref{prop::Ext^i_L_ramified}(ii), the induced maps have the following description.
\begin{enumerate}
\item[(i)] For $k\geq 0$, the map 
\[
(\Phi^{+}_{k})_{*} : \op{Ext}^i_{\FF[[L^{\times}/Z]]}((\op{ind}^{L^{\times}/Z}_{L^{\times}_{k}/Z} \overline{\sigma}_{k})^{\vee},\FF) \to  \op{Ext}^i_{\FF[[L^{\times}/Z]]}((\op{ind}^{L^{\times}/Z}_{L^{\times}_{k+1}/Z} \overline{\sigma}_{k+1})^{\vee},\FF)
\]
is the identity map if $i=0$, and is zero if $i=1$.

\item[(ii)] For $k \geq 1$, the map 
\[
(\Phi^{-}_{k} )_{*} : \op{Ext}^i_{\FF[[L^{\times}/Z]]}( (\op{ind}^{L^{\times}/Z}_{L^{\times}_{k}/Z} \overline{\sigma}_{k})^{\vee},\FF) \to  \op{Ext}^i_{\FF[[L^{\times}/Z]]}( (\op{ind}^{L^{\times}/Z}_{L^{\times}_{k-1}/Z} \overline{\sigma}_{k-1})^{\vee},\FF)
\]
is zero if $i=0$, and is the identity map if $i=1$.

\item[(iii)] If $1\leq r\leq p-1$, then the map 
\[
(\Phi^{-}_{0} )_{*}: \op{Ext}^i_{\FF[[L^{\times}/Z]]}( (\op{ind}^{L^{\times}/Z}_{L^{\times}_{0}/Z} \overline{\sigma}_{0})^{\vee},\FF) \to  \op{Ext}^i_{\FF[[L^{\times}/Z]]}( (\op{ind}^{L^{\times}/Z}_{L^{\times}_{0}/Z} \overline{\sigma}_{0})^{\vee},\FF)
\]
is zero for $i=0,1$.

\item[(iv)] If $r=0$, then the map 
\[
(\Phi^{-}_{0} )_{*}: \op{Ext}^i_{\FF[[L^{\times}/Z]]}((\op{ind}^{L^{\times}/Z}_{L^{\times}_{0}/Z} \overline{\sigma}_{0})^{\vee},\FF) \to  \op{Ext}^i_{\FF[[L^{\times}/Z]]}((\op{ind}^{L^{\times}/Z}_{L^{\times}_{0}/Z} \overline{\sigma}_{0})^{\vee},\FF)
\]
is multiplication by $\eta(\varpi_L)$ for $i=0,1$.
\end{enumerate}
\end{lemma}
\begin{proof}
Let $A$ denote the $\FF$-algebra $\FF[t]/t(2 + t)$. Then $( \op{ind}^{L^{\times}/Z}_{L_k^{\times}/Z} \overline{\sigma}_k)^{\vee} \simeq A[[z_0]]/z_0^{p^k(r+1)}$ by Lemma \ref{lemma::structure_sigma_k}(iii). We choose a projective resolution of $A[[z_0]]/z_0^{p^k(r+1)}$ in ${\rm Mod}^{\rm pro}_{L^{\times}/Z}(\FF)$:
\[
0\to P_{k,-1} :=A[[z_0]] \xrightarrow{z_0^{p^k(r+1)} } P_{k,0}:= A[[z_0]] \to A[[z_0]]/z_0^{p^k(r+1)} \to 0.
\] 
To compute $(\Phi_k^{\pm})_*$, we first lift the map $(\Phi_k^{\pm})^{\vee}$ to a map between projective resolutions, and then the maps $(\Phi_k^{\pm})_*$ are induced by the lifting maps (see \cite[Lemma 2.4.4]{Weibel}).

(i) By Proposition \ref{prop::Phi_k_L_ramified}(i), the map $(\Phi_k^+)^{\vee} : A[[z_0]]/ z_0^{p^{k+1}(r+1)} \to  A[[z_0]]/z_0^{p^k(r+1)}$ is multiplication by $u_k(z_0)$. This map clearly lifts to a map between projective resolutions by
\[
P_{k+1,0} \xrightarrow{ u_k(z_0) } P_{k,0},\quad P_{k+1,-1} \xrightarrow{u_k(z_0) z_0^{(p^{k+1} - p^k) (r+1)} } P_{k,-1}.
\]
Then (i) follows from Proposition \ref{prop::Ext^i_L_ramified}(i).

(ii) By Proposition \ref{prop::Phi_k_L_ramified}(ii), the map  $(\Phi_k^-)^{\vee} : A[[z_0]]/ z_0^{p^{k-1}(r+1)} \to  A[[z_0]]/z_0^{p^k(r+1)}$ is multiplication by $z_0^{(p^k - p^{k-1})(r+1)}$, and it lifts to a map between projective resolutions by
\[
P_{k-1,0} \xrightarrow{z_0^{(p^k - p^{k-1})(r+1)} } P_{k,0},\quad P_{k-1,-1} \xrightarrow{{\rm id} } P_{k,-1}.
\]
Then (ii) follows from Proposition \ref{prop::Ext^i_L_ramified}(i).

(iii) and (iv) The map $(\Phi_0^-)^{\vee} : A[[z_0]]/ z_0^{r+1} \to  A[[z_0]]/z_0^{r+1}$, given in Proposition \ref{prop::Phi_k_L_ramified}(iii),  also lifts to a map between projective resolutions
\[
P_{0,0} \xrightarrow{((-1)^r u^{-r}/r!) (1+t)  z_0^{r} } P_{0,0},\quad P_{0,-1} \xrightarrow{((-1)^r u^{-r}/r!) (1+t)  z_0^{r}} P_{0,-1}.
\]
Then (iii) and (iv) also follow from Proposition \ref{prop::Ext^i_L_ramified}(i).
\end{proof}

\begin{theorem} \label{thm::branching_ramified}
Let $\pi$ be a supersingular representation of ${\rm GL}_2(\QQ_p)$ with central character $\zeta$. Let $\eta: L^{\times} \to \FF^{\times}$ be a smooth character such that $\eta|_{Z} = \zeta$. Then there are quasi-isomorphisms
\[
\op{RHom}_{L^\times,\zeta}(\pi|_{L^\times},\eta) \simeq \FF^{\oplus 2}[-1],\qquad \op{RHom}_{L^\times,\zeta}( \eta, \pi|_{L^\times}) \simeq \FF^{\oplus 2} [0].
\]
\end{theorem}

\begin{proof}
By Lemma \ref{lemma::concentrate_one_deg}, to prove the first quasi-isomorphism, it suffices to prove the following results on Ext-groups:
\[
\op{dim}_{\FF}\op{Ext}^{i}_{L^\times,\zeta}(\pi|_{L^\times},\eta)=\left\{\begin{array}{cl}
        2,& \text{ if } i=1,\\
        0,& \text{ otherwise} .
        \end{array}\right.
\]

By twisting, we can assume $\pi = \pi(r,0) = \op{ind}_{KZ}^{G}(\sigma) /\Phi$ with $\sigma := \op{Sym}^{r}(\FF^2)$ and $0\leq r\leq p-1$.  The short exact sequence 
\begin{align*}
        0\to\op{ind}_{KZ}^{G}(\sigma)\xrightarrow{\Phi }\op{ind}_{KZ}^{G}(\sigma)\to \pi\to 0 
    \end{align*}
induces a long exact sequence
\[
\cdots \to \op{Ext}^{i}_{L^\times,\zeta}(\pi|_{L^\times},\eta) \to \op{Ext}^{i}_{L^\times,\zeta}(\op{ind}_{KZ}^{G}(\sigma),\eta) \xrightarrow{\Phi_i^{*} } \op{Ext}^{i}_{L^\times,\zeta}(\op{ind}_{KZ}^{G}(\sigma),\eta) \to \cdots.
\]
By \eqref{eq::Mackey_1} and \eqref{eq::Ext_reduce_mod_center}, we are reduced to understand the map
\[
\Phi_i^{*} :  \prod_{k\geq 0} \mathfrak{C}^i_k \to \prod_{k\geq 0}  \mathfrak{C}^i_k,
\] 
where $\mathfrak{C}^i_k : =  \op{Ext}^{i}_{L^\times/Z}( \op{ind}^{L^{\times}/Z}_{L_k^{\times}/Z} \overline{\sigma}_k, \FF) \simeq \op{Ext}^{i}_{\FF[[L^\times/Z ]] }(\FF, (\op{ind}^{L^{\times}/Z}_{L_k^{\times}/Z} \overline{\sigma}_k)^{\vee})$. By Proposition \ref{prop::Ext^i_L_ramified}(ii), $\mathfrak{C}^i_k $ is one-dimensional if $i=0,1$, and is zero if $i\geq 2$. The map $\Phi_i^*$ for $i = 0,1$ is given by the maps $(\Phi_k^{\pm})^{*}$ in Lemma \ref{lemma::Phi_Ext^i_L_ram_1}.

If $1\leq r \leq p-1$, then it follows from Lemma \ref{lemma::Phi_Ext^i_L_ram_1}(i)-(iii) that the maps $\Phi_0^{*}$ and $\Phi_1^{*}$ are given by
\begin{align*}
\Phi_0^{*}& : \prod_{k\geq 0} \mathfrak{C}^0_k \to \prod_{k\geq 0} \mathfrak{C}^0_k ,~ (a_0,a_1,a_2,\ldots) \mapsto (0,a_0,a_1,a_2,\ldots),\\
\Phi_1^{*}& : \prod_{k\geq 0} \mathfrak{C}^1_k \to \prod_{k\geq 0} \mathfrak{C}^1_k ,~ (a_0,a_1,a_2,a_3,\ldots) \mapsto (a_1,a_2,a_3,\ldots).
\end{align*}
Hence $\ker \Phi_0^{*} = {\rm coker} (\Phi_1^{*}) = 0$, and both ${\rm coker} (\Phi_0^{*})$ and $\ker \Phi_1^{*}$ are one-dimensional. This proves the Ext-dimension formula for $1\leq r\leq p-1$.

When $r=0$, it follows from Lemma \ref{lemma::Phi_Ext^i_L_ram_1}(i), (ii), and (iv) that the maps $\Phi_0^{*}$ and $\Phi_1^{*}$ are given by
\begin{align*}
\Phi_0^{*}& : \prod_{k\geq 0} \mathfrak{C}^0_k \to \prod_{k\geq 0} \mathfrak{C}^0_k ,~ (a_0,a_1,a_2,\ldots) \mapsto (\eta(\varpi_L)a_0,a_0,a_1,a_2,\ldots),\\
\Phi_1^{*}& : \prod_{k\geq 0} \mathfrak{C}^1_k \to \prod_{k\geq 0} \mathfrak{C}^1_k ,~ (a_0,a_1,a_2,a_3,\ldots) \mapsto (\eta(\varpi_L)a_0+a_1,a_2,a_3,\ldots).
\end{align*}
We also have $\ker \Phi_0^{*} = {\rm coker} (\Phi_1^{*}) = 0$, and both ${\rm coker} (\Phi_0^{*})$ and $\ker \Phi_1^{*}$ are one-dimensional. This finishes the computation of $\dim \op{Ext}^{i}_{L^\times,\zeta}(\pi|_{L^\times},\eta)$, and implies the first quasi-isomorphism.

Now we prove the second quasi-isomorphism. By Lemma \ref{lemma::concentrate_one_deg} again, it suffices to show
\[
\op{dim}_{\FF}\op{Ext}^{i}_{L^\times,\zeta}(\eta, \pi|_{L^\times})=\left\{\begin{array}{cl}
        2,& \text{ if } i=0,\\
        0,& \text{ otherwise}. 
        \end{array}\right.
\]

Since $\eta$ is one-dimensional, we have
\[
\op{Ext}^{i}_{L^\times,\zeta}(\eta,\op{ind}_{KZ}^{G}(\sigma)) = \bigoplus_{k\geq 0} \op{Ext}^{i}_{L^\times,\zeta}(\eta,\op{ind}_{L_k^{\times}}^{L^{\times}}(\sigma_k)) .
\]
As above, we are then reduced to understanding the following map $(\Phi_i)_*$ induced by $\Phi$
\[
(\Phi_i)_* :  \bigoplus_{k\geq 0} \op{Ext}^{i}_{\FF[[L^\times/Z]]}( (\op{ind}^{L^{\times}/Z}_{L_k^{\times}/Z} \overline{\sigma}_k)^{\vee}, \FF) \to \bigoplus_{k\geq 0} \op{Ext}^{i}_{\FF[[L^\times/Z]]}( ( \op{ind}^{L^{\times}/Z}_{L_k^{\times}/Z} \overline{\sigma}_k)^{\vee}, \FF),~ i=0,1.
\] 
Let $\mathfrak{C}^i_k : = \op{Ext}^{i}_{\FF[[L^\times/Z]]}( ( \op{ind}^{L^{\times}/Z}_{L_k^{\times}/Z} \overline{\sigma}_k)^{\vee}, \FF)$. Then we can write
\[
(\Phi_i)_*|_{\mathfrak{C}_k^i} = ((\Phi_k^-)_*, (\Phi_k^+)_*): \mathfrak{C}_k^i \to \mathfrak{C}_{k-1}^i \oplus \mathfrak{C}_{k+1}^i,
\] 
where the maps $(\Phi_k^{\pm})_*$ are given by Lemma \ref{lemma::Phi_Ext^i_L_ram_2}. The description of $(\Phi_i)_*$ is as follows. If $1\leq r \leq p-1$,
\begin{align*}
(\Phi_0)_*& : \bigoplus_{k\geq 0} \mathfrak{C}^0_k \to \bigoplus_{k\geq 0} \mathfrak{C}^0_k ,~ (a_0,a_1,a_2,\ldots) \mapsto (0,a_0,a_1,a_2,\ldots),\\
(\Phi_1)_*& : \bigoplus_{k\geq 0} \mathfrak{C}^1_k \to \bigoplus_{k\geq 0} \mathfrak{C}^1_k ,~ (a_0,a_1,a_2,a_3,\ldots) \mapsto (a_1,a_2,a_3,\ldots);
\end{align*}
and if $r=0$,
\begin{align*}
(\Phi_0)_*& : \bigoplus_{k\geq 0} \mathfrak{C}^0_k \to \bigoplus_{k\geq 0} \mathfrak{C}^0_k ,~ (a_0,a_1,a_2,\ldots) \mapsto (\eta (\varpi_L)a_0,a_0,a_1,a_2,\ldots),\\
(\Phi_1)_*& : \bigoplus_{k\geq 0} \mathfrak{C}^1_k \to \bigoplus_{k\geq 0} \mathfrak{C}^1_k ,~ (a_0,a_1,a_2,a_3,\ldots) \mapsto (\eta(\varpi_L)a_0+a_1,a_2,a_3,\ldots).
\end{align*}
Hence we always have $\ker (\Phi_0)_* = {\rm coker} (\Phi_1)_* = 0$, and both ${\rm coker} (\Phi_0)_*$ and $\ker (\Phi_1)_*$ are one-dimensional. This finishes the computation of $\op{dim}_{\FF}\op{Ext}^{i}_{L^\times,\zeta}(\eta, \pi|_{L^\times})$, and implies the second quasi-isomorphism.
\end{proof}

\subsubsection{The unramified case}\label{subsection::unramified} 
Assume $p\geq 3$. Let $L$ be the unramified quadratic extension of $\QQ_p$ in $\overline{\QQ}_p$. 
Choose a nonsquare unit $\delta\in \ZZ_p \backslash (\ZZ_p^{\times})^2$ such that $\mathcal O_L=\ZZ_p[\sqrt{\delta}]$. 
With the choice of a $\ZZ_p$-basis $\{\sqrt{\delta},1\}$ of $\mathcal O_L$, we have an embedding
\[\iota(a+\sqrt{\delta} b)=\left(\begin{array}{cl}
    a & b \\
    \delta b & a
    \end{array}\right),\]
for any $a,b\in \ZZ_p$. 
\begin{lemma}
    A set of representatives of $KZ\backslash G/L^\times$ is given by 
    \[\left\{\left(\begin{array}{cl}
        p^{-k} & 0 \\
        0 & 1
        \end{array}\right),\ k\in\ZZ_{\geq 0}\right\}. \]
\end{lemma}
\begin{proof}
Since $G = BL^\times$, one has a surjective map of quotient sets
        \[ 
        \pi: (KZ \cap B) \backslash B \twoheadrightarrow KZ \backslash BL^\times / L^\times. 
        \]
        A set of representatives of $(KZ \cap B) \backslash B$ is given by $\left( \begin{smallmatrix} p^{\ZZ} & \mathbb{Q}_p/\mathbb{Z}_p \\ 0 & 1 \end{smallmatrix} \right)$.
        
        First, we show that for any $\beta_1 = \left( \begin{smallmatrix} p^{k_1} & n_1 \\ 0 & 1 \end{smallmatrix} \right)$ with $\operatorname{val}_p(n_1) < 0$, there is $\beta_2 = \left( \begin{smallmatrix} p^{k_2} & 0 \\ 0 & 1 \end{smallmatrix} \right)$ such that $\pi(\beta_1) = \pi(\beta_2)$. If $\operatorname{val}_p(n_1) \ge k_1$, we take $k_2 = k_1$, and we take $l = 1 - p^{-k_1}n_1\sqrt{\delta} \in \mathcal{O}_L^\times$, then
        \begin{equation}\label{eq::representative_3}
            \beta_1 \iota(l) \beta_2^{-1} = \begin{pmatrix} 1 - p^{-2k_1}\delta n_1^2 & 0 \\ -\delta p^{-2k_1}n_1 & 1 \end{pmatrix} \in K. 
        \end{equation}        
        If $\operatorname{val}_p(n_1) \le k_1 - 1$, we take $k_2 = k_1 - 2\operatorname{val}_p(n_1)$, and we take $l = \delta n_1^{-1} - p^{k_1}n_1^{-2}\sqrt{\delta} \in L^\times$, then
        \begin{equation}\label{eq::representative_4}
            \beta_1 \iota(l) \beta_2^{-1} = \begin{pmatrix} 0 & \delta - p^{2k_1}n_1^{-2} \\ -\delta p^{2\operatorname{val}_p(n_1)}n_1^{-2} & \delta n_1^{-1} \end{pmatrix} \in K. 
        \end{equation}
        Now assume $\beta_i = \left( \begin{smallmatrix} p^{k_i} & 0 \\ 0 & 1 \end{smallmatrix} \right)$, $i = 1, 2$ such that $\pi(\beta_1) = \pi(\beta_2)$. Then there is $l = a + b\sqrt{\delta} \in L^\times$ such that
        \[ 
        \beta_1 \iota(l) \beta_2^{-1} = \begin{pmatrix} p^{k_1-k_2}a & p^{k_1}b \\ p^{-k_2}\delta b & a \end{pmatrix} \in K. 
        \]
        
        Since $\delta$ is a nonsquare unit, either $p^{k_1-k_2}a^2$ or $p^{k_1-k_2}\delta b^2$ must lie in $\mathbb{Z}_p^\times$. If $p^{k_1-k_2}a^2\in \mathbb{Z}_p^\times$, since both $p^{k_1-k_2}a$ and $a$ are in $\mathbb{Z}_p$, we must have $a \in \mathbb{Z}_p^\times$ and $k_1 = k_2$. If $p^{k_1-k_2}\delta b^2\in \mathbb{Z}_p^\times$, we must have $p^{k_1}b\in \mathbb{Z}_p^\times$ and $k_1+k_2=0$. This finishes the proof of the lemma. 
\end{proof}
For $k\in\ZZ_{\geq 0}$, let $s_k:=\left(\begin{smallmatrix}
    p^{-k} & 0 \\
    0 & 1
    \end{smallmatrix}\right)$
and $L_k^\times:=L^\times \bigcap s_{k}^{-1}Ks_{k}Z$. Let $\sigma_k:=(s_k^{-1})_{*}\op{Res}^{KZ}_{s_kL^\times s_k^{-1}\cap KZ}\sigma$; this is a representation of $L_k^\times$. By Mackey's formula \cite[Thm. 1.1]{Yamamoto}, we have $L^{\times}$-equivariant decompositions
\begin{equation} \label{eq::Mackey_2}
(\op{ind}_{KZ}^{G}\sigma)|_{L^\times}\cong \bigoplus_{s_k \in KZ\backslash G/L^\times}\op{ind}_{KZ}^{KZs_k L^\times}\sigma\cong \bigoplus_{k\geq 0}\op{ind}_{L_k^\times}^{L^\times}\sigma_k,
\end{equation}
where the second isomorphism is given by $[s_k^{-1},v]\mapsto [1,v]$, $v\in \sigma$.

The following lemma describes the restriction of the Hecke operator $\Phi$ to $ \op{ind}_{L_k^\times}^{L^\times}\sigma_k$. 

\begin{lemma}\label{lemma::Phi_L_unr}
    For $k\in \ZZ_{\geq 0}$,  denote by  $\Phi_k:=\Phi|_{\op{ind}_{L_k^\times}^{L^\times}\sigma_k}$. Then for $k\geq 1$, we can write $\Phi_k= (\Phi^{+}_k,  \Phi^{-}_k)$; and the morphisms $\Phi^{\pm}_k$ and $\Phi_0$ are given by 
    \begin{align*}
        \Phi^{+}_k:\op{ind}_{L_k^\times}^{L^\times}\sigma_k&\to \op{ind}_{L_{k+1}^\times}^{L^\times}\sigma_{k+1},~ k\geq 1 \\
        [1,X^{r-j}Y^j]&\mapsto \sum_{i=0}^{p-1}[1-\sqrt{\delta} p^{k}i, i^j X^r] ;\\
        \Phi^{-}_k:\op{ind}_{L_k^\times}^{L^\times}\sigma_k&\to \op{ind}_{L_{k-1}^\times}^{L^\times}\sigma_{k-1},~ k\geq 1 \\
        [1,X^{r-j}Y^j]&\mapsto [1, \delta_{j,r}Y^r]; \\  
        \Phi_0:\sigma_0&\to \op{ind}_{L_1^\times}^{L^\times}\sigma_1\\
        [1,(X+\sqrt{\delta}Y)^{r-j}(X-\sqrt{\delta}Y)^{j}]&\mapsto \sum_{i=0}^{p-1}[1-\sqrt{\delta} i, (-1)^j(1+i\sqrt\delta)^{-j}(1-i\sqrt\delta)^{j-r} X^r]+[\sqrt\delta,(-1)^j(\sqrt \delta)^{-r}X^r]. 
    \end{align*}
\end{lemma}
\begin{proof}
    The map $\Phi_k$ is determined by $\Phi ([s_k^{-1}, v ]),~v\in \sigma$. By \eqref{eq::formula_Phi}, we have 
\[\Phi\left(\left[\left(\begin{array}{cl}
    p^{-k} & 0 \\
    0 & 1
    \end{array}\right)^{-1},X^{r-j}Y^{j}\right]\right)=\left[\left(\begin{array}{cl}
        p^{-k+1} & 0 \\
        0 & 1
        \end{array}\right)^{-1},\delta_{j,r}Y^r\right]+\sum_{i=0}^{p-1}\left[\left(\begin{array}{cl}
        p^{-k-1} & p^{-1}i \\
        0 & 1
        \end{array}\right)^{-1},i^jX^r\right]. \]
    For $k \geq 1$, let $\Phi_k^{\pm}$ be given by: 
    \[\Phi_k^{+}\left(\left[\left(\begin{array}{cl}
        p^{-k} & 0 \\
        0 & 1
        \end{array}\right)^{-1},X^{r-j}Y^{j}\right]\right)=\sum_{i=0}^{p-1}\left[\left(\begin{array}{cl}
            p^{-k-1} & p^{-1}i \\
            0 & 1
            \end{array}\right)^{-1},i^jX^r\right],\]
    \[\Phi_k^{-}\left(\left[\left(\begin{array}{cl}
        p^{-k} & 0 \\
        0 & 1
        \end{array}\right)^{-1},X^{r-j}Y^{j}\right]\right)=\left[\left(\begin{array}{cl}
            p^{-k+1} & 0 \\
            0 & 1
            \end{array}\right)^{-1},\delta_{j,r}Y^r\right].\]
    Since $\vv(p^{-1} i )\geq -k-1$, by \eqref{eq::representative_3}, we have 
    \begin{equation}\label{eq::representative_unr}
        \begin{aligned}
            \left(\begin{array}{cl}
                p^{-k-1} & p^{-1}i \\
                0 & 1
                \end{array}\right)^{-1}&=\left(\begin{array}{cl}
                1 & -p^{k}i \\
                -p^{k}i\delta & 1
                \end{array}\right)\left(\begin{array}{cl}
                p^{-k-1} & 0 \\
                0 & 1
                \end{array}\right)^{-1} \left(\begin{array}{cl}
                1-p^{2k}i^2\delta & 0 \\
                -p^{1+2k}i\delta & 1
                \end{array}\right)^{-1}\\
                &=\iota(1-p^ki\sqrt{\delta})s_{k+1}^{-1}\left(\begin{array}{cl}
                    1-p^{2k}i^2\delta & 0 \\
                    -p^{1+2k}i\delta & 1
                    \end{array}\right)^{-1},
        \end{aligned}
    \end{equation}
    where $\left(\begin{smallmatrix}
                    1-p^{2k}i^2\delta & 0 \\
                    -p^{1+2k}i\delta & 1
                    \end{smallmatrix}\right)\in K$ and $\left(\begin{smallmatrix}
                    1-p^{2k}i^2\delta & 0 \\
                    -p^{1+2k}i\delta & 1
                   \end{smallmatrix}\right)^{-1}X^r=(1-p^{2k}i^2\delta)^{-r}X^r$. 
    Thus this gives the formula for $\Phi_k^{+}$ using the isomorphisms in \eqref{eq::Mackey_2}. 
    The formula for $\Phi_k^{-},\ k\geq 1$ follows directly from the definition. 

    For $\Phi_0$, by \eqref{eq::formula_Phi} we have 
    \begin{align*}
        &\Phi\left(\left[\left(\begin{array}{cl}
            1 & 0 \\
            0 & 1
            \end{array}\right)^{-1},(X+\sqrt{\delta}Y)^{r-j}(X-\sqrt{\delta}Y)^{j}\right]\right)\\
            =&\left[\left(\begin{array}{cl}
                p^{1} & 0 \\
                0 & 1
                \end{array}\right)^{-1},(-1)^j(\sqrt \delta)^r Y^r\right]+\sum_{i=0}^{p-1}\left[\left(\begin{array}{cl}
                p^{-1} & p^{-1}i \\
                0 & 1
                \end{array}\right)^{-1},(i\sqrt\delta+1)^{r-j}(-i\sqrt\delta+1)^j X^r\right]. 
    \end{align*}
    Then the formula for $\Phi_0$ follows from the equation \eqref{eq::representative_unr} for $k=0$ and the identity
    $$\left(\begin{array}{cl}
        p & 0 \\
        0 & 1
        \end{array}\right)^{-1}= \left(\begin{array}{cl}
            0 & 1 \\
            \delta & 0
            \end{array}\right)\left(\begin{array}{cl}
                p^{-1} & 0 \\
                0 & 1
                \end{array}\right)^{-1} \left(\begin{array}{cl}
                    0 & p \\
                    p\delta & 0
                    \end{array}\right)^{-1} ,$$
    where $\left(\begin{smallmatrix}
                    0 & p \\
                    p\delta & 0
                    \end{smallmatrix}\right)^{-1}\in KZ$ and $\left(\begin{smallmatrix}
                    0 & p \\
                    p\delta & 0
                    \end{smallmatrix}\right)^{-1}\cdot Y^r=\delta^{-r}X^r$.  
\end{proof}

 Let $\omega_2:L^\times \to \FF^\times$ be the character sending $x\in L^\times$ to $x p^{-v_L(x)} \pmod{p}$.

\begin{lemma}\label{lemma::action_of_L_k_unr}
    \begin{enumerate}
        \item[(i)]One has
        $$L^\times= L_0^\times\cong p^{\ZZ}\times \mu_{p^2-1}\times(1+p\ZZ_p+\sqrt{\delta}p\ZZ_p); $$
        $$L_k^\times\cong (1+\sqrt{\delta}p^k\ZZ_p)\cdot Z,\ k\in\ZZ_{>0},$$
        where $\mu_{p^2-1}$ denotes the group of $(p^2-1)$-st roots of unity. 
        \item[(ii)] The $L_0^\times$-representation $\sigma_0$ is completely reducible and it admits a decomposition 
        \[ 
        \sigma_0\cong\bigoplus_{j=0}^r \omega_2^{r+(p-1)j}. 
        \]
        where $(X+\sqrt{\delta}Y)^{r-j}(X-\sqrt{\delta}Y)^{j}$ is an $\omega_2^{r+(p-1)j}$-eigenvector of $L^{\times}$. 
        \item[(iii)]If $k\geq 1$, the action of $L_k^{\times}$ on $\sigma_k$ is given by 
        \[
        (1 + \sqrt{\delta} p^k b) p^n a \cdot X^{r-j} Y^j = \omega^r(a) X^{r-j} (bX + Y)^j,~ a\in \ZZ_p^{\times},~b\in \ZZ_p, ~n\in \ZZ. 
        \]
    \end{enumerate}
\end{lemma}
\begin{proof}
    The proof of (i) and (iii) is similar to that of Lemma \ref{lemma::action_of_L_k}. 
    (ii) follows from 
    $$\sigma_0\cong\bigoplus_{j=0}^{r}\FF(X+\sqrt{\delta}Y)^{r-j}(X-\sqrt{\delta}Y)^{j} $$
    and $L_0^\times$ acts on $(X+\sqrt{\delta}Y)^{r-j}(X-\sqrt{\delta}Y)^{j}$ via the character $\omega_2^{r-j}(\omega_2^{p})^{j}$. 
\end{proof}

Let $\eta: L^{\times} \to \FF^{\times}$ be a smooth character such that $\eta|_{Z} = \omega^r = : \zeta$. We write 
\begin{equation}\label{eq::eta_L_unr}
    \eta = \omega_2^{r+(p-1)n},~ \text{for some  $n  \in \{ 0,\ldots, p\}$}.
    \end{equation}
Then  $\zeta$ is the central character of $\op{ind}^{L^{\times}}_{L_k^{\times}} \sigma_k$. The representation $\sigma_k\otimes \eta^{-1}|_{L_k^{\times}}$ has trivial $Z$-action. Thus it is naturally a $L_k^{\times}/Z$-representation. By \eqref{eq::isom_ext_twist}, we have
\begin{equation} \label{eq::Ext_reduce_mod_center_unr}
\op{Ext}^i_{L^{\times},\zeta } (\op{ind}^{L^{\times}}_{L_k^{\times}} \sigma_k ,\eta) \simeq \op{Ext}^i_{L^{\times}/Z} (\op{ind}^{L^{\times}/Z}_{L_k^{\times}/Z} (\sigma_k \otimes\eta^{-1}|_{L_k^{\times}}), \FF ) ,~ \forall i,
\end{equation}
and
\[
\op{Ext}^i_{L^{\times},\zeta } (\eta, \op{ind}^{L^{\times}}_{L_k^{\times}} \sigma_k) \simeq \op{Ext}^i_{L^{\times}/Z} (\FF, \op{ind}^{L^{\times}/Z}_{L_k^{\times}/Z} (\sigma_k \otimes\eta^{-1}|_{L_k^{\times}}) ) ,~ \forall i. 
\]
Let $\overline{\sigma}_k$ denote the $L_k^{\times}/Z$-representation $\sigma_k \otimes\eta^{-1}|_{L_k^{\times}}$. Taking $\FF$-linear dual induces isomorphisms
\begin{align*}
\op{Ext}^i_{L^{\times}/Z} ( \op{ind}^{L^{\times}/Z}_{L_k^{\times}/Z} \overline{\sigma}_k,\FF ) & \simeq \op{Ext}^i_{ \FF[[L^{\times}/Z]]} (\FF, ( \op{ind}^{L^{\times}/Z}_{L_k^{\times}/Z} \overline{\sigma}_k)^{\vee} )\\
\op{Ext}^i_{L^{\times}/Z} (\FF, \op{ind}^{L^{\times}/Z}_{L_k^{\times}/Z} \overline{\sigma}_k ) & \simeq \op{Ext}^i_{ \FF[[L^{\times}/Z]]} (( \op{ind}^{L^{\times}/Z}_{L_k^{\times}/Z} \overline{\sigma}_k)^{\vee}, \FF ),~ \forall i.
\end{align*}

Choose a primitive Teichm\"{u}ller root $\xi\in\mu_{p^2-1}$ and write $\overline{\xi}$ for its class in $L^\times/Z$. 
Then $\overline{\xi}$ has order $p+1$. Let $C:=\langle\overline{\xi}\rangle$.
Then 
$$L^\times/Z\cong C\times L_1^\times/Z$$
and the group algebra $\FF[C]$ is isomorphic to $\FF[t']/((1+t')^{p+1} - 1)$ by identifying $\overline{\xi}^{-1}$ with $1+t'$. Let $t:= \overline{\xi} - 1$. Then  $\FF[C]$ is also isomorphic to $\FF[t]/((1+t)^{p+1} - 1)$. Let $\tau$ be an $\FF[C]$-module, and let $\tau^{\vee}$ be its $\FF$-linear dual. Then for $f\in \tau^{\vee},~ v\in \tau$, we have
\[
(1+t) f (v) = f(\overline{\xi}^{-1} v) = f( (1+t') v).
\]

For $k\geq 1$, let $\gamma_k : = (1 + \sqrt \delta p )^{p^{k-1}} \in L_k^{\times} $. 
\begin{lemma}\label{lemma::p-torsion_free_unramified}
\begin{itemize}
\item[(i)] The quotient group $L_k^{\times}/Z$ is isomorphic to $\ZZ_p$, for all $k\geq 1$.
    
\item[(ii)] The composite $i: \gamma_k^{\ZZ_p} \to L_k^{\times} \to  L_k^{\times}/Z$ is an isomorphism of $p$-adic Lie groups.
\end{itemize}
\end{lemma}
\begin{proof}
    Since $p\geq 3$ and the ramification degree $e_L$ of $L$ is 1, one has $\lfloor \frac{e_L}{p-1} \rfloor+1=1$. 
    By \cite[Prop. 4.3.12]{Coh07}, the $p$-adic logarithm map gives an isomorphism between $1+p^k \mc{O}_L$ and $p^k \mc{O}_L$, for all $k \geq 1$. 
    The same map also gives an isomorphism between $1+p^k \ZZ_p$ and $p^k \ZZ_p$, for all $k\geq 1$. Since $Z \cap (1 + p^{k} \mc{O}_L) = 1+p^{k} \ZZ_p$ and $\mc O_L\cong \ZZ_p+\sqrt{\delta}\ZZ_p$, the $p$-adic logarithm map induces an isomorphism
\[
L_k^{\times}/Z = (1 + p^{k} \mc{O}_L) Z / Z \simeq   p^{k} \mc{O}_L / p^{k} \ZZ_p \simeq \sqrt{\delta}p^{k} \ZZ_p.
\]
This proves (i).

The proof of (ii) is similar to that of Lemma \ref{lemma::p-torsion_free_ramified}(ii). 
\end{proof}
As a consequence of Lemma \ref{lemma::p-torsion_free_unramified}, 
\[
\FF[[L_k^{\times}/Z]] \simeq \FF[[z_k]],
\]
with $z_k : = \gamma_k - 1$ and $k\geq 1$.

We define a map $\alpha : \gamma_k^{\ZZ_p} \to U(\FF_p)$ by sending $\gamma_k$ to $\bigl(\begin{smallmatrix}
    1 & 1 \\ 0& 1
    \end{smallmatrix} \bigr) \in U(\FF_p)$. By Lemma \ref{lemma::structure_sigma_dual}, $\sigma$ is a representation of $\gamma_k^{\ZZ_p}$ via the map $\alpha$. Let $e_0 : = Y^r \in \overline{\sigma}_k$, $e_i : = z_k'^i e_0$, where $z_k': = \gamma_k^{-1} - 1$. In particular, $e_r = (-1)^r r! X^r$. Let $e_0^{\vee},\ldots, e_r^{\vee}$ be the dual basis of $\overline{\sigma}_k^{\vee}$, the $\FF$-linear dual of $\overline{\sigma}_k$. We have
    \[
    (z_k f) (v) = f(z_k'v),~ f\in \overline{\sigma}_k^{\vee},~v\in \overline{\sigma}_k.
    \]
Let 
\[
A : =\FF[C]= \FF[t]/(1+t)^{p+1} - 1.
\]

\begin{lemma} \label{lemma::structure_sigma_k_L_unr}
\begin{enumerate}
 \item [(i)] The $L_0^\times/Z$-representation $\overline{\sigma}_0$ is completely reducible and it admits a decomposition 
        \begin{equation} \label{eq::dec_L_0_sigma_0}
        \overline{\sigma}_{0}\cong\bigoplus_{j=0}^{r} \omega_2^{(p-1)(j-n)}.
        \end{equation}
\item[(ii)] For $k\geq 1$, the inclusion $\FF [[L_{k+1}^{\times}/Z ]] \to \FF [[L_k^{\times} /Z]] $ corresponds to the map $\FF[[z_{k+1}]] \to \FF[[z_k]]$ sending $z_{k+1}$ to $z_k^p$.

\item[(iii)] For $k\geq 1$, the $\FF[[L_k^{\times}/Z ]]$-module $\overline{\sigma}_k^{\vee}$ is cyclic, generated by $e_r^{\vee} $. In particular, $\overline{\sigma}_k^{\vee}$ is isomorphic to $\FF[[ z_k]]/  z_k^{r+1} $.

\item[(iv)] The $\FF$-algebra $\FF[[L^{\times} /Z]]$ is isomorphic to $A[[z_1]] $. And we have an isomorphism of $\FF[[L^{\times} /Z]]$-modules 
\[
( \op{ind}^{L^\times/Z }_{L_k^\times /Z } \overline{\sigma}_k)^{\vee} \simeq  \left\{\begin{array}{cl}
                                                            \bigoplus_{j=0}^r A[[z_1]]/ (z_1,1+t - \omega_2(\xi)^{(p-1)(n-j)})  ,~ & \text{if $k= 0$,}\\
                                                              A[[z_1]]/z_1^{p^{k-1}(r+1)},~ &\text{if $k\geq 1$}.
                                                                \end{array}\right.
\]
\end{enumerate}
\end{lemma}
\begin{proof}
   The first statement follows from Lemma \ref{lemma::action_of_L_k_unr}(ii). The proof of (ii) and (iii) is similar to that of Lemma \ref{lemma::structure_sigma_k}. 
   For (iv), since $L^\times/ Z \simeq C \times L^\times_1/Z$, we have $\FF[[L^\times /Z]]\simeq \FF[C ]\otimes_{\FF} \FF[[z_1]] \simeq A[[z_1]] $. The description of $( \op{ind}^{L^\times/Z }_{L_k^\times /Z } \overline{\sigma}_k)^{\vee}$ for $k\geq 1$ follows from (ii) and (iii). 
    Consider the character $\omega_2^i$ of $\mu_{p+1}$. For $f\in (\omega_2^i)^{\vee},~ v\in \omega_2^i$, we have
\[
(1+t) f (v) = f(\overline{\xi}^{-1} v) = f(\omega_2^{i} (\xi^{-1}) v) = \omega_2^{i} (\xi^{-1}) f(v) = \omega_2(\xi)^{-i} f(v).
\]
Hence $(1+t)f = \omega_2(\xi)^{-i} f$. We then have an isomorphism of $A$-modules
\[
(\omega_2^i)^{\vee} \simeq A/(1+t - \omega_2(\xi)^{-i}).
\]
In the decomposition $\overline{\sigma}_{0}\cong\bigoplus_{j=0}^{r} \omega_2^{(p-1)(j-n)}$, the vector
$(X+\sqrt \delta Y)^{r-j} (X-\sqrt \delta Y)^j$ is a $\omega_2^{(p-1)(j-n)}$-eigenvector. Therefore
\begin{equation} \label{eq::dec_sigma_0_dual_unr}
\FF ((X+\sqrt \delta Y)^{r-j} (X-\sqrt \delta Y)^j)^{\vee} \simeq A/ (1+t - \omega_2(\xi)^{(p-1)(n-j)}).
\end{equation}
The statement for $k=0$ follows.
\end{proof}

For $i\in \{0,\ldots, p\}$, we let 
\[
p_i(t) : = \frac{(1+t)^{p+1} -1}{1+t - \omega_2(\xi)^i} \in \FF[t].
\]

\begin{proposition} \label{prop::Ext^i_L_unr}
    Let $\eta: L^\times\to \FF^{\times}$ be the character given by \eqref{eq::eta_L_unr}.
    \begin{enumerate}
    \item[(i)] If $n \notin\{0,\ldots,r\}$, then 
    \[
     \op{Ext}^i_{ \FF[[L^\times/Z]]} (\overline{\sigma}_0^{\vee}, \FF ) \simeq  \op{Ext}^i_{ \FF[[L^\times/Z]]} (\FF ,\overline{\sigma}_0^{\vee} ) = 0,~\forall i.
    \]
    
    \item[(ii)]  If $0\leq n \leq r$, then
    \[
     \op{Ext}^i_{ \FF[[L^\times/Z]]} (\overline{\sigma}_0^{\vee}, \FF ) \simeq  \op{Ext}^i_{ \FF[[L^\times/Z]]} (\FF ((X+\sqrt \delta Y)^{r-n} (X-\sqrt \delta Y)^n)^{\vee}, \FF )  \simeq \left\{\begin{array}{cl}
                                                                 \FF  ,~ & \text{if $i = 0,1$,}\\
                                                                   0,~ &\text{else}
                                                                    \end{array}\right.
    \]
    and
    \[
     \op{Ext}^i_{ \FF[[L^\times/Z]]} ( \FF, \overline{\sigma}_0^{\vee} )  \simeq  \op{Ext}^i_{ \FF[[L^\times/Z]]} (\FF, \FF ((X+\sqrt \delta Y)^{r-n} (X-\sqrt \delta Y)^n)^{\vee} )  \simeq \left\{\begin{array}{cl}
                                                                  A  /t  ,~ & \text{if $i = 0,1$,}\\
                                                                   0,~ &\text{else}.
                                                                    \end{array}\right.
    \]
    
    \item[(iii)] If $k\geq 1$, then
    \[
     \op{Ext}^i_{ \FF[[L^\times/Z]]} (( \op{ind}^{L^\times/Z}_{L_k^\times / Z} \overline{\sigma}_k)^{\vee}, \FF ) \simeq \left\{\begin{array}{cl}
                                                                 \FF  ,~ & \text{if $i = 0,1$,}\\
                                                                   0,~ &\text{else}
                                                                    \end{array}\right.
    \]
    and
    \[
    \op{Ext}^i_{ \FF[[L^\times/Z]]} (\FF, ( \op{ind}^{L^\times/Z}_{L_k^\times/Z} \overline{\sigma}_k)^{\vee} ) =   \left\{\begin{array}{cl}
                                                         z_1^{p^{k-1}(r+1) - 1} p_0(t)  A [[z_1]]/z_1^{p^{k-1}(r+1)} ,~ & \text{if $i = 0$,}\\
    p_0(t)\frac{ A[[z_1]]/z_1^{p^{k-1}(r+1)}}{(z_1)}  ,~ & \text{if $i = 1$,}\\
                                                                   0,~ &\text{else}.
                                                                    \end{array}\right.
    \]
\end{enumerate}
    In particular, if nonzero, $\op{Ext}^i_{ \FF[[L^\times/Z]]} (( \op{ind}^{L^\times/Z}_{L_k^\times / Z} \overline{\sigma}_k)^{\vee}, \FF )$ and $\op{Ext}^i_{ \FF[[L^\times/Z]]} (\FF, ( \op{ind}^{L^\times/Z}_{L_k^\times/Z} \overline{\sigma}_k)^{\vee} )$ are one-dimensional.
    \end{proposition}
    \begin{proof}
    Let $P(i)_{\bullet}$ denote the following projective resolution of the $A[[z_1]]$-module $A/(1+t - \omega_2(\xi)^i)$:
    \begin{equation} \label{eq::proj_res_one_dim_L_unr}
    \begin{multlined} 
    \cdots   \xrightarrow{\bigl(\begin{smallmatrix}
    p_i(t) & z_1 \\            0 & 1+t - \omega_2(\xi)^i
    \end{smallmatrix} \bigr)}  A[[z_1]]\oplus A[[z_1]]  \xrightarrow{\bigl(\begin{smallmatrix}
    1+t - \omega_2(\xi)^i & -z_1 \\            0 & p_i(t)
    \end{smallmatrix} \bigr)}  A[[z_1]]\oplus A[[z_1]] \\ \xrightarrow{\bigl(\begin{smallmatrix}
    p_i(t) & z_1 \\            0 & 1+t - \omega_2(\xi)^i
    \end{smallmatrix} \bigr)}  A[[z_1]]\oplus A[[z_1]]  \xrightarrow{(-(1+t - \omega_2(\xi)^i),z_1) } A[[z_1]] \to A/(1+t - \omega_2(\xi)^i) \to 0.
    \end{multlined}
    \end{equation}
    
    By Lemma \ref{lemma::structure_sigma_k_L_unr}(iv),
    \[
     \op{Ext}^i_{ \FF[[L^\times/Z]]} (\overline{\sigma}_0^{\vee}, \FF )  \simeq \bigoplus_{j=0}^r  \op{Ext}^i_{A[[z_1]]} (  A/ (1+t - \omega_2(\xi)^{(p-1)(n-j)}) , \FF ) .
    \]
    To compute $\op{Ext}^i_{ A[[z_1]]} (  A/ (1+t - \omega_2(\xi)^{(p-1)(n-j)}) , \FF ) $, we use the projective resolution:
    \begin{equation*}
     P((p-1)(n-j))_{\bullet}  \to   A/ (1+t - \omega_2(\xi)^{(p-1)(n-j)}) \to 0.
    \end{equation*}
    Then 
    \begin{equation*}
    \begin{multlined}
    \op{Ext}^i_{ A[[z_1]]} (  A/ (1+t - \omega_2(\xi)^{(p-1)(n-j)}) , \FF ) \\
    = H^i ( \FF  \xrightarrow{ \bigl(\begin{smallmatrix}
    -(1+t - \omega_2(\xi)^{(p-1)(n-j)}) \\  z_1
    \end{smallmatrix} \bigr) } \FF \oplus \FF  \xrightarrow{\bigl(\begin{smallmatrix}
    p_{(p-1)(n-j)}(t) & 0 \\           z_1 & 1+t - \omega_2(\xi)^{(p-1)(n-j)}
    \end{smallmatrix} \bigr) }  \\
    \FF \oplus \FF  \xrightarrow{\bigl(\begin{smallmatrix}
    1+t - \omega_2(\xi)^{(p-1)(n-j)} & 0 \\          -z_1 & p_{(p-1)(n-j)}(t)
    \end{smallmatrix} \bigr)} \FF \oplus \FF  \xrightarrow{\bigl(\begin{smallmatrix}
    p_{(p-1)(n-j)}(t) & 0 \\           z_1 & 1+t - \omega_2(\xi)^{(p-1)(n-j)}
    \end{smallmatrix} \bigr) }  \cdots ).
    \end{multlined}
    \end{equation*}
    If $(p-1)(n-j) \neq 0$, then
    \begin{equation*}
    \begin{multlined}
    \op{Ext}^i_{ A[[z_1]]} (  A/ (1+t - \omega_2(\xi)^{(p-1)(n-j)}) , \FF ) \\ = H^i ( \FF  \xrightarrow{ \bigl(\begin{smallmatrix}
    -(1 - \omega_2(\xi)^{(p-1)(n-j)}) \\  0
    \end{smallmatrix} \bigr) } \FF \oplus \FF  \xrightarrow{\bigl(\begin{smallmatrix}
    0 & 0 \\          0 & 1 - \omega_2(\xi)^{(p-1)(n-j)}
    \end{smallmatrix} \bigr) }  \\
    \FF \oplus \FF  \xrightarrow{\bigl(\begin{smallmatrix}
    1 - \omega_2(\xi)^{(p-1)(n-j)} & 0 \\          0 & 0
    \end{smallmatrix} \bigr)} \FF \oplus \FF  \xrightarrow{\bigl(\begin{smallmatrix}
    0 & 0 \\          0 & 1 - \omega_2(\xi)^{(p-1)(n-j)}
    \end{smallmatrix} \bigr) }  \cdots ) .
    \end{multlined}
    \end{equation*}
    One checks that $\op{Ext}^i_{ A[[z_1]]} (  A/ (1+t - \omega_2(\xi)^{(p-1)(n-j)}) , \FF ) = 0$ for all $i$. \\
    When $(p-1)(n-j) = 0$,
    \begin{align*}
    \op{Ext}^i_{ A[[z_1]]} (  A/ t , \FF ) &= H^i ( \FF  \xrightarrow{ \bigl(\begin{smallmatrix}
    0 \\  0
    \end{smallmatrix} \bigr) } \FF \oplus \FF  \xrightarrow{\bigl(\begin{smallmatrix}
    -1 & 0 \\          0 & 0
    \end{smallmatrix} \bigr) } \FF \oplus \FF  \xrightarrow{\bigl(\begin{smallmatrix}
    0 & 0 \\          0 & -1
    \end{smallmatrix} \bigr)} \FF \oplus \FF  \xrightarrow{\bigl(\begin{smallmatrix}
    -1 & 0 \\          0 & 0
    \end{smallmatrix} \bigr) }  \cdots ) \\
    & \simeq \left\{\begin{array}{cl}
                                                                 \FF  ,~ & \text{if $i = 0,1$,}\\
                                                                   0,~ &\text{else}.
                                                                    \end{array}\right.
    \end{align*}
    This finishes the calculation of $ \op{Ext}^i_{ \FF[[L^\times/Z]]} (\overline{\sigma}_0^{\vee}, \FF ) $.
    
    To compute $\op{Ext}^i_{ \FF[[L^\times/Z]]} (\FF, \overline{\sigma}_0^{\vee} )$, we have
    \[
     \op{Ext}^i_{ \FF[[L^\times/Z]]} (\FF,\overline{\sigma}_0^{\vee} )  \simeq \bigoplus_{j=0}^r  \op{Ext}^i_{A[[z_1]]} (\FF,  A/ (1+t - \omega_2(\xi)^{(p-1)(n-j)}) ) .
    \]
    We use the projective resolution $ P(0)_{\bullet}$ of the trivial $A[[z_1]]$-module $\FF$. The proof is similar to the argument given above. This gives (i) and (ii). 
    
    For (iii), the proof is similar to that of Proposition \ref{prop::Ext^i_L_ramified}. It suffices to use the projective resolution
    \[
    0\to A[[z_1]]  \xrightarrow{z_1^{p^{k-1}(r+1)}} A[[z_1]] \to ( \op{ind}^{L^\times/Z}_{L_k^\times / Z} \overline{\sigma}_k)^{\vee} \simeq A[[z_1]] / z_1^{p^{k-1}(r+1)} \to 0
    \]
    in the calculation of $ \op{Ext}^i_{ \FF[[L^\times/Z]]} (( \op{ind}^{L^\times/Z}_{L_k^\times / Z} \overline{\sigma}_k)^{\vee}, \FF )$ for $k\geq 1$. The details are left to the reader. 
\end{proof}
The morphisms $\Phi_k^{\pm}$ and $\Phi_0$ in Lemma \ref{lemma::Phi_L_unr} descend to $\op{ind}^{L^{\times}/Z }_{L_k^{\times} /Z } \overline{\sigma}_k$. We have the following proposition.

\begin{proposition} \label{prop::Phi_k_L_unr}
\begin{enumerate}
\item[(i)] For $k\geq 1$, $\Phi^{+}_k:  \op{ind}_{L_k^\times/Z}^{L^\times/Z} \overline{\sigma}_k  \to \op{ind}_{L_{k+1}^\times /Z }^{L^\times /Z } \overline{\sigma}_{k+1}$ is injective. And the dual map 
\[
(\Phi^+_k)^{\vee}: A[[z_1]]/z_1^{p^{k}(r+1)} \to A[[z_1]]/z_1^{p^{k-1}(r+1)}
\]
is multiplication by some $u_k(z_1) \in  1 + z_1 \FF[[z_1]]$.

\item[(ii)] For $k \geq 2$, $\Phi^{-}_k :  \op{ind}_{L_k^\times/Z}^{L^\times /Z} \overline{\sigma}_k   \to \op{ind}_{L_{k-1}^\times /Z}^{L^\times /Z} \overline{\sigma}_{k-1}$ is surjective. And the dual map 
\[
(\Phi^-_k)^{\vee}: A[[z_1]]/z_1^{p^{k-2}(r+1)} \to A[[z_1]]/z_1^{p^{k-1}(r+1)}
\]
is multiplication by $z_1^{(p^{k-1} - p^{k-2})(r+1)}$.
\item[(iii)] The map 
\[
(\Phi_1^{-})^\vee: \bigoplus_{j=0}^r A/ (1+t - \omega_2(\xi)^{(p-1)(n-j)})\to A[[z_1]]/z_1^{r+1}
\]
is given by the matrix 
\[
\begin{bmatrix}
a_0(z_1) & a_1(z_1) & \cdots & a_r(z_1) 
\end{bmatrix},
\]
where $a_j(z_1) =  (2\sqrt\delta)^{-r}(-1)^{j}  \binom{r}{j}  \omega_2(\xi)^{(p-1)(n-j)}p_{(p-1)(n-j)}(t) z_1^r$.
\item[(iv)] The map 
\[
(\Phi_0)^{\vee}: A[[z_1]]/z_1^{r+1} \to \bigoplus_{j=0}^r A/ (1+t - \omega_2(\xi)^{(p-1)(n-j)})
\]
is given by the matrix 
\[
\frac{(-1)^r}{r!} \begin{bmatrix}
(-1)^{0} & (-1)^{1} & \cdots & (-1)^{r}
\end{bmatrix}^T.
\]
\end{enumerate}
\end{proposition} 
\begin{proof}
    The proof of (i) and (ii) is similar to that of Proposition \ref{prop::Phi_k_L_ramified}(i) and (ii). 
    
    (iii) Recall that the vectors $ (1+t')^i z_1'^l [1,e_0],~0\leq i \leq p,~ 0\leq l\leq r$ form a basis of $\op{ind}^{L^\times/Z}_{L_1^\times/Z} \overline{\sigma}_1$ and $(\op{ind}^{L^\times/Z}_{L_1^\times/Z} \overline{\sigma}_1)^{\vee}$ is generated by $(z_1'^r [1,e_0])^{\vee}$ by Lemma \ref{lemma::structure_sigma_k_L_unr}(iii) and (iv). We have
\begin{align*}
& (\Phi^-_{1})^{\vee} (((X+\sqrt\delta Y)^{r-j} (X-\sqrt\delta Y)^j)^{\vee}) ((1+t')^i z_1'^{l} [1,e_0]) \\ =  & ((X+\sqrt\delta Y)^{r-j} (X-\sqrt\delta Y)^j)^{\vee} (\Phi^-_{1} ((1+t')^i z_1'^{l} [1,e_0])) \\
 = & ((X+\sqrt\delta Y)^{r-j} (X-\sqrt\delta Y)^j)^{\vee} ((1+t')^i \Phi^-_{1} ( z_1'^{l} [1,e_0])) .
\end{align*}
By Lemma \ref{lemma::Phi_L_unr}, $\Phi^-_{1} ( z_1'^{l} [1,e_0]) = 0$ if $l\neq 0$. And if $l=0$,
\begin{align*}
((X+\sqrt\delta Y)^{r-j} (X-\sqrt\delta Y)^j)^{\vee} ((1+t')^i \Phi^-_{1} ( [1,e_0])) &=  ((X+\sqrt\delta Y)^{r-j} (X-\sqrt\delta Y)^j)^{\vee} ((1+t')^i  (Y)^r) \\
& = (2\sqrt\delta)^{-r}(-1)^{j}\omega_2(\xi)^{(p-1)(n-j)i}    \binom{r}{j}.
\end{align*}
Therefore,
\begin{align*}
(\Phi^-_{1})^{\vee} (((X+\sqrt\delta Y)^{r-j} (X-\sqrt\delta Y)^j)^{\vee})  &= (2\sqrt\delta)^{-r}(-1)^{j}  \binom{r}{j}\sum_{i=0}^{p} \omega_2(\xi)^{(p-1)(n-j)i}  ((1+t')^i [1,e_0])^{\vee} \\
& = (2\sqrt\delta)^{-r}(-1)^{j}  \binom{r}{j}\sum_{i=0}^{p} \omega_2(\xi)^{(p-1)(n-j)i} (1+t)^{-i} ([1,e_0])^{\vee}\\
&= (2\sqrt\delta)^{-r}(-1)^{j}  \binom{r}{j}\sum_{i=0}^{p} \omega_2(\xi)^{(p-1)(n-j)i} (1+t)^{-i} z_1^r (z_1'^r[1,e_0])^{\vee}\\
&= (2\sqrt\delta)^{-r}(-1)^{j}  \binom{r}{j} \omega_2(\xi)^{(p-1)(n-j)}p_{(p-1)(n-j)}(t) z_1^r  (z_1'^r[1,e_0])^{\vee} \\
& = a_j(z_1) (z_1'^r[1,e_0])^{\vee}.
\end{align*}
(iv) By Lemma \ref{lemma::Phi_L_unr}, we have
\begin{align*}
&\quad\ (\Phi_0)^{\vee} ( (z_1'^r [1,e_0])^{\vee}) ((X+\sqrt\delta Y)^{r-j} (X-\sqrt\delta Y)^j)  \\& = (z_1'^r [1,e_0])^{\vee} (\Phi_0 ((X+\sqrt\delta Y)^{r-j} (X-\sqrt\delta Y)^j))\\
& = (z_1'^r [1,e_0])^{\vee} \left( \sum_{i=0}^{p-1}[1-\sqrt{\delta} i, (-1)^j(1+i\sqrt\delta)^{-j}(1-i\sqrt\delta)^{j-r} X^r]+[\sqrt\delta,(-1)^j(\sqrt \delta)^{-r}X^r]  \right) \\
& = (z_1'^r [1,e_0])^{\vee} \left( (-1)^j \frac{(-1)^r}{r!} z_1'^r \left[ 1 , e_0 \right] \right) \\
& = \frac{(-1)^{r+j}}{r!}.
\end{align*}
Therefore, 
\[
(\Phi_0)^{\vee} ( (z_1'^r [1,e_0])^{\vee}) =\frac{(-1)^{r}}{r!}  \sum_{j=0}^r (-1)^j((X+\sqrt\delta Y)^{r-j} (X-\sqrt\delta Y)^j)^{\vee}.
\]
\end{proof}
The following two lemmas describe the induced maps between extension groups. Let $(\Phi^{\pm}_{k})^{*}$ and $(\Phi_{0})^{*}$ (resp. $(\Phi^{\pm}_{k})_{*}$ and $(\Phi_{0})_{*}$)  denote the map between the extension groups $\op{Ext}^i_{\FF[[L^{\times}/Z]]}(\FF, (\op{ind}^{L^{\times}/Z}_{L_{k}^{\times}/Z} \overline{\sigma}_{k})^{\vee})$ (resp. $\op{Ext}^i_{\FF[[L^{\times}/Z]]}((\op{ind}^{L^{\times}/Z}_{L^{\times}_{k}/Z} \overline{\sigma}_{k})^{\vee},\FF)$) induced by $(\Phi^{\pm}_{k})^{\vee}$ and $(\Phi_{0})^{\vee}$ .

\begin{lemma} \label{lemma::Phi_Ext^i_L_unr_1}
Let $\eta: L^\times\to \FF^{\times}$ be the character $\omega_2^{r+(p-1)n}$ given by \eqref{eq::eta_L_unr}. Under the isomorphisms in Proposition \ref{prop::Ext^i_L_unr}, the induced maps have the following description.
\begin{enumerate}
\item[(i)] For $k\geq 1$, the map 
\[
(\Phi^{+}_{k})^{*} : \op{Ext}^i_{\FF[[L^\times/Z]]}(\FF, (\op{ind}^{L^\times/Z}_{L^\times_{k+1}/Z} \overline{\sigma}_{k+1})^{\vee}) \to  \op{Ext}^i_{\FF[[L^\times/Z]]}(\FF, (\op{ind}^{L^\times/Z}_{L^\times_{k}/Z} \overline{\sigma}_{k})^{\vee})
\]
is zero if $i=0$, and is the identity map if $i=1$.

\item[(ii)] For $k \geq 2$, the map 
\[
(\Phi^{-}_{k} )^{*} : \op{Ext}^i_{\FF[[L^\times/Z]]}(\FF, (\op{ind}^{L^\times/Z}_{L^\times_{k-1}/Z} \overline{\sigma}_{k-1})^{\vee}) \to  \op{Ext}^i_{\FF[[L^\times/Z]]}(\FF, (\op{ind}^{L^\times/Z}_{L^\times_{k}/Z} \overline{\sigma}_{k})^{\vee})
\]
is the identity map if $i=0$, and is zero if $i=1$.

\item[(iii)] If $n\notin \{0,\ldots, r\}$, then the maps $(\Phi^{-}_{1} )^{*}$ and $~(\Phi_{0} )^{*}$ are zero.

\item[(iv)] If $n\in \{0,\ldots, r\}$, then the map 
\[
(\Phi^{-}_{1} )^{*} :  \op{Ext}^i_{\FF[[L^\times/Z]]}(\FF,  \overline{\sigma}_{0}^{\vee}) \to  \op{Ext}^i_{\FF[[L^\times/Z]]}(\FF, (\op{ind}^{L^\times/Z}_{L^\times_{1}/Z} \overline{\sigma}_{1})^{\vee}) 
\]
is multiplication by $(2\sqrt\delta)^{-r}(-1)^{n}  \binom{r}{n}$ if $i=0$; if $i=1$, it is zero if $r \geq 1$ and is the identity map if $r=0$. 

\item[(v)] If $n\in \{0,\ldots, r\}$, then the map 
\[
(\Phi_{0} )^{*} :  \op{Ext}^i_{\FF[[L^\times/Z]]}(\FF, (\op{ind}^{L^\times/Z}_{L^\times_{1}/Z} \overline{\sigma}_{1})^{\vee}) \to  \op{Ext}^i_{\FF[[L^\times/Z]]}(\FF, \overline{\sigma}_{0}^{\vee}) 
\]
is multiplication by $\frac{(-1)^{r+n}}{r!}$ if $i=1$; if $i=0$, it is zero if $r\geq 1$ and is the identity map if $r=0$.
\end{enumerate}
\end{lemma}

\begin{proof}
(i) and (ii) follow directly from Proposition \ref{prop::Phi_k_L_unr}(i) and (ii) and Proposition \ref{prop::Ext^i_L_unr}(iii). (iii) follows from Proposition \ref{prop::Ext^i_L_unr}(i). 

Assume now $n\in \{0,\ldots, r\}$. It follows from  Proposition \ref{prop::Phi_k_L_unr}(iii) and (the proof of) Proposition \ref{prop::Ext^i_L_unr} that the map $(\Phi^{-}_{1} )^{*}  $, which is between two one-dimensional spaces, is induced by multiplication by $a_{n}(z_1) = (2\sqrt\delta)^{-r}(-1)^{n}  \binom{r}{n}  z_1^r p_0(t)$. In degree $0$, $(\Phi^{-}_{1} )^{*}  $ sends the basis vector $1$ to the vector $ (2\sqrt\delta)^{-r}(-1)^{n}  \binom{r}{n}  z_1^r  p_0(t)\in   z_1^{r} p_0(t)  A [[z_1]]/z_1^{(r+1)} $, which is then multiplication by $ (2\sqrt\delta)^{-r}(-1)^{n}  \binom{r}{n} $. In degree $1$, the image of $1$ under $(\Phi^{-}_{1} )^{*}  $ is the vector $ (2\sqrt\delta)^{-r}(-1)^{n}  \binom{r}{n} z_1^r  p_0(t)\in   p_0(t)  \frac{A [[z_1]]/z_1^{(r+1)}}{z_1} $, which is zero if $r\geq 1$ and is $p_0(t)$ if $r=0$. This gives (iv).

The map $(\Phi_{0} )^{*}  $ is induced by multiplication by $(-1)^{r+n}/r!$. In degree $0$, it sends the base vector $z_1^r p_0(t) \pmod{z_1^{r+1}}$ to $(-1)^{r+n} z_1^r p_0(t)/r! \pmod{t,z_1} \in A/t$, which is $0$ if $r\geq 1$ and is $1$ if $r=0$ because $p_0(t) \equiv 1\pmod{t}$. In degree $1$, $(\Phi_{0} )^{*}  $ sends the base vector $p_0(t)$ to $(-1)^{r+n} p_0(t) /r! \pmod{t} \in A/t$, which is $(-1)^{r+n}/r! $. (v) then follows.

\end{proof}

\begin{lemma} \label{lemma::Phi_Ext^i_L_unr_2}
Let $\eta: L^\times\to \FF^{\times}$ be the character $\omega_2^{r+(p-1)n}$ given by \eqref{eq::eta_L_unr}.  Under the isomorphisms in Proposition \ref{prop::Ext^i_L_unr}, the induced maps have the following description.
\begin{enumerate}
\item[(i)] For $k\geq 1$, the map 
\[
(\Phi^{+}_{k})_{*} : \op{Ext}^i_{\FF[[L^\times/Z]]}(  (\op{ind}^{L^\times/Z}_{L^\times_{k}/Z} \overline{\sigma}_{k})^{\vee},\FF) \to  \op{Ext}^i_{\FF[[L^\times/Z]]}(  (\op{ind}^{L^\times/Z}_{L^\times_{k+1}/Z} \overline{\sigma}_{k+1})^{\vee},\FF)
\]
is the identity map if $i=0$, and is zero if $i=1$. 

\item[(ii)] For $k \geq 2$, the map 
\[
(\Phi^{-}_{k})_{*} : \op{Ext}^i_{\FF[[L^\times/Z]]}(  (\op{ind}^{L^\times/Z}_{L^\times_{k}/Z} \overline{\sigma}_{k})^{\vee},\FF) \to  \op{Ext}^i_{\FF[[L^\times/Z]]}( (\op{ind}^{L^\times/Z}_{L^\times_{k-1}/Z} \overline{\sigma}_{k-1})^{\vee},\FF)
\]
 is zero if $i=0$, and is the identity map if $i=1$.

\item[(iii)] If $n\notin \{0,\ldots, r\}$, then the maps $(\Phi^{-}_{1} )_{*}$ and $(\Phi_{0} )_{*}$ are zero.

\item[(iv)] If $n\in \{0,\ldots, r\}$, then the map 
\[
(\Phi^{-}_{1} )_{*} :  \op{Ext}^i_{\FF[[L^\times/Z]]}((\op{ind}^{L^\times/Z}_{L^\times_{1}/Z} \overline{\sigma}_{1})^{\vee}, \FF) \to  \op{Ext}^i_{\FF[[L^\times/Z]]}(  \overline{\sigma}_{0}^{\vee},\FF ) 
\]
is multiplication by $(2\sqrt\delta)^{-r}(-1)^{n}  \binom{r}{n}$ if $i = 1$; if $i=0$, it is zero if $r\geq 1$ and is the identity map if $r = 0$.  

\item[(v)] If $n\in \{0,\ldots, r\}$, then the map 
\[
(\Phi_{0} )_{*} :  \op{Ext}^i_{\FF[[L^\times/Z]]}(  \overline{\sigma}_{0}^{\vee},\FF) \to  \op{Ext}^i_{\FF[[L^\times/Z]]}( (\op{ind}^{L^\times/Z}_{L^\times_{1}/Z} \overline{\sigma}_{1})^{\vee},\FF) 
\]
is multiplication by $(-1)^{r+n}/r! $ if $i = 0$; if $i=1$, it is zero if $r\geq 1$ and is the identity map if $r = 0$.
\end{enumerate}
\end{lemma}

\begin{proof}
(i) For $k\geq 1$, we choose a projective resolution of $(\op{ind}^{L^\times/Z}_{L^\times_{k}/Z} \overline{\sigma}_{k})^{\vee}$:
\[
0\to P_{k,-1} :=A[[z_1]] \xrightarrow{z_1^{p^{k-1}(r+1)} } P_{k,0}:= A[[z_1]] \to (\op{ind}^{L^\times/Z}_{L^\times_{k}/Z} \overline{\sigma}_{k})^{\vee} \simeq A[[z_1]]/z_1^{p^{k-1}(r+1)} \to 0.
\] 
To compute the map  $(\Phi_{k}^{+})_*$, we first lift the map $(\Phi_{k}^{+})^{\vee}$ to a map between projective resolutions by taking
\[
P_{k+1,0} \xrightarrow{ u_k(z_1) } P_{k,0},\quad P_{k+1,-1} \xrightarrow{u_k(z_1) z_1^{(p^{k} - p^{k-1}) (r+1)} } P_{k,-1}.
\]
Applying ${\rm Hom}_{A[[z_1]]}(-,\FF)$ to the resolutions, we see that $(\Phi_{k}^{+})_*$ is multiplication by $u_k(z_1) $ (resp. $u_k(z_1) z_1^{(p^{k} - p^{k-1}) (r+1)} $) in degree $0$ (resp. $1$). (i) follows from Proposition \ref{prop::Phi_k_L_unr}(i) and Proposition \ref{prop::Ext^i_L_unr}(iii).

(ii) For $k\geq 2$, the map  $(\Phi_{k}^-)^{\vee}$ is multiplication by $z_1^{(p^{k-1} - p^{k-2})(r+1)}$, and it lifts to a map between projective resolutions
\[
P_{k-1,0} \xrightarrow{z_1^{(p^{k-1} - p^{k-2})(r+1)} } P_{k,0},\quad P_{k-1,-1} \xrightarrow{{\rm id} } P_{k,-1}.
\]
Then $(\Phi_{k}^-)_*$ is multiplication by $z_1^{(p^{k-1} - p^{k-2})(r+1)}$ (resp. $1$) in degree $0$ (resp. $1$). (ii) follows from Proposition \ref{prop::Ext^i_L_unr}(iii).

(iii) follows from Proposition \ref{prop::Ext^i_L_unr}(i).

(iv) By Proposition \ref{prop::Ext^i_L_unr}(ii), we are reduced to understand the map
\[
(\Phi^{-}_{1} )_{*} :  \op{Ext}^i_{\FF[[L^\times/Z]]}((\op{ind}^{L^\times/Z}_{L^\times_{1}/Z} \overline{\sigma}_{1})^{\vee}, \FF) \to  \op{Ext}^i_{\FF[[L^\times/Z]]}( \FF((X+\sqrt\delta Y)^{r-n} (X-\sqrt\delta Y)^n )^{\vee},\FF ).
\]
We take the projective resolution $P(0)_{\bullet}$ (given in \eqref{eq::proj_res_one_dim_L_unr}) of $ \FF((X+\sqrt\delta Y)^{r-n} (X-\sqrt\delta Y)^n )^{\vee}$. The composition 
\[
\FF((X+\sqrt\delta Y)^{r-n} (X-\sqrt\delta Y)^n )^{\vee} \to \overline{\sigma}_0^{\vee} \xrightarrow{
(\Phi^{-}_{1} )^{\vee}}   (\op{ind}^{L^\times/Z}_{L^\times_{1}/Z} \overline{\sigma}_{1})^{\vee}
\]
lifts to a map between projective resolutions by taking
\[
 P(0)_0  \xrightarrow{
a_{n}(z_1)} P_{1,0} ,\quad P(0)_{-1}  \xrightarrow{ \begin{bmatrix}
    0 & a_{n}(z_1)/z_1^r
\end{bmatrix} } P_{1,-1}.
\]
Then in degree $0$, $(\Phi_{1}^-)_*$ is multiplication by $a_{n}(z_1) = (2\sqrt\delta)^{-r}(-1)^{n}  \binom{r}{n}  z_1^r p_0(t)$, which is zero if $r\geq 1$ and is the identity map if $r=0$ because $p_0(t)\equiv 1\pmod{t}$. In degree $1$, $(\Phi_{1}^-)_*$ is multiplication by $a_{n}(z_1)/z_1^r = (2\sqrt\delta)^{-r}(-1)^{n}  \binom{r}{n}  p_0(t)$, which is multiplication by $(2\sqrt\delta)^{-r}(-1)^{n}  \binom{r}{n}$.

(v)  By Proposition \ref{prop::Ext^i_L_unr}(ii), we are reduced to understand the map
\[
(\Phi_{0} )_{*} :  \op{Ext}^i_{\FF[[L^\times/Z]]}( \FF((X+\sqrt\delta Y)^{r-n} (X-\sqrt\delta Y)^n)^{\vee},\FF) \to  \op{Ext}^i_{\FF[[L^\times/Z]]}( (\op{ind}^{L^\times/Z}_{L^\times_{1}/Z} \overline{\sigma}_{1})^{\vee},\FF) .
\] 
The composition
\[
 (\op{ind}^{L^\times/Z}_{L^\times_{1}/Z} \overline{\sigma}_{1})^{\vee} \xrightarrow{(\Phi_{0})^{\vee}} \overline{\sigma}_0^{\vee} \to  \FF((X+\sqrt\delta Y)^{r-n} (X-\sqrt\delta Y)^n )^{\vee} 
\]
lifts to a map between projective resolutions by taking
\[
 P_{1,0}  \xrightarrow{\frac{(-1)^{r+n}}{r!} }  P(0)_0 ,\quad  P_{1,-1}   \xrightarrow{ \begin{bmatrix}
    0 \\
 (-1)^{r+n} z_1^r/r!
\end{bmatrix}} P(0)_{-1}.
\]
Applying ${\rm Hom}_{A[[z_1]]}(-,\FF)$ to the resolutions, we see that in degree $0$, $(\Phi_{0} )_{*}$ is multiplication by $(-1)^{r+n} /r!$; and in degree $1$, $(\Phi_{0} )_{*}$ is multiplication by $(-1)^{r+n}  z_1^r/r!$, which is zero if $r \geq 1$ and is the identity if $r = 0$.
\end{proof}

\begin{theorem} \label{thm::branching_unr}
Let $\pi$ be a supersingular representation of ${\rm GL}_2(\QQ_p)$ with central character $\zeta$. Let $\eta: L^{\times} \to \FF^{\times}$ be a smooth character such that $\eta|_{Z} = \zeta$. Then there are quasi-isomorphisms
\[
\op{RHom}_{L^\times,\zeta}(\pi|_{L^\times},\eta) \simeq \FF^{\oplus 2}[-1],\qquad \op{RHom}_{L^\times,\zeta}( \eta, \pi|_{L^\times}) \simeq \FF^{\oplus 2} [0].
\]
\end{theorem}

\begin{proof}
By Lemma \ref{lemma::concentrate_one_deg}, to prove the first quasi-isomorphism, it suffices to prove
\[
\op{dim}_{\FF}\op{Ext}^{i}_{L^\times,\zeta}(\pi|_{L^\times},\eta)=\left\{\begin{array}{cl}
        2,& \text{ if } i=1,\\
        0,& \text{ otherwise}.
        \end{array}\right.
\]

By twisting, we can assume $\pi = \op{ind}_{KZ}^{G}(\sigma) /\Phi$ with $\sigma := \op{Sym}^{r}(\FF^2)$ and $0\leq r\leq p-1$. The short exact sequence 
\begin{align*}
        0\to\op{ind}_{KZ}^{G}(\sigma)\xrightarrow{\Phi }\op{ind}_{KZ}^{G}(\sigma)\to \pi\to 0 
    \end{align*}
induces a long exact sequence
\[
\cdots \to \op{Ext}^{i}_{L^\times,\zeta}(\pi|_{L^\times},\eta) \to \op{Ext}^{i}_{L^\times,\zeta}(\op{ind}_{KZ}^{G}(\sigma),\eta) \xrightarrow{\Phi_i^{*} } \op{Ext}^{i}_{L^\times,\zeta}(\op{ind}_{KZ}^{G}(\sigma),\eta) \to \cdots.
\]
By \eqref{eq::Mackey_2} and \eqref{eq::Ext_reduce_mod_center_unr}, we are reduced to understand the map
\[
\Phi_i^{*} :  \prod_{k\geq 0} \mathfrak{C}^i_k \to \prod_{k\geq 0}  \mathfrak{C}^i_k,
\] 
where $\mathfrak{C}^i_k : =  \op{Ext}^{i}_{L^\times/Z}( \op{ind}^{L^{\times}/Z}_{L_k^{\times}/Z} \overline{\sigma}_k, \FF) \simeq \op{Ext}^{i}_{\FF[[L^\times/Z ]] }(\FF, (\op{ind}^{L^{\times}/Z}_{L_k^{\times}/Z} \overline{\sigma}_k)^{\vee})$. By Proposition \ref{prop::Ext^i_L_unr}(ii) and (iii), we just need to compute the map $\Phi_i^*$ for $i = 0,1$, which is given by the maps $(\Phi_k^{\pm})^{*}$ and $(\Phi_0)^{*}$ in Lemma \ref{lemma::Phi_Ext^i_L_unr_1}. 

We first assume $n\notin \{0,\ldots, r \}$. Then $\mathfrak{C}^i_0 = 0$ by Proposition \ref{prop::Ext^i_L_unr}(i) and it follows that  
\begin{align*}
(\Phi_k^+)^{*} &: \mathfrak{C}^0_{k+1} \to \mathfrak{C}^0_k \text{ is the zero map for $k\geq 1$}, \\
(\Phi_k^+)^{*} & : \mathfrak{C}^1_{k+1} \to \mathfrak{C}^1_k \text{ is the identity map for $k\geq 1$}, \\
(\Phi_k^-)^{*}& : \mathfrak{C}^0_{k-1} \to \mathfrak{C}^0_k\text{ is the identity map for $k\geq 2$}, \\
(\Phi_k^-)^{*} & : \mathfrak{C}^1_{k-1} \to \mathfrak{C}^1_k\text{ is  the zero map for $k\geq 2$}.
\end{align*}
We then have 
\begin{align*}
\Phi_0^{*}& : \prod_{k\geq 1} \mathfrak{C}^0_k \to \prod_{k\geq 1} \mathfrak{C}^0_k ,~ (a_1,a_2,\ldots) \mapsto  (0,a_1,a_2,\ldots),\\
\Phi_1^{*}& : \prod_{k\geq 1} \mathfrak{C}^1_k \to \prod_{k\geq 1} \mathfrak{C}^1_k ,~ (a_1,a_2,a_3,\ldots) \mapsto (a_2,a_3,\ldots).
\end{align*}
Then $\ker \Phi_0^{*} = \op{coker} (\Phi_1^{*}) = 0$, and both $\op{coker} (\Phi_0^{*})$ and $\ker \Phi_1^{*}$ are one-dimensional.

If $0\leq n\leq r$ and $r \geq 1$, then the difference from the preceding case is that $\mathfrak{C}^i_0 $ is one-dimensional for $i=0,1$, and 
\begin{align*}
(\Phi_0)^{*} &: \mathfrak{C}^0_{1} \to \mathfrak{C}^0_0 \text{ is the zero map}, \\
(\Phi_0)^{*} & : \mathfrak{C}^1_{1} \to \mathfrak{C}^1_0 \text{ is multiplication by $ (-1)^{r+n} /r!$}, \\
(\Phi_1^-)^{*}& : \mathfrak{C}^0_{0} \to \mathfrak{C}^0_1\text{ is multiplication by $ (2\sqrt\delta)^{-r}(-1)^{n}  \binom{r}{n} $}, \\
(\Phi_1^-)^{*} & : \mathfrak{C}^1_{0} \to \mathfrak{C}^1_1\text{ is the zero map}.
\end{align*}
We then have 
\begin{align*}
\Phi_0^{*}& : \prod_{k\geq 0} \mathfrak{C}^0_k \to \prod_{k\geq 0} \mathfrak{C}^0_k ,~ (a_0,a_1,a_2,\ldots) \mapsto   (0,(2\sqrt\delta)^{-r}(-1)^{n}  \binom{r}{n} a_0,a_1,a_2,\ldots),\\
\Phi_1^{*}& : \prod_{k\geq 0} \mathfrak{C}^1_k \to \prod_{k\geq 0} \mathfrak{C}^1_k ,~ (a_0,a_1,a_2,\ldots) \mapsto   ((-1)^{r+n}a_1 /r! ,a_2,a_3,\ldots).
\end{align*}
Hence $\ker \Phi_0^{*} = \op{coker} (\Phi_1^{*}) = 0$, and both $\op{coker} (\Phi_0^{*})$ and $\ker \Phi_1^{*}$ are one-dimensional.

Finally, if $n = r =0$, then $(\Phi_k^+)^{*}$ for $k\geq 1$ and $(\Phi_k^-)^{*}$ for $k\geq 2$ are as before. By Lemma \ref{lemma::Phi_Ext^i_L_unr_1}
\begin{align*}
(\Phi_0)^{*} &: \mathfrak{C}^0_{1} \to \mathfrak{C}^0_0 \text{ is the identity map}, \\
(\Phi_0)^{*} & : \mathfrak{C}^1_{1} \to \mathfrak{C}^1_0 \text{ is multiplication by $ (-1)^{r+n} /r!$}, \\
(\Phi_1^-)^{*}& : \mathfrak{C}^0_{0} \to \mathfrak{C}^0_1\text{ is multiplication by $ (2\sqrt\delta)^{-r}(-1)^{n}  \binom{r}{n} $}, \\
(\Phi_1^-)^{*} & : \mathfrak{C}^1_{0} \to \mathfrak{C}^1_1\text{ is the identity map}.
\end{align*}

We then have 
\begin{align*}
\Phi_0^{*}& : \prod_{k\geq 0} \mathfrak{C}^0_k \to \prod_{k\geq 0} \mathfrak{C}^0_k ,~ (a_0,a_1,a_2,\ldots) \mapsto   (a_1,(2\sqrt\delta)^{-r}(-1)^{n}  \binom{r}{n} a_0,a_1,a_2,\ldots),\\
\Phi_1^{*}& : \prod_{k\geq 0} \mathfrak{C}^1_k \to \prod_{k\geq 0} \mathfrak{C}^1_k ,~ (a_0,a_1,a_2,\ldots) \mapsto   ((-1)^{r+n}a_1 /r! ,a_0+a_2,a_3,\ldots).
\end{align*}
Then, as above, $\ker \Phi_0^{*} = \op{coker} \Phi_1^{*} = 0$, and both $\op{coker} \Phi_0^{*}$ and $\ker \Phi_1^{*}$ are one-dimensional. This finishes the computation of $\op{dim}_{\FF}\op{Ext}^{i}_{L^\times,\zeta}(\pi|_{L^\times},\eta)$.

Now we prove the second quasi-isomorphism. By Lemma \ref{lemma::concentrate_one_deg} again, it suffices to prove
\[
\op{dim}_{\FF}\op{Ext}^{i}_{L^\times,\zeta}(\eta, \pi|_{L^\times})=\left\{\begin{array}{cl}
        2,& \text{ if } i=0,\\
        0,& \text{ otherwise}. 
        \end{array}\right.
\]

Since each term of the projective resolution $P(0)_{\bullet}$ of the trivial $A[[z_1]]$-module $\FF$ is finitely generated, we have
\[
\op{Ext}^{i}_{L^\times,\zeta}(\eta,\op{ind}_{KZ}^{G}(\sigma)) = \bigoplus_{k\geq 0} \op{Ext}^{i}_{L^\times,\zeta}(\eta,\op{ind}_{L_k^{\times}}^{L^{\times}}(\sigma_k)) .
\]
As above, we are then reduced to understand the following map $\Phi_{i,*}$ induced by $\Phi$
\[
\Phi_{i,*} :  \bigoplus_{k\geq 0} \op{Ext}^{i}_{\FF[[L^\times/Z]]}( (\op{ind}^{L^{\times}/Z}_{L_k^{\times}/Z} \overline{\sigma}_k)^{\vee}, \FF) \to \bigoplus_{k\geq 0} \op{Ext}^{i}_{\FF[[L^\times/Z]]}( ( \op{ind}^{L^{\times}/Z}_{L_k^{\times}/Z} \overline{\sigma}_k)^{\vee}, \FF),~ i=0,1.
\] 
Let $\mathfrak{C}^i_k : = \op{Ext}^{i}_{\FF[[L^\times/Z]]}( ( \op{ind}^{L^{\times}/Z}_{L_k^{\times}/Z} \overline{\sigma}_k)^{\vee}, \FF)$.  By Proposition \ref{prop::Ext^i_L_unr}, we just need to compute the map $\Phi_{i,*}$ for $i = 0,1$, which is given by the maps $(\Phi_k^{\pm})_{*}$ and $(\Phi_0)_{*}$ in Lemma \ref{lemma::Phi_Ext^i_L_unr_2}.  

We first assume $n\notin \{0,\ldots, r \}$. Then $\mathfrak{C}^i_0 = 0$ by Proposition \ref{prop::Ext^i_L_unr}(i) and it follows that  
\begin{align*}
(\Phi_k^+)_{*} &: \mathfrak{C}^0_{k} \to \mathfrak{C}^0_{k+1}  \text{ is the identity map for $k\geq 1$} , \\
(\Phi_k^+)_{*} & : \mathfrak{C}^1_{k} \to \mathfrak{C}^1_{k+1} \text{ is the zero map  for $k\geq 1$}, \\
(\Phi_k^-)_{*}& : \mathfrak{C}^0_{k} \to \mathfrak{C}^0_{k-1} \text{ is the zero map for $k\geq 2$},\\
(\Phi_k^-)_{*} & : \mathfrak{C}^1_{k} \to \mathfrak{C}^1_{k-1} \text{ is the identity map for $k\geq 2$}.
\end{align*}
We then have 
\begin{align*}
\Phi_{0,*}& : \bigoplus_{k\geq 1} \mathfrak{C}^0_k \to \bigoplus_{k\geq 1} \mathfrak{C}^0_k ,~ (a_1,a_2,\ldots) \mapsto  (0,a_1,a_2,\ldots),\\
\Phi_{1,*}& : \bigoplus_{k\geq 1} \mathfrak{C}^1_k \to \bigoplus_{k\geq 1} \mathfrak{C}^1_k ,~ (a_1,a_2,a_3,\ldots) \mapsto (a_2,a_3,\ldots).
\end{align*}
Then $\ker \Phi_{0,*} = \op{coker} \Phi_{1,*} = 0$, and both $\op{coker} \Phi_{0,*}$ and $\ker \Phi_{1,*}$ are one-dimensional. The dimension of $\op{Ext}^{i}_{L^\times,\zeta}(\eta, \pi|_{L^\times})$ follows.

If $0\leq n\leq r$ and $1\leq r$, then the difference from the preceding case is that $\mathfrak{C}^i_0 $ is one-dimensional for $i=0,1$, and 
\begin{align*}
(\Phi_0)_{*} &: \mathfrak{C}^0_{0} \to \mathfrak{C}^0_1 \text{ is multiplication by $ (-1)^{r+n} /r!$}, \\
(\Phi_0)_{*} & : \mathfrak{C}^1_{0} \to \mathfrak{C}^1_1 \text{ is the zero map}, \\
(\Phi_1^-)_{*}& : \mathfrak{C}^0_{1} \to \mathfrak{C}^0_0\text{ is the zero map}, \\
(\Phi_1^-)_{*} & : \mathfrak{C}^1_{1} \to \mathfrak{C}^1_0\text{ is multiplication by $ (2\sqrt\delta)^{-r}(-1)^{n}  \binom{r}{n} $}.
\end{align*}
We then have 
\begin{align*}
\Phi_{0,*}& : \bigoplus_{k\geq 0} \mathfrak{C}^0_k \to \bigoplus_{k\geq 0} \mathfrak{C}^0_k ,~ (a_0,a_1,a_2,\ldots) \mapsto   (0,(-1)^{r+n}/r!a_0,a_1,a_2,\ldots),\\
\Phi_{1,*}& : \bigoplus_{k\geq 0} \mathfrak{C}^1_k \to \bigoplus_{k\geq 0} \mathfrak{C}^1_k ,~ (a_0,a_1,a_2,\ldots) \mapsto   ( (2\sqrt\delta)^{-r}(-1)^{n}  \binom{r}{n} a_1,a_2,a_3,\ldots).
\end{align*}
Hence $\ker \Phi_{0,*} = \op{coker} \Phi_{1,*} = 0$, and both $\op{coker} \Phi_{0,*}$ and $\ker \Phi_{1,*}$ are one-dimensional.

Finally, if $n = r =0$, then $(\Phi_k^+)_{*}$ for $k\geq 1$ and $(\Phi_k^-)_{*}$ for $k\geq 2$ are as before. By Lemma \ref{lemma::Phi_Ext^i_L_unr_2}
\begin{align*}
(\Phi_0)_{*} &: \mathfrak{C}^0_{0} \to \mathfrak{C}^0_1 \text{ is multiplication by $ (-1)^{r+n} /r!$}, \\
(\Phi_0)_{*} & : \mathfrak{C}^1_{0} \to \mathfrak{C}^1_1 \text{ is the identity map}, \\
(\Phi_1^-)_{*}& : \mathfrak{C}^0_{1} \to \mathfrak{C}^0_0\text{ is the identity map}, \\
(\Phi_1^-)_{*} & : \mathfrak{C}^1_{1} \to \mathfrak{C}^1_0\text{ is multiplication by $ (2\sqrt\delta)^{-r}(-1)^{n}  \binom{r}{n} $}.
\end{align*}

We then have 
\begin{align*}
\Phi_{0,*}& : \bigoplus_{k\geq 0} \mathfrak{C}^0_k \to \bigoplus_{k\geq 0} \mathfrak{C}^0_k ,~ (a_0,a_1,a_2,\ldots) \mapsto   (a_1,(-1)^{r+n} /r! a_0,a_1,a_2,\ldots),\\
\Phi_{1,*}& : \bigoplus_{k\geq 0} \mathfrak{C}^1_k \to \bigoplus_{k\geq 0} \mathfrak{C}^1_k ,~ (a_0,a_1,a_2,\ldots) \mapsto   ((2\sqrt\delta)^{-r}(-1)^{n}  \binom{r}{n}a_1 ,a_0+a_2,a_3,\ldots).
\end{align*}
As above, $\ker \Phi_{0,*} = \op{coker} \Phi_{1,*} = 0$, and both $\op{coker} \Phi_{0,*}$ and $\ker \Phi_{1,*}$ are one-dimensional.
\end{proof}

%% file: chapter_4.tex
\section{Homological branching for mod $p$ representations: the split tori case} \label{sec::split_torus}

Let $F$ be a finite extension of $\QQ_p$ of degree $d$. In this section, we consider the homological branching laws for an absolutely irreducible admissible smooth $\FF$-representation $\pi$ of $G := {\rm GL}_2(F)$ with respect to the diagonal torus $T$. In \S \ref{sec::PS_St_split}, we assume $\pi$ is either a principal series representation or a special series representation of $G$. In \S \ref{sec::supersingular_split}, we assume $F = \QQ_p$ and $\pi$ is a supersingular representation of $G$. 

Let $\zeta: Z\to \FF^{\times}$ denote the central character of $\pi$. Let $\eta:T \to \FF^{\times}$ be a smooth character. By Lemma  \ref{lemma::Borel--Wallach}, if $\eta|_{Z}$ is different from $\zeta$, then
\[
\op{RHom}_{T }(\pi|_{T},\eta) \simeq \op{RHom}_{T}(\eta, \pi|_{T}) \simeq 0 . 
\]
Therefore, we always assume $\eta|_{Z} = \zeta$ from now on, and we consider the derived Hom-complexes
\[
\op{RHom}_{T,\zeta }(\pi|_{T},\eta) ,\qquad \op{RHom}_{T,\zeta}(\eta, \pi|_{T}) .
\]
In this section, the complexes $\op{RHom}_{T,\zeta }(\pi|_{T},\eta) $ and $\op{RHom}_{T,\zeta}(\eta, \pi|_{T}) $ are computed explicitly, except in exceptional cases. In exceptional cases, we determine when the complexes are objects in $\op{D^b_{fd}}(\FF)$ and compute their Euler–Poincar\'e characteristic.

\begin{definition} \label{def::exceptional_T}
\begin{itemize}
\item[(i)] If $\pi$ is an irreducible principal series representation or a special series representation, we say $\eta$ is {\em exceptional for $\pi$} if either $\pi = \op{Ind}_B^G \chi$ with $ \chi \neq s^{-1}_* \chi$ and $\eta\in \{\chi , s^{-1}_* \chi\}$, or $\pi = \op{St}\otimes\chi_1\circ\op{det}$ and $\eta = \chi_1\boxtimes \chi_1$. Otherwise we say $\eta$ is {\em nonexceptional for $\pi$}.

\item[(ii)] If $F = \QQ_p$ with $p\geq 3$ and $\pi$ is an irreducible supersingular representation of $\op{GL}_2(\QQ_p)$, every $\eta$ is called {\em nonexceptional for $\pi$}.
\end{itemize}
\end{definition}

\subsection{Principal series and special series} \label{sec::PS_St_split}
Assume $p\geq 2$. Let $\chi=\chi_1\boxtimes\chi_2 : T\to \FF^{\times}$ be a smooth character of the diagonal torus $T$ of $G$. Let $\zeta: = \chi_1\chi_2 $ be the central character of the parabolic induction $\op{Ind}_B^G \chi$. We start with some lemmas on the restriction of $\op{Ind}_B^G \chi$ and $\op{St}$ to $T$. 

Let $s= \bigl(\begin{smallmatrix}
0 & 1 \\ 1& 0
\end{smallmatrix} \bigr) $, $u=  \bigl(\begin{smallmatrix}
1 & 1 \\ 0& 1
\end{smallmatrix} \bigr)$. 

\begin{lemma}\label{T coset}
    We have a double coset decomposition $$G=B\bigsqcup BsT\bigsqcup Bsu T,$$ with $B$ and $BsT$ closed $B\times T$-orbits in $G$. 
\end{lemma}
\begin{proof}
    This follows from $BsT=\left\{  \bigl(\begin{smallmatrix}
a & b \\ c& 0
\end{smallmatrix} \bigr)  \in G\right\}$ and $Bsu T=\left\{  \bigl(\begin{smallmatrix}
a & b \\ c& d
\end{smallmatrix} \bigr)  \in G\mid c\neq 0, d\neq 0 \right\}$. 
\end{proof}

\begin{lemma} \label{lem::eigenspace_Ind}
Assume $\chi_1\neq \chi_2$. Then for any smooth character $\eta :T \to \FF^{\times}$. 
\[
\op{Hom}_T(\eta, (\op{Ind}_B^G\chi)|_T) = 0.
\]
\end{lemma}
\begin{proof}
 Assume $\eta$ is a smooth character of $T$ such that 
\begin{equation} \label{eq::eigenspace_T_Ind}
\op{Hom}_T(\eta, (\op{Ind}_B^G\chi)|_T) \neq 0.
\end{equation}
Then there is nonzero function $f\in \op{Ind}_B^G\chi$ such that $T$ acts on $f$ by $\eta$. We first claim that $f(s) \neq 0$. Since $f$ is locally constant and nonzero, $f$ is nonzero on the open dense orbit $BsuT$. Then $f(su) \neq 0$. Let $t_n =  \bigl(\begin{smallmatrix}
 \varpi_F^n & 0 \\ 0 & 1
\end{smallmatrix} \bigr)$. Then
\[
\eta(t_n^{-1})f(su) = (t_n^{-1}\cdot f)(su) =  f(s t_{n}^{-1} t_n u t_n^{-1}) = f(\bigl(\begin{smallmatrix}
 1 & 0 \\ 0 & \varpi_F^{-n}
\end{smallmatrix} \bigr) s  (t_n u t_n^{-1})) = \chi (\bigl(\begin{smallmatrix}
 1 & 0 \\ 0 & \varpi_F^{-n}
\end{smallmatrix} \bigr))  f(s (t_n u t_n^{-1})).
\]
Since $t_n u t_n^{-1} =  \bigl(\begin{smallmatrix}
 1 & \varpi_F^n \\ 0 & 1
\end{smallmatrix} \bigr)$ and $f$ is locally constant, $f(s (t_n u t_n^{-1})) = f(s)$ for $n\gg 0$. We deduce that $f(s) \neq 0$.

Next we show that if \eqref{eq::eigenspace_T_Ind} holds, then $\eta = s_*^{-1}\chi$. For any $t\in T$, on the one hand
\[
(t\cdot f) (s) = f(st) = f((sts^{-1}) s) = s_*^{-1}\chi (t) f(s).
\]
On the other hand, since $T$ acts on $f$ by $\eta$, 
\[
(t\cdot f) (s) = \eta(t) f(s).
\]
Since $f(s) \neq 0$, we must have $\eta = s_*^{-1}\chi$.

Now we prove 
    $$\op{Hom}_T(s_*^{-1}\chi, (\op{Ind}_B^G\chi)|_T)=0.$$
 Assume there is a nonzero function $f \in \op{Ind}_B^G \chi$ such that $T$ acts on $f$ by $s_*^{-1}\chi$. Since $f$ is locally constant, $f(su) \neq 0 $ and there is an integer $n$ such that $f$ is constant on $  \bigl(\begin{smallmatrix}
1 & 0 \\ \varpi_F^n\mc O_F & 1
\end{smallmatrix} \bigr)$. For $k\geq n$ and $x\in \mc{O}_F^{\times}$, let $t=  \bigl(\begin{smallmatrix}
x \varpi_F^k & 0 \\ 0 & 1
\end{smallmatrix} \bigr)$ and $b=  \bigl(\begin{smallmatrix}
 - x^{-1}\varpi_F^{-k} & x^{-1}\varpi_F^{-k} \\ 0 & 1
\end{smallmatrix} \bigr)$. Then $bsut =\bigl(\begin{smallmatrix}
1 & 0 \\ x\varpi_F^k  & 1
\end{smallmatrix} \bigr) \in  \bigl(\begin{smallmatrix}
1 & 0 \\ \varpi_F^n\mc O_F & 1
\end{smallmatrix} \bigr)$. Then
\[
f (\bigl(\begin{smallmatrix}
1 & 0 \\ x\varpi_F^k  & 1
\end{smallmatrix} \bigr)) = f(bsu t)=f(su)\chi(b)s_*^{-1}\chi(t)=f(su)\chi_1(-1)(\chi_1^{-1}\chi_2)(x)(\chi_1^{-1}\chi_2)(\varpi_F)^k.
\]
Since $f$ is constant on $  \bigl(\begin{smallmatrix}
1 & 0 \\ \varpi_F^n\mc O_F & 1
\end{smallmatrix} \bigr)$, we must have $\chi_1(\varpi_F) = \chi_2(\varpi_F)$ and $\chi_1(x) = \chi_2(x)$ for all $x\in \mc{O}_F^{\times} $. This contradicts to the assumption $\chi_1\neq \chi_2$. 
\end{proof}

\begin{lemma}\label{lemma::PS_T}
    We have a short exact sequence in $\op{Mod}_T^{\op{sm}}(\FF)$ 
    \begin{equation}\label{st}
        0\to \op{c-Ind}_{F^\times}^T\chi_1\chi_2\to (\op{Ind}_B^G\chi)|_T\to \chi\oplus s_*^{-1}\chi \to 0. 
    \end{equation}
If $\chi_1\neq \chi_2$, \eqref{st} is nonsplit. And the images of $(\op{Ind}_B^G\chi)|_T$ in 
    $\op{Ext}^{1}_{T}(\chi,\op{c-Ind}_{F^\times}^T\chi_1\chi_2)$ and $\op{Ext}^{1}_{T}(s_*^{-1}\chi,\op{c-Ind}_{F^\times}^T\chi_1\chi_2)$ are both nonzero. 
\end{lemma}
\begin{proof}
    It is straightforward to check that $s^{-1}Bs\cap T=T$, $(su)^{-1}Bsu \cap T=F^\times$, and $(su)^{-1}_*\chi|_{F^\times}=\chi_1\chi_2$. By Lemma \ref{lemma:doublecoset_ind}, we have isomorphisms of $T$-representations 
    $$\op{c-Ind}_B^{BsT}\chi\cong s_*^{-1}\chi,\quad \op{c-Ind}_B^{Bsu T}\chi\cong \op{c-Ind}_{F^\times}^T\chi_1\chi_2. $$
    Since $Bsu T\subseteq Bsu T\bigsqcup B\subseteq G$ and $Bsu T\subseteq Bsu T\bigsqcup BsT\subseteq G$ are both open filtrations of $B\times T$-orbits, the exact sequence follows from Lemma \ref{lemma::filtration_ind}.  Lemma \ref{lem::eigenspace_Ind} implies that  if $\chi_1\neq \chi_2$,  \eqref{st} is nonsplit and the images of $(\op{Ind}_B^G\chi)|_T$ in 
    $\op{Ext}^{1}_{T}(\chi,\op{c-Ind}_{F^\times}^T\chi_1\chi_2)$ and $\op{Ext}^{1}_{T}(s_*^{-1}\chi,\op{c-Ind}_{F^\times}^T\chi_1\chi_2)$ are both nonzero. 
\end{proof}

\begin{lemma}\label{T inv}
    The subspace of $T$-invariants $(\op{Ind}_B^G \FF)^T$ is one-dimensional and consists of constant functions. 
\end{lemma}
\begin{proof}
    Assume $f\in (\op{Ind}_B^G \FF)^T$. Then $f$ takes constant value on each orbit $B$, $BsT$ and $BsuT$. Since the orbit $BsuT$ is open and dense in $G$ and since $f$ is locally constant, $f$ is in fact constant on $G$.  
\end{proof}

\begin{lemma} 
    We have a nonsplit short exact sequence in $\op{Mod}^{\mr{sm}}_{T}(\FF)$
 \begin{equation} \label{eq::ex_eq_St_T}
0\to \op{c-Ind}_{F^\times}^T \chi_1^2 \to (\op{St}\otimes(\chi_1\circ\op{det}))|_T\to \chi_1\boxtimes\chi_1 \to 0.
\end{equation}
\end{lemma}
\begin{proof}
It suffices to assume $\chi_1 = \FF$ is the trivial character. By Lemma \ref{lemma::PS_T}, we have a short exact sequence 
\begin{equation*}
0\to \op{c-Ind}_{F^\times}^T\FF\to (\op{Ind}_B^G\FF)|_T\to \FF\oplus \FF \to 0.
\end{equation*}
    Since $\op{St}|_T=(\op{Ind}_B^G \FF)|_T/\FF $ and $\op{Hom}_T(\FF,\op{c-Ind}_{F^\times}^T \FF)=0$, the Snake Lemma gives a short exact sequence 
\begin{equation}  \label{eq::ex_eq_St_T_2}     
0\to \op{c-Ind}_{F^\times}^T \FF \to \op{St}|_T\to \FF \to 0.
\end{equation}
We prove that this short exact sequence is nonsplit. Assume there is a nonconstant function $f\in \op{Ind}_B^G \FF$ such that $\overline{f} \in \op{St}^T=(\op{Ind}_B^G \FF /\FF)^T$. Then for any $t\in T$, $t\cdot f - f = c_t$ is a constant. Then
\[
c_t = (t\cdot f) (\op{Id}_2) - f(\op{Id}_2) = f(t) - f(\op{Id}_2) = 0.
\]
It follows that $f\in  (\op{Ind}_B^G \FF)^T$. This contradicts Lemma \ref{T inv}. Therefore $ \op{St}^T = 0$ and the short exact sequence \eqref{eq::ex_eq_St_T_2} is nonsplit.
\end{proof}

\begin{lemma} \label{lemma::branching_cInd^T}
Let $\eta: T\to \FF^{\times}$ be a smooth character such that $\eta|_{F^{\times}} = \chi_1\chi_2$. Then there are quasi-isomorphisms
\[
 \op{RHom}_{T,\zeta} (\op{c-Ind}_{F^\times}^T\chi_1\chi_2,\eta) \simeq \FF[-d],\qquad \op{RHom}_{T,\zeta}(\eta, \op{c-Ind}_{F^\times}^T\chi_1\chi_2) \simeq \FF[-1].
\]
\end{lemma}

\begin{proof}
By \eqref{eq::isom_ext_twist} and Lemma \ref{lemma::proj_formula},
\[
\op{Ext}^i_{T,\zeta}(\op{c-Ind}_{F^{\times}}^{T}\chi_1\chi_2,\eta) \simeq \op{Ext}^i_{T/F^\times}(\op{c-Ind}_{\{1\}}^{T/F^\times}\FF,\FF) \simeq \op{Hom}_{\mathrm{D}(T/F^\times)}(\op{c-Ind}_{\{1\}}^{T/F^\times}\FF,\FF[i]).
\]
Since the functor $\mathcal{R}_{\{1\}}^{T/F^\times}$ is right adjoint to $\op{c-Ind}_{\{1\}}^{T/F^\times}$, by Lemma \ref{lemma::rightadjoint}, we have
\[
\op{Hom}_{\mathrm{D}(T/F^\times)}(\op{c-Ind}_{\{1\}}^{T/F^\times}\FF,\FF[i]) \simeq \op{Hom}_{\mathrm{D}(\{1\})}(\FF, \mathcal{R}_{\{1\}}^{T/F^\times}(\FF[i])) \simeq \op{Hom}_{\mathrm{D}(\{1\})}(\FF, \FF[i-d]) \simeq \op{Ext}^{i-d}_{\{1\}}(\FF,\FF),
\] 
which is one-dimensional if $i = d$ and is zero if $ i\neq d$. It follows from Lemma \ref{lemma::concentrate_one_deg} that there is a quasi-isomorphism
\[
 \op{RHom}_{T,\zeta} (\op{c-Ind}_{F^\times}^T\chi_1\chi_2,\eta) \simeq \FF[-d].
\]

For the second quasi-isomorphism, by \eqref{eq::isom_ext_twist} and Lemma \ref{lemma::proj_formula} again,
\[
\op{Ext}^i_{T,\zeta}(\eta,\op{c-Ind}_{F^{\times}}^{T}\chi_1\chi_2) \cong  \op{Ext}^i_{T/F^\times}(\FF,\op{c-Ind}_{\{1\}}^{T/F^\times}\FF) \cong \op{Ext}^i_{F^\times}(\FF,\op{c-Ind}_{\varpi_F^{\ZZ}}^{F^\times}\op{c-Ind}_{\{1\}}^{\varpi_F^{\ZZ}}\FF),
\]
where the second isomorphism uses $T/F^\times\cong F^\times \cong \varpi_F^{\ZZ}\times \mc O_{F}^{\times}$. Since $F^{\times}/\varpi_F^{\ZZ}\cong \mc O_{F}^{\times}$ is compact, the functor $\op{c-Ind}_{\varpi_F^{\ZZ}}^{F^\times} = \op{Ind}_{\varpi_F^{\ZZ}}^{F^\times}$ admits a left adjoint $\op{Res}^{F^\times}_{\varpi_F^{\ZZ}}$. Hence we have
\begin{align*}
\op{Ext}^i_{F^\times}(\FF,\op{c-Ind}_{\varpi_F^{\ZZ}}^{F^\times}\op{c-Ind}_{\{1\}}^{\varpi_F^{\ZZ}}\FF) & \simeq \op{Hom}_{\op{D}(F^{\times})} (\FF, \op{c-Ind}_{\varpi_F^{\ZZ}}^{F^\times}\op{c-Ind}_{\{1\}}^{\varpi_F^{\ZZ}}\FF[i]) \simeq \op{Hom}_{\op{D}(\varpi^{\ZZ} )} (\FF, \op{c-Ind}_{\{1\}}^{\varpi_F^{\ZZ}}\FF[i]) \\
& \simeq \op{Ext}^i_{\varpi_F^{\ZZ}}(\FF,\op{c-Ind}_{\{1\}}^{\varpi_F^{\ZZ}}\FF).
\end{align*}
The extension group $ \op{Ext}^i_{\varpi_F^{\ZZ}}(\FF,\op{c-Ind}_{\{1\}}^{\varpi_F^{\ZZ}}\FF)$ is isomorphic to the group cohomology $H^i(\ZZ , \bigoplus_{\ZZ} \FF )$ (where $\ZZ$ acts on $\bigoplus_{\ZZ} \FF $ by translating the index set $\ZZ$), which is zero if $i\neq 1$, and is one-dimensional if $i=1$. It follows from Lemma \ref{lemma::concentrate_one_deg} that there is a quasi-isomorphism
\[
\op{RHom}_{T,\zeta}(\eta, \op{c-Ind}_{F^\times}^T\chi_1\chi_2) \simeq \FF[-1].
\]
\end{proof}

\begin{theorem}\label{thm::PS_T}
    Let $\pi=\op{Ind}_B^G\chi$ with $\chi_1\neq \chi_2$. Let $\eta:T\to \FF^\times$ be a smooth character such that $\eta|_{F^{\times}} = \zeta$. Assume $\eta\notin \{\chi,s_*^{-1}\chi\}$. Then there are quasi-isomorphisms
\[
\op{RHom}_{T,\zeta}(\pi|_{T},\eta) \simeq \FF[-d],\qquad \op{RHom}_{T,\zeta}(\eta, \pi|_{T}) \simeq \FF[-1].
\]
\end{theorem}
\begin{proof}
Since $\eta\notin \{\chi,s_*^{-1}\chi\}$, Lemma \ref{lemma::Borel--Wallach} gives
\[
\op{RHom}_{T,\zeta}(\chi\oplus s_*^{-1}\chi,\eta)\simeq  \op{RHom}_{T,\zeta}(\eta,\chi\oplus s_*^{-1}\chi) \simeq 0.
\]
Then the short exact sequence \eqref{st} induces quasi-isomorphisms
\begin{align*}
 \op{RHom}_{T,\zeta}((\op{Ind}_{B}^{G}\chi)|_{T},\eta) & \cong  \op{RHom}_{T,\zeta}(\op{c-Ind}_{F^{\times}}^{T}\chi_1\chi_2,\eta)\\
 \op{RHom}_{T,\zeta}(\eta,(\op{Ind}_{B}^{G}\chi)|_{T})&\cong   \op{RHom}_{T,\zeta}(\eta,\op{c-Ind}_{F^{\times}}^{T}\chi_1\chi_2).
\end{align*}
The results then follow from Lemma \ref{lemma::branching_cInd^T}.
\end{proof}

When $\eta\in \{\chi,s_*^{-1}\chi\}$, under the assumption that $T/Z \simeq F^{\times}$ has no $p$-torsion (in particular $p\neq 2$), the Euler–Poincar\'e characteristics $\chi (\op{RHom}_{T,\zeta}(\pi|_{T},\eta))$ and $\chi (\op{RHom}_{T,\zeta}(\eta, \pi|_{T}))$ are obtained in the following theorem.

\begin{theorem} \label{thm::ps_T_EP}
    Let $\pi=\op{Ind}_B^G\chi$ with $\chi_1\neq \chi_2$ and $\eta\in \{\chi,s_*^{-1}\chi\}$.  Assume $F$ doesn't contain any primitive $p$-th root of unity. Then the complexes  $\op{RHom}_{T,\zeta} (\pi|_{T}, \eta) $ and $\op{RHom}_{T,\zeta} (\eta, \pi|_{T}) $ are objects in $\op{D^b_{fd}}(\FF)$, and 
\[
\chi ( \op{RHom}_{T,\zeta} (\pi|_{T}, \eta))  =(-1)^d,\qquad \chi ( \op{RHom}_{T,\zeta} (\eta, \pi|_{T}))  =-1.
\]
 \end{theorem}
\begin{proof}

The short exact sequence \eqref{st} induces distinguished triangles 
\begin{equation*} \label{eq::long_exact_seq_ps_1}
 \op{RHom}_{T,\zeta}(\chi\oplus s_*^{-1}\chi,\eta)\to \op{RHom}_{T,\zeta}((\op{Ind}_{B}^{G}\chi)|_{T},\eta)\to  \op{RHom}_{T,\zeta}(\op{c-Ind}_{F^{\times}}^{T}\chi_1\chi_2,\eta) \xrightarrow{+1},
    \end{equation*}
and
\begin{equation*} \label{eq::long_exact_seq_ps_2}
  \op{RHom}_{T,\zeta}(\eta,\op{c-Ind}_{F^{\times}}^{T}\chi_1\chi_2)\to \op{RHom}_{T,\zeta}(\eta,(\op{Ind}_{B}^{G}\chi)|_{T})\to \op{RHom}_{T,\zeta}(\eta,\chi\oplus s_*^{-1}\chi)  \xrightarrow{+1}.
    \end{equation*}

Since $\chi\neq s_*^{-1}\chi$ and $\eta\in \{\chi,s_*^{-1}\chi\}$, we have quasi-isomorphisms 
\[
\op{RHom}_{T,\zeta}(\chi\oplus s_*^{-1}\chi,\eta)\simeq  \op{RHom}_{T/F^\times}(\FF,\FF) \simeq \op{RHom}_{T,\zeta}(\eta, \chi\oplus s_*^{-1}\chi).
\]
Under our assumption that $T/F^{\times} \cong F^{\times}$ has no $p$-torsion, the pro-$p$ completion $\widehat{F^\times}$ is isomorphic to $\ZZ_p^{d+1}$. Then $\op{RHom}_{T/F^\times}(\FF,\FF)  \in \op{D^b_{fd}(\FF)}$ and $ \chi (\op{RHom}_{T/F^\times}(\FF,\FF))= 0$ by Lemma \ref{lemma::fg_abel}. Then all statements of the lemma follow from Lemma \ref{lemma::branching_cInd^T}.
\end{proof}

Now we assume $\pi =\op{St}\otimes(\chi_1\circ\op{det}) $ is a special series representation of $G$.

\begin{theorem} \label{thm::sp_T}
    Let $\pi=\op{St}\otimes(\chi_1\circ\op{det})$. Let $\eta :T\to \FF^{\times}$ be a smooth character such that $\eta|_{F^{\times}} = \zeta = \chi_1^{2}$. Assume $\eta\neq \chi_1\boxtimes\chi_1$. Then there are quasi-isomorphisms
\[
\op{RHom}_{T,\zeta}(\pi|_{T},\eta) \simeq \FF[-d],\qquad \op{RHom}_{T,\zeta}(\eta, \pi|_{T}) \simeq \FF[-1].
\]
\end{theorem}
\begin{proof}
Since $\eta\neq \chi_1\boxtimes\chi_1$, Lemma \ref{lemma::Borel--Wallach} gives
\[
\op{RHom}_{T,\zeta}(\chi_1\boxtimes\chi_1,\eta) \simeq \op{RHom}_{T,\zeta}(\eta,\chi_1\boxtimes\chi_1) \simeq 0.
\] 
Therefore the short exact sequence \eqref{eq::ex_eq_St_T} induces quasi-isomorphisms
\begin{align*}
 \op{RHom}_{T,\zeta}(\pi|_{T},\eta) & \cong  \op{RHom}_{T,\zeta}(\op{c-Ind}_{F^{\times}}^{T}\chi_1^2,\eta) \\
 \op{RHom}_{T,\zeta}(\eta, \pi|_{T}) & \cong  \op{RHom}_{T,\zeta}(\eta , \op{c-Ind}_{F^{\times}}^{T}\chi_1^2).
\end{align*}
The results then follow from Lemma \ref{lemma::branching_cInd^T}.
\end{proof}

\begin{theorem}  \label{thm::sp_T_EP}
    Let $\pi=\op{St}\otimes(\chi_1\circ\op{det})$ and $\eta=\chi_1\boxtimes\chi_1$. Assume $F$ doesn't contain any primitive $p$-th root of unity. Then the complexes  $\op{RHom}_{T,\zeta} (\pi|_{T}, \eta) $ and $\op{RHom}_{T,\zeta} (\eta, \pi|_{T}) $ are objects in $\op{D^b_{fd}}(\FF)$, and 
\[
\chi ( \op{RHom}_{T,\zeta} (\pi|_{T}, \eta))  =(-1)^d,\qquad \chi ( \op{RHom}_{T,\zeta} (\eta, \pi|_{T}))  =-1.
\]
\end{theorem}
\begin{proof}
As in the proof of Theorem \ref{thm::ps_T_EP}, the assumption on $F$ implies that $\op{RHom}_{T,\zeta}( \eta,\chi_1\boxtimes \chi_1)$ and $\op{RHom}_{T,\zeta}(\chi_1\boxtimes \chi_1,\eta)$ are in $\op{D^b_{fd}(\FF)}$, and 
\[
\chi(\op{RHom}_{T,\zeta}( \chi_1\boxtimes \chi_1,\eta) )= \chi(\op{RHom}_{T,\zeta}( \eta,\chi_1\boxtimes \chi_1)) = 0.
\]
Then all statements follow from distinguished triangles induced by the short exact sequence \eqref{eq::ex_eq_St_T} and Lemma \ref{lemma::branching_cInd^T}. 
\end{proof}

\subsection{Supersingular representations of ${\rm GL}_2(\QQ_p)$} \label{sec::supersingular_split}

In this subsection, we study the homological branching for supersingular representations of $G:= {\rm GL}_2(\QQ_p)$ with respect to the diagonal torus $T$ of $G$. In \S\ref{sec::restriction_to_T_0}, as a preparation for the main result, we look at the restriction of supersingular representations to the open compact subgroup $T_0$  defined below. The main result (Theorem \ref{thm::branching_ss_T}) is proved in \S\ref{sec::restriction_to_T}. We assume $p\geq 3$ in this section. 


\subsubsection{Restriction to $T_0$} \label{sec::restriction_to_T_0}

 Let 
\[
T_0 := T(\ZZ_p)Z = \left(\begin{array}{cl}
            \ZZ_p^{\times} & 0 \\
            0 & 1
            \end{array}\right) Z,\quad T_k : =  \left(\begin{array}{cl}
           1+ p^k \ZZ_p & 0 \\
            0 & 1
            \end{array}\right) Z,~ k\geq 1,
\]
be subgroups of $T$.

\begin{lemma}\label{lemma::representatives_T_0}
    A set of representatives for the double coset $KZ\backslash G/ T_0$ is given by 
    $$\left\{ \left(\begin{array}{cl}
            p^m & 0 \\
            0 & 1
            \end{array}\right) , ~\left(\begin{array}{cl}
        p^m & p^{-k} \\
        0 & 1
        \end{array}\right),\ m\in\ZZ,k\in\ZZ_{\geq 1} \right\}. $$
\end{lemma}
\begin{proof}
Since $G=KB$, a set of representatives for $KZ \backslash G \simeq (KZ \cap B) \backslash B$ is given by  $\bigl(\begin{smallmatrix}
p^{\ZZ} &\QQ_p/ \ZZ_p \\ 0& 1
\end{smallmatrix} \bigr)$. It is straightforward to check that $\ZZ_p^{\times}$ has one orbit on each nonzero valuation stratum of $\QQ_p/\ZZ_p$, represented by $p^{-k}$, and a separate zero orbit.
\end{proof}

For $m\in\ZZ$, let $s_{m,0}:=\bigl(\begin{smallmatrix}
    p^{m} & 0 \\
    0 & 1
    \end{smallmatrix} \bigr)$ and let $s_{m,k}:= \bigl(\begin{smallmatrix}
    p^{m} & p^{-k} \\
    0 & 1
    \end{smallmatrix} \bigr),~ k\geq 1$. Let $T_{m,k} : = T_0 \bigcap s_{m,k}^{-1}KZ s_{m,k}$, $m\in \ZZ,~k\geq 0$. Let  $\sigma := \op{Sym}^{r}(\FF^2)$, $0\leq r\leq p-1$. Let $\sigma_{m,k}:= (s_{m,k}^{-1})_{*}\op{Res}^{KZ}_{s_{m,k} T_0 s_{m,k}^{-1}\cap KZ }(\sigma)$, a finite-dimensional representation of $T_{m,k}$. Let $\omega^r:Z\to \FF^{\times}$ be the character introduced in \S\ref{section::sm_rep_GL_2}, and let $\chi : \bigl(\begin{smallmatrix}
    \ZZ_p^{\times} &  \\
    & 1
    \end{smallmatrix} \bigr)  \to \FF^{\times}$ be the character sending $ \bigl(\begin{smallmatrix}
    a &  \\
    & 1
    \end{smallmatrix} \bigr)  $  to $\omega(a)$. We denote the character $\omega^r$ by $\zeta$.

\begin{lemma} \label{lemma::T_{m,k}}
\begin{enumerate}
\item[(i)]   For $m\in \ZZ$, we have $T_{m,k} = T_k$ for $k\geq 0$.

\item[(ii)] The $T_{m,0}$-representation $\sigma_{m,0}$ is semisimple, with decomposition
\[
\sigma_{m,0}\cong\bigoplus_{j=0}^{r} \chi^{r-j }\boxtimes \zeta .
\]

\item[(iii)] If $k\geq 1$, the action of $T_{m,k}$ on $\sigma_{m,k}$ is given by
\begin{equation}\label{eq::action_T_k}
\left(\begin{array}{cl}
            1+p^{k}b & 0  \\
            0 & 1
            \end{array}\right) p^n a \cdot X^{r-j} Y^j = \zeta(a) X^{r-j} (-bX + Y)^j   ,~ a\in \ZZ_p^{\times},~b\in \ZZ_p, ~n\in \ZZ. 
\end{equation}
\end{enumerate}
\end{lemma}
\begin{proof}
(i) is trivial if $k=0$. Assume $k\geq 1$. Let $ \bigl(\begin{smallmatrix}
    a & b \\
   c & d
    \end{smallmatrix} \bigr) \in K$ such that $s_{m,k}^{-1}  \bigl(\begin{smallmatrix}
    a & b \\
   c & d
    \end{smallmatrix} \bigr)s_{m,k} \in T(\ZZ_p) $. Then $bp^k + a-d = c = 0$ and $a,d\in \ZZ_p^{\times}$. It follows that $a-d \in p^k \ZZ_p$. Then 
\[
s_{m,k}^{-1}   \left(\begin{array}{cl}
                   a & b \\
                    c & d
                    \end{array}\right) s_{m,k} =   \left(\begin{array}{cl}
                   a &  \\
                     & d
                    \end{array}\right)\in  \left(\begin{array}{cl}
                  1+p^k \ZZ_p &  \\
                     & 1
                    \end{array}\right)  \ZZ_p^{\times}.
\]

(ii) is trivial and (iii) follows from the identity 
\[
s_{m,k}   \left(\begin{array}{cl}
                   1+p^k b &  \\
                     & 1
                    \end{array}\right) s_{m,k}^{-1} = \left(\begin{array}{cl}
                   1+p^k b & -b  \\
                     & 1
                    \end{array}\right).
\]
\end{proof}

By Lemma \ref{lemma::T_{m,k}}, the representation $\sigma_{m,k}$ of $T_{m,k} = T_k$ does not depend on $m$, so we denote it by $\sigma_k$. Mackey's formula \cite[Thm. 1.1]{Yamamoto}, together with $\op{ind}_{T_{m,k}}^{T_0}\sigma_{m,k} = \op{ind}_{T_{k}}^{T_0}\sigma_{k} $, gives the following $T_0$-equivariant decomposition
\begin{equation} \label{eq::Mackey_T_0}
(\op{ind}_{KZ}^{G}\sigma)|_{T_0 }\cong \bigoplus_{s_{m,k} \in KZ\backslash G/ T_0}\op{ind}_{KZ}^{KZs_{m,k} T_0}\sigma\cong \bigoplus_{m\in \ZZ,~k\geq 0}\op{ind}_{T_{m,k}}^{T_0}\sigma_{m,k} \cong  \bigoplus_{m\in \ZZ}\left( \bigoplus_{k\geq 0}\op{ind}_{T_{k}}^{T_0}\sigma_{k}\right).
\end{equation}

The following lemma describes the restriction of the Hecke operator $\Phi$ to $\op{ind}_{T_{m,k}}^{T_0}\sigma_{m,k}$. 

\begin{lemma}\label{lemma::Phi_T_0}
    For $m\in\ZZ$, $k\in \ZZ_{\geq 0}$, write $\Phi_{m,k} : = \Phi|_{\op{ind}_{T_{m,k}}^{T_0}\sigma_{m,k}}$. Then $\Phi_{m,k} = (\Phi_{m,k}^{+},\Phi_{m,k}^{-})$, where the morphisms $\Phi^{\pm}_{m,k}$ are given by
\begin{align*}
\Phi^{+}_{m,k}:  \op{ind}_{T_{m,k}}^{T_0}\sigma_{m,k} & \to \op{ind}_{T_{m-1,k+1}}^{T_0}\sigma_{m-1,k+1}, ~k \geq 1\\ 
 [1,X^{r-j}Y^j] & \mapsto \sum_{i=0}^{p-1} \left(\begin{array}{cl}
            1+p^{k}i &  \\
             & 1
            \end{array}\right)\left[ 1, i^j X^r \right] ;\\
\Phi^{-}_{m,k}:  \op{ind}_{T_{m,k}}^{T_0}\sigma_{m,k} & \to \op{ind}_{T_{m+1,k-1}}^{T_0}\sigma_{m+1,k-1},~ k\geq 2 \\
[1,X^{r-j}Y^j] & \mapsto [1, \delta_{j,r}Y^r] ; \\
\Phi^{-}_{m,1}:  \op{ind}_{T_{m,1}}^{T_0}\sigma_{m,1} & \to \op{ind}_{T_{m+1,0}}^{T_0}\sigma_{m+1,0} \\
[1,X^{r-j}Y^j] & \mapsto [1, \delta_{j,r}(-X+Y)^r] ; \\
\Phi^{+}_{m,0}:  \op{ind}_{T_{m,0}}^{T_0}\sigma_{m,0} & \to \op{ind}_{T_{m-1,1}}^{T_0}\sigma_{m-1,1}\\ 
 [1,X^{r-j}Y^j] & \mapsto \sum_{i=1}^{p-1}\left(\begin{array}{cl}
            i &  \\
             & 1
            \end{array}\right) \left[ 1 , i^{j-r} X^r \right] ;\\
\Phi^{-}_{m,0} = (\Phi^{-,1}_{m,0},\Phi^{-,2}_{m,0}):  \op{ind}_{T_{m,0}}^{T_0}\sigma_{m,0} & \to \op{ind}_{T_{m+1,0}}^{T_0}\sigma_{m+1,0} \oplus \op{ind}_{T_{m-1,0}}^{T_0}\sigma_{m-1,0}\\ 
[1,X^{r-j}Y^j] & \mapsto \left( [1, \delta_{j,r}Y^r] , [1, \delta_{j,0}X^r] \right).
\end{align*}
\end{lemma}
\begin{proof}
   The map $\Phi_{m,k}$ is determined by $\Phi ([s_{m,k}^{-1}, v ]),~v\in \sigma$. For  $k\geq 1$, \eqref{eq::formula_Phi} gives 
$$\Phi\left( \left[ \left(\begin{array}{cl}
    p^{m} & p^{-k} \\
    0 & 1
    \end{array}\right)^{-1},X^{r-j}Y^{j}\right] \right) = \left[ \left(\begin{array}{cl}
        p^{m+1} & p^{-k+1} \\
        0 & 1
        \end{array}\right)^{-1},\delta_{j,r}Y^r \right]+\sum_{i=0}^{p-1}\left[ \left(\begin{array}{cl}
        p^{m} & p^{-k}+i \\
        0 & p
        \end{array}\right)^{-1},i^jX^r \right], $$
   whereas, for $k = 0$,
    $$\Phi \left( \left[ \left(\begin{array}{cl}
    p^{m} & 0 \\
    0 & 1
    \end{array}\right)^{-1},X^{r-j}Y^{j}\right] \right) = \left[\left(\begin{array}{cl}
        p^{m+1} & 0 \\
        0 & 1
        \end{array}\right)^{-1},\delta_{j,r}Y^r\right]+\sum_{i=0}^{p-1} \left[\left(\begin{array}{cl}
        p^{m-1} & p^{-1}i \\
        0 & 1
        \end{array}\right)^{-1},i^jX^r \right]. $$

For $k\geq 1$, let 
\begin{align*}
\Phi^+_{m,k}([s_{m,k}^{-1}, X^{r-j}Y^j]) & : = \sum_{i=0}^{p-1}\left[ \left(\begin{array}{cl}
        p^{m} & p^{-k}+i \\
        0 & p
        \end{array}\right)^{-1},i^jX^r \right] \\
\Phi^-_{m,k}([s_{m,k}^{-1}, X^{r-j}Y^j] )& : =  \left[ \left(\begin{array}{cl}
        p^{m+1} & p^{-k+1} \\
        0 & 1
        \end{array}\right)^{-1},\delta_{j,r}Y^r \right].
\end{align*}
Then the formula for $\Phi^-_{m,k}$ with $k\geq 2$ is clear. For $\Phi^{-}_{m,1}$, it suffices to use the identity
    $$\left(\begin{array}{cl}
                p^{m+1} & 1 \\
                0 & 1
                \end{array}\right)^{-1} =\left(\begin{array}{cl}
                p^{m+1} & 0 \\
                0 & 1
                \end{array}\right)^{-1} \left(\begin{array}{cl}
                1 & -1 \\
                0 & 1
                \end{array}\right). $$
Since  
    $$\left(\begin{array}{cl}
        p^{m} & p^{-k}+i \\
        0 & p
        \end{array}\right)^{-1}= \left(\begin{array}{cl}
            1+p^{k}i &  \\
             & 1
            \end{array}\right)\left(\begin{array}{cl}
                p^{m-1} & p^{-k-1} \\
                0 & 1
                \end{array}\right)^{-1} \left(\begin{array}{cl}
                    p(1+p^{k}i) & 0 \\
                    0 & p
                    \end{array}\right)^{-1}, $$
we obtain the result for $\Phi^+_{m,k}$ with $k\geq 1$. 

For $k=0$, let
\begin{align*}
\Phi^+_{m,0} ([s_{m,0}^{-1}, X^{r-j} Y^j]) &:= \sum_{i=1}^{p-1} \left[\left(\begin{array}{cl}
        p^{m-1} & p^{-1}i \\
        0 & 1
        \end{array}\right)^{-1},i^jX^r \right]\\
\Phi^-_{m,0} ([s_{m,0}^{-1}, X^{r-j} Y^j]) &:= \left[\left(\begin{array}{cl}
        p^{m+1} & 0 \\
        0 & 1
        \end{array}\right)^{-1},\delta_{j,r}Y^r\right]+ \left[\left(\begin{array}{cl}
        p^{m-1} & 0 \\
        0 & 1
        \end{array}\right)^{-1}, \delta_{j,0}X^r \right].
\end{align*}
Then the formula for $\Phi^{-}_{m,0}$ follows. Finally, if $i \neq 0$, we have
$$\left(\begin{array}{cl}
                p^{m-1} & p^{-1}i \\
                0 & 1
                \end{array}\right)^{-1} =\left(\begin{array}{cl}
                i & 0 \\
                0 & 1
                \end{array}\right) \left(\begin{array}{cl}
                p^{m-1} & p^{-1} \\
                0 & 1
                \end{array}\right)^{-1} 
                \left(\begin{array}{cl}
                    i^{-1} & 0 \\
                    0 & 1
                    \end{array}\right),$$
which gives the formula for $\Phi^{+}_{m,0}$.
\end{proof}

Let $\eta: T_0\to \FF^{\times}$ be a smooth character such that $\eta|_{Z} = \zeta$. We write
\begin{equation}\label{eq::eta_T_0}
\eta = \chi^{n}\boxtimes \zeta,~ \text{for some  $n  \in \{ 0,\ldots, p-2\}$}.
\end{equation}
The representation $\sigma_{k}\otimes \eta^{-1}|_{T_{k}}$ has trivial $Z$-action, and hence is naturally a $T_k/Z$-representation. We denote it by $\overline{\sigma}_k$. By \eqref{eq::isom_ext_twist}, we have isomorphisms
\begin{equation} \label{eq::Ext_reduce_mod_center_T}
\op{Ext}^i_{T_0,\zeta } (\op{ind}^{T_0}_{T_{k}} \sigma_{k} ,\eta) \simeq \op{Ext}^i_{T_0 /Z} (\op{ind}^{T_0/Z}_{T_k/Z}
\overline{\sigma}_{k}, \FF ) \simeq \op{Ext}^i_{ \FF[[T_0/Z]]} (\FF, ( \op{ind}^{T_0/Z}_{T_k/Z} \overline{\sigma}_k)^{\vee} ) ,~ \forall i,
\end{equation}
and
\[
\op{Ext}^i_{T_0,\zeta } (\eta, \op{ind}^{T_0}_{T_k} \sigma_k) \simeq \op{Ext}^i_{T_0/Z} (\FF, \op{ind}^{T_0/Z}_{T_k/Z} \overline{\sigma}_k )\simeq  \op{Ext}^i_{ \FF[[T_0/Z]]} (( \op{ind}^{T_0/Z}_{T_k/Z} \overline{\sigma}_k)^{\vee}, \FF ) ,~ \forall i, 
\]
where $\op{Ext}^i_{ \FF[[T_0/Z]]}$ stands for the $i$-th extension group in ${\rm Mod}^{\rm pro}_{T_0/Z}(\FF)$. Because $T_0/Z$ is compact, the category ${\rm Mod}^{\rm pro}_{T_0/Z}(\FF)$ is the same as the category of compact $\FF[[T_0/Z]]$-modules.

\begin{lemma} \label{lemma::action_T_0_sigma_bar}
\begin{enumerate}
\item[(i)] The $T_0/Z$-representation $\overline{\sigma}_0$ is semisimple, with decomposition
\begin{equation*} \label{eq::dec_T_0_sigma_0}
\overline{\sigma}_{0}\cong\bigoplus_{j=0}^{r} \chi^{r- j -n },
\end{equation*}
where $n$ is given by \eqref{eq::eta_T_0}, and $X^{r-j} Y^j$ is a $\chi^{r-j-n}$-eigenvector.

\item[(ii)] If $k\geq 1$, $\overline{\sigma}_k$ has an $\FF$-basis $\{ X^{r-j} Y^j,~ 0\leq j\leq r\}$, with $T_k/Z$-action given by 
\begin{equation*} \label{eq::action_of_T_k/Z}
(1 + p^k b)  \cdot X^{r-j} Y^j =  X^{r-j} (-bX + Y)^j,~b\in \ZZ_p. 
\end{equation*}
\end{enumerate}
\end{lemma}
\begin{proof}
The first statement follows from Lemma \ref{lemma::T_{m,k}}(ii) and the second statement is a consequence of Lemma \ref{lemma::T_{m,k}}(iii).
\end{proof}

Let $\mu_{p-1} \subset \ZZ_p^{\times}$ be the subgroup of $(p-1)$st roots of unity and let $\zeta_{p-1} \in \mu_{p-1}$ be a primitive $(p-1)$st root. Then the group algebra $\FF[\mu_{p-1}]$ is isomorphic to $\FF[t']/((1+t')^{p-1} - 1)$ by identifying $\zeta_{p-1}^{-1}$ with $1+t'$. Let $t:= \zeta_{p-1} - 1$. Then  $\FF[\mu_{p-1}]$ is also isomorphic to $\FF[t]/((1+t)^{p-1} - 1)$. Let $\tau$ be an $\FF[\mu_{p-1}]$-module, and let $\tau^{\vee}$ be its $\FF$-linear dual. Then for $f\in \tau^{\vee},~ v\in \tau$, we have
\begin{equation}\label{eq::action_of_t_T_0}
((1+t)\cdot f) (v) = f(\zeta_{p-1}^{-1} v) = f( (1+t') v).
\end{equation}

For $k\geq 1$, let $\gamma_k : =  \bigl(\begin{smallmatrix}
    1-p & 0 \\
    0 & 1
    \end{smallmatrix} \bigr)^{p^{k-1}} \in T_k $.

\begin{lemma}
\begin{itemize}
\item[(i)] The quotient group $T_k/Z$ is isomorphic to $\ZZ_p$, for all $k\geq 1$.

\item[(ii)] The composite $\gamma_k^{\ZZ_p} \to T_k \to  T_k/Z$ is an isomorphism of $p$-adic Lie groups.
\end{itemize}
\end{lemma}
\begin{proof}
(i) Since $p\geq 3$, the $p$-adic logarithm map induces an isomorphism between $T_k / Z \simeq 1+ p^k \ZZ_p$ and $p^k \ZZ_p$, for all $k\geq 1$, see \cite[Prop. 4.3.12]{Coh07}. The proof of (ii) is the same as that of Lemma \ref{lemma::p-torsion_free_ramified}(ii).  
\end{proof}

As a consequence, for $k\geq 1$
\[
\FF[[T_k/Z]] \simeq \FF[[z_k]],
\]
with $z_k : = \gamma_k - 1$. 

We define a map $\alpha_k : \gamma_k^{\ZZ_p} \to U(\FF_p)$ by sending $\gamma_k$ to $\bigl(\begin{smallmatrix}
1 & 1 \\ 0& 1
\end{smallmatrix} \bigr) \in U(\FF_p)$. By Lemma \ref{lemma::structure_sigma_dual}, $\sigma$ is a representation of $\gamma_k^{\ZZ_p}$ via the map $\alpha_k$. Comparing with the action explicitly given in Lemma \ref{lemma::action_T_0_sigma_bar}(ii), $\sigma$ is isomorphic to $\overline{\sigma}_k$ as $\gamma_k^{\ZZ_p}$-representations. Let $z_k' : = \gamma_k^{-1} - 1 \in \FF[[T_k/Z]]$. Let 
\[
e_0 : = Y^r \in \overline{\sigma}_k,~ e_i : = z_k'^i e_0 \in \overline{\sigma}_k,~ 0\leq i\leq r.
\]
Let $\overline{\sigma}_k^{\vee}$ denote the $\FF$-linear dual of $\overline{\sigma}_k$. For $f\in \overline{\sigma}_k^{\vee}$ and $v\in \overline{\sigma}_k $, we have
\begin{equation} \label{action_of_z_k_T_0}
(z_k f) (v) = f(z_k' v).
\end{equation}
Let $e_0^{\vee},\ldots, e_r^{\vee}$ be the dual basis of $\overline{\sigma}_k^{\vee}$. Then $e_i^{\vee} = z_k^{r-i} e_r^{\vee}$.

Let 
\[
A : = \FF[t]/((1+t)^{p-1} - 1).
\]

\begin{lemma} \label{lemma::structure_sigma_k_T_0}
\begin{enumerate}
\item[(i)] For $k\geq 1$, the inclusion $\FF [[T_{k+1}/Z ]] \to \FF [[T_k /Z]] $ corresponds to the map $\FF[[z_{k+1}]] \to \FF[[z_k]]$ sending $z_{k+1}$ to $z_k^p$.

\item[(ii)] For $k\geq 1$, the $\FF[[T_k/Z ]]$-module $\overline{\sigma}_k^{\vee}$ is cyclic, generated by $e_r^{\vee} $. In particular, $\overline{\sigma}_k^{\vee}$ is isomorphic to $\FF[[ z_k]]/  z_k^{r+1} $.

\item[(iii)] The $\FF$-algebra $\FF[[T_0 /Z]]$ is isomorphic to $A[[z_1]] $. Moreover, we have an isomorphism of  $\FF[[T_0 /Z]]$-modules 
\begin{equation} \label{eq::dual_ind_T_0_description}
( \op{ind}^{T_0/Z }_{T_k /Z } \overline{\sigma}_k)^{\vee} \simeq  \left\{\begin{array}{cl}
                                                            \bigoplus_{j=0}^r A/ (1+t - \omega( \zeta_{p-1})^{n+j -r})  ,~ & \text{if $k= 0$,}\\
                                                              A[[z_1]]/z_1^{p^{k-1}(r+1)},~ &\text{if $k\geq 1$}.
                                                                \end{array}\right.
\end{equation}
\end{enumerate}
\end{lemma}
\begin{proof}
The proof of (i) and (ii) is similar to that of Lemma \ref{lemma::structure_sigma_k}. For (iii), since $T_0 / Z \simeq \mu_{p-1} \times T_1/Z$, we have 
\[
\FF[[T_0 /Z]]\cong \FF[\mu_{p-1} ] \widehat{\otimes}_{\FF} \FF[[z_1]] \cong A[[z_1]] .
\]
The description of $( \op{ind}^{T_0/Z }_{T_k /Z } \overline{\sigma}_k)^{\vee}$ for $k\geq 1$ follows from (i) and (ii).

Consider the character $\chi^i$ of $\mu_{p-1}$. For $f\in (\chi^i)^{\vee},~ v\in \chi^i$, we have
\[
((1+t) f) (v) = f(\zeta_{p-1}^{-1} v) = f(\chi^{i} (\zeta_{p-1}^{-1}) v) = \chi^{i} (\zeta_{p-1}^{-1}) f(v) = \omega(\zeta_{p-1})^{-i} f(v).
\]
Hence $(1+t)f = \omega(\zeta_{p-1})^{-i} f$. We then have an isomorphism of $A$-modules
\[
(\chi^i)^{\vee} \simeq A/(1+t - \omega(\zeta_{p-1})^{-i}).
\]
In the decomposition $\overline{\sigma}_{0}\cong\bigoplus_{j=0}^{r} \chi^{r-j-n }$, the vector
$X^{r-j} Y^j$ is a $\chi^{r-j-n }$-eigenvector. Therefore
\begin{equation} \label{eq::dec_sigma_0_dual}
\FF (X^{r-j} Y^j)^{\vee} \simeq A/ (1+t - \omega(\zeta_{p-1})^{n+j -r}).
\end{equation}
The statement for $k=0$ follows.
\end{proof}

\quash{
\begin{lemma} \label{lemma::dual_T_0}
The $\FF$-linear dual map induces isomorphisms
\begin{align*}
\op{Ext}^i_{T_0/Z} (\FF, \op{ind}^{T_0/Z}_{T_k/Z} \overline{\sigma}_k ) & \simeq \op{Ext}^i_{ \FF[[T_0/Z]]} (( \op{ind}^{T_0/Z}_{T_k/Z} \overline{\sigma}_k)^{\vee}, \FF )\\
\op{Ext}^i_{T_0/Z} ( \op{ind}^{T_0/Z}_{T_k/Z} \overline{\sigma}_k,\FF ) & \simeq \op{Ext}^i_{ \FF[[T_0/Z]]} (\FF, ( \op{ind}^{T_0/Z}_{T_k/Z} \overline{\sigma}_k)^{\vee} ),~ \forall i.
\end{align*}
\end{lemma}
\begin{proof}
This follows from the fact that the $\FF$-linear dual map induces an anti-equivalence of categories between smooth representations of $T_0/Z$ and the pseudo-compact $\FF[[T_0/Z]] $-modules, see \cite[Thm. 1.5]{Koh}.
\end{proof}
}

For $i\in \ZZ/(p-1)\ZZ$, we let 
\[
p_i(t) : = \frac{(1+t)^{p-1} -1}{1+t - \omega(\zeta_{p-1})^i} \in \FF[t].
\]
Let $m(r,n)$ be the number of integers $0\leq j\leq r$ such that $n+j - r\equiv 0 \pmod{p-1}$. Then
\[
m(r,n) : = \left\{\begin{array}{cl}
                                                             0  ,~ & \text{if $n>r$,}\\
                                                               2,~ &\text{if $(r,n) = (p-1, 0)$,} \\
1,~ &\text{otherwise when $n\leq r$.} 
                                                                \end{array}\right.
\]

\begin{proposition} \label{prop::Ext^i_T_0}
Let $\eta: T_0\to \FF^{\times}$ be the character given by \eqref{eq::eta_T_0}.
\begin{enumerate}
\item[(i)] If $k = 0$, then
\[
 \op{Ext}^i_{ \FF[[T_0/Z]]} (\overline{\sigma}_0^{\vee}, \FF ) \simeq \left\{\begin{array}{cl}
                                                             \FF^{\oplus m(r,n)}  ,~ & \text{if $i = 0,1$,}\\
                                                               0,~ &\text{else}
                                                                \end{array}\right.
\]
and
\[
 \op{Ext}^i_{ \FF[[T_0/Z]]} ( \FF, \overline{\sigma}_0^{\vee} )    \simeq \left\{\begin{array}{cl}
                                                              (A  /t)^{\oplus m(r,n)}  ,~ & \text{if $i = 0,1$,}\\
                                                               0,~ &\text{else}.
                                                                \end{array}\right.
\]

\item[(ii)] If $k\geq 1$ then
\[
 \op{Ext}^i_{ \FF[[T_0/Z]]} (( \op{ind}^{T_0/Z}_{T_k / Z} \overline{\sigma}_k)^{\vee}, \FF ) \simeq \left\{\begin{array}{cl}
                                                             \FF  ,~ & \text{if $i = 0,1$,}\\
                                                               0,~ &\text{else}
                                                                \end{array}\right.
\]
and
\[
\op{Ext}^i_{ \FF[[T_0/Z]]} (\FF, ( \op{ind}^{T_0/Z}_{T_k/Z} \overline{\sigma}_k)^{\vee} ) =   \left\{\begin{array}{cl}
                                                     z_1^{p^{k-1}(r+1) - 1} p_0(t)  \left( A [[z_1]]/z_1^{p^{k-1}(r+1)} \right) ,~ & \text{if $i = 0$,}\\
p_0(t)\frac{ A[[z_1]]/z_1^{p^{k-1}(r+1)}}{(z_1)}  ,~ & \text{if $i = 1$,}\\
                                                               0,~ &\text{else}.
                                                                \end{array}\right.
\]
In particular, $\op{Ext}^i_{ \FF[[T_0/Z]]} (\FF, ( \op{ind}^{T_0/Z}_{T_k/Z} \overline{\sigma}_k)^{\vee} )$ is one-dimensional if $i = 0,1$.
\end{enumerate}
\end{proposition}
\begin{proof}
Let $P(i)_{\bullet}$ denote the following projective resolution of $A/(1+t - \omega(\zeta_{p-1})^i)$ in ${\rm Mod}^{\rm pro}_{T_0/Z}(\FF)$:
\begin{equation} \label{eq::proj_res_one_dim_T_0}
\begin{multlined} 
\cdots   \xrightarrow{\bigl(\begin{smallmatrix}
p_i(t) & z_1 \\            0 & 1+t - \omega(\zeta_{p-1})^i
\end{smallmatrix} \bigr)}  A[[z_1]]\oplus A[[z_1]]  \xrightarrow{\bigl(\begin{smallmatrix}
1+t - \omega(\zeta_{p-1})^i & -z_1 \\            0 & p_i(t)
\end{smallmatrix} \bigr)}  A[[z_1]]\oplus A[[z_1]] \\ \xrightarrow{\bigl(\begin{smallmatrix}
p_i(t) & z_1 \\            0 & 1+t - \omega(\zeta_{p-1})^i
\end{smallmatrix} \bigr)}  A[[z_1]]\oplus A[[z_1]]  \xrightarrow{(-(1+t - \omega(\zeta_{p-1})^i),z_1) } A[[z_1]] \to A/(1+t - \omega(\zeta_{p-1})^i) \to 0.
\end{multlined}
\end{equation}

By Lemma \ref{lemma::structure_sigma_k_T_0}(iii),
\[
 \op{Ext}^i_{ \FF[[T_0/Z]]} (\overline{\sigma}_0^{\vee}, \FF )  \simeq \bigoplus_{j=0}^r  \op{Ext}^i_{A[[z_1]]} (  A/ (1+t - \omega( \zeta_{p-1})^{n+j -r}) , \FF ) .
\]
To compute $\op{Ext}^i_{ A[[z_1]]} (  A/ (1+t - \omega( \zeta_{p-1})^{n+j -r}) , \FF ) $, we use the projective resolution:
\begin{equation*}
 P(n+j-r)_{\bullet}  \to   A/ (1+t - \omega( \zeta_{p-1})^{n+j -r}) \to 0.
\end{equation*}
Then 
\begin{equation*}
\begin{multlined}
\op{Ext}^i_{ A[[z_1]]} (  A/ (1+t - \omega( \zeta_{p-1})^{n+j -r}) , \FF ) = H^i ( \FF  \xrightarrow{ \bigl(\begin{smallmatrix}
-(1+t - \omega(\zeta_{p-1})^{n+j -r}) \\  z_1
\end{smallmatrix} \bigr) } \FF \oplus \FF \\ \xrightarrow{\bigl(\begin{smallmatrix}
p_{n+j -r}(t) & 0 \\           z_1 & 1+t - \omega(\zeta_{p-1})^{n+j -r}
\end{smallmatrix} \bigr) }  
\FF \oplus \FF  \xrightarrow{\bigl(\begin{smallmatrix}
1+t - \omega(\zeta_{p-1})^{n+j -r} & 0 \\          -z_1 & p_{n+j -r}(t)
\end{smallmatrix} \bigr)} \FF \oplus \FF \\  \xrightarrow{\bigl(\begin{smallmatrix}
p_{n+j -r}(t) & 0 \\           z_1 & 1+t - \omega(\zeta_{p-1})^{n+j -r}
\end{smallmatrix} \bigr) }  \cdots ).
\end{multlined}
\end{equation*}
If $n+j-r \not\equiv 0 \pmod{p-1}$, then $\omega( \zeta_{p-1})^{n+j -r} \neq 1$ and
\begin{equation*}
\begin{multlined}
\op{Ext}^i_{ A[[z_1]]} (  A/ (1+t - \omega( \zeta_{p-1})^{n+j -r}) , \FF ) = H^i ( \FF  \xrightarrow{ \bigl(\begin{smallmatrix}
-(1 - \omega(\zeta_{p-1})^{n+j -r}) \\  0
\end{smallmatrix} \bigr) } \FF \oplus \FF  \xrightarrow{\bigl(\begin{smallmatrix}
0 & 0 \\          0 & 1 - \omega(\zeta_{p-1})^{n+j -r}
\end{smallmatrix} \bigr) }  \\
\FF \oplus \FF  \xrightarrow{\bigl(\begin{smallmatrix}
1 - \omega(\zeta_{p-1})^{n+j -r} & 0 \\          0 & 0
\end{smallmatrix} \bigr)} \FF \oplus \FF  \xrightarrow{\bigl(\begin{smallmatrix}
0 & 0 \\          0 & 1 - \omega(\zeta_{p-1})^{n+j -r}
\end{smallmatrix} \bigr) }  \cdots ) .
\end{multlined}
\end{equation*}
It is straightforward to check that $\op{Ext}^i_{ A[[z_1]]} (  A/ (1+t - \omega( \zeta_{p-1})^{n+j -r}) , \FF ) = 0$ for all $i$.

If $n+j-r \equiv 0 \pmod{p-1}$, then
\begin{align*}
\op{Ext}^i_{ A[[z_1]]} (  A/ t , \FF ) &= H^i ( \FF  \xrightarrow{ \bigl(\begin{smallmatrix}
0 \\  0
\end{smallmatrix} \bigr) } \FF \oplus \FF  \xrightarrow{\bigl(\begin{smallmatrix}
-1 & 0 \\          0 & 0
\end{smallmatrix} \bigr) } \FF \oplus \FF  \xrightarrow{\bigl(\begin{smallmatrix}
0 & 0 \\          0 & -1
\end{smallmatrix} \bigr)} \FF \oplus \FF  \xrightarrow{\bigl(\begin{smallmatrix}
-1 & 0 \\          0 & 0
\end{smallmatrix} \bigr) }  \cdots ) \\
& \simeq \left\{\begin{array}{cl}
                                                             \FF  ,~ & \text{if $i = 0,1$,}\\
                                                               0,~ &\text{else}.
                                                                \end{array}\right.
\end{align*}
This finishes the calculation of $ \op{Ext}^i_{ \FF[[T_0/Z]]} (\overline{\sigma}_0^{\vee}, \FF ) $.

To compute $\op{Ext}^i_{ \FF[[T_0/Z]]} (\FF, \overline{\sigma}_0^{\vee} )$, we have
\[
 \op{Ext}^i_{ \FF[[T_0/Z]]} (\FF,\overline{\sigma}_0^{\vee} )  \simeq \bigoplus_{j=0}^r  \op{Ext}^i_{A[[z_1]]} (\FF,  A/ (1+t - \omega( \zeta_{p-1})^{n+j -r}) ) .
\]
We use the projective resolution $ P(0)_{\bullet}$ of the trivial module $\FF$. The proof is similar to the argument given above. This proves (i). 

For (ii), the proof is similar to that of Proposition \ref{prop::Ext^i_L_ramified}. It suffices to use the following projective resolution of $( \op{ind}^{T_0/Z}_{T_k / Z} \overline{\sigma}_k)^{\vee}$ in ${\rm Mod}^{\rm pro}_{T_0/Z}(\FF)$
\[
0\to A[[z_1]]  \xrightarrow{z_1^{p^{k-1}(r+1)}} A[[z_1]] \to ( \op{ind}^{T_0/Z}_{T_k / Z} \overline{\sigma}_k)^{\vee} \simeq A[[z_1]] / z_1^{p^{k-1}(r+1)} \to 0
\]
in the calculation of $ \op{Ext}^i_{ \FF[[T_0/Z]]} (( \op{ind}^{T_0/Z}_{T_k / Z} \overline{\sigma}_k)^{\vee}, \FF )$ for $k\geq 1$. The details are left to the reader. 
\end{proof}

To simplify the notation, we let
\[
(\op{ind}^{T_0/Z}_{T_k/Z} \overline{\sigma}_k)_m : =\op{ind}^{T_0/Z }_{T_{m,k} /Z } \overline{\sigma}_{k}.
\]
The morphisms $\Phi_{m,k}^{\pm}$ in Lemma \ref{lemma::Phi_T_0} induce maps on $(\op{ind}^{T_0/Z}_{T_k/Z} \overline{\sigma}_k)_m$.  By abuse of notation, we denote the induced maps by $\Phi_{m,k}^{\pm}$.  We have the following proposition.

\begin{proposition} \label{prop::Phi_k_T_0}
The following statements hold:
\begin{enumerate}
\item[(i)] For $k\geq 1$, the map $\Phi^{+}_{m,k}:  (\op{ind}^{T_0/Z}_{T_k/Z} \overline{\sigma}_k)_m  \to (\op{ind}^{T_0/Z}_{T_{k+1}/Z} \overline{\sigma}_{k+1})_{m-1}$ is injective. Under the isomorphism \eqref{eq::dual_ind_T_0_description}, the dual map 
\[
(\Phi^+_{m,k})^{\vee}: A[[z_1]]/z_1^{p^{k}(r+1)} \to A[[z_1]]/z_1^{p^{k-1}(r+1)}
\]
is multiplication by some $u_k(z_1) \in  1 + z_1 \FF[[z_1]]$.

\item[(ii)] For $k \geq 2$, $\Phi^{-}_{m,k} : (\op{ind}^{T_0/Z}_{T_k/Z} \overline{\sigma}_k)_m \to (\op{ind}^{T_0/Z}_{T_{k-1}/Z} \overline{\sigma}_{k-1})_{m+1}$ is surjective. Under the isomorphism \eqref{eq::dual_ind_T_0_description}, the dual map 
\[
(\Phi^-_{m,k})^{\vee}: A[[z_1]]/z_1^{p^{k-2}(r+1)} \to A[[z_1]]/z_1^{p^{k-1}(r+1)}
\]
is multiplication by $z_1^{(p^{k-1} - p^{k-2})(r+1)}$.

\item[(iii)] Under the isomorphism \eqref{eq::dual_ind_T_0_description}, the map 
\[
(\Phi^-_{m,1})^{\vee}:  \bigoplus_{j=0}^r A/ (1+t - \omega( \zeta_{p-1})^{n+j -r})  \to A[[z_1]]/z_1^{r+1}
\]
is given by the matrix
\[
\begin{bmatrix}
a_0(z_1) & a_1(z_1) & \cdots & a_r(z_1) 
\end{bmatrix},
\]
where $a_j(z_1) =  (-1)^{r-j}  \binom{r}{j} \frac{(1+t)^{p-1} - 1}{\omega(\zeta_{p-1})^{r-j-n}(1+t) - 1 } z_1^r  $.

\item[(iv)] Under the isomorphism \eqref{eq::dual_ind_T_0_description}, the map 
\[
(\Phi^+_{m,0})^{\vee}: A[[z_1]]/z_1^{r+1} \to  \bigoplus_{j=0}^r A/ (1+t - \omega( \zeta_{p-1})^{n+j -r}) 
\]
is given by the matrix 
\[
\frac{(-1)^r}{r!} \begin{bmatrix}
1 & 1 & \cdots & 1 
\end{bmatrix}^T.
\]

\item[(v)] Under the isomorphism \eqref{eq::dual_ind_T_0_description}, the map 
\[
(\Phi^{-,1}_{m,0})^{\vee}:  \bigoplus_{j=0}^r A/ (1+t - \omega( \zeta_{p-1})^{n+j -r})  \to  \bigoplus_{j=0}^r A/ (1+t - \omega( \zeta_{p-1})^{n+j -r}) 
\]
is given by the matrix
\[
\begin{bmatrix}
0 & \cdots & \cdots & 0 \\
\vdots & \ddots & \vdots & \vdots \\
\vdots & \cdots & 0 & 0 \\
0 & \cdots & 0 & 1
\end{bmatrix},
\]
and the map  
\[
(\Phi^{-,2}_{m,0})^{\vee}:  \bigoplus_{j=0}^r A/ (1+t - \omega( \zeta_{p-1})^{n+j -r})  \to  \bigoplus_{j=0}^r A/ (1+t - \omega( \zeta_{p-1})^{n+j -r}) 
\]
is given by the matrix 
\[
\begin{bmatrix}
1 & 0 & \cdots & 0 \\
0 & 0 & \cdots & \vdots \\
\vdots & \vdots & \ddots & \vdots \\
0 & \cdots & \cdots &0
\end{bmatrix}.
\]
\end{enumerate}
\end{proposition}

\begin{proof}
The proof of (i) and (ii) is similar to that of Proposition \ref{prop::Phi_k_L_ramified}(i) and (ii).

(iii) Recall that the vectors $ (1+t')^i z_1'^l [1,e_0],~0\leq i \leq p-2,~ 0\leq l\leq r$ form a basis of $(\op{ind}^{T_0/Z}_{T_1/Z} \overline{\sigma}_1)_m$, and their dual vectors $ ((1+t')^i z_1'^l [1,e_0])^{\vee}$ form a basis of $(\op{ind}^{T_0/Z}_{T_1/Z} \overline{\sigma}_1)_m^{\vee}$. By \eqref{eq::action_of_t_T_0} and \eqref{action_of_z_k_T_0}
\[
(1+t)z_1 \cdot ((1+t')^i z_1'^l [1,e_0])^{\vee} = ((1+t')^{i-1} z_1'^{l-1} [1,e_0])^{\vee},~1\leq l \leq r.
\]
It follows that $(\op{ind}^{T_0/Z}_{T_1/Z} \overline{\sigma}_1)_m^{\vee}$ is generated by $(z_1'^r [1,e_0])^{\vee}$ because $1+t$ is an invertible element in $A$. We have
\begin{align*}
(\Phi^-_{m,1})^{\vee} ((X^{r-j} Y^j)^{\vee}) ((1+t')^i z_1'^{l} [1,e_0]) &=  (X^{r-j} Y^j)^{\vee} (\Phi^-_{m,1} ((1+t')^i z_1'^{l} [1,e_0])) \\
& =  (X^{r-j} Y^j)^{\vee} ((1+t')^i \Phi^-_{m,1} ( z_1'^{l} [1,e_0])) .
\end{align*}
By Lemma \ref{lemma::Phi_T_0}, $\Phi^-_{m,1} ( z_1'^{l} [1,e_0]) = 0$ if $l\neq 0$; and if $l=0$,
\begin{align*}
(X^{r-j} Y^j)^{\vee} ((1+t')^i \Phi^-_{m,1} ( [1,e_0])) &=  (X^{r-j} Y^j)^{\vee} ((1+t')^i  (-X+Y)^r) \\
& = (-1)^{r-j}\omega(\zeta_{p-1})^{(-r+j+n)i}    \binom{r}{j},
\end{align*}
where the second equality follows from the fact that $X^{r-j} Y^j$ is a $\chi^{r-j-n}$-eigenvector for the action of $T_0/Z$. Therefore,
\begin{align*}
(\Phi^-_{m,1})^{\vee} ((X^{r-j} Y^j)^{\vee})  &= (-1)^{r-j}  \binom{r}{j}\sum_{i=0}^{p-2} \omega(\zeta_{p-1})^{(-r+j+n)i}  ((1+t')^i [1,e_0])^{\vee} \\
& = (-1)^{r-j}  \binom{r}{j}\sum_{i=0}^{p-2} \omega(\zeta_{p-1})^{(-r+j+n)i} (1+t)^{-i} ([1,e_0])^{\vee}\\
&= (-1)^{r-j}  \binom{r}{j}\sum_{i=0}^{p-2} \omega(\zeta_{p-1})^{(-r+j+n)i} (1+t)^{-i} z_1^r (z_1'^r[1,e_0])^{\vee}\\
&= (-1)^{r-j}  \binom{r}{j} z_1^r \frac{(1+t)^{p-1} - 1}{\omega(\zeta_{p-1})^{r-j-n}(1+t) - 1 }  (z_1'^r[1,e_0])^{\vee} \\
& = a_j(z_1) (z_1'^r[1,e_0])^{\vee}.
\end{align*}

(iv) 
Tensor the maps of Lemma \ref{lemma::Phi_T_0} by $\eta^{-1}$ and transport them through the projection-formula identification
\[
(\op{ind}^{T_0}_{T_{k}} \sigma_{k})\otimes \eta^{-1} \to \op{ind}^{T_0}_{T_{k}} (\sigma_{k}\otimes\eta^{-1}),~[g,v]\otimes 1 \mapsto \eta(g) [g,v].
\]
We have
\begin{align*}
(\Phi^+_{m,0})^{\vee} ( (z_1'^r [1,e_0])^{\vee}) (X^{r-j} Y^j)  & = (z_1'^r [1,e_0])^{\vee} (\Phi^+_{m,0} (X^{r-j} Y^j))\\
& = (z_1'^r [1,e_0])^{\vee} \left( \sum_{i=1}^{p-1}\left(\begin{array}{cl}
            i &  \\
             & 1
            \end{array}\right) i^n  \left[ 1 , i^{j-r} X^r \right] \right) \\
& = (z_1'^r [1,e_0])^{\vee} \left( \sum_{i=0}^{p-2}  \zeta_{p-1}^{i(r-j -n)}\left(\begin{array}{cl}
            \zeta_{p-1}^{-i} &  \\
             & 1
            \end{array}\right) \left[ 1 , X^r \right] \right) \\
& = (z_1'^r [1,e_0])^{\vee} \left( \sum_{i=0}^{p-2}  \zeta_{p-1}^{i(r-j -n )} (1+t')^{i}  \left[ 1 , X^r \right] \right) \\
& = (z_1'^r [1,e_0])^{\vee} \left( \sum_{i=0}^{p-2}  \zeta_{p-1}^{i(r-j  -n)} (1+t')^{i} \frac{(-1)^r}{r!} z_1'^r \left[ 1 , e_0 \right] \right) \\
& = \frac{(-1)^r}{r!}.
\end{align*}
Therefore, 
\[
(\Phi^+_{m,0})^{\vee} ( (z_1'^r [1,e_0])^{\vee}) =\frac{(-1)^r}{r!}  \sum_{j=0}^r (X^{r-j} Y^j)^{\vee}.
\]

(v) follows easily from the expressions of $\Phi_{m,0}^{-,1}$ and $\Phi_{m,0}^{-,2}$ in Lemma \ref{lemma::Phi_T_0}.
\end{proof}

The following two lemmas describe the induced maps of $\Phi^{\pm}_{m,k}$ between extension groups. Let $(\Phi^{\pm}_{m,k})^{*}$ (resp. $(\Phi^{\pm}_{m,k})_{*}$) denote the maps between $\op{Ext}^i_{\FF[[T_0/Z]]}(\FF, (\op{ind}^{T_0/Z}_{T_{k}/Z} \overline{\sigma}_{k})_{m}^{\vee})$ (resp. $\op{Ext}^i_{\FF[[T_0/Z]]}((\op{ind}^{T_0/Z}_{T_{k}/Z} \overline{\sigma}_{k})_{m}^{\vee},\FF)$) induced by $(\Phi^{\pm}_{m,k})^{\vee}$.
 
\begin{lemma} \label{lemma::Phi_Ext^i_T_0_1}
Let $\eta: T_0\to \FF^{\times}$ be the character $\chi^{n}\boxtimes \zeta$ given by \eqref{eq::eta_T_0}. Under the isomorphisms in Proposition \ref{prop::Ext^i_T_0}, the induced maps have the following descriptions.
In (iv)--(vii), assume $(r,n)\ne(p-1,0)$.
\begin{enumerate}
\item[(i)] For $k\geq 1$, the map
\[
(\Phi^{+}_{m,k})^{*} : \op{Ext}^i_{\FF[[T_0/Z]]}(\FF, (\op{ind}^{T_0/Z}_{T_{k+1}/Z} \overline{\sigma}_{k+1})_{m-1}^{\vee}) \to \op{Ext}^i_{\FF[[T_0/Z]]}(\FF, (\op{ind}^{T_0/Z}_{T_{k}/Z} \overline{\sigma}_{k})_{m}^{\vee})
\]
is zero if $i=0$, and is the identity map if $i=1$.

\item[(ii)] For $k\geq 2$, the map
\[
(\Phi^{-}_{m,k})^{*} : \op{Ext}^i_{\FF[[T_0/Z]]}(\FF, (\op{ind}^{T_0/Z}_{T_{k-1}/Z} \overline{\sigma}_{k-1})_{m+1}^{\vee}) \to \op{Ext}^i_{\FF[[T_0/Z]]}(\FF, (\op{ind}^{T_0/Z}_{T_{k}/Z} \overline{\sigma}_{k})_{m}^{\vee})
\]
is the identity map if $i=0$, and is zero if $i=1$.

\item[(iii)] If $n\notin \{0,\ldots,r\}$, then the maps $(\Phi^{-}_{m,1})^{*},~(\Phi^{+}_{m,0})^{*},~(\Phi^{-,1}_{m,0})^{*},~(\Phi^{-,2}_{m,0})^{*}$ are all zero.

\item[(iv)] If $n\in \{0,\ldots,r\}$, then the map
\[
(\Phi^{-}_{m,1})^{*} : \op{Ext}^i_{\FF[[T_0/Z]]}(\FF, (\overline{\sigma}_{0})_{m+1}^{\vee}) \to \op{Ext}^i_{\FF[[T_0/Z]]}(\FF, (\op{ind}^{T_0/Z}_{T_{1}/Z} \overline{\sigma}_{1})_{m}^{\vee})
\]
is multiplication by $(-1)^n\binom{r}{n}$ if $i=0$; if $i=1$, it is zero if $r\geq 1$ and is the identity map if $r=0$.

\item[(v)] If $n\in \{0,\ldots,r\}$, then the map
\[
(\Phi^{+}_{m,0})^{*} : \op{Ext}^i_{\FF[[T_0/Z]]}(\FF, (\op{ind}^{T_0/Z}_{T_{1}/Z} \overline{\sigma}_{1})_{m-1}^{\vee}) \to \op{Ext}^i_{\FF[[T_0/Z]]}(\FF, (\overline{\sigma}_{0})_{m}^{\vee})
\]
is multiplication by $\frac{(-1)^{r+1}}{r!}$ if $i=1$; if $i=0$, it is zero if $r\geq 1$ and is multiplication by $-1$ if $r=0$.

\item[(vi)] If $n\in \{0,\ldots,r\}$, then the map
\[
(\Phi^{-,1}_{m,0})^{*} : \op{Ext}^i_{\FF[[T_0/Z]]}(\FF, (\overline{\sigma}_{0})_{m+1}^{\vee}) \to \op{Ext}^i_{\FF[[T_0/Z]]}(\FF, (\overline{\sigma}_{0})_{m}^{\vee})
\]
is the identity map if $n=0$, and is zero if $n\ne0$.

\item[(vii)] If $n\in \{0,\ldots,r\}$, then the map
\[
(\Phi^{-,2}_{m,0})^{*} : \op{Ext}^i_{\FF[[T_0/Z]]}(\FF, (\overline{\sigma}_{0})_{m-1}^{\vee}) \to \op{Ext}^i_{\FF[[T_0/Z]]}(\FF, (\overline{\sigma}_{0})_{m}^{\vee})
\]
is the identity map if $n=r$, and is zero if $n\ne r$.

\end{enumerate}
If $(r,n)=(p-1,0)$, use the ordered bases induced by
$\FF(X^{p-1})^{\vee}$ and $\FF(Y^{p-1})^{\vee}$ for the boundary Ext groups,
and the bases $z_1^{p-1}p_0(t)$ and $p_0(t)$ for the $k=1$ block
in degrees $0$ and $1$, respectively. Then
\[
\begin{array}{c|c|c}
\text{Map} & i=0 & i=1 \\ \hline
(\Phi^{-}_{m,1})^{*}
 & \begin{bmatrix}1&1\end{bmatrix} & 0 \\
(\Phi^{+}_{m,0})^{*}
 & 0 & \begin{bmatrix}1\\1\end{bmatrix} \\
(\Phi^{-,1}_{m,0})^{*}
 & \begin{bmatrix}0&0\\0&1\end{bmatrix}
 & \begin{bmatrix}0&0\\0&1\end{bmatrix} \\
(\Phi^{-,2}_{m,0})^{*}
 & \begin{bmatrix}1&0\\0&0\end{bmatrix}
 & \begin{bmatrix}1&0\\0&0\end{bmatrix}.
\end{array}
\]
\end{lemma}

\begin{proof}
(i) and (ii) follow directly from parts (i) and (ii) of Proposition \ref{prop::Phi_k_T_0} and Proposition \ref{prop::Ext^i_T_0}(ii). (iii) follows from Proposition \ref{prop::Ext^i_T_0}(i).

Assume now $n\in\{0,\ldots,r\}$ and $(r,n)\ne(p-1,0)$.
It follows from Proposition \ref{prop::Phi_k_T_0}(iii) and (the proof of) Proposition \ref{prop::Ext^i_T_0} that the map $(\Phi^{-}_{m,1})^{*}$, which is between two one-dimensional spaces for $i=0,1$, is induced by multiplication by $a_{r-n}(z_1)=(-1)^n\binom{r}{n}z_1^rp_0(t)$.
In degree $0$, $(\Phi^{-}_{m,1})^{*}$ sends the basis vector $1$ to the vector
$(-1)^n\binom{r}{n}z_1^rp_0(t)\in z_1^rp_0(t)(A[[z_1]]/z_1^{r+1})$,
which is then multiplication by $(-1)^n\binom{r}{n}$.
In degree $1$, the image of $1$ under $(\Phi^{-}_{m,1})^{*}$ is the vector
$(-1)^n\binom{r}{n}z_1^rp_0(t)\in p_0(t)\frac{A[[z_1]]/z_1^{r+1}}{z_1}$,
which is zero if $r\geq1$ and is $p_0(t)$ if $r=0$. This gives (iv).

The map $(\Phi^{+}_{m,0})^{*}$ is induced by multiplication by $(-1)^r/r!$.
In degree $0$, it sends the base vector $z_1^rp_0(t)\pmod{z_1^{r+1}}$ to
$(-1)^rz_1^rp_0(t)/r!\pmod{t,z_1}\in A/t$,
which is $0$ if $r\geq1$ and is $-1$ if $r=0$ because $p_0(t)\equiv-1\pmod t$.
In degree $1$, $(\Phi^{+}_{m,0})^{*}$ sends the base vector $p_0(t)$ to
$(-1)^rp_0(t)/r!\pmod t\in A/t$, which is $(-1)^{r+1}/r!$. (v) then follows.

By Proposition \ref{prop::Ext^i_T_0}(i) and the decomposition \eqref{eq::dec_sigma_0_dual},
\[
\op{Ext}^i_{\FF[[T_0/Z]]}(\FF,\overline{\sigma}_0^{\vee})
\simeq
\op{Ext}^i_{\FF[[T_0/Z]]}(\FF,\FF(X^nY^{r-n})^{\vee}).
\]
Then (vi) and (vii) follow from Proposition \ref{prop::Phi_k_T_0}(v).

For $(r,n)=(p-1,0)$, both $j=0$ and $j=r$ contribute, and
$a_0(z_1)=a_r(z_1)=z_1^rp_0(t)$ by Proposition \ref{prop::Phi_k_T_0}(iii).
Applying the same calculations to both summands gives the first two rows
of the table, using $r>0$, $p_0(0)=-1$, and $(p-1)!=-1$ in $\FF$.
The last two rows follow from the projections in
Proposition \ref{prop::Phi_k_T_0}(v).
\end{proof}

\begin{lemma} \label{lemma::Phi_Ext^i_T_0_2}
Let $\eta: T_0\to \FF^{\times}$ be the character $\chi^{n}\boxtimes \zeta$ given by \eqref{eq::eta_T_0}.  Under the isomorphisms in Proposition \ref{prop::Ext^i_T_0}, the induced maps have the following descriptions.
In (iv)--(vii), assume $(r,n)\ne(p-1,0)$.
\begin{enumerate}
\item[(i)] For $k\geq 1$, the map 
\[
(\Phi^{+}_{m,k})_{*} : \op{Ext}^i_{\FF[[T_0/Z]]}(  (\op{ind}^{T_0/Z}_{T_{k}/Z} \overline{\sigma}_{k})_{m}^{\vee},\FF) \to  \op{Ext}^i_{\FF[[T_0/Z]]}(  (\op{ind}^{T_0/Z}_{T_{k+1}/Z} \overline{\sigma}_{k+1})_{m-1}^{\vee},\FF)
\]
is the identity map if $i=0$, and is zero if $i=1$. 

\item[(ii)] For $k \geq 2$, the map 
\[
(\Phi^{-}_{m,k})_{*} : \op{Ext}^i_{\FF[[T_0/Z]]}(  (\op{ind}^{T_0/Z}_{T_{k}/Z} \overline{\sigma}_{k})_{m}^{\vee},\FF) \to  \op{Ext}^i_{\FF[[T_0/Z]]}( (\op{ind}^{T_0/Z}_{T_{k-1}/Z} \overline{\sigma}_{k-1})_{m+1}^{\vee},\FF)
\]
is zero if $i=0$, and is the identity map if $i=1$.

\item[(iii)] If $n\notin \{0,\ldots, r\}$, then the maps $(\Phi^{-}_{m,1} )_{*},~(\Phi^{+}_{m,0} )_{*},~(\Phi^{-,1}_{m,0} )_{*},~(\Phi^{-,2}_{m,0} )_{*} $ are all zero.

\item[(iv)] If $n\in \{0,\ldots, r\}$, then the map 
\[
(\Phi^{-}_{m,1} )_{*} :  \op{Ext}^i_{\FF[[T_0/Z]]}((\op{ind}^{T_0/Z}_{T_{1}/Z} \overline{\sigma}_{1})_{m}^{\vee}, \FF) \to  \op{Ext}^i_{\FF[[T_0/Z]]}( ( \overline{\sigma}_{0})_{m+1}^{\vee},\FF ) 
\]
is multiplication by $(-1)^{n+1}\binom{r}{n}$ if $i = 1$; if $i=0$, it is zero if $r\geq 1$ and is multiplication by $-1$ if $r = 0$.  

\item[(v)] If $n\in \{0,\ldots, r\}$, then the map 
\[
(\Phi^{+}_{m,0} )_{*} :  \op{Ext}^i_{\FF[[T_0/Z]]}( ( \overline{\sigma}_{0})_{m}^{\vee},\FF) \to  \op{Ext}^i_{\FF[[T_0/Z]]}( (\op{ind}^{T_0/Z}_{T_{1}/Z} \overline{\sigma}_{1})_{m-1}^{\vee},\FF)\]
is multiplication by $(-1)^{r}/r! $ if $i = 0$; if $i=1$, it is zero if $r\geq 1$ and is the identity map if $r = 0$.

\item[(vi)] If $n\in \{0,\ldots, r\}$, then the map 
\[
(\Phi^{-,1}_{m,0} )_{*}: \op{Ext}^i_{\FF[[T_0/Z]]}( ( \overline{\sigma}_{0})_{m}^{\vee}, \FF) \to \op{Ext}^i_{\FF[[T_0/Z]]}( ( \overline{\sigma}_{0})_{m+1}^{\vee},\FF)
\]
is the identity map if $n = 0$, and is zero if $n\neq 0$.

\item[(vii)] If $n\in \{0,\ldots, r\}$, then the map 
\[
(\Phi^{-,2}_{m,0} )_{*}: \op{Ext}^i_{\FF[[T_0/Z]]}( ( \overline{\sigma}_{0})_{m}^{\vee},\FF) \to \op{Ext}^i_{\FF[[T_0/Z]]}( ( \overline{\sigma}_{0})_{m-1}^{\vee},\FF) 
\]
is  the identity map if $n = r$, and is zero if $n\neq r$.
\end{enumerate}
If $(r,n)=(p-1,0)$, use the ordered bases induced by
$\FF(X^{p-1})^{\vee}$ and $\FF(Y^{p-1})^{\vee}$ for the boundary Ext groups,
and the bases obtained from the two-term resolution of the $k=1$ block.
Then
\[
\begin{array}{c|c|c}
\text{Map} & i=0 & i=1 \\ \hline
(\Phi^{-}_{m,1})_{*}
 & 0 & \begin{bmatrix}-1\\-1\end{bmatrix} \\
(\Phi^{+}_{m,0})_{*}
 & \begin{bmatrix}-1&-1\end{bmatrix} & 0 \\
(\Phi^{-,1}_{m,0})_{*}
 & \begin{bmatrix}0&0\\0&1\end{bmatrix}
 & \begin{bmatrix}0&0\\0&1\end{bmatrix} \\
(\Phi^{-,2}_{m,0})_{*}
 & \begin{bmatrix}1&0\\0&0\end{bmatrix}
 & \begin{bmatrix}1&0\\0&0\end{bmatrix}.
\end{array}
\]
\end{lemma}
\begin{proof}
(i) For $k\geq 1$, we choose a projective resolution of $(\op{ind}^{T_0/Z}_{T_{k}/Z} \overline{\sigma}_{k})_m^{\vee}$:
\[
0\to P_{k,-1} :=A[[z_1]] \xrightarrow{z_1^{p^{k-1}(r+1)} } P_{k,0}:= A[[z_1]] \to (\op{ind}^{T_0/Z}_{T_{k}/Z} \overline{\sigma}_{k})_m^{\vee} \simeq A[[z_1]]/z_1^{p^{k-1}(r+1)} \to 0.
\] 
To compute the map  $(\Phi_{m,k}^{+})_*$, we first lift the map $(\Phi_{m,k}^{+})^{\vee}$ to a map between projective resolutions by taking
\[
P_{k+1,0} \xrightarrow{ u_k(z_1) } P_{k,0},\quad P_{k+1,-1} \xrightarrow{u_k(z_1) z_1^{(p^{k} - p^{k-1}) (r+1)} } P_{k,-1}.
\]
Applying ${\rm Hom}_{A[[z_1]]}(-,\FF)$ to the resolutions, we see that $(\Phi_{m,k}^{+})_*$ is multiplication by $u_k(z_1) $ (resp. $u_k(z_1) z_1^{(p^{k} - p^{k-1}) (r+1)} $) in degree $0$ (resp. $1$). (i) follows from Proposition \ref{prop::Phi_k_T_0}(i) and Proposition \ref{prop::Ext^i_T_0}(ii).

(ii) For $k\geq 2$, the map  $(\Phi_{m,k}^-)^{\vee}$ is multiplication by $z_1^{(p^{k-1} - p^{k-2})(r+1)}$, and it lifts to a map between projective resolutions
\[
P_{k-1,0} \xrightarrow{z_1^{(p^{k-1} - p^{k-2})(r+1)} } P_{k,0},\quad P_{k-1,-1} \xrightarrow{{\rm id} } P_{k,-1}.
\]
Then $(\Phi_{m,k}^-)_*$ is multiplication by $z_1^{(p^{k-1} - p^{k-2})(r+1)}$ (resp. $1$) in degree $0$ (resp. $1$). (ii) follows from Proposition \ref{prop::Ext^i_T_0}(ii).

(iii) follows from Proposition \ref{prop::Ext^i_T_0}(i).

For (iv)--(vii), assume $n\in\{0,\ldots,r\}$ and $(r,n)\ne(p-1,0)$.

(iv) By Proposition \ref{prop::Ext^i_T_0}(i), it remains to determine the map
\[
(\Phi^{-}_{m,1} )_{*} :  \op{Ext}^i_{\FF[[T_0/Z]]}((\op{ind}^{T_0/Z}_{T_{1}/Z} \overline{\sigma}_{1})_m^{\vee}, \FF) \to  \op{Ext}^i_{\FF[[T_0/Z]]}( \FF(X^{n} Y^{r-n} )^{\vee},\FF ).
\]
We take the projective resolution $P(0)_{\bullet}$ (given in \eqref{eq::proj_res_one_dim_T_0}) of $ \FF(X^{n} Y^{r-n} )^{\vee}$. The composition 
\[
\FF(X^{n} Y^{r-n} )^{\vee} \to \overline{\sigma}_0^{\vee} \xrightarrow{
(\Phi^{-}_{m,1} )^{\vee}}   (\op{ind}^{T_0/Z}_{T_{1}/Z} \overline{\sigma}_{1})_m^{\vee}
\]
lifts to a map between projective resolutions by taking
\[
 P(0)_0  \xrightarrow{
a_{r-n}(z_1)} P_{1,0} ,\quad P(0)_{-1}  \xrightarrow{ \begin{bmatrix}
    0 & a_{r-n}(z_1)/z_1^r
\end{bmatrix} } P_{1,-1}.
\]
Then in degree $0$, $(\Phi_{m,1}^-)_*$ is multiplication by $a_{r-n}(z_1) = (-1)^n \binom{r}{n}p_0(t) z_1^r$, which is zero if $r\geq 1$ and is multiplication by $-1$ if $r=0$ because $p_0(t)\equiv -1\pmod{t}$. In degree $1$, $(\Phi_{m,1}^-)_*$ is multiplication by $a_{r-n}(z_1)/z_1^r = (-1)^n \binom{r}{n}p_0(t) $, which is multiplication by $(-1)^{n+1} \binom{r}{n}$.

(v)  By Proposition \ref{prop::Ext^i_T_0}(i), it remains to determine the map
\[
(\Phi^{+}_{m,0} )_{*} :  \op{Ext}^i_{\FF[[T_0/Z]]}( \FF(X^{n} Y^{r-n})^{\vee},\FF) \to  \op{Ext}^i_{\FF[[T_0/Z]]}( (\op{ind}^{T_0/Z}_{T_{1}/Z} \overline{\sigma}_{1})_{m-1}^{\vee},\FF) .
\] 
The composition
\[
 (\op{ind}^{T_0/Z}_{T_{1}/Z} \overline{\sigma}_{1})_{m-1}^{\vee} \xrightarrow{(\Phi^+_{m,0})^{\vee}} \overline{\sigma}_0^{\vee} \to  \FF(X^{n} Y^{r-n} )^{\vee} 
\]
lifts to a map between projective resolutions by taking
\[
 P_{1,0}  \xrightarrow{\frac{(-1)^r}{r!} }  P(0)_0 ,\quad  P_{1,-1}   \xrightarrow{ \begin{bmatrix}
    0 \\
 (-1)^r z_1^r/r!
\end{bmatrix}} P(0)_{-1}.
\]
Applying ${\rm Hom}_{A[[z_1]]}(-,\FF)$ to the resolutions, we see that in degree $0$, $(\Phi^{+}_{m,0} )_{*}$ is multiplication by $(-1)^r /r!$; and in degree $1$, $(\Phi^{+}_{m,0} )_{*}$ is multiplication by $(-1)^r  z_1^r/r!$, which is zero if $r \geq 1$ and is the identity if $r = 0$.

(vi) The map $(\Phi^{-,1}_{m,0})^{\vee}$ lifts to a map between projective resolutions by taking
\[
\bigoplus_{j=0}^r P(n+j -r)_0   \xrightarrow{\alpha } \bigoplus_{j=0}^r P(n+j -r)_0, 
\]
where $\alpha$ is given by the matrix
\[
\begin{bmatrix}
0 & \cdots & \cdots & 0 \\
\vdots & \ddots & \vdots & \vdots \\
\vdots & \cdots & 0 & 0 \\
0 & \cdots & 0 & 1
\end{bmatrix} ;
\]
and
\[
\bigoplus_{j=0}^r P(n+j -r)_{-1}    \xrightarrow{ \beta = (\beta_{i,j})_{0\leq i,j\leq r}} \bigoplus_{j=0}^r P(n+j -r)_{-1},  
\]
where $\beta_{i,j}$ is a $2\times 2$ matrix, which is zero if $(i,j) \neq (r,r)$ and is the identity matrix if  $(i,j) = (r,r)$.
This gives (vi) by applying ${\rm Hom}_{A[[z_1]]}(-,\FF)$ to the resolutions.

The proof of (vii) is similar to that of (vi).
For $(r,n)=(p-1,0)$, both $j=0$ and $j=r$ contribute, and
$a_0(z_1)=a_r(z_1)=z_1^rp_0(t)$ by Proposition \ref{prop::Phi_k_T_0}(iii).
Applying the calculation in (iv) to both summands gives zero in degree $0$
and the column $\begin{bmatrix}-1&-1\end{bmatrix}^{T}$ in degree $1$,
since $r>0$ and $p_0(0)=-1$.
Similarly, the calculation in (v) gives the row
$\begin{bmatrix}-1&-1\end{bmatrix}$ in degree $0$ and zero in degree $1$,
since $(-1)^r/r!=-1$ in $\FF$.
The projection maps in (vi) and (vii) give the last two rows of the table.
\end{proof}

\subsubsection{Restriction to $T$} \label{sec::restriction_to_T}

\begin{lemma} \label{lemma::representatives_T}
    A set of representatives for $KZ\backslash G/T$ is given by 
    $$\left\{\left(\begin{array}{cl}
            1 & 0 \\
            0 & 1
            \end{array}\right),~\left(\begin{array}{cl}
        1 & p^{-k} \\
        0 & 1
        \end{array}\right),~k\in\ZZ_{\geq 1} \right\}. $$
\end{lemma}
\begin{proof}
This follows easily from Lemma \ref{lemma::representatives_T_0}.
\end{proof}

Let $s_{0}:=\bigl(\begin{smallmatrix}
    1 & 0 \\
    0 & 1
    \end{smallmatrix} \bigr)$ and $s_{k}:= \bigl(\begin{smallmatrix}
    1 & p^{-k} \\
    0 & 1
    \end{smallmatrix} \bigr)$ for $k\geq 1$. It is straightforward to check that 
\[
T \bigcap s_{k}^{-1}KZ s_{k} = T_ k,~ k\in \ZZ_{\geq 0},
\]
which is given in Lemma \ref{lemma::T_{m,k}}. Let $\sigma_{k}$ denote the $T_{k}$-representation $(s_{k}^{-1})_{*}\op{Res}^{KZ}_{s_{k}T s_{k}^{-1}\cap KZ}\sigma$. Mackey's formula \cite[Thm. 1.1]{Yamamoto} gives $T$-equivariant isomorphisms
\[
(\op{ind}_{KZ}^{G}\sigma)|_{T}\cong \bigoplus_{s_k\in KZ\backslash G/T}\op{ind}_{KZ}^{KZ s_k T}\sigma\cong \bigoplus_{k\geq 0}\op{ind}_{T_{k}}^{T}\sigma_{k}.
\] 

The restriction of the Hecke operator $\Phi$ to $\op{ind}_{T_{k}}^{T}\sigma_{k}$ is given by the following lemma.

\begin{lemma}\label{lemma::Phi_T}
    For $k\in \ZZ_{\geq 0}$, let $\Phi_{k}: = \Phi|_{\op{ind}_{T_{k}}^{T}\sigma_{k}}$. Then $\Phi_{k}= (\Phi^{+}_{k}, \Phi^{-}_{k})$, where the morphisms $\Phi_k^{\pm}$ are given by 
\begin{align*}
\Phi^{+}_{k}:  \op{ind}_{T_{k}}^{T}\sigma_{k} & \to \op{ind}_{T_{k+1}}^{T}\sigma_{k+1}, ~k \geq 1\\ 
 [1,X^{r-j}Y^j] & \mapsto \sum_{i=0}^{p-1} \left(\begin{array}{cl}
            p(1+p^{k}i) &  \\
             & 1
            \end{array}\right)\left[1, i^j X^r \right] ;\\
\Phi^{-}_{k}:  \op{ind}_{T_{k}}^{T}\sigma_{k} & \to \op{ind}_{T_{k-1}}^{T}\sigma_{k-1},~ k\geq 2 \\
[1,X^{r-j}Y^j] & \mapsto \left(\begin{array}{cl}
            p^{-1} &  \\
             & 1
            \end{array}\right) [1, \delta_{j,r}Y^r] ; \\
\Phi^{-}_{1}:  \op{ind}_{T_{1}}^{T}\sigma_{1} & \to \op{ind}_{T_{0}}^{T}\sigma_{0} \\
[1,X^{r-j}Y^j] & \mapsto  \left(\begin{array}{cl}
            p^{-1} &  \\
             & 1
            \end{array}\right) [1, \delta_{j,r}(-X+Y)^r] ; \\
\Phi^{+}_{0}:  \op{ind}_{T_{0}}^{T}\sigma_{0} & \to \op{ind}_{T_{1}}^{T}\sigma_{1}\\ 
 [1,X^{r-j}Y^j] & \mapsto \sum_{i=1}^{p-1} \left(\begin{array}{cl}
            pi &  \\
             & 1
            \end{array}\right)\left[1, i^{j-r} X^r \right] ;\\
\Phi^{-}_{0} :  \op{ind}_{T_{0}}^{T}\sigma_{0} & \to \op{ind}_{T_{0}}^{T}\sigma_{0} \\ 
[1,X^{r-j}Y^j] & \mapsto   \left(\begin{array}{cl}
            p^{-1} &  \\
             & 1
            \end{array}\right)  \left[1, \delta_{j,r}Y^r \right] +   \left(\begin{array}{cl}
            p &  \\
             & 1
            \end{array}\right) [1, \delta_{j,0}X^r] .
\end{align*}
\end{lemma}

\begin{proof}
The proof is similar to that of Lemma \ref{lemma::Phi_T_0}.
\end{proof}

\begin{proposition} \label{prop::Ext_T_0_to_T_1}
Let $\eta: T\to \FF^{\times}$ be a smooth character such that $\eta|_{Z} = \zeta$. 
\begin{enumerate}
\item[(i)] Frobenius reciprocity induces a natural isomorphism
\[
\op{Ext}^{i}_{T,\zeta}(\op{ind}_{T_{k}}^{T}\sigma_{k},\eta) \cong  \op{Ext}^{i}_{T_0,\zeta} (\op{ind}_{T_{k}}^{T_0}\sigma_{k},\op{Res}^{T}_{T_0}\eta).
\]
\item[(ii)] Recall the maps $\Phi_{0,k}^{\pm},~ \Phi_{0,0}^{-,1}$ and $\Phi_{0,0}^{-,2}$ given in Lemma \ref{lemma::Phi_T_0}. The following diagrams are commutative:
\[
\text{for $k\geq 0$,} \xymatrix@C=10em{
        \op{Ext}^{i}_{T,\zeta}(\op{ind}_{T_{k+1}}^{T}\sigma_{k+1},\eta)\ar[r]^{(\Phi_k^{+})^*}\ar[d]^{\cong}&\op{Ext}^{i}_{T,\zeta}(\op{ind}_{T_{k}}^{T}\sigma_{k},\eta)\ar[d]^{\cong}\\
       \op{Ext}^{i}_{T_0,\zeta} (\op{ind}_{T_{k+1}}^{T_0}\sigma_{k+1},\op{Res}^{T}_{T_0}\eta) \ar[r]^{\eta \left(\left(\begin{smallmatrix}
            p &  \\
             & 1
        \end{smallmatrix}\right)\right)(\Phi_{0,k}^{+})^*}&\op{Ext}^{i}_{T_0,\zeta} (\op{ind}_{T_{k}}^{T_0}\sigma_{k},\op{Res}^{T}_{T_0}\eta)
    } ;
\]
\[
\text{for $k\geq 1$,} \xymatrix@C=10em{
        \op{Ext}^{i}_{T,\zeta}(\op{ind}_{T_{k-1}}^{T}\sigma_{k-1},\eta)\ar[r]^{(\Phi_k^{-})^*}\ar[d]^{\cong}&\op{Ext}^{i}_{T,\zeta}(\op{ind}_{T_{k}}^{T}\sigma_{k},\eta)\ar[d]^{\cong}\\
       \op{Ext}^{i}_{T_0,\zeta} (\op{ind}_{T_{k-1}}^{T_0}\sigma_{k-1},\op{Res}^{T}_{T_0}\eta) \ar[r]^{\eta \left(\left(\begin{smallmatrix}
            p^{-1} &  \\
             & 1
        \end{smallmatrix}\right)\right)(\Phi_{0,k}^{-})^*}&\op{Ext}^{i}_{T_0,\zeta} (\op{ind}_{T_{k}}^{T_0}\sigma_{k},\op{Res}^{T}_{T_0}\eta)
    };
\]
\[
\text{and \quad}  \xymatrix@C=10em{
        \op{Ext}^{i}_{T,\zeta}(\op{ind}_{T_{0}}^{T}\sigma_{0},\eta)\ar[r]^{(\Phi_0^{-})^*}\ar[d]^{\cong}&\op{Ext}^{i}_{T,\zeta}(\op{ind}_{T_{0}}^{T}\sigma_{0},\eta)\ar[d]^{\cong}\\
       \op{Ext}^{i}_{T_0,\zeta} (\sigma_{0},\op{Res}^{T}_{T_0}\eta) \ar[r]^{\eta \left(\left(\begin{smallmatrix}
            p^{-1} &  \\
             & 1
        \end{smallmatrix}\right)\right)(\Phi_{0,0}^{-,1})^* + \eta \left(\left(\begin{smallmatrix}
            p &  \\
             & 1
        \end{smallmatrix}\right)\right)(\Phi_{0,0}^{-,2})^*}&\op{Ext}^{i}_{T_0,\zeta} (\sigma_{0},\op{Res}^{T}_{T_0}\eta)
    }.
\]
\end{enumerate}
\end{proposition}
\begin{proof}
(i) Since $T_0$ is an open subgroup of $T$, the restriction functor $ \op{Res}^{T}_{T_0}$ is the right adjoint of $ \op{ind}_{T_{0}}^{T}$. The adjunction map is given by Frobenius reciprocity. Since $ \op{ind}_{T_{0}}^{T}$ is exact, $ \op{Res}^{T}_{T_0}$ preserves injective objects. Since the category of smooth $\FF$-representations of $T$ with fixed $Z$-action $\zeta$ has enough injective objects, we choose an injective resolution $I^{\bullet}$ of $\eta$. Then 
\begin{align*}
\op{Ext}^{i}_{T,\zeta}(\op{ind}_{T_{k}}^{T}\sigma_{k},\eta) & = H^i (\op{Hom}_{T,\zeta} (\op{ind}_{T_{k}}^{T}\sigma_{k},I^{\bullet} )) \\
&=  H^i (\op{Hom}_{T,\zeta} (\op{ind}_{T_{0}}^{T}\op{ind}_{T_{k}}^{T_0} \sigma_{k},I^{\bullet} )) \\
&\cong H^i (\op{Hom}_{T_0,\zeta} (\op{ind}_{T_{k}}^{T_0}\sigma_{k}, \op{Res}^{T}_{T_0} I^{\bullet} )) \\
& = \op{Ext}^{i}_{T_0,\zeta} (\op{ind}_{T_{k}}^{T_0}\sigma_{k},\op{Res}^{T}_{T_0}\eta),
\end{align*}
where the third isomorphism is Frobenius reciprocity, and the final equality holds because $\op{Res}^{T}_{T_0} I^{\bullet}$ is an injective resolution of $\op{Res}^{T}_{T_0}\eta$ in the category of smooth $\FF$-representations of $T_0$ with fixed $Z$-action $\zeta$.

(ii) We prove the assertion for $(\Phi_k^{+})^* $ with $k\geq 1$; the remaining cases are similar. By naturality of Frobenius reciprocity, it suffices to check the commutativity for $i=0$. Note that since $T$ is abelian, the map $ \Phi_k^+$ is a composition of two maps $\alpha$ and $\beta$ given by
\begin{align*}
 \op{ind}_{T_{k}}^{T}\sigma_{k}   \xrightarrow{ \alpha} & \op{ind}_{T_{k}}^{T}\sigma_{k} \xrightarrow{ \beta} \op{ind}_{T_{k+1}}^{T}\sigma_{k+1} \\ 
 [1,X^{r-j}Y^j]  \mapsto  &  \left(\begin{array}{cl}
            p &  \\
             & 1
            \end{array}\right) [1,X^{r-j}Y^j ] \\
& [1,X^{r-j}Y^j] \mapsto \sum_{i=0}^{p-1} \left(\begin{array}{cl}
            1+p^{k}i &  \\
             & 1
            \end{array}\right)\left[1, i^j X^r \right] .
\end{align*}
Under the isomorphism in (i), which is given by Frobenius reciprocity, it is clear that the map $ \beta^*$ is identified with the map $(\Phi_{0,k}^{+})^* $; and the map $\alpha^* : \op{Hom}_{T,\zeta}( \op{ind}_{T_{k}}^{T}\sigma_{k} ,\eta) \to \op{Hom}_{T,\zeta}( \op{ind}_{T_{k}}^{T}\sigma_{k} ,\eta) $ is multiplication by $\eta \left(\left(\begin{smallmatrix}
            p &  \\
             & 1
        \end{smallmatrix}\right)\right)$.
\end{proof}

\begin{proposition} \label{prop::Ext_T_0_to_T_2}
Let $\eta: T\to \FF^{\times}$ be a smooth character such that $\eta|_{Z} = \zeta$. 
\begin{enumerate}
\item[(i)] There is a natural isomorphism
\[
\op{Ext}^{i}_{T,\zeta}( \eta, \op{ind}_{T_{k}}^{T}\sigma_{k}) \cong  \op{Ext}^{i-1}_{T_0,\zeta} ( \op{Res}^{T}_{T_0}\eta, \op{ind}_{T_{k}}^{T_0}\sigma_{k}),~ i\geq 1.
\]
\item[(ii)]  Recall the maps $\Phi_{0,k}^{\pm},~ \Phi_{0,0}^{-,1}$ and $\Phi_{0,0}^{-,2}$ given in Lemma \ref{lemma::Phi_T_0}. The following diagrams are commutative:
\[
\text{for $k\geq 0$,} \xymatrix@C=10em{
        \op{Ext}^{i}_{T,\zeta}( \eta, \op{ind}_{T_{k}}^{T}\sigma_{k})\ar[r]^{(\Phi_k^{+})_{*}}\ar[d]^{\cong}&\op{Ext}^{i}_{T,\zeta}(\eta , \op{ind}_{T_{k+1}}^{T}\sigma_{k+1})\ar[d]^{\cong}\\
      \op{Ext}^{i-1}_{T_0,\zeta} ( \op{Res}^{T}_{T_0}\eta, \op{ind}_{T_{k}}^{T_0}\sigma_{k}) \ar[r]^{\eta \left(\left(\begin{smallmatrix}
            p &  \\
             & 1
        \end{smallmatrix}\right)\right)(\Phi_{0,k}^{+})_{*}}& \op{Ext}^{i-1}_{T_0,\zeta} ( \op{Res}^{T}_{T_0}\eta, \op{ind}_{T_{k+1}}^{T_0}\sigma_{k+1})
    };
\]
\[
\text{for $k\geq 1$,} \xymatrix@C=10em{
      \op{Ext}^{i}_{T,\zeta}( \eta, \op{ind}_{T_{k}}^{T}\sigma_{k}) \ar[r]^{(\Phi_k^{-})_*}\ar[d]^{\cong}&\op{Ext}^{i}_{T,\zeta}(\eta, \op{ind}_{T_{k-1}}^{T}\sigma_{k-1})\ar[d]^{\cong}\\
       \op{Ext}^{i-1}_{T_0,\zeta} ( \op{Res}^{T}_{T_0}\eta, \op{ind}_{T_{k}}^{T_0}\sigma_{k}) \ar[r]^{\eta \left(\left(\begin{smallmatrix}
            p^{-1} &  \\
             & 1
        \end{smallmatrix}\right)\right)(\Phi_{0,k}^{-})_{*}}& \op{Ext}^{i-1}_{T_0,\zeta} ( \op{Res}^{T}_{T_0}\eta, \op{ind}_{T_{k-1}}^{T_0}\sigma_{k-1})
    };
\]
\[
\text{and \quad}  \xymatrix@C=10em{
        \op{Ext}^{i}_{T,\zeta}(\eta, \op{ind}_{T_{0}}^{T}\sigma_{0})\ar[r]^{(\Phi_0^{-})_{*}}\ar[d]^{\cong}&\op{Ext}^{i}_{T,\zeta}( \eta, \op{ind}_{T_{0}}^{T}\sigma_{0})\ar[d]^{\cong}\\
       \op{Ext}^{i-1}_{T_0,\zeta} ( \op{Res}^{T}_{T_0}\eta,\sigma_{0}) \ar[r]^{\eta \left(\left(\begin{smallmatrix}
            p^{-1} &  \\
             & 1
        \end{smallmatrix}\right)\right)(\Phi_{0,0}^{-,1})_{*} + \eta \left(\left(\begin{smallmatrix}
            p &  \\
             & 1
        \end{smallmatrix}\right)\right)(\Phi_{0,0}^{-,2})_{*}}&\op{Ext}^{i-1}_{T_0,\zeta} (\op{Res}^{T}_{T_0}\eta ,\sigma_{0})
    }.
\]
\end{enumerate}
\end{proposition}
\begin{proof}
(i) Since $T = T_0 \times \left(\begin{smallmatrix}
            p &   \\
             & 1
        \end{smallmatrix}\right)^{\ZZ} $, the Hochschild--Serre spectral sequence gives
\[
E_2^{pq} = H^{p}(\left(\begin{smallmatrix}
            p &   \\
             & 1
        \end{smallmatrix}\right)^{\ZZ}, \op{Ext}^{q}_{T_0,\zeta} ( \op{Res}^{T}_{T_0}\eta,  \op{Res}^{T}_{T_0} \op{ind}_{T_{k}}^{T}\sigma_{k}) )  \Rightarrow \op{Ext}^{p+q}_{T,\zeta}( \eta, \op{ind}_{T_{k}}^{T}\sigma_{k}) .
\]
Since 
\[
\op{Res}^{T}_{T_0} \op{ind}_{T_{k}}^{T}\sigma_{k}\simeq \bigoplus_{\left(\begin{smallmatrix}
            p &   \\
             & 1
        \end{smallmatrix}\right)^{\ZZ}}  \op{ind}_{T_{k}}^{T_0}\sigma_{k},
\]
we have
\[
\op{Ext}^{q}_{T_0,\zeta} ( \op{Res}^{T}_{T_0}\eta,  \op{Res}^{T}_{T_0} \op{ind}_{T_{k}}^{T}\sigma_{k}) \simeq \bigoplus_{\left(\begin{smallmatrix}
            p &   \\
             & 1
        \end{smallmatrix}\right)^{\ZZ}} \op{Ext}^{q}_{T_0,\zeta} ( \op{Res}^{T}_{T_0}\eta,  \op{ind}_{T_{k}}^{T_0}\sigma_{k}).
\]
Note that the cyclic group $\left(\begin{smallmatrix}
            p &   \\
             & 1
        \end{smallmatrix}\right)^{\ZZ}$ acts on the index set by translation. Therefore,
\[
E_2^{pq} \simeq \left\{\begin{array}{cl}
        \op{Ext}^{q}_{T_0,\zeta} ( \op{Res}^{T}_{T_0}\eta,  \op{ind}_{T_{k}}^{T_0}\sigma_{k}),& \text{ if } p=1,\\
        0,& \text{ otherwise}. 
        \end{array}\right.
\]

(ii) The proof is similar to that of Proposition \ref{prop::Ext_T_0_to_T_1}(ii).
\end{proof}

\begin{theorem} \label{thm::branching_ss_T}
Let $\pi$ be a supersingular representation of ${\rm GL}_2(\QQ_p)$ with central character $\zeta$. Let $\eta: T \to \FF^{\times}$ be a smooth character such that $\eta|_Z = \zeta$. Then there are quasi-isomorphisms
\[
 \op{RHom}_{T ,\zeta}(\pi|_{T},\eta) \simeq \FF^{\oplus 2}[-1],\qquad  \op{RHom}_{T ,\zeta}(\eta, \pi|_{T}) \simeq \FF^{\oplus 2}[-1].
\]
\end{theorem}
\begin{proof}
By Lemma \ref{lemma::concentrate_one_deg}, to prove the first quasi-isomorphism, it suffices to prove
\[
\dim_{\FF} \op{Ext}^{i}_{T ,\zeta}(\pi|_{T},\eta)=\left\{\begin{array}{cl}
        2,& \text{ if } i=1,\\
        0,& \text{ otherwise}.
        \end{array}\right.
\]

By twisting, we can assume $\pi = \op{ind}_{KZ}^{G}(\sigma) /\Phi$ with $\sigma := \op{Sym}^{r}(\FF^2)$ and $0\leq r\leq p-1$. The short exact sequence 
\begin{align*}
        0\to\op{ind}_{KZ}^{G}(\sigma)\xrightarrow{\Phi }\op{ind}_{KZ}^{G}(\sigma)\to \pi\to 0 
    \end{align*}
induces a long exact sequence
\[
\cdots \to \op{Ext}^{i}_{T,\zeta}(\pi|_{T},\eta) \to \op{Ext}^{i}_{T,\zeta}(\op{ind}_{KZ}^{G}(\sigma),\eta) \xrightarrow{\Phi_i^{*} } \op{Ext}^{i}_{T, \zeta}(\op{ind}_{KZ}^{G}(\sigma),\eta) \to \cdots.
\]
We are reduced to understanding the map
\[
\Phi_i^{*} :  \prod_{k\geq 0} \mathfrak{C}^i_k \to \prod_{k\geq 0}  \mathfrak{C}^i_k,
\] 
where $\mathfrak{C}^i_k : =  \op{Ext}^{i}_{T,\zeta}( \op{ind}^{T}_{T_k} \sigma_k, \eta)$. By Proposition \ref{prop::Ext_T_0_to_T_1}(i),
\[
\mathfrak{C}^i_k\cong  \op{Ext}^{i}_{T_0,\zeta} (\op{ind}_{T_{k}}^{T_0}\sigma_{k},\op{Res}^{T}_{T_0}\eta) \simeq \op{Ext}^i_{T_0 /Z} (\op{ind}^{T_0/Z}_{T_k/Z} (\sigma_{k} \otimes\eta^{-1}|_{T_{k}}), \FF ).
\]
Let $\overline{\sigma}_k$ denote the $T_k/Z$-representation $\sigma_k \otimes\eta^{-1}|_{T_k}$. Then 
\[
\mathfrak{C}^i_k \simeq \op{Ext}^{i}_{\FF[[T_0/Z ]] }(\FF, (\op{ind}^{T_0 /Z}_{T_k/Z} \overline{\sigma}_k)^{\vee}),
\]
which is given by Proposition \ref{prop::Ext^i_T_0}. We just need to compute the map $\Phi_i^*$ for $i = 0,1$, which is induced by the maps $(\Phi_k^{\pm})^{*}$ in Proposition \ref{prop::Ext_T_0_to_T_1}. We use Proposition \ref{prop::Ext_T_0_to_T_1}(ii) and Lemma \ref{lemma::Phi_Ext^i_T_0_1} to compute these maps. Set
\[
\lambda : = \eta\left( \left(\begin{array}{cl}
            p &  \\
             & 1
            \end{array}\right) \right).
\]

We first assume $n\notin \{0,\ldots, r \}$. Then $\mathfrak{C}^i_0 = 0$ by Proposition \ref{prop::Ext^i_T_0}(i) and it follows that  
\begin{align*}
(\Phi_k^+)^{*} &: \mathfrak{C}^0_{k+1} \to \mathfrak{C}^0_k \text{ is the zero map for $k\geq 1$}, \\
(\Phi_k^+)^{*} & : \mathfrak{C}^1_{k+1} \to \mathfrak{C}^1_k \text{ is multiplication by $\lambda$ for $k\geq 1$}, \\
(\Phi_k^-)^{*}& : \mathfrak{C}^0_{k-1} \to \mathfrak{C}^0_k\text{ is multiplication by $\lambda^{-1} $ for $k\geq 2$}, \\
(\Phi_k^-)^{*} & : \mathfrak{C}^1_{k-1} \to \mathfrak{C}^1_k\text{ is  the zero map for $k\geq 2$}.
\end{align*}
We then have 
\begin{align*}
\Phi_0^{*}& : \prod_{k\geq 1} \mathfrak{C}^0_k \to \prod_{k\geq 1} \mathfrak{C}^0_k ,~ (a_1,a_2,\ldots) \mapsto \lambda^{-1} (0,a_1,a_2,\ldots),\\
\Phi_1^{*}& : \prod_{k\geq 1} \mathfrak{C}^1_k \to \prod_{k\geq 1} \mathfrak{C}^1_k ,~ (a_1,a_2,a_3,\ldots) \mapsto \lambda (a_2,a_3,\ldots).
\end{align*}
Then $\ker \Phi_0^{*} = {\rm coker} (\Phi_1^{*}) = 0$, and both ${\rm coker} (\Phi_0^{*})$ and $\ker \Phi_1^{*}$ are one-dimensional.

If $1\leq n\leq r-1$, then the difference from the preceding case is that $\mathfrak{C}^i_0 $ is one-dimensional for $i=0,1$, and 
\begin{align*}
(\Phi_0^+)^{*} &: \mathfrak{C}^0_{1} \to \mathfrak{C}^0_0 \text{ is the zero map}, \\
(\Phi_0^+)^{*} & : \mathfrak{C}^1_{1} \to \mathfrak{C}^1_0 \text{ is multiplication by $\lambda (-1)^{r+1} /r!$}, \\
(\Phi_1^-)^{*}& : \mathfrak{C}^0_{0} \to \mathfrak{C}^0_1\text{ is multiplication by $\lambda^{-1} (-1)^n\binom{r}{n} $}, \\
(\Phi_1^-)^{*} & : \mathfrak{C}^1_{0} \to \mathfrak{C}^1_1\text{ is the zero map}.
\end{align*}
We then have 
\begin{align*}
\Phi_0^{*}& : \prod_{k\geq 0} \mathfrak{C}^0_k \to \prod_{k\geq 0} \mathfrak{C}^0_k ,~ (a_0,a_1,a_2,\ldots) \mapsto  \lambda^{-1} (0,(-1)^n\binom{r}{n} a_0,a_1,a_2,\ldots),\\
\Phi_1^{*}& : \prod_{k\geq 0} \mathfrak{C}^1_k \to \prod_{k\geq 0} \mathfrak{C}^1_k ,~ (a_0,a_1,a_2,\ldots) \mapsto  \lambda ((-1)^{r+1}a_1 /r! ,a_2,a_3,\ldots).
\end{align*}
Hence $\ker \Phi_0^{*} = {\rm coker} (\Phi_1^{*}) = 0$, and both ${\rm coker} (\Phi_0^{*})$ and $\ker \Phi_1^{*}$ are one-dimensional.

If $n=r \geq 1$, we further have 
\[
(\Phi_0^-)^{*} : \mathfrak{C}^i_{0} \to \mathfrak{C}^i_0 \text{ is multiplication by $\lambda$.}
\]
We then have 
\begin{align*}
\Phi_0^{*}& : \prod_{k\geq 0} \mathfrak{C}^0_k \to \prod_{k\geq 0} \mathfrak{C}^0_k ,~ (a_0,a_1,a_2,\ldots) \mapsto  \lambda^{-1} (0,(-1)^n\binom{r}{n} a_0,a_1,a_2,\ldots) + \lambda (a_0,0,\ldots ) ,\\
\Phi_1^{*}& : \prod_{k\geq 0} \mathfrak{C}^1_k \to \prod_{k\geq 0} \mathfrak{C}^1_k ,~ (a_0,a_1,a_2,\ldots) \mapsto  \lambda ((-1)^{r+1}a_1/r! ,a_2,a_3,\ldots) +\lambda(a_0,0,\ldots ).
\end{align*}
Then $\ker \Phi_0^{*} = {\rm coker} (\Phi_1^{*}) = 0$, and both ${\rm coker} (\Phi_0^{*})$ and $\ker \Phi_1^{*}$ are one-dimensional.

If $n=0$ and $1\leq r\leq p-2$, the boundary self-map is multiplication by $\lambda^{-1}$, and we obtain
\begin{align*}
\Phi_0^*(a_0,a_1,a_2,\ldots)
 &=\lambda^{-1}(a_0,a_0,a_1,a_2,\ldots),\\
\Phi_1^*(a_0,a_1,a_2,\ldots)
 &=(\lambda^{-1}a_0+\lambda(-1)^{r+1}a_1/r!,
       \lambda a_2,\lambda a_3,\ldots).
\end{align*}
The first map is injective with image defined by equality of the first two coordinates. The second map is surjective with one-dimensional kernel, since $(-1)^{r+1}/r!\neq0$. Thus the same kernel and cokernel dimensions hold.

If $n=0$ and $r=p-1$, then $\mathfrak{C}_0^i$ is two-dimensional for $i=0,1$. Write its coordinates as $(x,y)$, corresponding respectively to $(X^{p-1})^\vee$ and $(Y^{p-1})^\vee$. By the boundary formulas in Lemma \ref{lemma::Phi_Ext^i_T_0_1}, using $(p-1)!=-1$ in $\FF$, we have
\begin{align*}
\Phi_0^*(x,y,a_1,a_2,\ldots)
 &=(\lambda x,\lambda^{-1}y,
       \lambda^{-1}(x+y),\lambda^{-1}a_1,\ldots),\\
\Phi_1^*(x,y,a_1,a_2,\ldots)
 &=(\lambda x+\lambda a_1,\lambda^{-1}y+\lambda a_1,
       \lambda a_2,\lambda a_3,\ldots).
\end{align*}
The first map is injective, and its image consists of the sequences $(u,v,b_1,b_2,\ldots)$ satisfying $b_1=\lambda^{-2}u+v$, so its cokernel is one-dimensional. The second map is surjective with kernel generated by $(-1,-\lambda^2,1,0,\ldots)$. Thus the same kernel and cokernel dimensions hold.

Finally, if $n = r =0$, then $(\Phi_k^+)^{*}$ for $k\geq 1$ and $(\Phi_k^-)^{*}$ for $k\geq 2$ are as before. By Lemma \ref{lemma::Phi_Ext^i_T_0_1}, the boundary maps are as follows
\begin{align*}
(\Phi_0^+)^{*} &: \mathfrak{C}^i_{1} \to \mathfrak{C}^i_0 \text{ is multiplication by $-\lambda $}, \\
(\Phi_1^-)^{*} & : \mathfrak{C}^i_{0} \to \mathfrak{C}^i_1\text{ is multiplication by $\lambda^{-1} $},\\
(\Phi_0^-)^{*} & : \mathfrak{C}^i_{0} \to \mathfrak{C}^i_0 \text{ is multiplication by $\lambda + \lambda^{-1}$.}
\end{align*}
We then have 
\begin{align*}
\Phi_0^{*} :  \prod_{k\geq 0} \mathfrak{C}^0_k & \to \prod_{k\geq 0} \mathfrak{C}^0_k ,\\
 (a_0,a_1,a_2,\ldots)& \mapsto \lambda^{-1} \left( (1+\lambda^2) a_0 -\lambda^2 a_1, a_0,a_1,a_2,\ldots \right) ,
\end{align*}
and
\begin{align*}
\Phi_1^{*} : \prod_{k\geq 0} \mathfrak{C}^1_k & \to \prod_{k\geq 0} \mathfrak{C}^1_k , \\
 (a_0,a_1,a_2,\ldots) & \mapsto \lambda \left( (1+ \lambda^{-2})a_0 - a_1 ,a_2 + \lambda^{-2} a_0,a_3,\ldots \right) .
\end{align*}
Then, as above, $\ker \Phi_0^{*} = {\rm coker} (\Phi_1^{*}) = 0$, and both ${\rm coker} (\Phi_0^{*})$ and $\ker \Phi_1^{*}$ are one-dimensional. The dimension of $\op{Ext}^{i}_{T ,\zeta}(\pi|_{T},\eta)$ then follows.

Now we prove the second quasi-isomorphism. By Lemma \ref{lemma::concentrate_one_deg} again, it suffices to show
\[
\dim_{\FF} \op{Ext}^{i}_{T ,\zeta}(\eta, \pi|_{T})=\left\{\begin{array}{cl}
        2,& \text{ if } i = 1,\\
        0,& \text{ otherwise}. 
        \end{array}\right.
\]

Since $\eta$ is one-dimensional, we have
\[
\op{Ext}^{i}_{T,\zeta}(\eta,\op{ind}_{KZ}^{G}(\sigma)) = \bigoplus_{k\geq 0} \op{Ext}^{i}_{T,\zeta}(\eta,\op{ind}_{T_k}^{T}(\sigma_k)) .
\]
For the second calculation, we put $ \mathfrak{C}^i_k : =  \op{Ext}^{i}_{T,\zeta}(\eta,\op{ind}_{T_k}^{T}(\sigma_k))$, which is isomorphic to $\op{Ext}^{i-1}_{T_0,\zeta} ( \op{Res}^{T}_{T_0}\eta, \op{ind}_{T_{k}}^{T_0}\sigma_{k})$ by Proposition \ref{prop::Ext_T_0_to_T_2}(i). As above, we are then reduced to understanding the following map $(\Phi_i)_*$ induced by $\Phi$,
\[
(\Phi_i)_* :  \bigoplus_{k\geq 0} \mathfrak{C}^i_k \to \bigoplus_{k\geq 0} \mathfrak{C}^i_k.
\] 
Then by Proposition \ref{prop::Ext^i_T_0} $\mathfrak{C}^i_k \simeq \op{Ext}^{i-1}_{T_0 /Z} (\FF, \op{ind}^{T_0/Z}_{T_k/Z} \overline{\sigma}_{k} )$ which is zero if $i\notin\{1,2\}$. We then need to compute the map $(\Phi_i)_*$ for $i = 1,2$, which is induced by the maps $(\Phi_k^{\pm})_{*}$ in Proposition \ref{prop::Ext_T_0_to_T_2}. We use Proposition \ref{prop::Ext_T_0_to_T_2}(ii) and Lemma \ref{lemma::Phi_Ext^i_T_0_2} to compute these maps.

We first assume $n\notin \{0,\ldots, r \}$. Then $\mathfrak{C}^i_0 = 0$ by Proposition \ref{prop::Ext^i_T_0}(i) and it follows that  
\begin{align*}
(\Phi_k^+)_{*} &: \mathfrak{C}^1_{k} \to \mathfrak{C}^1_{k+1}  \text{ is multiplication by $\lambda$ for $k\geq 1$} , \\
(\Phi_k^+)_{*} & : \mathfrak{C}^2_{k} \to \mathfrak{C}^2_{k+1} \text{ is the zero map for $k\geq 1$}, \\
(\Phi_k^-)_{*}& : \mathfrak{C}^1_{k} \to \mathfrak{C}^1_{k-1} \text{ is  the zero map for $k\geq 2$},\\
(\Phi_k^-)_{*} & : \mathfrak{C}^2_{k} \to \mathfrak{C}^2_{k-1} \text{ is multiplication by $\lambda^{-1} $ for $k\geq 2$}.
\end{align*}
We then have 
\begin{align*}
(\Phi_1)_*& : \bigoplus_{k\geq 1} \mathfrak{C}^1_k \to \bigoplus_{k\geq 1} \mathfrak{C}^1_k ,~ (a_1,a_2,\ldots) \mapsto \lambda (0,a_1,a_2,\ldots),\\
(\Phi_2)_*& : \bigoplus_{k\geq 1} \mathfrak{C}^2_k \to \bigoplus_{k\geq 1} \mathfrak{C}^2_k ,~ (a_1,a_2,a_3,\ldots) \mapsto \lambda^{-1}  (a_2,a_3,\ldots).
\end{align*}
Then $\ker (\Phi_1)_* = {\rm coker} (\Phi_2)_* = 0$, and both ${\rm coker} (\Phi_1)_*$ and $\ker (\Phi_2)_*$ are one-dimensional. 

If $1\leq n\leq r-1$, then the difference from the preceding case is that $\mathfrak{C}^i_0 $ is one-dimensional for $i=1,2$, and 
\begin{align*}
(\Phi_0^+)_{*} &: \mathfrak{C}^1_{0} \to \mathfrak{C}^1_1  \text{ is multiplication by $\lambda (-1)^{r} /r!$}, \\
(\Phi_0^+)_{*} & : \mathfrak{C}^2_{0} \to \mathfrak{C}^2_1 \text{ is the zero map}, \\
(\Phi_1^-)_{*}& : \mathfrak{C}^1_{1} \to \mathfrak{C}^1_0 \text{ is the zero map}, \\
(\Phi_1^-)_{*} & : \mathfrak{C}^2_{1} \to \mathfrak{C}^2_0 \text{ is multiplication by $\lambda^{-1} (-1)^{n+1}\binom{r}{n} $}.
\end{align*}
We then have 
\begin{align*}
(\Phi_1)_{*}& : \bigoplus_{k\geq 0} \mathfrak{C}^1_k \to \bigoplus_{k\geq 0} \mathfrak{C}^1_k ,~ (a_0,a_1,a_2,\ldots) \mapsto \lambda (0, (-1)^{r}a_0 /r!,a_1,a_2,\ldots),\\
(\Phi_2)_{*}& : \bigoplus_{k\geq 0} \mathfrak{C}^2_k \to \bigoplus_{k\geq 0} \mathfrak{C}^2_k ,~ (a_0,a_1,a_2,\ldots) \mapsto \lambda^{-1} ((-1)^{n+1}\binom{r}{n} a_1 ,a_2,a_3,\ldots).
\end{align*}
Hence $\ker (\Phi_1)_* = {\rm coker} (\Phi_2)_* = 0$, and both ${\rm coker} (\Phi_1)_*$ and $\ker (\Phi_2)_*$ are one-dimensional. 

If $n=r \geq 1$, we further have 
\[
(\Phi_0^-)_{*} : \mathfrak{C}^i_{0} \to \mathfrak{C}^i_0 \text{ is multiplication by $\lambda$.}
\]
We then have 
\begin{align*}
(\Phi_1)_{*}& : \bigoplus_{k\geq 0} \mathfrak{C}^1_k \to \bigoplus_{k\geq 0} \mathfrak{C}^1_k ,~ (a_0,a_1,a_2,\ldots) \mapsto  \lambda (a_0, (-1)^{r}a_0 /r!,a_1,a_2,\ldots) ,\\
(\Phi_2)_{*}& : \bigoplus_{k\geq 0} \mathfrak{C}^2_k \to \bigoplus_{k\geq 0} \mathfrak{C}^2_k ,~ (a_0,a_1,a_2,\ldots) \mapsto \lambda^{-1}((-1)^{n+1}\binom{r}{n} a_1 ,a_2,a_3,\ldots)+\lambda (a_0,0,\ldots ).
\end{align*}
Then $\ker (\Phi_1)_* = {\rm coker} (\Phi_2)_* = 0$, and both ${\rm coker} (\Phi_1)_*$ and $\ker (\Phi_2)_*$ are one-dimensional.

If $n=0$ and $1\leq r\leq p-2$, the boundary self-map is multiplication by $\lambda^{-1}$, and we obtain
\begin{align*}
(\Phi_1)_*(a_0,a_1,a_2,\ldots)
 &=(\lambda^{-1}a_0,\lambda(-1)^ra_0/r!,
       \lambda a_1,\lambda a_2,\ldots),\\
(\Phi_2)_*(a_0,a_1,a_2,\ldots)
 &=\lambda^{-1}(a_0-a_1,a_2,a_3,\ldots).
\end{align*}
The first map is injective with image defined by
$b_1=\lambda^2(-1)^r b_0/r!$, so its cokernel is one-dimensional. The second map is surjective with kernel generated by $(1,1,0,\ldots)$. Thus the same kernel and cokernel dimensions hold.

If $n=0$ and $r=p-1$, then $\mathfrak{C}_0^i$ is two-dimensional for $i=1,2$. Use the coordinates $(x,y)$ associated with $(X^{p-1})^\vee$ and $(Y^{p-1})^\vee$ in Lemma \ref{lemma::Phi_Ext^i_T_0_2}. Its boundary formulas give
\begin{align*}
(\Phi_1)_*(x,y,a_1,a_2,\ldots)
 &=(\lambda x,\lambda^{-1}y,-\lambda(x+y),
       \lambda a_1,\lambda a_2,\ldots),\\
(\Phi_2)_*(x,y,a_1,a_2,\ldots)
 &=(\lambda x-\lambda^{-1}a_1,
       \lambda^{-1}y-\lambda^{-1}a_1,
       \lambda^{-1}a_2,\lambda^{-1}a_3,\ldots).
\end{align*}
The first map is injective, and its image consists of the sequences $(u,v,b_1,b_2,\ldots)$ satisfying $b_1=-u-\lambda^2v$, so its cokernel is one-dimensional. The second map is surjective, with kernel generated by $(\lambda^{-2},1,1,0,\ldots)$. Thus the same kernel and cokernel dimensions hold.

Finally, if $n = r =0$, then $(\Phi_k^+)_{*}$ for $k\geq 1$ and $(\Phi_k^-)_{*}$ for $k\geq 2$ are as before. By Lemma \ref{lemma::Phi_Ext^i_T_0_2}, the boundary maps are as follows
\begin{align*}
(\Phi_0^+)_{*} &: \mathfrak{C}^i_{0} \to \mathfrak{C}^i_1 \text{ is multiplication by $\lambda$}, \\
(\Phi_1^-)_{*} & : \mathfrak{C}^i_{1} \to \mathfrak{C}^i_0\text{ is multiplication by $-\lambda^{-1}$},\\
(\Phi_0^-)_{*} & : \mathfrak{C}^i_{0} \to \mathfrak{C}^i_0 \text{ is multiplication by $\lambda +\lambda^{-1}$.}
\end{align*}
We then have 
\begin{align*}
(\Phi_1)_{*} :  \bigoplus_{k\geq 0} \mathfrak{C}^1_k & \to \bigoplus_{k\geq 0} \mathfrak{C}^1_k ,\\
 (a_0,a_1,a_2,\ldots)& \mapsto \lambda \left( (1+\lambda^{-2}) a_0 -\lambda^{-2} a_1, a_0,a_1,a_2,\ldots \right) ,
\end{align*}
and
\begin{align*}
(\Phi_2)_{*} : \bigoplus_{k\geq 0} \mathfrak{C}^2_k & \to \bigoplus_{k\geq 0} \mathfrak{C}^2_k , \\
 (a_0,a_1,a_2,\ldots) & \mapsto  \lambda^{-1}\left( (1+ \lambda^2)a_0 - a_1 ,a_2 + \lambda^2 a_0,a_3,\ldots \right) .
\end{align*}
Then, as above, $\ker (\Phi_1)_* = {\rm coker} (\Phi_2)_* = 0$, and both ${\rm coker} (\Phi_1)_*$ and $\ker (\Phi_2)_*$ are one-dimensional. The dimension of $\op{Ext}^{i}_{T ,\zeta}(\eta, \pi|_{T})$ then follows.
\end{proof}

\section{The mod $p$ local Langlands correspondence and homological branching} \label{sec::mod_p_LLC_branching}

\subsection{The mod $p$ local Langlands correspondence for $\op{GL}_2(\QQ_p)$} \label{sec::mod_p_LLC}

We recall the mod $p$ local Langlands correspondence for $\op{GL}_2(\QQ_p)$, see \cite{Bre03} \cite{Colmez} \cite{Paskunas-BM}. Assume $p\geq 5$ and the coefficient field $\FF$ is sufficiently large. Let $\brho: {\rm Gal}(\overline{\QQ}_p/\QQ_p) \to {\rm GL}_2(\FF)$ be a continuous two-dimensional representation. The correspondence associates to $\brho$ an admissible smooth $\FF$-representation  $\pi(\brho)$ of $G:={\rm GL}_2(\QQ_p)$. We normalize the correspondence so that the central character of $\pi(\brho)$ is $\omega^{-1} \det\brho$ via the local Artin map, where $\omega$ is the mod $p$ cyclotomic character. 

Let $I_p$ be the inertia subgroup of ${\rm Gal}(\overline{\QQ}_p/\QQ_p)$. Let $\omega_2 : I_p \to \mathbb{F}^\times$ be Serre's fundamental character of niveau $2$ defined by  
\[\omega_2(g) = \frac{g(p^{1/(p^2-1)})}{p^{1/(p^2-1)}} \pmod{p} ,~ g \in I_p.\]
We have $\omega_2^{p^2-1} = 1$ and $\omega_2^{p+1} = \omega|_{I_p}$. For $s \in \{1, \dots, p\}$, let $\operatorname{ind}(\omega_2^s)$ denote the unique irreducible representation of ${\rm Gal}(\overline{\QQ}_p/\QQ_p)$ of dimension $2$ over $\mathbb{F}$ with determinant $\omega^s$ and whose restriction to $I_p$ is $\omega_2^s \oplus \omega_2^{ps}$. Then every absolutely irreducible representation defined over $\FF$ is of the form $\operatorname{ind}(\omega_2^s) \otimes \chi $, $s \in \{1, \dots, p\}$.

\begin{enumerate}
\item[(i)] If $\brho \simeq \operatorname{ind}(\omega_2^{r+1}) \otimes \chi $ is absolutely irreducible, then $\pi(\brho)\simeq \pi(r,0)\otimes(\chi\circ \op{det})$ is the irreducible supersingular representation of $G$ recalled in Theorem \ref{thm::Breuil}.
\item[(ii)]  If $\brho  \simeq \bigl(\begin{smallmatrix}
\chi_1 & * \\ 0& \chi_2
\end{smallmatrix} \bigr)$ with $\chi_1 \chi_2^{-1} \neq \omega^{\pm 1}$, then there is a short exact  sequence (split if and only if $\brho$ is semisimple)
\[
0 \to {\rm Ind}_{B}^{G} \chi_2\boxtimes \chi_1 \omega^{-1} \to \pi(\brho) \to {\rm Ind}_{B}^{G} \chi_1\boxtimes \chi_2 \omega^{-1} \to 0.
\]
\item[(iii)] If $\brho\simeq \bigl(\begin{smallmatrix}
\chi\omega & * \\ 0 & \chi
\end{smallmatrix} \bigr)$ is nonsplit, then there is a short exact sequence
\[
0\to  {\rm St} \otimes \chi\circ \det \to \pi(\brho) \to \kappa \otimes \chi\circ \det \to 0,
\]
where $\kappa$ is a nonsplit extension $0 \to \FF  \to \kappa \to {\rm Ind}_{B}^{G} \omega \boxtimes \omega^{-1} \to 0$.
\item[(iv)]
 If $\brho  \simeq \bigl(\begin{smallmatrix}
\chi & * \\ 0& \chi\omega
\end{smallmatrix} \bigr)$ is nonsplit, then there is a short exact sequence
\[
0\to {\rm Ind}_{B}^{G} \chi\omega \boxtimes \chi\omega^{-1} \to \pi(\brho) \to \tau_1 \otimes \chi\circ \det \to 0,
\]
where $\tau_1$ is the unique nonsplit extension $0 \to {\rm St} \to \tau_1 \to \FF^{\oplus 2} \to 0$ such that $\op{Hom}_G(\FF , \tau_1) = 0$.
\item[(v)] If $\brho\simeq  \chi\omega\oplus \chi$, then $\pi(\brho)  = (\FF \oplus \op{St} \oplus  {\rm Ind}_{B}^{G} \omega \boxtimes \omega^{-1})\otimes\chi\circ\det$.
\end{enumerate}

\subsection{Homological branching for $\pi(\brho)$} \label{sec::branching_pi(brho)}

Let $L$ be a quadratic extension of $\QQ_p$. Let $T$ be the diagonal torus of $G$. Let $\zeta$ denote the central character of $\pi(\brho)$.

\begin{definition} \label{def::exceptional_pi_brho}
Let $H = L^{\times}$ or $T$. Let $\eta: H \to \FF^{\times}$ be a smooth character such that $\eta|_Z = \zeta $. We say that $\eta$ is {\em nonexceptional for $\pi(\brho)$} if it is nonexceptional, in the sense of Definition \ref{def::exceptional_L} and Definition \ref{def::exceptional_T}, for every infinite-dimensional irreducible constituent of $\pi(\brho)$. Otherwise we say $\eta$ is {\em exceptional for $\pi(\brho)$}.
\end{definition}

The following corollary is an immediate consequence of Theorems \ref{nonsplitPS}, \ref{thm::St_nonsplit_nongeneric}, \ref{thm::branching_ramified} and \ref{thm::branching_unr}. We note that when $p\geq 5$, $L^{\times}/\QQ_p^{\times}$ is $p$-torsion-free.

\begin{corollary} \label{cor::relation_with_mod_p_LLC_1}
    Let $\eta: L^{\times} \to \FF^{\times}$ be a smooth character such that $\eta|_{\QQ_p^{\times}} = \zeta$. 
    \begin{enumerate}
	\item [(i)] If $\eta$ is nonexceptional for $\pi(\brho)$, then there are quasi-isomorphisms
\[
 \op{RHom}_{L^\times,\zeta}(\pi(\brho)|_{L^\times},\eta) \simeq \FF^{\oplus 2} [-1],\qquad \op{RHom}_{L^\times,\zeta}(\eta , \pi(\brho)|_{L^\times}) \simeq \FF^{\oplus 2} [0].
\]

\item [(ii)] If $\eta$ is exceptional for $\pi(\brho)$, then $\op{RHom}_{L^\times,\zeta}(\pi(\brho)|_{L^\times}, \eta)$ and $\op{RHom}_{L^\times,\zeta}(\eta,\pi(\brho)|_{L^\times})$ are in $\op{D^b_{fd}}(\FF)$, and
\[
\chi(\op{RHom}_{L^\times,\zeta}(\pi(\brho)|_{L^\times}, \eta) ) = -2,\qquad \chi(\op{RHom}_{L^\times,\zeta}(\eta,\pi(\brho)|_{L^\times})) =2.
\]
\end{enumerate}
\end{corollary}
\begin{proof}
When $\brho$ is irreducible, the assertion is an immediate consequence of Theorems \ref{thm::branching_ramified} and \ref{thm::branching_unr}.  When $\brho$ is reducible, $\pi(\brho)$ has exactly two infinite-dimensional constituents, so for nonexceptional characters, the assertion follows from Theorem \ref{nonsplitPS}; for exceptional characters, the assumption $p\geq 5$ guarantees that $L^{\times}/\QQ_p^{\times}$ is $p$-torsion-free, so Theorem \ref{thm::St_nonsplit_nongeneric} gives the results,  because the $\op{RHom}$-complexes given by the finite-dimensional constituents of $\pi(\brho)$, which are just characters, have Euler--Poincar\'e characteristic zero.
\end{proof}

\begin{corollary}  \label{cor::relation_with_mod_p_LLC_2}
    Let $\eta: T \to \FF^{\times}$ be a smooth character of $T$ such that $\eta|_{\QQ_p^\times}=\zeta$. 
    \begin{enumerate}
	\item [(i)] If $\eta$ is nonexceptional for $\pi(\brho)$, then there are quasi-isomorphisms
\[
\op{RHom}_{T,\zeta}(\pi(\brho)|_{T},\eta) \simeq \op{RHom}_{T,\zeta}(\eta, \pi(\brho)|_{T}) \simeq \FF^{\oplus 2}[-1].
\]

\item [(ii)] If $\eta$ is exceptional for $\pi(\brho)$, then $\op{RHom}_{T,\zeta}(\pi(\brho)|_{T}, \eta)$ and $\op{RHom}_{T,\zeta}(\eta,\pi(\brho)|_{T})$ are in $\op{D^b_{fd}(\FF)}$ and 
\[
\chi(\op{RHom}_{T,\zeta}(\pi(\brho)|_{T}, \eta)) = \chi(\op{RHom}_{T,\zeta}(\eta,\pi(\brho)|_{T})) =-2.
\]
\end{enumerate}
\end{corollary}
\begin{proof}
The proof is similar to that of Corollary \ref{cor::relation_with_mod_p_LLC_1}, using instead Theorems \ref{thm::PS_T}, \ref{thm::ps_T_EP}, \ref{thm::sp_T}, \ref{thm::sp_T_EP} and \ref{thm::branching_ss_T}. We note that when $p\geq 5$, $\QQ_p^{\times}$ is $p$-torsion-free.
\end{proof}

%% file: chapter_5.tex
\section{Homological branching for $p$-adic Banach space representations} \label{sec::banach}

Let $E$ be a sufficiently large finite extension of $\QQ_p$ with ring of integers $\mc{O}$, uniformizer $\varpi$ and residue field $\FF$. Let $\Pi$ be a unitary admissible absolutely irreducible $E$-Banach space representation of $G= \op{GL}_2(F)$, where $F$ is a finite extension of $\QQ_p$ of degree $d$. Recall that $\Pi$ is {\em ordinary} if it is a subquotient of a unitary parabolic induction of a unitary character. In this section, we assume $\Pi$ is ordinary if $F\neq \QQ_p$. 

The aim of this section is to study the homological branching of $\Pi$ with respect to nonexceptional unitary characters (see Definition \ref{def::generic_Banach_L}) of $H$, where $H$ is the split or nonsplit torus.

\subsection{The semisimple reduction of $\Pi$}

Since $\Pi$ is unitary, there is an open bounded $G$-invariant $\mc{O}$-lattice $\Pi^0 \subset \Pi$. If $\Pi$ is ordinary, then $\Pi$ is a subquotient of $({\rm Ind}_B^G \delta)^{\rm cont}$, where $\delta$ is some unitary character and the superscript means taking continuous induction. Then $(\Pi^0/\varpi)^{\rm ss} \subseteq ({\rm Ind}_B^G (\delta^0/\varpi))^{\rm ss}$, where $\delta^0$ is an invariant $\mc{O}$-lattice of $\delta$. When $G = {\rm GL}_2(\QQ_p)$ and $\Pi$ is any unitary admissible absolutely irreducible Banach representation, the reduction $(\Pi^0/\varpi)^{\rm ss}$ is determined by Pa{\v{s}}k{\={u}}nas \cite{Paskunas-IHES} ($p\geq 5$) and Colmez--Dospinescu--Pa{\v{s}}k{\={u}}nas \cite{CDP} ($p=2,3$). 

\begin{theorem} (\cite[Thm. 1.1]{Paskunas-IHES}, \cite[Thms. 1.4 and 1.9]{CDP}) \label{thm::Paskunas_Banach_1}
Let $\Pi$ be a unitary admissible absolutely irreducible $E$-Banach space representation of ${\rm GL}_2(\QQ_p)$. Then $\Pi^0 /\varpi $ is of finite length. Moreover, one of the following holds:
\begin{enumerate}
\item[(i)] $\Pi^0/\varpi$ is absolutely irreducible supersingular;

\item[(ii)] $(\Pi^0/\varpi)^{\rm ss} \subseteq ( {\rm Ind}_B^G \chi_1\boxtimes \chi_2\omega^{-1})^{\rm ss} \oplus ( {\rm Ind}_B^G \chi_2\boxtimes \chi_1\omega^{-1})^{\rm ss} $ for some smooth characters $\chi_1,\chi_2: \QQ_p^{\times} \to \FF^{\times}$.
\end{enumerate} 
Moreover, $\Pi$ is ordinary if and only if the inclusion in (ii) is not an isomorphism.
\end{theorem}

\begin{remark}
The representation 
\[
\pi\{\chi_1,\chi_2\} : =  ( {\rm Ind}_B^G \chi_1\boxtimes \chi_2\omega^{-1})^{\rm ss} \oplus ( {\rm Ind}_B^G \chi_2\boxtimes \chi_1\omega^{-1})^{\rm ss}
\]
is explicitly described in \cite[Lemma 2.14]{CDP}.
\end{remark}

\quash{
\begin{lemma} \label{lem::reduction_of_non-ordinary}
Let $\Pi$ be an infinite dimensional unitary admissible absolutely irreducible $E$-Banach space representation of $G$ with a central character and let $\Pi^0$ be an open bounded $G$-invariant lattice in $\Pi$. Then $\Pi$ is ordinary if and only if $(\Pi^0/\varpi)^{\rm ss}$ contains exactly one infinite dimensional irreducible constituent.
\end{lemma}
\begin{proof}
This follows directly from \cite[Thm. 1.9]{CDP} and the properties of Colmez's Montreal functor.
\end{proof}
}

From now on, we assume $\Pi$ is an {\em infinite-dimensional} unitary admissible absolutely irreducible $E$-Banach space representation of $G$. We fix an open bounded $G$-invariant $\mc{O}$-lattice $\Pi^0$ in $\Pi$. Let $\zeta: Z\to E^{\times}$ (resp. $\zeta^0 : Z\to \mc{O}^{\times}$) denote the central character of $\Pi$ (resp. $\Pi^0$), which exists by \cite{Dospinescu-Schraen}. 

\begin{definition} \label{def::generic_Banach_L}
Let $H \in \{ L^{\times},T \}$. Let $\eta: H \to  E^{\times}$ be a unitary character such that  $\eta|_Z = \zeta$. Let $\eta^0 $ be an $H$-invariant $\mc{O}$-lattice in $\eta$ such that $\eta^0|_Z = \zeta^0$. We say $\eta$ is {\em nonexceptional for $\Pi$} if $\eta^0/\varpi$ is nonexceptional for every infinite-dimensional irreducible constituent of $(\Pi^0/\varpi)^{\rm ss}$ in the sense of Definition \ref{def::exceptional_L} and Definition \ref{def::exceptional_T}.
\end{definition}

Set
\[
m_{\Pi}:= \left\{\begin{array}{cl}
        2,& \text{ if $\Pi$ is nonordinary and $G = {\rm GL}_2(\QQ_p)$},\\
        1,& \text{ if $\Pi$ is ordinary}. 
        \end{array}\right.
\]

\begin{proposition} \label{prop:dim_reduction_L}
Let $\eta: L^{\times} \to  E^{\times}$ be a unitary character that is nonexceptional for $\Pi$. Then
\[
\dim_{\FF}{\rm Ext}_{\op{Mod}^{\rm sm}_{L^{\times},\zeta}(\mc{O}/\varpi)}^i (\Pi^0/\varpi, \eta^0 /\varpi ) = \left\{\begin{array}{cl}
        m_{\Pi} ,& \text{ if } i = d,\\
        0,& \text{ otherwise}. 
        \end{array}\right.
\]
and
\[
\dim_{\FF}{\rm Ext}_{\op{Mod}^{\rm sm}_{L^{\times},\zeta}(\mc{O}/\varpi)}^i (\eta^0 /\varpi, \Pi^0/\varpi ) = \left\{\begin{array}{cl}
        m_{\Pi},& \text{ if } i = 0,\\
        0,& \text{ otherwise}. 
        \end{array}\right.
\]
\end{proposition}
\begin{proof}
If $\Pi$ is nonordinary, the statement is a consequence of Theorems \ref{nonsplitPS}, \ref{thm::branching_ramified} and \ref{thm::branching_unr}. The statement for ordinary $\Pi$ follows from Theorem \ref{nonsplitPS} and the fact that $(\Pi^0/\varpi)^{\rm ss} \subseteq ({\rm Ind}_B^G (\delta^0/\varpi))^{\rm ss}$ for some continuous unitary character $\delta$.
\end{proof}

\begin{proposition}  \label{prop:dim_reduction_T}
Let $\eta: T \to  E^{\times}$ be a unitary character that is nonexceptional for $\Pi$. Then
\[
\dim_{\FF}{\rm Ext}_{\op{Mod}^{\rm sm}_{T,\zeta}(\mc{O}/\varpi)}^i (\Pi^0/\varpi , \eta^0 /\varpi ) = \left\{\begin{array}{cl}
        m_{\Pi},& \text{ if } i = d,\\
        0,& \text{ otherwise}. 
        \end{array}\right.,
\]
and
\[
\dim_{\FF}{\rm Ext}_{\op{Mod}^{\rm sm}_{T,\zeta}(\mc{O}/\varpi)}^i (\eta^0 /\varpi, \Pi^0/\varpi) = \left\{\begin{array}{cl}
        m_{\Pi},& \text{ if } i = 1,\\
        0,& \text{ otherwise}. 
        \end{array}\right..
\]
\end{proposition}
\begin{proof}
If $\Pi$ is nonordinary, the statement is a consequence of Theorems \ref{thm::PS_T}, \ref{thm::sp_T} and \ref{thm::branching_ss_T}.  The statement for ordinary $\Pi$ follows from Theorems \ref{thm::PS_T}, \ref{thm::sp_T} and the fact that $(\Pi^0/\varpi)^{\rm ss} \subseteq ({\rm Ind}_B^G (\delta^0/\varpi))^{\rm ss}$ for some continuous unitary character $\delta$.
\end{proof}

\subsection{Lemmas on extension groups} \label{section::bananch_lemma_extn}

Let $(A,\mathfrak{m})$ be a complete local Noetherian $\mc{O}$-algebra with residue field $\FF$. For a locally profinite group $G$, recall the category ${\rm Mod}^{\rm pro}_G(A)$ consisting of profinite augmented $G$-representations over $A$, see \S \ref{sec::notation}. If $G$ is moreover profinite, let ${\rm Mod}^{\rm pro.fg}_G(A)$ denote the  full subcategory of ${\rm Mod}^{\rm pro}_G(A)$ consisting of finitely generated  $A[[G]]$-modules.

\quash{
\begin{lemma} \label{lemma::module_ext_1} \mar{Remove this lemma.}
Let $G$ be a profinite group. Let $M_1,M_2 \in {\rm Mod}^{\rm pro.fg}_G(A)$. Then 
\[
{\rm Ext}_{{\rm Mod}^{\rm pro}_G(A)}^i (M_1, M_2 ) \simeq {\rm Ext}_{{\rm Mod}^{\rm pro.fg}_G(A)}^i (M_1, M_2),~\forall i.
\]
\end{lemma}
\begin{proof}
Since $M_1$ is finitely generated, we can find a projective resolution of $M_1$ by finite free $A[[G]]$-modules.
\end{proof}
}

\begin{lemma} \label{lemma::module_ext_2}
Let $G$ be a profinite group. Let $M_1,M_2 \in {\rm Mod}^{\rm pro}_G(A)$ such that $M_1$ is flat over $A$. Then for any $n$,
\[
{\rm Ext}_{{\rm Mod}^{\rm pro}_G(A)}^i (M_1, M_2 \otimes_A A /\mf{m}^n) \simeq {\rm Ext}_{{\rm Mod}^{\rm pro}_G(A /\mf{m}^n)}^i (M_1\otimes_A A /\mf{m}^n, M_2\otimes_A A /\mf{m}^n),~\forall i.
\]
\end{lemma}
\begin{proof}
We choose a resolution $F_{\bullet}$ of $ M_1$ by free $ A[[G]]$-modules:
\[
\cdots\to F_1 \xrightarrow{d_1} F_0 \xrightarrow{d_0} M_1 \to 0.
\]
Let $K_i: = \ker(d_i : F_i \to F_{i-1})$ with the convention $F_{-1}: = M_1$. Since $M_1$ is flat over $A$, ${\rm Tor}_j^A(M_1, A/\mf{m}^n) = 0$ for $j\geq 1$. Since $F_0$ is free over $A$, we deduce that ${\rm Tor}_j^A(K_0, A/\mf{m}^n) = 0$ for $j\geq 1$. By induction, ${\rm Tor}_j^A(K_i, A/\mf{m}^n) = 0$ for $j\geq 1$ and any $i$. This implies that 
\[
K_i\otimes_A A/\mf{m}^n \simeq \ker(d_i\otimes 1 : F_i\otimes_A A/\mf{m}^n \to F_{i-1}\otimes_A A/\mf{m}^n).
\]
Thus the complex $F_{\bullet}\otimes_A A/\mf{m}^n$ is exact and it is a resolution of $ M_1\otimes_A A /\mf{m}^n$ by free $(A/\mf{m}^n) [[G]]$-modules. Then
\begin{align*}
{\rm Ext}_{{\rm Mod}^{\rm pro}_G(A)}^i (M_1, M_2 \otimes_A A /\mf{m}^n) & \simeq H^i({\rm Hom}_G (F_{\bullet}, M_2 \otimes_A A /\mf{m}^n)) \\
&\simeq  H^i({\rm Hom}_G (F_{\bullet}\otimes_A A /\mf{m}^n, M_2 \otimes_A A /\mf{m}^n)) \\
& \simeq {\rm Ext}_{{\rm Mod}^{\rm pro}_G(A /\mf{m}^n)}^i (M_1\otimes_A A /\mf{m}^n, M_2\otimes_A A /\mf{m}^n).
\end{align*}
\end{proof}

The following variant of Lemma \ref{lemma::module_ext_2} will be used in the proof of Proposition \ref{prop::vanishing_higher_Ext_L_2}(i).

\begin{lemma} \label{lemma::Group_Ext_3}
Let $G$ be a locally profinite group and $Z$ be a subgroup of $Z_G$, where $Z_G$ is the center of $G$, such that $G/Z$ is profinite. Let $\zeta: Z\to A^{\times}$ be a character which can be extended to a character of $G$. Let $M_1,M_2 \in {\rm Mod}^{\rm pro}_{G,\zeta}(A)$ such that $M_1$ is flat over $A$. Then for any $n$,
\[
{\rm Ext}_{{\rm Mod}^{\rm pro}_{G,\zeta}(A)}^i (M_1, M_2 \otimes_A A /\mf{m}^n) \simeq {\rm Ext}_{{\rm Mod}^{\rm pro}_{G,\zeta}(A /\mf{m}^n)}^i (M_1\otimes_A A /\mf{m}^n, M_2\otimes_A A /\mf{m}^n),~\forall i.
\]
\end{lemma}
\begin{proof}
Twisting by $\zeta^{-1}$ induces an equivalence of categories between ${\rm Mod}^{\rm pro}_{G,\zeta}(A)$ and $ {\rm Mod}^{\rm pro}_{G/Z}(A)$. Then the lemma follows from Lemma \ref{lemma::module_ext_2}.
\end{proof}

\begin{lemma} \label{lemma::Group_Ext_2}
Let $G$ be a locally profinite group. Let $j\leq n$ be two positive integers. For $\pi \in \op{Mod}^{\rm sm}_{G}(A/ \mf{m}^j)$, we have an isomorphism
\[
{\rm Ext}_{\op{Mod}^{\rm sm}_{G}(A/ \mf{m}^n)}^i (A/ \mf{m}^n, \pi) \simeq {\rm Ext}_{\op{Mod}^{\rm sm}_{G}(A/ \mf{m}^j)}^i (A/ \mf{m}^j, \pi),~\forall i.
\]
\end{lemma}
\begin{proof} 
This is \cite[Thm. 1.1]{Fust} (see also \cite[Prop. 2.2.6]{Emerton-ordinary-2} when $G$ is compact). In fact, \cite[Thm. 1.1]{Fust} shows that 
\[
{\rm Ext}_{\op{Mod}^{\rm sm}_{G}(A/ \mf{m}^n)}^i (A/ \mf{m}^n, \pi) \simeq H^i(G ,\pi),
\]
where $H^i (G, -)$ denotes the continuous group cohomology. By the proof of \cite[Cor. 4.4]{Fust}, 
\[
H^i(G ,\pi) \simeq H^i (C^{\bullet}(G,\pi)^G ),
\]
where $C^{\bullet}(G,\pi)$ is the continuous cochain complex defined in \cite[\S2]{Fust}. Since $\pi \in \op{Mod}^{\rm sm}_{G}(A/ \mf{m}^j)$, the complex $C^{\bullet}(G,\pi)$ is in fact a complex of $A/\mf{m}^j$-modules. Hence by \cite[Thm. 1.1]{Fust} again, 
\[
H^i (C^{\bullet}(G,\pi)^G )  \simeq {\rm Ext}_{\op{Mod}^{\rm sm}_{G}(A/ \mf{m}^j)}^i (A/ \mf{m}^j, \pi) .
\]  
\end{proof}

Assume $G$ is a compact $p$-adic Lie group. The Iwasawa algebra $\mc{O}[[G]]$ is a Noetherian ring by Lazard \cite[V.2.2.4]{Lazard}; and so is $\mc{O}[[G]][1/p]$. Let  ${\rm Mod}^{\rm pro.fl}_G(\mc{O})$  denote the full subcategory of ${\rm Mod}^{\rm pro}_G(\mc{O})$ consisting of objects which are flat over $\mc{O}$. Let ${\rm Mod}^{\rm pro.fl}_{G,{\rm fg}}(\mc{O})$ denote the intersection of ${\rm Mod}^{\rm pro.fl}_G(\mc{O})$ and ${\rm Mod}^{\rm pro.fg}_G(\mc{O})$. For any additive category $\mf{A}$, let $\mf{A}_{\QQ}$ denote the additive category with the same objects as $\mf{A}$ and morphisms
\[
{\rm Hom}_{\mf{A}_{\QQ}}(A,B): = {\rm Hom}_{\mf{A}}(A,B) \otimes_{\ZZ} \QQ.
\]
The functor $\Pi\mapsto (\Pi^0)^d $ induces anti-equivalences of categories (see \cite[Thms. 2.3 and 3.5]{ST-banach})
\[
{\rm Ban}_{G}(E) \simeq {\rm Mod}^{\rm pro.fl}_G(\mc{O})_{\QQ},\quad {\rm Ban}_{G}^{\rm adm}(E) \simeq {\rm Mod}^{\rm pro.fl}_{G,{\rm fg}}(\mc{O})_{\QQ},
\] 
where ${\rm Ban}_{G}(E)$ (resp. ${\rm Ban}_{G}^{\rm adm}(E)$) denotes the category of all (resp. admissible) $E$-Banach space representations of $G$. Via the functor $M\mapsto M[1/p]$, the category ${\rm Mod}^{\rm pro.fl}_{G,{\rm fg}}(\mc{O})_{\QQ}$ is equivalent to ${\rm Mod}^{\rm fg}(\mc{O}[[G]][1/p] )$, the category of finitely generated $\mc{O}[[G]][1/p]$-modules. Since $\mc{O}[[G]][1/p]$ is a Noetherian ring, the category ${\rm Mod}^{\rm fg}(\mc{O}[[G]][1/p] )$ is an abelian category; and so is ${\rm Mod}^{\rm pro.fl}_{G,{\rm fg}}(\mc{O})_{\QQ}$.

\begin{lemma} \label{lemma::ext_banach_module_cpt}
Let $G$ be a compact $p$-adic Lie group. Let $\Pi_1,\Pi_2 \in {\rm Ban}^{\rm adm}_{G}(E)$. Then
\[
{\rm Ext}_{{\rm Ban}^{\rm adm}_{G}(E)}^i (\Pi_1, \Pi_2) \simeq {\rm Ext}_{\op{Mod}^{\rm pro}_{G}(\mc{O})}^i ((\Pi_2^0)^d, (\Pi_1^0)^d)[1/p],~\forall i.
\] 
\end{lemma}
\begin{proof}
Since the category ${\rm Ban}^{\rm adm}_{G}(E) $ is anti-equivalent to ${\rm Mod}^{\rm pro.fl}_{G,{\rm fg}}(\mc{O})_{\QQ}$, it suffices to show
\[
{\rm Ext}_{{\rm Mod}^{\rm pro.fl}_{G,{\rm fg}}(\mc{O})_{\QQ}}^i ((\Pi_2^0)^d, (\Pi_1^0)^d) \simeq {\rm Ext}_{\op{Mod}^{\rm pro}_{G}(\mc{O})}^i ((\Pi_2^0)^d, (\Pi_1^0)^d)[1/p],~\forall i.
\]
We choose a resolution $F_{\bullet}$ of $(\Pi_2^0)^d$ by finite free $\mc{O}[[G]]$-modules. Since $F_i[1/p]$ is projective in ${\rm Mod}^{\rm pro.fl}_{G,{\rm fg}}(\mc{O})_{\QQ}$, $F_i$ is projective in ${\rm Mod}^{\rm pro.fl}_{G,{\rm fg}}(\mc{O})_{\QQ}$. Then
\begin{equation*}
\begin{multlined}
{\rm Ext}_{{\rm Mod}^{\rm pro.fl}_{G,{\rm fg}}(\mc{O})_{\QQ}}^i ((\Pi_2^0)^d, (\Pi_1^0)^d)  \simeq H^i ({\rm Hom}_G (F_{\bullet}, (\Pi_1^0)^d)\otimes \QQ) \\
  \simeq  H^i ({\rm Hom}_G (F_{\bullet}, (\Pi_1^0)^d))\otimes \QQ \simeq {\rm Ext}_{\op{Mod}^{\rm pro}_{G}(\mc{O})}^i ((\Pi_2^0)^d, (\Pi_1^0)^d)[1/p].
\end{multlined}
\end{equation*}
\end{proof}

As in Lemma \ref{lemma::Group_Ext_3}, we have a variant of Lemma \ref{lemma::ext_banach_module_cpt}:

\begin{lemma} \label{lemma::ext_banach_module}
Let $G$ be a $p$-adic Lie group and $Z$ be a subgroup of $Z_G$, where $Z_G$ is the center of $G$, such that $G/Z$ is a compact $p$-adic Lie group. Let $\zeta: Z\to \mc{O}^{\times}$ be a character which can be extended to a character of $G$. Let $\Pi_1,\Pi_2 \in {\rm Ban}^{\rm adm}_{G,\zeta}(E)$. Then
\[
{\rm Ext}_{{\rm Ban}^{\rm adm}_{G,\zeta}(E)}^i (\Pi_1, \Pi_2) \simeq {\rm Ext}_{\op{Mod}^{\rm pro}_{G,\zeta}(\mc{O})}^i ((\Pi_2^0)^d, (\Pi_1^0)^d)[1/p],~\forall i.
\] 
\end{lemma}

\subsection{Lemmas in commutative algebra}

\begin{lemma} \label{lemma::CA_2}
Let $(R, \mathfrak{m})$ be a complete Noetherian commutative local ring. Let $\{M_n\}_{n \geq 1}$ be a projective system of finite $R$-modules in which each $M_n$ is a finite free $R/\mathfrak{m}^n$-module of rank $r$. We assume the transition maps $\pi_n: M_n \to M_{n-1}$ induce the isomorphisms $M_n \otimes_R R/\mathfrak{m}^{n-1} \cong M_{n-1}$. Then $M = \varprojlim M_n$ is a free $R$-module of rank $r$.
\end{lemma}
\begin{proof}
We fix a basis $ \{ e_{1,1}, e_{1,2}, \ldots, e_{1,r} \}$ of $M_1$ over $R /\mathfrak{m} $. We choose by induction a basis of $M_n$ which is compatible with the transition map. Suppose we have a basis $\{ e_{n-1,1}, \dots, e_{n-1,r} \}$ for $M_{n-1}$ over $R/\mathfrak{m}^{n-1}$. Since the transition map $\pi_n: M_n \to M_{n-1}$ is surjective, there exist
\[ e_{n,j} \in M_n \quad \text{such that} \quad \pi_n(e_{n,j}) = e_{n-1,j} \quad \text{for } j=1, \dots, r. \]
The set $\{ e_{n,j} \pmod{\mathfrak{m}M_n} \}$ gives the basis $ \{ e_{1,1}, e_{1,2}, \ldots, e_{1,r} \}$ for $M_n/\mathfrak{m}M_n \cong M_1$. By Nakayama's Lemma, $\{e_{n,1}, \dots, e_{n,r}\}$ generate $M_n$. Since $M_n$ is free of rank $r$ and since any surjective endomorphism of a finitely generated $R$-module is an isomorphism, any set of $r$ generators is a basis. Thus, $\{e_{n,1}, \dots, e_{n,r}\}$ is a basis for $M_n$.

By induction, we obtain $r$ sequences $e_j = (e_{n,j})_{n \geq 1}$ for $j=1, \dots, r$. By the compatibility of our choices, $e_j \in \varprojlim M_n = M$. Let $x \in M$, where $x = (x_n)_{n \geq 1}$ with $x_n \in M_n$. Since each $x_n$ can be uniquely written as:
\[ x_n = \sum_{j=1}^r a_{n,j} e_{n,j}, \quad a_{n,j} \in R/\mathfrak{m}^n , \]
the condition $\pi_n(x_n) = x_{n-1}$ implies that $a_{n,j} \equiv a_{n-1,j} \pmod{\mathfrak{m}^{n-1}}$. Thus, for each $j$, the sequence $(a_{n,j})_{n \geq 1}$ defines a unique element $a_j \in \varprojlim R/\mathfrak{m}^n$.

Since $R$ is complete, $R \cong \varprojlim R/\mathfrak{m}^n$. Thus, every $x \in M$ has a unique representation:
\[ x = \sum_{j=1}^r a_j e_j, \quad a_j \in R. \]
This proves that $M$ is a free $R$-module of rank $r$.
\end{proof}

\begin{corollary} \label{lemma::CA_1}
Let $\{M_n\}_{n\geq 1}$ be a projective system of $\mc{O}$-modules such that each $M_n$ is a finite $\mc{O} /\varpi^n$-module. We assume there is an integer $r$ such that for each $n\geq 1$, $\op{length}_{\mc{O}} M_n = r n$, and the transition map $ \pi_n: M_n \to M_{n-1}$ induces an isomorphism $M_n \otimes_{\mc{O}} \mc{O}/\varpi^{n-1} \cong M_{n-1}$. Then $M = \varprojlim M_n$ is a free $\mc{O}$-module of rank $r$.
\end{corollary}
\begin{proof}
By Lemma \ref{lemma::CA_2}, it suffices to show $M_n$ is a finite free $\mc{O} /\varpi^n$-module of rank $r$. Since finite flatness implies freeness in this setup (see \cite[Thm. 7.10]{Mat-CRT}), it suffices to show $M_n$ is a flat $\mc{O} /\varpi^n$-module. By \cite[Thm. 22.3(3)]{Mat-CRT} it suffices to check the natural surjective map
\[
f : \varpi \mc{O}/\varpi^n \mc{O} \otimes_{\mc{O}/\varpi^n } M_n \twoheadrightarrow \varpi M_n 
\]
is an isomorphism.

By assumption, the composite $\pi_2 \circ \cdots \circ \pi_n$ induces an isomorphism 
\[
M_n \otimes_{\mc{O}} \mc{O} /\varpi \cong M_{1}.
\]
Thus $M_n/\varpi M_n $ has dimension $r$ over $\FF$.  Nakayama’s lemma shows that $M_n$ is generated by $r$ elements over $\mathcal O/\varpi^n$. Thus there is a surjection $(\mathcal O/\varpi^n)^{\oplus r}\twoheadrightarrow M_n$, and 
\[
\op{length}_{\mc{O}} \left( \varpi \mc{O}/\varpi^n \mc{O} \otimes_{\mc{O}/\varpi^n } M_n \right) \leq r(n-1).
\]
On the other hand, by the short exact sequence $ 0\to \varpi M_n \to M_n \to M_n/\varpi M_n \cong M_1 \to 0$,
\[
\op{length}_{\mc{O}}\varpi M_n = \op{length}_{\mc{O}}  M_n - \op{length}_{\mc{O}} M_1 = r(n-1).
\]
Thus $f$ must be an isomorphism.
\end{proof}

\subsection{The case of nonsplit tori}
Let $G : = {\rm GL}_2(F)$, where $F$ is a finite extension of $\QQ_p$ of degree $d$. Let $\Pi$ be an {\em infinite-dimensional} unitary admissible absolutely irreducible $E$-Banach space representation of $G$. Let $\zeta$ denote the central character of $\Pi$. Recall that we require $\Pi$ is ordinary if $F\neq \QQ_p$.

Let $\eta: L^{\times} \to  E^{\times}$ be a continuous character such that  $\eta|_Z = \zeta$. Since $\zeta$ is unitary and $L^{\times}/Z$ is compact, $\eta$ is unitary. Let $\eta^0 :  L^{\times} \to \mc{O}^{\times}$ be the unique invariant $\mc{O}$-lattice of $\eta$ such that $\eta^0|_Z = \zeta^0$.

\begin{proposition} \label{prop::vanishing_higher_Ext_L_2}
Let $j\leq n$ be two positive integers. 
\begin{itemize}
\item[(i)] We have an isomorphism
\[
{\rm Ext}_{\op{Mod}^{\rm sm}_{L^{\times},\zeta}(\mc{O}/\varpi^n)}^i ( \Pi^0/\varpi^j \Pi^0, \eta^0 /\varpi^n) \simeq  {\rm Ext}_{\op{Mod}^{\rm sm}_{L^{\times},\zeta}(\mc{O}/\varpi^j)}^i ( \Pi^0/\varpi^j \Pi^0, \eta^0 /\varpi^j), ~\forall i .
\]
\item[(ii)] If $\eta$ is nonexceptional for $\Pi$ in the sense of Definition \ref{def::generic_Banach_L}, then 
\[
\op{length}_{\mc{O}} {\rm Ext}_{\op{Mod}^{\rm sm}_{L^{\times},\zeta}(\mc{O}/\varpi^j)}^i ( \Pi^0/\varpi^j \Pi^0, \eta^0 /\varpi^j) = \left\{\begin{array}{cl}
        m_{\Pi}j ,& \text{ if } i = d,\\
        0,& \text{ otherwise}. 
        \end{array}\right.
\]
\end{itemize}
\end{proposition}

\begin{proof}
The Pontryagin duality induces an isomorphism
\[
{\rm Ext}_{\op{Mod}^{\rm sm}_{L^{\times},\zeta}(\mc{O}/\varpi^n)}^i ( \Pi^0/\varpi^j \Pi^0,\eta^0/\varpi^n) \simeq {\rm Ext}_{\op{Mod}^{\rm pro}_{L^{\times},\zeta}(\mc{O}/\varpi^n)}^i ( (\eta^0/\varpi^n)^\vee , (\Pi^0/\varpi^j \Pi^0)^{\vee}).
\]
It follows from Lemma \ref{lemma::Group_Ext_3} that
\[
{\rm Ext}_{\op{Mod}^{\rm pro}_{L^{\times},\zeta}(\mc{O}/\varpi^n)}^i ( (\eta^0/\varpi^n)^\vee , (\Pi^0/\varpi^j \Pi^0)^{\vee}) \simeq {\rm Ext}_{\op{Mod}^{\rm pro}_{L^{\times},\zeta}(\mc{O}/\varpi^j)}^i ( (\eta^0/\varpi^j)^\vee , (\Pi^0/\varpi^j \Pi^0)^{\vee}).
\]
This proves (i) by applying  Pontryagin duality again.

Assume $\eta$ is nonexceptional for $\Pi$. By (i), we have
\begin{equation} \label{eq::in_vanishing_Ext_L_2}
 {\rm Ext}_{\op{Mod}^{\rm sm}_{L^{\times},\zeta}(\mc{O}/\varpi^j)}^i ( \Pi^0/\varpi \Pi^0,\eta^0/\varpi^j)  \simeq  {\rm Ext}_{\op{Mod}^{\rm sm}_{L^{\times},\zeta}(\FF)}^i ( \Pi^0/\varpi \Pi^0,\eta^0/\varpi), 
\end{equation}
which is zero if $i\neq d$ and has dimension $m_{\Pi}$ if $i=d$ by Proposition \ref{prop:dim_reduction_L}.

The short exact sequence 
\[
0\to \varpi \Pi^0/\varpi^j \Pi^0 \to \Pi^0/\varpi^j \Pi^0 \to  \Pi^0/\varpi\Pi^0 \to 0
\]
induces a long exact sequence
\begin{equation*}
\begin{multlined}
\cdots \to {\rm Ext}_{\op{Mod}^{\rm sm}_{L^{\times},\zeta}(\mc{O}/\varpi^j)}^i ( \Pi^0/\varpi \Pi^0,\eta^0/\varpi^j) \to {\rm Ext}_{\op{Mod}^{\rm sm}_{L^{\times},\zeta}(\mc{O}/\varpi^j)}^i ( \Pi^0/\varpi^j \Pi^0,\eta^0/\varpi^j) \\ \to {\rm Ext}_{\op{Mod}^{\rm sm}_{L^{\times},\zeta}(\mc{O}/\varpi^j)}^i ( \varpi\Pi^0/\varpi^j \Pi^0,\eta^0/\varpi^j) \to \cdots.
\end{multlined}
\end{equation*}
Note that $\varpi\Pi^0/\varpi^j \Pi^0\simeq \Pi^0/\varpi^{j-1} \Pi^0$. Then (ii) follows from  \eqref{eq::in_vanishing_Ext_L_2} by d\'evissage.
\end{proof}

\begin{proposition} \label{prop::vanishing_higher_Ext_L_1}
Let $j\leq n$ be two positive integers. 
\begin{itemize}
\item[(i)] We have an isomorphism
\[
{\rm Ext}_{\op{Mod}^{\rm sm}_{L^{\times},\zeta}(\mc{O}/\varpi^n)}^i (\eta^0 /\varpi^n, \Pi^0/\varpi^j \Pi^0) \simeq  {\rm Ext}_{\op{Mod}^{\rm sm}_{L^{\times},\zeta}(\mc{O}/\varpi^j)}^i (\eta^0 /\varpi^j, \Pi^0/\varpi^j \Pi^0 ), ~\forall i .
\]
\item[(ii)] If $\eta$ is nonexceptional for $\Pi$ in the sense of Definition \ref{def::generic_Banach_L}, then 
\[
\op{length}_{\mc{O}} {\rm Ext}_{\op{Mod}^{\rm sm}_{L^{\times},\zeta}(\mc{O}/\varpi^j)}^i (\eta^0 /\varpi^j, \Pi^0/\varpi^j \Pi^0) = \left\{\begin{array}{cl}
        m_{\Pi}j ,& \text{ if } i = 0,\\
        0,& \text{ otherwise}. 
        \end{array}\right.
\]
\end{itemize}
\end{proposition}

\begin{proof}
The proof is similar to that of Proposition \ref{prop::vanishing_higher_Ext_L_2}. To prove (i), it suffices to replace Lemma \ref{lemma::Group_Ext_3} by Lemma \ref{lemma::Group_Ext_2}. For (ii), by filtering $\Pi^0/\varpi^j \Pi^0$ we are reduced to showing that if $\eta$ is nonexceptional for $\Pi$, then
\begin{equation*} 
 {\rm Ext}_{\op{Mod}^{\rm sm}_{L^{\times},\zeta}(\mc{O}/\varpi^j)}^i (\eta^0/\varpi^j,  \Pi^0/\varpi \Pi^0)  \simeq  {\rm Ext}_{\op{Mod}^{\rm sm}_{L^{\times},\zeta}(\FF)}^i ( \eta^0/\varpi,\Pi^0/\varpi \Pi^0) 
\end{equation*}
is zero if $i\neq 0$ and has dimension $m_{\Pi}$ if $i=0$. This follows from Proposition \ref{prop:dim_reduction_L}.
\end{proof}

Set
\[
M_n: = {\rm Ext}_{\op{Mod}^{\rm sm}_{L^{\times},\zeta}(\mc{O}/\varpi^n)}^d ( \Pi^0/ \varpi^{n} \Pi^0,  \eta^0/\varpi^n  ),\quad M'_n: = {\rm Hom}_{L^{\times}}(\eta^0/\varpi^n , \Pi^0/ \varpi^{n} \Pi^0).
\]

\begin{proposition} \label{prop::flatness_for_L}
Assume $\eta$ is nonexceptional for $\Pi$. Then $\{M_n\}_{n\geq 1}$ and $\{M'_n\}_{n\geq 1}$ form projective systems satisfying the conditions in Corollary \ref{lemma::CA_1}.
\end{proposition}
\begin{proof}
By Proposition \ref{prop::vanishing_higher_Ext_L_2} and Proposition \ref{prop::vanishing_higher_Ext_L_1},
\[
\op{length}_{\mc{O}}  M_n = \op{length}_{\mc{O}} M_n' = m_{\Pi}n.
\]
Thus it suffices to check there are natural  transition maps, which give isomorphisms 
\[
M_n \otimes_{\mc{O}} \mc{O} /\varpi^{n-1} \simeq M_{n-1},\quad M'_n \otimes_{\mc{O}} \mc{O} /\varpi^{n-1} \simeq M'_{n-1}.
\]

Let $0<j <n$. We consider the following commutative diagram with exact rows and exact columns
\begin{equation} \label{eq::diag_L_1}
\xymatrix@C=2em{
 &  0 \ar[d] &   \\
    & \varpi^{n-j} \Pi^0/ \varpi^n \Pi^0 \ar[d] &   \\
       &  \Pi^0/ \varpi^n \Pi^0 \ar[rd]^{\varpi^j}\ar[d]^{\varpi^j}&   \\
 0 \ar[r] &   \varpi^{j} \Pi^0/ \varpi^{n} \Pi^0 \ar[r] \ar[d] &    \Pi^0/ \varpi^{n} \Pi^0   \ar[r]  &  \Pi^0/ \varpi^{j} \Pi^0  \ar[r] &0\\
 & 0 & }.
\end{equation}

By Proposition \ref{prop::vanishing_higher_Ext_L_2}, we have for $i\neq d$
\[
{\rm Ext}_{L^{\times},\zeta}^i (  \varpi^j \Pi^0/\varpi^n \Pi^0, \eta^0/\varpi^n) = {\rm Ext}_{L^{\times},\zeta}^i (  \varpi^{n-j} \Pi^0/\varpi^n \Pi^0, \eta^0/\varpi^n) = {\rm Ext}_{L^{\times},\zeta}^i (  \Pi^0/\varpi^j \Pi^0, \eta^0 /\varpi^n) =0 ,
\]
where ${\rm Ext}_{L^{\times},\zeta}^i $ stands for ${\rm Ext}_{\op{Mod}^{\rm sm}_{L^{\times},\zeta}(\mc{O}/\varpi^n)}^i $.
We deduce from  \eqref{eq::diag_L_1} the following diagram with exact rows and exact columns 
\begin{equation*} 
\xymatrix@C=1em{
        & &  & 0 \ar[d] &   \\
 0 \ar[r] &  {\rm Ext}_{L^{\times},\zeta}^d ( \Pi^0/ \varpi^{j} \Pi^0,  \eta^0/\varpi^n  )  \ar[r] &  {\rm Ext}_{L^{\times},\zeta}^d (   \Pi^0/ \varpi^{n} \Pi^0, \eta^0 /\varpi^n  )  \ar[rd]^{\varpi^j }    \ar[r]  &  {\rm Ext}_{L^{\times},\zeta}^d (  \varpi^j \Pi^0/ \varpi^{n} \Pi^0 , \eta^0/\varpi^n  )  \ar[r] \ar[d]^{\varpi^j} &0 \\ 
 &   & & {\rm Ext}_{L^{\times},\zeta}^d (   \Pi^0/ \varpi^{n} \Pi^0, \eta^0 /\varpi^n  )  \ar[d]  & \\ 
 &   & & {\rm Ext}_{L^{\times},\zeta}^d (  \varpi^{n-j} \Pi^0/ \varpi^{n} \Pi^0, \eta^0 /\varpi^n  )  \ar[d]  & \\ 
 &   & & 0  &  }.
\end{equation*}
Then 
\[
M_n /\varpi^j \simeq {\rm Ext}_{L^{\times},\zeta}^d (  \varpi^{n-j} \Pi^0/ \varpi^{n} \Pi^0, \eta^0 /\varpi^n  ) \simeq {\rm Ext}_{L^{\times},\zeta}^d (   \Pi^0/ \varpi^{j} \Pi^0, \eta^0 /\varpi^n  ) \simeq M_j.
\]
This proves the assertion for $M_n$. To prove the assertion for $M'_n$, by Proposition \ref{prop::vanishing_higher_Ext_L_1} we have
\[
{\rm Ext}_{\op{Mod}^{\rm sm}_{L^{\times},\zeta}(\mc{O}/\varpi^n)}^1 (\eta^0/\varpi^n,  \varpi^j \Pi^0/\varpi^n \Pi^0) = {\rm Ext}_{\op{Mod}^{\rm sm}_{L^{\times},\zeta}(\mc{O}/\varpi^n)}^1 (\eta^0 /\varpi^n,  \varpi^{n-j}\Pi^0/\varpi^n \Pi^0) =0 .
\]
We deduce from  \eqref{eq::diag_L_1} the following diagram with exact rows and exact columns 
\begin{equation*} 
\xymatrix@C=1em{
    &  0 \ar[d] &   \\
    &  {\rm Hom}_{L^{\times}} ( \eta^0 /\varpi^n , \varpi^{n-j} \Pi^0/ \varpi^n \Pi^0 ) \ar[d] &   \\
     & {\rm Hom}_{L^{\times}} ( \eta^0 /\varpi^n , \Pi^0/ \varpi^n \Pi^0 )  \ar[rd]^{\varpi^j }\ar[d]^{\varpi^j}&   \\
 0 \ar[r] &  {\rm Hom}_{L^{\times}} ( \eta^0/\varpi^n , \varpi^{j} \Pi^0/ \varpi^{n} \Pi^0 )  \ar[r] \ar[d]&  {\rm Hom}_{L^{\times}} ( \eta^0 /\varpi^n ,  \Pi^0/ \varpi^{n} \Pi^0 )      \ar[r]  &  {\rm Hom}_{L^{\times}} ( \eta^0/\varpi^n ,  \Pi^0/ \varpi^{j} \Pi^0 )  \ar[r] &0 \\ 
 & 0 &}.
\end{equation*}
Then
\[
M'_n /\varpi^j \simeq {\rm Hom}_{L^{\times}} ( \eta^0/\varpi^n ,  \Pi^0/ \varpi^{j} \Pi^0 ) \simeq M'_j.
\]
\end{proof}

\begin{corollary} \label{cor::limit_free_L}
Assume $\eta$ is nonexceptional for $\Pi$. 
\begin{enumerate}
 \item[(i)] $ \plim_n {\rm Ext}_{\op{Mod}^{\rm sm}_{L^{\times},\zeta}(\mc{O}/\varpi^n)}^i ( \Pi^0/ \varpi^{n} \Pi^0,  \eta^0/\varpi^n  )$ is zero if $i\neq d$; if $i=d$, it is a free $\mc{O}$-module of rank $m_{\Pi}$.
\item[(ii)] $ \plim_n {\rm Ext}_{\op{Mod}^{\rm sm}_{L^{\times},\zeta}(\mc{O}/\varpi^n)}^i( \eta^0/\varpi^n  , \Pi^0/ \varpi^{n} \Pi^0) $ is zero if $i\neq 0$; if $i=0$, it is a free $\mc{O}$-module of rank $m_{\Pi}$.
\end{enumerate}
\end{corollary}
\begin{proof}
Both (i) and (ii) follow from Corollary \ref{lemma::CA_1}, Propositions \ref{prop::flatness_for_L}, \ref{prop::vanishing_higher_Ext_L_2}(ii) and \ref{prop::vanishing_higher_Ext_L_1}(ii).
\end{proof}

The following theorem is an immediate consequence of Corollary \ref{cor::limit_free_L}.

\begin{theorem} \label{thm::Bananch_nonsplit_tori_1}
Let $\eta: L^{\times} \to  E^{\times}$ be a continuous character such that $\eta|_Z = \zeta$. Assume $\eta$ is nonexceptional for $\Pi$. 
\[
\dim_E \big( \plim_n {\rm Ext}_{\op{Mod}^{\rm sm}_{ L^{\times},\zeta}(\mc{O}/\varpi^n)}^i ( \Pi^0/ \varpi^{n} \Pi^0,  \eta^0/\varpi^n  ) \big)[1/p]  = \left\{\begin{array}{cl}
        m_{\Pi} ,& \text{ if } i = d,\\
        0,& \text{ else },
        \end{array}\right. 
\]
and 
\[
\dim_E  \big(\plim_n {\rm Ext}_{\op{Mod}^{\rm sm}_{ L^{\times},\zeta}(\mc{O}/\varpi^n)}^i( \eta^0/\varpi^n  , \Pi^0/ \varpi^{n} \Pi^0)\big)[1/p]   = \left\{\begin{array}{cl}
        m_{\Pi} ,& \text{ if } i = 0,\\
        0,& \text{ else }.
        \end{array}\right..
\]
\end{theorem}

In the following, we interpret Theorem \ref{thm::Bananch_nonsplit_tori_1} as results on extension groups in the abelian category ${\rm Ban}_{L^{\times},\zeta}^{\rm adm}(E)$.

\begin{lemma} \label{lemma::adm_L}
The Banach space representation $\Pi|_{L^{\times}}$ is an admissible representation of $L^{\times}$.
\end{lemma}
\begin{proof}
It suffices to show that $\Pi^0/\varpi \Pi^0$ is an admissible representation of $L^{\times}$. The character $\zeta^0/\varpi:Z\to \FF^{\times}$ is the central character of  $\Pi^0/\varpi \Pi^0$. We extend $\zeta^0/\varpi$ to a character of $(1+\varpi_L \mc{O}_L)Z$, still denoted by $\zeta^0/\varpi$, by letting $1+\varpi_L \mc{O}_L$ act trivially. It suffices to show the $\FF$-vector space
\[
\op{Hom}_{1+\varpi_L \mc{O}_L} (\FF, \Pi^0/\varpi \Pi^0) \simeq \op{Hom}_{(1+\varpi_L \mc{O}_L)Z } (\zeta^0/\varpi, \Pi^0/\varpi \Pi^0) 
\]
is finite-dimensional. By Frobenius reciprocity,
\[
 \op{Hom}_{(1+\varpi_L \mc{O}_L)Z } (\zeta^0/\varpi, \Pi^0/\varpi \Pi^0)  \simeq  \op{Hom}_{L^{\times} } ( \op{c-Ind}_{(1+\varpi_L \mc{O}_L)Z}^{L^{\times}}\zeta^0/\varpi, \Pi^0/\varpi \Pi^0).
\]
Since $(1+\varpi_L \mc{O}_L)Z $ has finite index in $L^{\times}$, the induced representation $\op{c-Ind}_{(1+\varpi_L \mc{O}_L)Z}^{L^{\times}}\zeta^0/\varpi$ is an extension of finitely many smooth characters $\overline{\eta}$ of $L^{\times}$ satisfying $\overline{\eta}|_Z = \zeta^0/\varpi$. We are then reduced to show that for any such $\overline{\eta}$, the $\FF$-vector space $ \op{Hom}_{L^{\times} } ( \overline{\eta}, \Pi^0/\varpi \Pi^0)$ is finite-dimensional. If $(\Pi^0/\varpi \Pi^0)^{\rm ss}$ does not contain any special series, this is a consequence of Proposition \ref{prop:dim_reduction_L} because $\overline{\eta}$ is nonexceptional; if $(\Pi^0/\varpi \Pi^0)^{\rm ss}$ contains a special series, say $ \op{St}\otimes \chi_1\circ\op{det}$ for some mod $p$ smooth character $\chi_1$, the short exact sequence \eqref{eq::twist_St_2}  implies that the $\FF$-vector space $ \op{Hom}_{L^{\times} } ( \overline{\eta}, \op{St}\otimes \chi_1\circ\op{det})$ is finite-dimensional. 
\end{proof}

\begin{proposition} \label{prop::ext_banach_L}
Assume $\eta$ is nonexceptional for $\Pi$. Then
\begin{enumerate}
\item[(i)] ${\rm Ext}_{\op{Mod}^{\rm pro}_{L^{\times},\zeta}(\mc{O})}^i (   (\eta^0)^d, (\Pi^0)^d ) \simeq \plim_n {\rm Ext}_{\op{Mod}^{\rm sm}_{L^{\times},\zeta}(\mc{O}/\varpi^n)}^i ( \Pi^0/ \varpi^{n} \Pi^0,  \eta^0/\varpi^n  )$.

\item[(ii)] ${\rm Ext}_{\op{Mod}^{\rm pro}_{L^{\times},\zeta}(\mc{O})}^i (  (\Pi^0)^d,   (\eta^0)^d ) \simeq \plim_n {\rm Ext}_{\op{Mod}^{\rm sm}_{L^{\times},\zeta}(\mc{O}/\varpi^n)}^i (  \eta^0/\varpi^n ,\Pi^0/ \varpi^{n} \Pi^0)$.
\end{enumerate}
\end{proposition}

\begin{proof}
The following proof is inspired by the proof of \cite[Prop. 6.9]{Fust}. We choose a projective resolution $P_{\bullet}$ of $ (\eta^0)^d $ in $\op{Mod}^{\rm pro}_{L^{\times},\zeta}(\mc{O})$. Let $C^i_n : = {\rm Hom}_{L^{\times}}(P_i,  (\Pi^0)^d/\varpi^n)$. Then $\cdots \to C^{\bullet}_1 \to C^{\bullet}_0$ is a tower of cochain complex satisfying Mittag--Leffler condition, because $P_i$ is projective. By the variant of \cite[Thm. 3.5.8]{Weibel}, there is short exact sequence
\begin{equation*}
\begin{multlined}
0 \to \plim\nolimits_n^1 {\rm Ext}_{\op{Mod}^{\rm pro}_{L^{\times},\zeta}(\mc{O})}^{i-1} (   (\eta^0)^d, (\Pi^0)^d /\varpi^n)   \to {\rm Ext}_{\op{Mod}^{\rm pro}_{L^{\times},\zeta}(\mc{O})}^i (   (\eta^0)^d, (\Pi^0)^d ) \\  \to \plim\nolimits_n {\rm Ext}_{\op{Mod}^{\rm pro}_{L^{\times},\zeta}(\mc{O})}^i (   (\eta^0)^d, (\Pi^0)^d /\varpi^n ) \to 0.
\end{multlined}
\end{equation*}
By Lemma \ref{lemma::Group_Ext_3},
\begin{align*}
{\rm Ext}_{\op{Mod}^{\rm pro}_{L^{\times},\zeta}(\mc{O})}^{i} (   (\eta^0)^d, (\Pi^0)^d /\varpi^n)  & \simeq {\rm Ext}_{\op{Mod}^{\rm pro}_{L^{\times},\zeta}(\mc{O}/\varpi^n)}^{i} (   (\eta^0)^d/\varpi^n, (\Pi^0)^d /\varpi^n) \\
&\simeq {\rm Ext}_{\op{Mod}^{\rm sm}_{L^{\times},\zeta}(\mc{O}/\varpi^n)}^{i} (   ((\Pi^0)^d /\varpi^n)^{\vee}, ((\eta^0)^d/\varpi^n)^{\vee}) \\
&\simeq {\rm Ext}_{\op{Mod}^{\rm sm}_{L^{\times},\zeta}(\mc{O}/\varpi^n)}^{i} (   \Pi^0 /\varpi^n, \eta^0/\varpi^n),
\end{align*}
where the last isomorphism follows from \cite[(4)]{Paskunas-JL}. 

The projective system $ \{A_n\}_n$ with 
\[
A_n: = {\rm Ext}_{\op{Mod}^{\rm pro}_{L^{\times},\zeta}(\mc{O})}^{i-1} (   (\eta^0)^d, (\Pi^0)^d /\varpi^n) (\simeq {\rm Ext}_{\op{Mod}^{\rm sm}_{L^{\times},\zeta}(\mc{O}/\varpi^n)}^{i-1} (   \Pi^0 /\varpi^n, \eta^0/\varpi^n) )
\]
is a projective system of finite $\mc{O}$-modules by Proposition \ref{prop::vanishing_higher_Ext_L_2}(ii). Then by \cite[Thm. 1]{Jensen}, $  \plim\nolimits_n^1 A_n = 0$. Therefore,
\begin{align*}
{\rm Ext}_{\op{Mod}^{\rm pro}_{L^{\times},\zeta}(\mc{O})}^i (   (\eta^0)^d, (\Pi^0)^d ) & \simeq \plim\nolimits_n {\rm Ext}_{\op{Mod}^{\rm pro}_{L^{\times},\zeta}(\mc{O})}^i (   (\eta^0)^d, (\Pi^0)^d /\varpi^n )\\
&\simeq \plim\nolimits_n {\rm Ext}_{\op{Mod}^{\rm sm}_{L^{\times},\zeta}(\mc{O}/\varpi^n)}^{i} (   \Pi^0 /\varpi^n, \eta^0/\varpi^n).
\end{align*}
This gives (i). The proof of (ii) is similar, using Proposition \ref{prop::vanishing_higher_Ext_L_1}(ii) instead of Proposition \ref{prop::vanishing_higher_Ext_L_2}(ii). 
\end{proof}

\begin{theorem} \label{thm::Bananch_nonsplit_tori}
Under the same condition as in Theorem \ref{thm::Bananch_nonsplit_tori_1}, we have isomorphisms
\[
{\rm Ext}_{{\rm Ban}_{L^{\times},\zeta}^{\rm adm}(E)}^i (\Pi, \eta) \simeq \big( \plim_n {\rm Ext}_{\op{Mod}^{\rm sm}_{L^{\times},\zeta}(\mc{O}/\varpi^n)}^i ( \Pi^0/ \varpi^{n} \Pi^0,  \eta^0/\varpi^n  ) \big)[1/p],
\]
and
\[
{\rm Ext}_{{\rm Ban}_{L^{\times},\zeta}^{\rm adm}(E)}^i (\eta, \Pi) \simeq \big( \plim_n {\rm Ext}_{\op{Mod}^{\rm sm}_{L^{\times},\zeta}(\mc{O}/\varpi^n)}^i ( \eta^0/\varpi^n , \Pi^0/ \varpi^{n} \Pi^0 ) \big)[1/p].
\]
\end{theorem}
\begin{proof}
By Lemma \ref{lemma::ext_banach_module} and Proposition \ref{prop::ext_banach_L}(i)
\begin{align*}
{\rm Ext}_{{\rm Ban}_{L^{\times},\zeta}^{\rm adm}(E)}^i (\Pi, \eta) &\simeq {\rm Ext}_{\op{Mod}^{\rm pro}_{L^{\times},\zeta}(\mc{O})}^i ((\eta^0)^d, (\Pi^0)^d)[1/p] \\
& \simeq \big( \plim_n {\rm Ext}_{\op{Mod}^{\rm sm}_{L^{\times},\zeta}(\mc{O}/\varpi^n)}^i ( \Pi^0/ \varpi^{n} \Pi^0,  \eta^0/\varpi^n  ) \big)[1/p].
\end{align*}
The proof for ${\rm Ext}_{{\rm Ban}_{L^{\times},\zeta}^{\rm adm}(E)}^i (\eta,\Pi)$ is similar; it suffices to replace Proposition \ref{prop::ext_banach_L}(i) by Proposition \ref{prop::ext_banach_L}(ii).
\end{proof}

\subsection{The case of split tori} 
Let $\eta: T \to  E^{\times}$ be a unitary character such that  $\eta|_Z = \zeta$. Let $\eta^0 :  T \to \mc{O}^{\times}$ be the unique invariant $\mc{O}$-lattice of $\eta$ such that $\eta^0|_Z = \zeta^0$.

We first prove an analogue of Proposition \ref{prop::vanishing_higher_Ext_L_2}. 

\begin{proposition} \label{prop::vanishing_higher_Ext_T_1}
Let $j\leq n$ be two positive integers. 
\begin{itemize}
\item[(i)] We have an isomorphism
\[
{\rm Ext}_{\op{Mod}^{\rm sm}_{T,\zeta}(\mc{O}/\varpi^n)}^i ( \Pi^0/\varpi^j \Pi^0, \eta^0 /\varpi^n) \simeq  {\rm Ext}_{\op{Mod}^{\rm sm}_{T,\zeta}(\mc{O}/\varpi^j)}^i ( \Pi^0/\varpi^j \Pi^0, \eta^0 /\varpi^j), ~\forall i .
\]
\item[(ii)] If $\eta$ is nonexceptional for $\Pi$ in the sense of Definition \ref{def::generic_Banach_L}, then 
\[
\op{length}_{\mc{O}} {\rm Ext}_{\op{Mod}^{\rm sm}_{T,\zeta}(\mc{O}/\varpi^j)}^i ( \Pi^0/\varpi^j \Pi^0, \eta^0 /\varpi^j) = \left\{\begin{array}{cl}
        m_{\Pi}j ,& \text{ if } i = d,\\
        0,& \text{ otherwise}. 
        \end{array}\right.
\]
\end{itemize}
\end{proposition}

\begin{proof}
(i) Twisting by the rank one $\mc{O}/\varpi^n$-module $(\eta^0)^{-1}/\varpi^n$, we have isomorphisms
\begin{align*}
{\rm Ext}_{\op{Mod}^{\rm sm}_{T,\zeta}(\mc{O}/\varpi^n)}^i ( \Pi^0/\varpi^j \Pi^0,\eta^0/\varpi^n) & \simeq {\rm Ext}_{\op{Mod}^{\rm sm}_{T/Z}(\mc{O}/\varpi^n)}^i ( \Pi^0/\varpi^j \Pi^0 \otimes (\eta^0)^{-1}/\varpi^n,\mc{O}/\varpi^n ) \\
& \simeq {\rm Ext}_{\op{Mod}^{\rm sm}_{T/Z}(\mc{O}/\varpi^n)}^i ( \Pi^0/\varpi^j \Pi^0 \otimes (\eta^0)^{-1}/\varpi^j,\mc{O}/\varpi^n).
\end{align*}

Let $V : = \Pi^0/\varpi^j \Pi^0 \otimes (\eta^0)^{-1}/\varpi^j \in \op{Mod}^{\rm sm}_{T/Z}(\mc{O}/\varpi^n)$. Let $T_0 : = T(\mc{O}_F)Z$. Consider the Hochschild--Serre spectral sequence 
\[
E_2^{p,q} =H^p ( T/T_0 ,  {\rm Ext}^q_{\op{Mod}^{\rm sm}_{T_0 /Z }(\mc{O}/\varpi^n)} ( V  , \mc{O}/\varpi^n) ) \Rightarrow {\rm Ext}_{\op{Mod}^{\rm sm}_{T/Z}(\mc{O}/\varpi^n)}^{p+q} ( V,\mc{O}/\varpi^n).
\]
Since $T_0/Z$ is profinite, Lemma \ref{lemma::module_ext_2} gives an isomorphism
\[
{\rm Ext}^q_{\op{Mod}^{\rm sm}_{T_0 /Z }(\mc{O}/\varpi^n)} ( V  , \mc{O}/\varpi^n)  \simeq {\rm Ext}^q_{\op{Mod}^{\rm sm}_{T_0 /Z }(\mc{O}/\varpi^j)} ( V  , \mc{O}/\varpi^j).  
\]
We obtain the desired isomorphism by applying the Hochschild--Serre spectral sequence again.

(ii) can be proved in the same way as Proposition \ref{prop::vanishing_higher_Ext_L_2}(ii), by using Proposition \ref{prop:dim_reduction_T} instead of Proposition \ref{prop:dim_reduction_L}.
\end{proof}

\begin{proposition} \label{prop::vanishing_higher_Ext_T_2}
Let $j\leq n$ be two positive integers. 
\begin{itemize}
\item[(i)] We have an isomorphism
\[
{\rm Ext}_{\op{Mod}^{\rm sm}_{T,\zeta}(\mc{O}/\varpi^n)}^i (\eta^0 /\varpi^n, \Pi^0/\varpi^j \Pi^0) \simeq  {\rm Ext}_{\op{Mod}^{\rm sm}_{T,\zeta}(\mc{O}/\varpi^j)}^i (\eta^0 /\varpi^j, \Pi^0/\varpi^j \Pi^0 ), ~\forall i .
\]
\item[(ii)] If $\eta$ is nonexceptional for $\Pi$ in the sense of Definition \ref{def::generic_Banach_L}, then 
\[
\op{length}_{\mc{O}}{\rm Ext}_{\op{Mod}^{\rm sm}_{T,\zeta}(\mc{O}/\varpi^j)}^i (\eta^0 /\varpi^j, \Pi^0/\varpi^j \Pi^0) = \left\{\begin{array}{cl}
        m_{\Pi}j ,& \text{ if } i = 1,\\
        0,& \text{ otherwise}. 
        \end{array}\right.
\]
\end{itemize}
\end{proposition}
\begin{proof}
The proof is similar to that of Proposition \ref{prop::vanishing_higher_Ext_L_1}.
\end{proof}

Set
\[
M_n: = {\rm Ext}_{\op{Mod}^{\rm sm}_{T,\zeta}(\mc{O}/\varpi^n)}^d ( \Pi^0/ \varpi^{n} \Pi^0,  \eta^0/\varpi^n  ),\quad M'_n: = {\rm Ext}_{\op{Mod}^{\rm sm}_{T,\zeta}(\mc{O}/\varpi^n)}^1(\eta^0/\varpi^n , \Pi^0/ \varpi^{n} \Pi^0).
\]

\begin{proposition} \label{prop::flatness_for_T}
Assume $\eta$ is nonexceptional for $\Pi$. Then $\{M_n\}_{n\geq 1}$ and $\{M'_n\}_{n\geq 1}$ form projective systems  satisfying the conditions in Corollary \ref{lemma::CA_1}.
\end{proposition}
\begin{proof}
The proof is similar to that of Proposition \ref{prop::flatness_for_L}.
\end{proof}

\begin{corollary} \label{cor::limit_free_T}
Assume $\eta$ is nonexceptional for $\Pi$. 
\begin{enumerate}
 \item[(i)] $ \plim_n {\rm Ext}_{\op{Mod}^{\rm sm}_{T,\zeta}(\mc{O}/\varpi^n)}^i ( \Pi^0/ \varpi^{n} \Pi^0,  \eta^0/\varpi^n  )$ is zero if $i\neq d$; if $i=d$, it is a free $\mc{O}$-module of rank $m_{\Pi}$.
\item[(ii)] $ \plim_n {\rm Ext}_{\op{Mod}^{\rm sm}_{T,\zeta}(\mc{O}/\varpi^n)}^i( \eta^0/\varpi^n  , \Pi^0/ \varpi^{n} \Pi^0) $ is zero if $i\neq 1$; if $i=1$, it is a free $\mc{O}$-module of rank $m_{\Pi}$.
\end{enumerate}
\end{corollary}

\begin{proof}
This follows from Corollary \ref{lemma::CA_1}, Propositions \ref{prop::flatness_for_T},  \ref{prop::vanishing_higher_Ext_T_1}(ii) and \ref{prop::vanishing_higher_Ext_T_2}(ii).
\end{proof}

The following theorem is an immediate consequence of Corollary \ref{cor::limit_free_T}.
\begin{theorem} \label{thm::Bananch_split_tori}
Let $\eta: T \to  E^{\times}$ be a unitary character such that $\eta|_Z = \zeta$. Assume $\eta$ is nonexceptional for $\Pi$. 
\[
\dim_E \big( \plim_n {\rm Ext}_{\op{Mod}^{\rm sm}_{ T,\zeta}(\mc{O}/\varpi^n)}^i ( \Pi^0/ \varpi^{n} \Pi^0,  \eta^0/\varpi^n  ) \big)[1/p]  = \left\{\begin{array}{cl}
        m_{\Pi} ,& \text{ if } i = d,\\
        0,& \text{ else },
        \end{array}\right. 
\]
and 
\[
\dim_E  \big(\plim_n {\rm Ext}_{\op{Mod}^{\rm sm}_{T,\zeta}(\mc{O}/\varpi^n)}^i( \eta^0/\varpi^n  , \Pi^0/ \varpi^{n} \Pi^0)\big)[1/p]   = \left\{\begin{array}{cl}
        m_{\Pi} ,& \text{ if } i = 1,\\
        0,& \text{ else }.
        \end{array}\right..
\]
\end{theorem}

\begin{remark}
Unlike the nonsplit-torus case, restriction of $\Pi$ to the diagonal torus $T$ does not in general preserve admissibility; therefore we do not interpret Theorem \ref{thm::Bananch_split_tori} as an Ext-calculation in $\operatorname{Ban}^{\mathrm{adm}}_T(E)$.
\end{remark}